\documentclass[9pt]{amsart}

\usepackage{amssymb}
\usepackage{amsthm}
\usepackage{amsmath}
 \numberwithin{equation}{section} 
 \usepackage{amssymb} 

\usepackage{booktabs}

\usepackage{enumerate}

\usepackage{algorithm,algpseudocode}

\usepackage{xy}

\usepackage{xypic}

\usepackage{tikz}

\usetikzlibrary{trees,arrows}

\usepackage[linguistics]{forest}

\newtheorem{theorem*}{Theorem}
\newtheorem{conjecture*}[theorem*]{Conjecture}

\newtheorem{theorem}{Theorem}[section]

\newtheorem{corollary}[theorem]{Corollary}

\newtheorem{conjecture}[theorem]{Conjecture}

\theoremstyle{remark}
\newtheorem{remarks}[theorem]{Remarks}

\renewcommand{\mod}{\operatorname{mod}}

\newcommand{\rad}{\operatorname{rad}}

\newcommand{\id}{\operatorname{id}}

\newcommand{\charact}{\operatorname{char}}

\newcommand{\bA}{\mathbb{A}}

\newcommand{\bD}{\mathbb{D}}
\newcommand{\bE}{\mathbb{E}}
\newcommand{\bF}{\mathbb{F}}

\newcommand{\bL}{\mathbb{L}}

\newcommand{\EE}{\mathbb{E}}

\newcommand{\calD}{\mathcal{D}}
\newcommand{\calI}{\mathcal{I}}

\begin{document}

\title{An algorithm for  the periodicity of  deformed preprojective algebras of Dynkin types $\mathbb{E}_6$, $\mathbb{E}_7$ and $\mathbb{E}_8$}         


{
\def\thefootnote{}

\footnote{This research was supported by
	the Research Grant
	DEC-2011/02/A/ST1/00216 of the Polish National Science Center.}
}


\author[J. Białkowski]{Jerzy Bia\l kowski}

\address[Jerzy Bia\l kowski]{Faculty of Mathematics and Computer Science,
	Nicolaus Copernicus University,
	Chopina~12/18,
	87-100 Toru\'n,
	Poland}

\begin{abstract}
	We construct a numeric algorithm for completing the proof of a conjecture 
	asserting that
	all deformed preprojective algebras of generalized Dynkin
	type are periodic. 
	In particular,  we obtain an  algorithmic   
	procedure showing that  
	non-trivial
	deformed preprojective algebras 
	of Dynkin types $\mathbb{E}_7$ and $\mathbb{E}_8$
	exist only in characteristic 2. 
	As a consequence, we show that   
	deformed preprojective algebras of Dynkin types 
	$\mathbb{E}_6$, $\mathbb{E}_7$ and $\mathbb{E}_8$
	are periodic and 
	we obtain an  algorithm  for a  classification 
	of such algebras, up to algebra isomorphism.  
	We do it  by a  reduction of the conjecture 
	to a solution of   a system of equations associated 
	with the problem of the existence of a suitable  algebra 
	isomorphism  $\varphi_f:  P^f(\bE_n)  \to P(\bE_n)$
	described in 
	Theorem~\ref{th:hom1}. 
	One also shows that our   algorithmic approach to the conjecture is also applicable to 
	the classification of the mesh algebras of generalized Dynkin type.	

\bigskip



\end{abstract}


\dedicatory{\bfseries Dedicated to the memory of Professor Andrzej Skowro\'nski (1950-2020)}

\maketitle

\section{Introduction and main results}
\label{sec:intro}

In this paper, by $K$ we denote a fixed algebraically closed field.
Moreover, by an algebra we mean 
a finite-dimensional associative basic 
connected $K$-algebra with  identity, 
if not stated otherwise.

We recall from \cite{ESk} that the deformed preprojective algebras of generalized Dynkin type
$\bA_n (n \geq 2)$, 
$\bD_n (n \geq 4)$, 
$\bE_6$, $\bE_7$, $\bE_8$, 
$\bL_n (n \geq 1)$
(defined later in this section) were introduced 
as a new class of periodic
$K$-algebras.
Unfortunately, the original proof of their periodicity had a gap, see   
\cite{B:E6,B:soc}.  
However, 
 this proof was correct 
 for deformed preprojective $K$-algebras of generalized Dynkin
type which are not isomorphic to preprojective $K$-algebras of generalized Dynkin, but 
only in the case of $K$ being of positive characteristic. One of the main aims of this article is to complete the proof by showing  that deformed preprojective $K$-algebras of generalized Dynkin
type 
 that  are  not isomorphic to preprojective $K$-algebras of generalized Dynkin
type exist only for $K$ of positive characteristic, see also 
Conjecture~\ref{conj:1.6}, implicitly  stated in \cite{ESk}.

We do it by applying  the numeric   algorithm   presented in 
Section~\ref{sec:outline}.
On this way we obtain an affirmative solution of 
Conjecture~\ref{conj:1.6}
in a simple combinatorial data   by showing that  
non-trivial
deformed preprojective algebras 
of Dynkin types $\mathbb{E}_7$ and $\mathbb{E}_8$
exist only in characteristic 2. 
As a consequence, we show that   
deformed preprojective algebras of Dynkin types 
$\mathbb{E}_6$, $\mathbb{E}_7$ and $\mathbb{E}_8$
are periodic and 
we obtain an  algorithm  for a  classification of such algebras, 
up to algebra isomorphism.  
We do it  by a  reduction of the     conjecture to a solution of   a system of equations
associated with the problem of the existence of a suitable  
$K$-algebra isomorphism of the form  $\varphi_f:  P^f(\bE_n)  \to P(\bE_n)$
described in 
Theorem~\ref{th:hom1}, 
where  
$P(\bE_n)$ is the preprojective algebra 
\eqref{eq:1.1} 
of Dynkin type $\bE_n$ and $P^f(\bE_n) $ is   the deformed  preprojective algebra 
\eqref{eq:1.3} 
of Dynkin type $\bE_n$. 
 On the other hand,  our   algorithmic approach to the conjecture is also applicable to 
 the classification of the mesh algebras of generalized Dynkin type.  
 The solution we obtain here shows an essential application  of the computer algebra technique and computer computations  in solving difficult and important theoretical  problems of high complexity of modern representation theory.
 
We recall from \cite{ESk} and \cite{ASS,SS1,SS2,SY} 
that the preprojective algebras are playing a crucial role 
in representation theory of finite-dimensional algebras $R$,  
of their derived categories ${\calD}^b(\mod R)$, 
of their representation types, 
and in  the Auslander-Reiten theory. 
Also they play an important role in a categorical study 
of isolated  surface singularities of finite type and tame type. 
Our results of the paper can be viewed as a highly non-trivial  
application of the discrete mathematics technique, 
computer algebra technique,  
and computer calculation in the modern representation theory 
and  the Auslander-Reiten theory that were 
successfully developed  during last fifty years.

 On the other hand, we feel that our  computational algorithmic approach  leading to  an affirmative solution of 
 Conjecture~\ref{conj:1.6}
 can be also applied in: (i) the  Coxeter-type classification of $\Phi$-mesh root systems in the sense of \cite{S1,S2}, (ii) the   Grothendieck group recognition for derived categories ${\calD}^b(\mod R)$ of a module category $\mod R$, (iii) the derived category  study  of the  hypersurface singularities  $f=  x_1^a+x_2^b+x_3 ^a$ classified recently in \cite{S4,S5},
in particular the tame singularities such that their orbifold characteristic $\chi_{(a,b,c)} = \frac{1}{a} + \frac{1}{b} +\frac{1}{c}$  is zero, and (iv) in the    Coxeter spectral classification of connected  symmetrizable integer Cartan matrices of Dynkin types  $ \mathbb{E}_6$, $\mathbb{E}_7$ and $\mathbb{E}_8$ and their generalized mesh root systems recently studied in \cite{S3}.

To introduce the reader to the problems we study and to outline  the contents of this paper, we recall from \cite{ESk} 
that, by   the classification of 
deformed preprojective algebras of generalized Dynkin type 
$\bL_n$ 
  for an algebraically closed field $K$ of characteristic different from $2$,  each of the deformed preprojective $K$-algebras of type $\bL_n$ 
is isomorphic to the preprojective $K$-algebra $ P(\bL_n)$ of type $\bL_n$.
Moreover, each of the deformed preprojective algebras of type $\bA_n$
by the definition
is isomorphic to the preprojective algebra $ P(\bA_n)$ of type $\bA_n$.
Furthermore, by \cite{B:E6},
all
deformed preprojective algebras of Dynkin type $\bE_6$ are isomorphic 
to the canonical preprojective algebra $ P(\bE_6)$  of type $\bE_6$.
 
On the other hand, it is known that 
for $K$ with $\charact(K) = 2$
there exist  
deformed preprojective $K$-algebras of the Dynkin types 
$\bD_n$,
$\bE_7$ and $\bE_8$ (see  \cite[Theorem]{B:soc})
not isomorphic 
to the canonical preprojective $K$-algebras of these types.
But the complete classification of  these algebras seems to be a difficult problem.
 
In our  forthcoming article 
\cite{B:Dn}
we  
prove
that,   for $K$ of characteristic different from $2$, every deformed preprojective $K$-algebra of Dynkin type $\bD_n$,
with  $n \geq 4$,
is isomorphic to the preprojective $K$-algebra $ P(\bD_n)$ of Dynkin type $\bD_n$.
The present  paper is devoted to obtain  the dual fact for the remaining types 
$\bE_7$ and $\bE_8$.
Unfortunately, calculations for these types 
are rather length and much more complicated than in the previous one and the main 
 difficulty that appears  is to find 
 a precise  formulae or patterns 
that are easy to formulate and to use in the proof
handled ``manually''.    
Theretofore we reduce  the problem of the 
existence 
of non-trivial
deformed preprojective algebras $ P(\bE_n)$
of Dynkin types $\mathbb{E}_7$ and $\mathbb{E}_8$ to a computational problem  solvable in an algorithmic way, namely  to the problem of solving 
an associated 
system
of equations.
 
We note that our main  interest is to find  a suitable  form
of coefficients of the  solution 
 to be able to 
determine in which characteristics the expected  solutions exist. Moreover, we would like  to obtain possibly simple and ``human readable''
exemplary formulae for  a $K$-algebra  isomorphism $\varphi_f:  P^f(\bE_n)  \to P(\bE_n)$
described in 
Theorem~\ref{th:hom1} 
between
deformed preprojective algebra of type $\mathbb{E}_n$ 
and the 
preprojective algebra of type $\mathbb{E}_n$
for $n \in\{7,8\}$. Fortunately, we are able to  show that  such an isomorphism  $\varphi_f$ exists in case when the characteristic of $K$ is    different from $2$.

In order to do that we present some
methods of imposing additional restrictions
on the associated systems of equations.   
However, the computations in the case of type $\mathbb{E}_8$
remain complicated. To get them more effective  we need to modify the   obtained sets 
of equations to be able to perform computations  
in a reasonable time, but a  selection  of suitable modification steps  is quite non-trivial part of our work.

We recall from \cite{ESk} that the 
\emph{preprojective algebra $P(\bE_n)$ of Dynkin type $\bE_n$},
for $n \in \{6,7,8\}$, 
is the bound quiver $K$-algebra 
\begin{gather}
\label{eq:1.1}
P(\bE_n):=  KQ_{\mathbb{E}_{n}}/{\calI}(\mathbb{E}_{n})
\end{gather}
given by the quiver
$$
    \begin{array}{c} Q_{\mathbb{E}_{n}}: \\ \end{array}
   \quad \vcenter{
    \xymatrix@C=.8pc@R=2pc{
        && && 0 \ar@<.5ex>[d]^{a_0} \\
        1 \ar@<.5ex>[rr]^{a_1} && 2 \ar@<.5ex>[ll]^{\bar{a}_1} \ar@<.5ex>[rr]^{a_2} &&
        3
         \ar@<.5ex>[ll]^{\bar{a}_2} \ar@<.5ex>[rr]^{a_3} \ar@<.5ex>[u]^{\bar{a}_0} &&
         4 \ar@<.5ex>[ll]^{\bar{a}_3} \ar@{-}@<.5ex>[rr]^{a_4} &&
         \cdots \ar@<.5ex>[ll]^{\bar{a}_4}
         \ar@<.5ex>[rr]^(.45){a_{n-2}} &&
         n-1 \ar@{-}@<.5ex>[ll]^(.55){\bar{a}_{n-2}} 
    }
   }$$
and the ideal ${\calI}(\mathbb{E}_{n})$ of relations
\begin{gather*}
  a_0 \bar{a}_0 = 0, 
  \quad 
  a_1 \bar{a}_1 = 0, 
  \quad 
 \bar{a}_1 a_1 + {a}_2 \bar{a}_2 = 0,
  \quad  
 \bar{a}_0 a_0 + \bar{a}_2 a_2 + a_3 \bar{a}_3 = 0,
 \\
 \bar{a}_i a_i + {a}_{i+1} \bar{a}_{i+1} = 0, \mbox{ for }i \in \{3,\dots,n-3\},
  \quad 
  \bar{a}_{n-2} {a}_{n-2} = 0.  
\end{gather*}
We also use  in the paper the local commutative $K$-algebra
\begin{gather}
\label{eq:1.2}
   R(\bE_n) = K \langle x, y \rangle /
   \left( x^2, y^3, (x+y)^{n-3} \right)
\end{gather}
that  is isomorphic to the $K$-algebra $e_3 P(\bE_n) e_3$, where $e_3$
is the primitive idempotent in $P(\bE_n)$ associated to the vertex
$3$ of $Q_{\bE_n}$.

An element $f$ from the square $\rad^2 R(\bE_n)$ of the radical
$\rad \ R(\bE_n)$ of $R(\bE_n)$ 
is said to be \emph{admissible}
if $f$ satisfies the following condition
\begin{gather}
\label{eq:1.2*}
\big(x+y + f(x,y)\big)^{n-3} = 0.
\end{gather}
Following  \cite[Section~7]{ESk}, for a given admissible element 
$f \in {\rm rad}^2R(\bE_n)$, we define the 
\emph{deformed preprojective $K$-algebra of Dynkin type ${\bE_n}$} 
\begin{gather}
\label{eq:1.3}
  P^f(\bE_n) :=  KQ_{\mathbb{E}_{n}}/{\calI}(\mathbb{E}_{n},f)
\end{gather}
to be the bound quiver $K$-algebra  given by the quiver $Q_{\bE_n}$ and the   ideal ${\calI}(\mathbb{E}_{n},f)$ of relations
\begin{gather*}
  a_0 \bar{a}_0 = 0, 
  \quad 
  a_1 \bar{a}_1 = 0, 
  \quad 
 \bar{a}_1 a_1 + {a}_2 \bar{a}_2 = 0,
  \quad  
 \bar{a}_0 a_0 + \bar{a}_2 a_2 + a_3 \bar{a}_3 
  + f(\bar{a}_0 a_0, \bar{a}_2 a_2) = 0,
 \\
 \bar{a}_i a_i + {a}_{i+1} \bar{a}_{i+1} = 0, \mbox{ for }i \in \{3,\dots,n-3\},
  \quad 
  \bar{a}_{n-2} {a}_{n-2} = 0.   
\end{gather*}

For  a  definition of 
the deformed preprojective algebras of generalized Dynkin types
$\bA_n (n \geq 2)$, 
$\bD_n (n \geq 4)$, 
$\bL_n (n \geq 1)$, 
and a discussion of their  applications  the reader is  
  referred to 
   \cite{ESk}.

\addtocounter{theorem}{4}
 
\begin{remarks}
\label{rem:1.4}
(a) Observe that $P^f(\bE_n)$ is obtained from $P(\bE_n)$
by deforming the relation at the exceptional vertex $3$ of $Q_{\bE_n}$,
and $P^f(\bE_n) = P(\bE_n)$ if $f = 0$.

(b) In the original definition of the admissible element $f$
for the types 
$\mathbb{E}_6, \mathbb{E}_7, \mathbb{E}_8$ 
the 
condition 
\eqref{eq:1.2*}
was missing (see \cite[Remark]{B:E6} for details).

(c) 
The corrected definition of the admissible element has been
presented in 
\cite[Section~9]{ESk} (on page 238) in the definition of deformed mesh algebras of generalized Dynkin type.
As a consequence of this mistake, the originally constructed   algebras (see 
  \cite[Corollary]{B:E6})
were not deformed 
preprojective algebras of generalized Dynkin
and hence the proof of a corresponding theorem 
 in the case of types ${\bE_6}$, ${\bE_7}$, and ${\bE_8}$
was not correct, see 
  \cite[Corollary]{B:E6} for its correction and more details.
\end{remarks}

To complete and correct the proof,   we construct in 
Section~\ref{sec:outline}
an  algorithm to 
test whether for a given 
type $\bE_n \in \{\bE_6, \bE_7, \bE_8\}$
every deformed preprojective algebra  of type $\bE_n$
is isomorphic to the preprojective algebra $P(\bE_n)$ 
\eqref{eq:1.1}
of type $\bE_n$.

By applying the algorithm we construct 
in Appendices \ref{app:E7} and \ref{app:E8}  
a $K$-algebra isomorphism $\varphi: P^f(\bE_n)\to  P(\bE_n)$ defined in 
Theorem~\ref{th:hom1} 
for  the types $\bE_7$ and $\bE_8$. 
As a consequence  
we obtain the following theorem, 
which is
one of the main results of this paper.
 
\begin{theorem}
\label{th:1.5}
Assume that $n \in \{7,8\}$ and $K$ is a field 
of characteristic different from $2$.
Then every deformed preprojective $K$-algebra 
of Dynkin type $\bE_n$ 
is isomorphic to the preprojective $K$-algebra 
$P(\bE_n)$ 
\eqref{eq:1.1}
of Dynkin type $\bE_n$.
\end{theorem}

We note that the calculations for the types $\bE_7$ and $\bE_8$ 
(presented in Appendices \ref{app:E7} and \ref{app:E8},  respectively)
are much more complicated 
than the ones
for the types $\bF_4$ and $\bE_6$,   
given in \cite{B:E6} and  \cite{B:F4},
and cannot be handled manually.
We refer to 
Tables~\ref{t:PEn}, \ref{t:coefficients-ff}--\ref{t:coefficients-sort} 
for the comparison of their structure in the case of types $\bE_n$, for $n \in\{6,7,8\}$. 
 
In the forthcoming article 
\cite{B:Dn}
we    prove
that every deformed preprojective algebra   of Dynkin type $\bD_n$,
for $n \geq 4$, over a field of characteristic different from $2$, 
is isomorphic (as a $K$-algebra) to the preprojective algebra $P (\bD_n)$ of Dynkin type $\bD_n$,
and hence we complete the proof of the following  conjecture. 
As  a consequence   we prove that 
every deformed preprojective algebra of generalized Dynkin type 
is periodic.

\begin{conjecture}[\cite{B:Dn},\cite{ESk}%
	\footnote{Our conjecture is implicitly stated in \cite{ESk}
	and in the present form is announced in \cite{B:Dn}.}]
\label{conj:1.6}
Every deformed preprojective $K$-algebra 
of generalized Dynkin type $\Delta$
over a field $K$ of characteristic different 
from $2$ is isomorphic to the preprojective 
$K$-algebra of corresponding Dynkin type $\Delta$.
\end{conjecture}

\begin{remarks}
\label{rem:1.7}
(a) It is easy to see  that under a minor modifications our algorithm
  can be also used  
in testing  whether or not,  for a given 
(arbitrary) generalized Dynkin type $\Delta$, 
every deformed preprojective algebra of type $\Delta$
is isomorphic to the preprojective algebra $P (\Delta)$ of type $\Delta$, and 
to test whether or not, for a given 
(arbitrary) mesh type $\Delta$ (in the sense of \cite{ESk})
every deformed mesh algebra of type $\Delta$
is isomorphic to the mesh algebra of type $\Delta$. 

(b) We recall 
 that 
 by the definition
 class of deformed mesh algebras of generalized Dynkin type
introduced in \cite{ESk} 
contains all deformed preprojective algebras 
of generalized Dynkin type.

(c) Our  algorithm was already successfully used  in \cite{B:E6} for 
  the  classification of the 
deformed preprojective $K$-algebras of Dynkin type $\bE_6$
and 
the 
deformed mesh $K$-algebras of 
type $\bF_4$ in \cite{B:F4}.
\end{remarks}

The article is organized as follows. 
Section~\ref{sec:hom} contains 
mathematical background 
for the algorithmic approach of the problem
and theorems proving the correctness of the algorithm
presented in the next sections.
Section~\ref{sec:outline}
presents an outline of the algorithm used in testing  if  
$P(\mathbb{E}_n)$ and $P^f(\mathbb{E}_n)$ are isomorphic as $K$-algebras.
Section~\ref{sec:data} contains a detailed 
description of 
the 
data structures used in the calculations.
In Section~\ref{sec:adm}
we describe the calculations associated with the admissible condition
for the algebras $P^f(\mathbb{E}_n)$.
Section~\ref{sec:baseEn} contains
an algorithm   constructing   a suitable $K$-basis $B$ of the preprojective  $P(\mathbb{E}_n)$ and for a suitable  presentation
of elements of $P(\mathbb{E}_n)$ in the basis $B$.
In Section~\ref{sec:eq} we construct
 a  required set of equations
and we discuss 
 a relationship between  the  complexity of 
 its equations  
 and 
the choice of the base $B$ of the algebra $P(\mathbb{E}_n)$.
The final Section~\ref{sec:solv}
contains  a discussion of  solving 
the obtained system of equations.

In Appendices \ref{app:E7} and \ref{app:E8} we 
present 
 applications of our  algorithms to the types $\bE_7$ and $\bE_8$, as well as we discuss 
details of the calculations 
(performed according to our  algorithmic procedures)
for the types $\bE_7$ and $\bE_8$, respectively.

The reader is referred to the monographs 
 \cite{ASS,SS1,SS2,SY} for a general background on the representation theory and selfinjective algebras. We also refer   to \cite{S1}--\cite{S5} for a discussion of  potential areas of application   mentioned  earlier in this section of the algorithmic procedures constructed in our paper.

\section{On algebra homomorphisms from $P^f(\mathbb{E}_n)$ to $P(\mathbb{E}_n)$}
\label{sec:hom}

From now on we assume that $n\in\{6,7,8\}$  
and $f$ is 
a fixed
admissible element from $\rad^2 R(\bE_n)$.
We usually identify the bound quiver $Q_{\bE_n}$
with the quadruple  $Q_{\bE_n} = (Q_0, Q_1, s, t)$, where
$Q_0 = \{0,\dots,n-1\}$ is the set of vertices of $Q_{\bE_n}$,
$Q_1 = \{a_0,\dots,a_{n-2},\bar{a}_0,\dots$, $\bar{a}_{n-2}\}$ 
is the set of arrows of $Q_{\bE_n}$, and
$s,t : Q_1 \to Q_0$ are the functions 
assigning to every arrow its source and target, respectively.

To simplify the notation we set in this section
$a_{n+i-1} = \bar{a_i}$ for $i \in \{ 0, \dots, n-2\}$.
We also denote by $m$ the length of the maximal non-zero path
in $P(\bE_n)$.
Hence we set
\[
    m = \left\{\begin{array}{r@{\mbox{\,\ for }}c}
        10 & n = 6,\\
        16 & n = 7,\\
        28 & n = 8.
    \end{array}\right.
\]

An important role in our study is played by 
 a $K$-algebra 
homomorphism 
\begin{gather}
\label{eq:hom0}
\varphi : P^f(\mathbb{E}_n) \to P(\mathbb{E}_n)
\end{gather}
such that 
for every arrow $a$ we have
$\varphi(a) = a + w_a$  
for some element $w_a \in e_{s(a)} \rad^2 P(\mathbb{E}_n) e_{t(a)}$.

Observe that in that case these $w_a$, for $a \in Q_1$, in fact induce the
homomorphism $\varphi$.
Hence each of these homomorphisms is induced by 
the equalities
\begin{gather}
\label{eq:hom}
  \varphi(a_k) = a_k +
  \sum_{l = 2}^m \sum_{j = 0}^{i(k,l) -1} \alpha(k,l,j) a_{i(k,l,j,0)} \dots a_{i(k,l,j,l -1)},
  \mbox{\ \  for } k = 0,\dots,2n-3,
\end{gather}
for some coefficients 
$\alpha(k,l,j) \in K$
with
$k = 0,\dots,2n-3$, $l = 2,\dots,m$, $j = 0,\dots,{i(k,l) -1}$,
where 
for $k \in \{0,\dots,2n-3\}$
\[
  \{ a_k \} \cup \Big\{ a_{i(k,l,j,0)} \dots a_{i(k,l,j,l -1)} \,\big|\, 
l \in \{2,\dots,m\}, j \in \{0,\dots, i(k,l) -1\}\} \Big\}
\]
form a basis of $e_{s(a_k)} P(\mathbb{E}_n) e_{t(a_k)}$
with some 
$i(k,l) \in \mathbb{N}$ 
(for $k \in \{0,\dots,2n-3\}$, $l \in \{2,\dots,m\}$), 
and $i(k,l,j,s) \in \{ 0, \dots, 2n-3\}$, 
(for $k,s \in \{0,\dots,2n-3\}$, $l \in \{2,\dots,m\}$, 
$j \in \{0,\dots, i(k,l) -1\}$).

We note that assumption of the existence of a homomorphism 
\eqref{eq:hom0}
of the above form is in fact very strong, 
and in the general case, finding such a homomorphism is not easy.
The conditions for the existence of such a homomorphism are 
described by 
Corollary~\ref{cor:hom2}.
In the subsequent sections we construct 
the algorithm that checks them.

The following theorem explains the importance of these homomorphisms.

\begin{theorem}
\label{th:hom1}
\label{th:2.1}
If 
$\varphi : P^f(\mathbb{E}_n) \to P(\mathbb{E}_n)$
is  a $K$-algebra homomorphism 
given by the equalities \eqref{eq:hom} 
then $\varphi$ is invertible, and hence is a $K$-algebra isomorphism.
\end{theorem}

\begin{proof}
We construct inductively elements
$\beta(k,l,j) \in K$
for $k = 0,\dots,2n-3$, $l = 2,\dots,m$, $j = 0,\dots,{i(k,l) -1}$,
such that 
\begin{gather}
\label{eq:ind0}
 \varphi\bigg(a_k + 
  \sum_{l = 2}^m \sum_{j = 0}^{i(k,l) -1} \beta(k,l,j) a_{i(k,l,j,0)} \dots a_{i(k,l,j,l -1)}
 \bigg) = a_k,
\end{gather}
for $k \in \{0,\dots,2n-3\}$.
We set
$R_{k,l} = e_{s(a_k)} \rad^{m-l-1} P(\mathbb{E}_n) e_{t(a_k)}$
and
observe that 
$$ \varphi(a_k) + R_{k,1} = a_k + R_{k,1}.$$
Assume now that
\begin{gather}
\label{eq:ind1}
 \varphi\bigg(a_k + 
  \sum_{l = 2}^{t-1} \sum_{j = 0}^{i(k,l) -1} \beta(k,l,j) a_{i(k,l,j,0)} \dots a_{i(k,l,j,l -1)}
 \bigg) 
 +R_{k,t-1} = a_k + R_{k,t-1},
\end{gather}
for some $t \in \{2,\dots,m-1\}$.
Now we construct 
$\beta(k,t,j) \in K$, 
for $k \in \{0,\dots,2n-3\}$, $j \in \{0,\dots,i(k,l) -1\}$,
such that
\begin{gather}
\label{eq:ind2}
 \varphi\bigg(a_k + 
  \sum_{l = 2}^{t} \sum_{j = 0}^{i(k,l) -1} \beta(k,l,j) a_{i(k,l,j,0)} \dots a_{i(k,l,j,l -1)}
 \bigg) 
 +R_{k,t} = a_k + R_{k,t}.
\end{gather}
%
Observe that   \eqref{eq:ind1} yields
\begin{gather*}
\label{eq:ind3pre}
 \varphi\bigg(a_k + 
  \sum_{l = 2}^{t-1} \sum_{j = 0}^{i(k,l) -1} \beta(k,l,j) a_{i(k,l,j,0)} \dots a_{i(k,l,j,l -1)}
 \bigg) 
 - a_k
 +R_{k,t-1} = 0 + R_{k,t-1},
\end{gather*}
for $k \in \{0,\dots,2n-3\}$,
and hence
there exist elements $\gamma_{k,t,j} \in K$, 
with
$k \in \{0,\dots,2n-3\}$, $j \in \{0,\dots,i(k,l) -1\}$,
such that
\begin{align*}
 \varphi\bigg(a_k + 
  \sum_{l = 2}^{t-1} \sum_{j = 0}^{i(k,l) -1} \beta(k,l,j) a_{i(k,l,j,0)} \dots a_{i(k,l,j,l -1)}
 \bigg) 
 - a_k
 +R_{k,t} 
= 
 \sum_{j = 0}^{i(k,t) -1} \gamma_{k,t,j} a_{i(k,t,j,0)} \dots a_{i(k,t,j,t -1)} 
 + R_{k,t},
\end{align*}
for $k \in \{0,\dots,2n-3\}$.
Therefore, if we set  
$\beta(k,t,j) = -  \gamma_{k,t,j}$, 
for $k \in \{0,\dots,2n-3\}$, $j \in \{0,\dots,i(k,l) -1\}$,
then  \eqref{eq:ind2} is satisfied. 
Hence, by induction we prove  that 
\eqref{eq:ind2} is satisfied also for $t = m$.
But in this case \eqref{eq:ind2} is equivalent to
\eqref{eq:ind0}.

Consequently,  \eqref{eq:ind0} is satisfied and  we obtain
\begin{gather}
 a_k = 
 \varphi(a_k) + 
  \sum_{l = 2}^m \sum_{j = 0}^{i(k,l) -1} \beta(k,l,j) \varphi(a_{i(k,l,j,0)}) \dots \varphi(a_{i(k,l,j,l -1)}),
\end{gather}
for $k \in \{0,\dots,2n-3\}$, and hence the set 
$\{\varphi(a_i)\}_{i \in \{0,\dots,2n-3\}}$ generates the $K$-algebra 
$P(\mathbb{E}_n)$.

Now we show  that  the $K$-linear map
$\psi : P(\mathbb{E}_n) \to P^f(\mathbb{E}_n)$  
defined  by 
\begin{gather}
\label{eq:ind5}
 \psi (a_k) = a_k + 
  \sum_{l = 2}^m \sum_{j = 0}^{i(k,l) -1} \beta(k,l,j) a_{i(k,l,j,0)} \dots a_{i(k,l,j,l -1)}
,
\end{gather}
for $k \in \{0,\dots,2n-3\}$, is a $K$-algebra  homomorphism that is inverse to $\varphi$.
Indeed, $\psi$ is well defined,
because $\varphi$ is a homomorphism
and \eqref{eq:ind0} is satisfied.
Moreover, we have
$\varphi(\psi(a_k)) = a_k$
for $k \in \{0,\dots,2n-3\}$.
Hence $\varphi \circ \psi= \id_{P(\mathbb{E}_n)}$.

To complete the proof that $\varphi$ (and hence also $\psi$)
is an isomorphism it suffices to show that 
$P^f(\mathbb{E}_n)$ is generated by
elements
$\{\psi(a_i)\}_{i \in \{0,\dots,2n-3\}}$. 
Recall that 
$P^f(\mathbb{E}_n)$ is generated by elements
$\{a_k\}_{k \in \{0,\dots,2n-3\}}$,
so it suffices show that 
for each $k \in \{0,\dots,2n-3\}$
element $a_k \in P^f(\mathbb{E}_n)$
is generated by the elements
$\{\psi(a_i)\}_{i \in \{0,\dots,2n-3\}}$. 
In particular, we may dually construct inductively elements
$\alpha'(k,l,j) \in K$
for $k = 0,\dots,2n-3$, $l = 2,\dots,m$, $j = 0,\dots,{i(k,l) -1}$,
such that 
\begin{gather*}
 \psi\bigg(a_k + 
  \sum_{l = 2}^m \sum_{j = 0}^{i(k,l) -1} \alpha'(k,l,j) a_{i(k,l,j,0)} \dots a_{i(k,l,j,l -1)}
 \bigg) = a_k,
\end{gather*}
for $k \in \{0,\dots,2n-3\}$.
\end{proof}

To apply above theorem to the computer algorithm
we need a tool to look for a homomorphism
given by the equalities of the form presented in \eqref{eq:hom}.
Hence we need to formulate the conditions for 
the existence of such a homomorphism 
in a form that can be checked by such an algorithm.
The following consequence of Theorem~\ref{th:hom1}
is being useful in the construction of that algorithm.

\begin{corollary}
\label{cor:hom2}
Let $i(k,l) \in \mathbb{N}$ 
for $k \in \{0,\dots,2n-3\}$, $l \in \{2,\dots,m\}$, 
and $i(k,l,j,s) \in \{ 0, \dots, 2n-3\}$, 
for $k,s \in \{0,\dots,2n-3\}$, $l \in \{2,\dots,m\}$, 
$j \in \{0,\dots, i(k,l) -1\}$
such that 
$$\{ a_k \} \cup \Big\{ a_{i(k,l,j,0)} \dots a_{i(k,l,j,l -1)} \,\big|\, 
l \in \{2,\dots,m\}, j \in \{0,\dots, i(k,l) -1\}\} \Big\}$$
form a basis of $e_{s(a_k)} P(\mathbb{E}_n) e_{t(a_k)}$
for each
$k \in \{0,\dots,2n-3\}$.
Assume that there exist coefficients 
$\alpha(k,l,j) \in K$
for $k = 0,\dots,2n-3$, $l = 2,\dots,m$, $j = 0,\dots,{i(k,l) -1}$
satisfying the equalities:
\begin{gather*}
\delta(k_1,k_2) =0
\mbox{ for } (k_1,k_2) \in \{ (0,n-1), (1,n), (2n-3, n-2)\};
\\
\delta(k+n-2,k-1)+\delta(k,k+n-1) =0
\mbox{ for } k \in \{ 2, 4, \dots, n-2 \};
\\
\delta(n-1,0)+\delta(n+1,2)+\delta(2n-3,n-2) +
f\big(\delta(n-1,0),\delta(n+1,2)\big) = 0  
\end{gather*}
where $\delta : \{0,\dots,n-3\}^2 \rightarrow P(\mathbb{E}_n)$
are defined as follow
\begin{align*}
  \delta(k_1,k_2) &=
   \sum_{l = 2}^m \sum_{j = 0}^{i(k_1,l) -1} \alpha(k_1,l,j)   a_{i(k_1,l,j,0)} \dots a_{i(k_1,l,j,l -1)} a_{k_2}
  + 
   \sum_{l = 2}^m \sum_{j = 0}^{i(k_2,l) -1} \alpha(k_2,l,j)   a_{k_1} a_{i(k_2,l,j,0)} \dots a_{i(k_2,l,j,l -1)}
 \\&\quad
  + 
   \bigg( \sum_{l = 2}^m \sum_{j = 0}^{i(k_1,l) -1} \alpha(k_1,l,j)  a_{i(k_1,l,j,0)} \dots a_{i(k_1,l,j,l -1)}\bigg)  
   \,
   \bigg(\sum_{l = 2}^m \sum_{j = 0}^{i(k_2,l) -1} \alpha(k_2,l,j) a_{i(k_2,l,j,0)} \dots a_{i(k_2,l,j,l -1)}\bigg) .
\end{align*}
Then the map
$\varphi : P^f(\mathbb{E}_n) \to P(\mathbb{E}_n)$
defined by setting 
$$
  \varphi(a_k) = a_k +
  \sum_{l = 2}^m \sum_{j = 0}^{i(k,l) -1} \alpha(k,l,j) a_{i(k,l,j,0)} \dots a_{i(k,l,j,l -1)}
$$
for $k = 0,\dots,2n-3$,
is a  $K$-algebra isomorphism.
\end{corollary}

\begin{proof}
It follows from the assumption on $\delta$
that $\varphi : P^f(\mathbb{E}_n) \to P(\mathbb{E}_n)$ 
defined on arrows as in the claim 
is a well defined homomorphism.
Hence the corollary is an immediate consequence of  Theorem~\ref{th:hom1}.
\end{proof}

We note that 
Corollary~\ref{cor:hom2}
says that 
for a given admissible element $f \in \rad^2 R(\bE_n)$
we can find an isomorphism
$\varphi : P^f(\mathbb{E}_n) \to P(\mathbb{E}_n)$
if we can solve a particular system of
$\sum_{i=0}^{n-1} \dim \rad^3 e_{i} P(\mathbb{E}_n) e_{i}$
equations over $K$
with
$\sum_{k=0}^{2n-3} \dim \rad^2 e_{s(a_k)} P(\mathbb{E}_n) e_{t(a_k)}$
variables.
Details of  the construction of these equations are presented 
in Sections~\ref{sec:baseEn} and \ref{sec:eq}.
Dimensions of $e_i P(E_n) e_j$ over $K$, for $i,j \in \{0,\dots,n-1\}$,
are shown in Table~\ref{t:dimeiAej}.

On the other hand, if we want
to show that such a system of equations
is solvable for all admissible elements
$f \in \rad^2 R(\bE_n)$, then we also need
some tool (preferable equations) to determine
if a given element $f \in \rad^2 R(\bE_n)$
is admissible.
We note 
that not all elements $f \in \rad^2 R(\bE_n)$
are admissible (see \cite[Remark]{B:E6} for details). 
Observe, that if
 we denote by $B$ the basis of $\rad^2 R(\bE_n)$
then we can identify element $f \in \rad^2 R(\bE_n)$
with coefficients $\theta_b \in K$, for $b\in B$, such that
$f = \sum_{b \in B} \theta_b b$.
We show later in Section~\ref{sec:adm} how to
construct equations for $\theta_b$, $b\in B$, over $K$
equivalent with 
$\sum_{b \in B} \theta_b b$ 
being an admissible element. 

We note also that in the above approach we may in fact 
reduce 
the scope of the calculations to
$P(\mathbb{E}_n)$.
Obviously, it is convenient,
because 
calculations in $P^f(\mathbb{E}_n)$
are much more complicated.

\section{Alghoritm}
\label{sec:outline}

We present here an  outline of the constructed algorithm.  Details of the particular steps of
this algorithm are presented in the  subsequent
sections.

We split the  calculations  into following three parts:

\begin{enumerate}[I.]
\setlength{\itemsep}{0pt}
 \item
 Computing the equations corresponding to the admissibility condition. 
 \item
 Computing the equations equivalent to 
 the existence of coefficients from 
 Corollary~\ref{cor:hom2}.
 \item
 Solving the system composed of equations obtained in Parts I and II.
\end{enumerate}

\noindent
In Part I we can extinguish the following steps:
\begin{enumerate}[1.]
\setlength{\itemsep}{0pt}
 \item
Choosing  
a base $B$ of $R(\mathbb{E}_n)$, 
and
a base $B'$ of $\rad^2 R(\mathbb{E}_n)$ 
such that $B' \subset B$,
and computing 
the presentation of (all) non-zero elements of $\rad^2 R(\mathbb{E}_n)$ in $B$.

 \item
Constructing the presentation of 
an arbitrary element $f \in \rad^2 R(\mathbb{E}_n)$
through the base elements.
We are identifying $f \in \rad^2 R(\mathbb{E}_n)$ with 
elements
$\theta_b \in K$, for $b \in B'$,
such that
$f= \sum_{b \in B'} \theta_b b$.

 \item
 Calculating  $(x+y+f(x,y))^{n-3}$.
In other words, we are constructing terms 
$\omega_b$ over $K$, for ${b \in B'}$,
such that 
$\sum_{b \in B'} \omega_b b = 
\big( \sum_{b \in B\setminus B'} b + \sum_{b \in B'} \theta_b b\big)^{n-3}$.

 \item
From the previous step we obtain
the set of equations $\{\omega_b = 0 \}_{b \in B'}$.
We reduce it to obtain the set $\Omega$ of independent equations.

 \item
 Further reduction of the obtained system of equations.
For each equation of $\Omega$ we chose 
the variable to be substituted in the 
equations obtained in Part~II.
\end{enumerate}

\noindent
In Part II we have the following steps:
\begin{enumerate}[1.]
\setlength{\itemsep}{0pt}
 \item
 Computation of 
\begin{itemize}
 \item
a base $B$ of $P(\mathbb{E}_n)$, 
 \item
the set $E$ of non-zero paths of $P(\mathbb{E}_n)$,
 \item
the presentation of elements from $E$ in $B$.
\end{itemize}

 \item
Constructing the presentation of arbitrary homomorphism 
$\varphi : P^f(\mathbb{E}_n) \to P(\mathbb{E}_n)$ 
defined on arrows.
We identify $\varphi$ with the
set of variables 
\[
  \big\{\alpha_{a,b} \in K\,|\,  
   a \in Q_1,b \in B, b \mbox{ is a path from $s(a)$ to $t(a)$} \big\}
.
\]

 \item
Calculating the actions of $\varphi$
on the paths from the relations of $P^f(\mathbb{E}_n)$.

 \item
 Deriving from the previous step 
 the equations in $P(\mathbb{E}_n)$ 
 corresponding
 to the relations of $P^f(\mathbb{E}_n)$.
\end{enumerate}

\section{Data structures}
\label{sec:data}

In this section we consider the choice of data structures.
Our priority is 
to reduce the 
computational and memory complexity of the necessary calculations.

We note that we need to choose the data structure in such a way,
so we can be able to
 \begin{itemize}
\setlength{\itemsep}{0pt}
 \item
  quickly perform operations 
  such as sum, product, multiplication by scalar,
  on elements from 
  $P(\mathbb{E}_n)$ (respectively, on elements from $R(\mathbb{E}_n)$); 
 \item
  quickly determine if the result of such an operation
  is non-zero;
 \item
  effectively store the presentation of elements 
  from $P(\mathbb{E}_n)$ (respectively, elements from $R(\mathbb{E}_n)$) 
  in the limited amount of memory.
\end{itemize}
We recall that in $P(\mathbb{E}_8)$ we have 
14 different arrows
and
non-zero paths of length 28,
so 
``brutal force'' approach by
computing and storing coefficients for all 
(not necessarily non-zero) paths of 
length less or equal 28 would require 
a big amount of memory.

\begin{table}[htbp]
\small
  \caption{Dimensions of $e_i A e_j$ over $K$ for $A \in \{P(\mathbb{E}_6),P(\mathbb{E}_7),P(\mathbb{E}_8)\}$}
  \label{t:dimeiAej}
  \begin{center}
\begin{tabular}{@{\ \ } c|cccccc}
\multicolumn{7}{@{}c@{}}{$\dim e_i P(\mathbb{E}_6) e_j$} \\
\midrule
    & \multicolumn{6}{c}{$j\ \,$} \\
  $i$  & 0 & 1 & 2 & 3 & 4 & 5 \\
\cmidrule{2-7}
$0$&     4 &     2 &     4 &     6 &     4 &     2 \\
$1$&     2 &     2 &     3 &     4 &     3 &     2 \\
$2$&     4 &     3 &     6 &     8 &     6 &     3 \\
$3$&     6 &     4 &     8 &    \!12\! &     8 &     4 \\
$4$&     4 &     3 &     6 &     8 &     6 &     3 \\
$5$&     2 &     2 &     3 &     4 &     3 &     2 \\\bottomrule
\multicolumn{7}{c}{}\\
\multicolumn{7}{c}{}\\
\end{tabular}\quad
\begin{tabular}{@{\ \ } c|ccccccc}
\multicolumn{8}{@{}c@{}}{$\dim e_i P(\mathbb{E}_7) e_j$} \\
\midrule
    & \multicolumn{7}{c}{$j\ \,$} \\
 $i$  & 0 & 1 & 2 & 3 & 4 & 5 & 6 \\
\cmidrule{2-8}
$0$&     7 &     4 &     8 &    12 &     9 &     6 &     3 \\
$1$&     4 &     4 &     6 &     8 &     6 &     4 &     2 \\
$2$&     8 &     6 &    12 &    16 &    12 &     8 &     4 \\
$3$&    12 &     8 &    16 &    24 &    18 &    12 &     6 \\
$4$&     9 &     6 &    12 &    18 &    15 &    10 &     5 \\
$5$&     6 &     4 &     8 &    12 &    10 &     8 &     4 \\
$6$&     3 &     2 &     4 &     6 &     5 &     4 &     3 \\\bottomrule
\multicolumn{8}{c}{}\\
\end{tabular}\quad
\begin{tabular}{@{\ \ } c|cccccccc@{\ \ } }
\multicolumn{9}{@{}c@{}}{$\dim e_i P(\mathbb{E}_8) e_j$} \\
\midrule
    & \multicolumn{8}{c}{$j\ \,$} \\
 $i$  & 0 & 1 & 2 & 3 & 4 & 5 & 6 & 7 \\
\cmidrule{2-9}
$0$&    16 &    10 &    20 &    30 &    24 &    18 &    12 &     6 \\
$1$&    10 &     8 &    14 &    20 &    16 &    12 &     8 &     4 \\
$2$&    20 &    14 &    28 &    40 &    32 &    24 &    16 &     8 \\
$3$&    30 &    20 &    40 &    60 &    48 &    36 &    24 &    12 \\
$4$&    24 &    16 &    32 &    48 &    40 &    30 &    20 &    10 \\
$5$&    18 &    12 &    24 &    36 &    30 &    24 &    16 &     8 \\
$6$&    12 &     8 &    16 &    24 &    20 &    16 &    12 &     6 \\
$7$&     6 &     4 &     8 &    12 &    10 &     8 &     6 &     4 \\\bottomrule
\end{tabular}
\end{center}
\end{table}

We note that the basis 
(of both $R(\mathbb{E}_n)$ and $P(\mathbb{E}_n)$) 
consist of relatively small amount of elements
(recall that $\dim_K R(\mathbb{E}_n) = \dim_K e_3 P(\mathbb{E}_n) e_3$, 
and see Table~\ref{t:dimeiAej} for the dimensions
of $e_i P(E_n) e_j$ over $K$, for $i,j \in \{0,\dots,n-1\}$).
Hence we can store them in 
(respectively)
two and four-dimensional tables:
\begin{itemize}
\setlength{\itemsep}{0pt}
 \item
 we store the basis of $R(\mathbb{E}_n)$ in the 2-dimensional 
 table, where the dimensions corresponds to
 \begin{itemize}
\setlength{\itemsep}{0pt}
  \item
   length $m$ of the element,
  \item
   number $r$ of the elements of length $m$;   
 \end{itemize}
 \item
 we store the basis of $P(\mathbb{E}_n)$ in the 4-dimensional 
 table, where the dimensions corresponds to
 \begin{itemize}
\setlength{\itemsep}{0pt}
  \item
   length $m$ of the path,
  \item
   source $s$ of the path,
  \item
   target $t$ of the path,
  \item
   number $r$ of the elements from $s$ to $t$ of length $m$.   
 \end{itemize}
\end{itemize}
We note that the last dimension in the case of both bases has
variable number of indices and we identify 
these tables
with compositions of one-dimensional tables.

The presentation of elements of $R(\mathbb{E}_n)$
we compute as a binary search tree,
where 
\begin{itemize}
\setlength{\itemsep}{0pt}
 \item
left branches corresponds to the element $x$,
 \item
right branches corresponds to the element $y$,
 \item
each of the nodes
correspond to the element being a composition of 
elements corresponding to branches on the path 
(from root)
to that node,
 \item
all leaves and only leaves correspond to zero elements,
 \item
with every node which is not a leaf we associate the vector
of coefficients of the corresponding element in the base 
of $R(\mathbb{E}_n)$.
\end{itemize}
We identify the leaves with null pointers and denote
them by symbol ``$\boxtimes$''.
We note that if we want our trees to be compatible with 
the commonly accepted definition of a tree, then we should
remove all nodes ``$\boxtimes$''.
In our case we
keep them in order to emphasize their role
to indicate all zero-paths.
We note also that in the construction of such a tree
(see Algorithm~\ref{alg:base-REn})
there appear some similar nodes to ``$\boxtimes$'',
of a special
role such 
as 
``under computation''.

For $R(\mathbb{E}_6)$ such a tree is of the  shape presented on Figure~\ref{fig:E6-paths}.
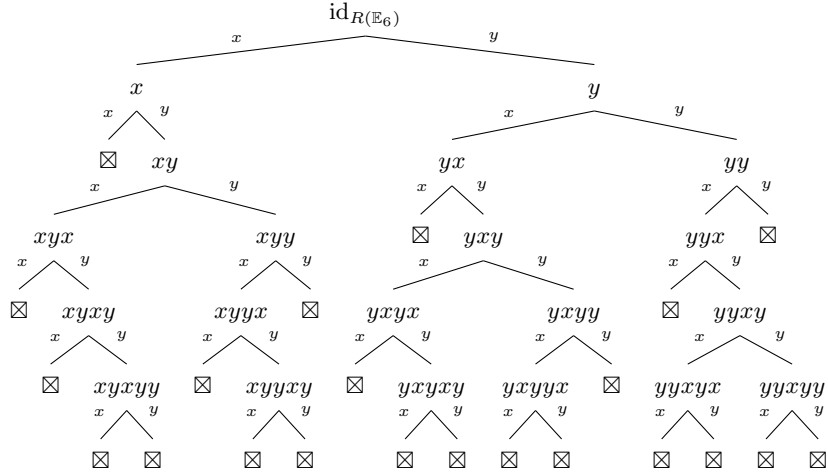
\begin{figure}[h]
\center
\begin{forest}
  [$\id_{R(\mathbb{E}_6)}$
    [$x$,edge label={node[midway,above left,font=\scriptsize]{$x$}}
      [$\boxtimes$,edge label={node[midway,above left,font=\scriptsize]{$x$}}]
      [$xy$,edge label={node[midway,above right,font=\scriptsize]{$y$}}
        [$xyx$,edge label={node[midway,above left,font=\scriptsize]{$x$}}
          [$\boxtimes$,edge label={node[midway,above left,font=\scriptsize]{$x$}}]
          [$xyxy$,edge label={node[midway,above right,font=\scriptsize]{$y$}}
            [$\boxtimes$,edge label={node[midway,above left,font=\scriptsize]{$x$}}]
            [$xyxyy$,edge label={node[midway,above right,font=\scriptsize]{$y$}}
              [$\boxtimes$,edge label={node[midway,above left,font=\scriptsize]{$x$}}]
              [$\boxtimes$,edge label={node[midway,above right,font=\scriptsize]{$y$}}]
            ]]]
        [$xyy$,edge label={node[midway,above right,font=\scriptsize]{$y$}}
          [$xyyx$,edge label={node[midway,above left,font=\scriptsize]{$x$}}
            [$\boxtimes$,edge label={node[midway,above left,font=\scriptsize]{$x$}}]
            [$xyyxy$,edge label={node[midway,above right,font=\scriptsize]{$y$}}
              [$\boxtimes$,edge label={node[midway,above left,font=\scriptsize]{$x$}}]
              [$\boxtimes$,edge label={node[midway,above right,font=\scriptsize]{$y$}}]
            ]]
          [$\boxtimes$,edge label={node[midway,above right,font=\scriptsize]{$y$}}]]]]
    [$y$,edge label={node[midway,above right,font=\scriptsize]{$y$}}
      [$yx$,edge label={node[midway,above left,font=\scriptsize]{$x$}}
        [$\boxtimes$,edge label={node[midway,above left,font=\scriptsize]{$x$}}]
        [$yxy$,edge label={node[midway,above right,font=\scriptsize]{$y$}}
          [$yxyx$,edge label={node[midway,above left,font=\scriptsize]{$x$}}
            [$\boxtimes$,edge label={node[midway,above left,font=\scriptsize]{$x$}}]
            [$yxyxy$,edge label={node[midway,above right,font=\scriptsize]{$y$}}
              [$\boxtimes$,edge label={node[midway,above left,font=\scriptsize]{$x$}}]
              [$\boxtimes$,edge label={node[midway,above right,font=\scriptsize]{$y$}}]
          ]]
          [$yxyy$,edge label={node[midway,above right,font=\scriptsize]{$y$}}
            [$yxyyx$,edge label={node[midway,above left,font=\scriptsize]{$x$}}
              [$\boxtimes$,edge label={node[midway,above left,font=\scriptsize]{$x$}}]
              [$\boxtimes$,edge label={node[midway,above right,font=\scriptsize]{$y$}}]
            ]
            [$\boxtimes$,edge label={node[midway,above right,font=\scriptsize]{$y$}}]]]]
      [$yy$,edge label={node[midway,above right,font=\scriptsize]{$y$}}
        [$yyx$,edge label={node[midway,above left,font=\scriptsize]{$x$}}
          [$\boxtimes$,edge label={node[midway,above left,font=\scriptsize]{$x$}}]
          [$yyxy$,edge label={node[midway,above right,font=\scriptsize]{$y$}}
            [$yyxyx$,edge label={node[midway,above left,font=\scriptsize]{$x$}}
              [$\boxtimes$,edge label={node[midway,above left,font=\scriptsize]{$x$}}]
              [$\boxtimes$,edge label={node[midway,above right,font=\scriptsize]{$y$}}]
            ]
            [$yyxyy$,edge label={node[midway,above right,font=\scriptsize]{$y$}}
              [$\boxtimes$,edge label={node[midway,above left,font=\scriptsize]{$x$}}]
              [$\boxtimes$,edge label={node[midway,above right,font=\scriptsize]{$y$}}]
            ]]]
        [$\boxtimes$,edge label={node[midway,above right,font=\scriptsize]{$y$}}]]]
  ]
\end{forest}
  \caption{Tree with non-zero paths of $R(\mathbb{E}_6)$}
  \label{fig:E6-paths}
\end{figure}
On this figure we may observe the association of nodes of this tree
with compositions of generators $x$, $y$,
where $\boxtimes$ denote the leaves. 

The basis of $R(\EE_6)$ over $K$ is stored in the following table 
$$B = \big\{ 
  \{1_{R(\mathbb{E}_6)}\},
  \{x,y\},
  \{xy,  yx, yy\},
  \{xyx,  xyy,  yxy\},
  \{xyxy,  yxyy\},
  \{xyxyy\}
\big\} .$$

Presentation of elements in the above basis is stored in the tree
illustrated on Figure~\ref{fig:E6-pres}.
\begin{figure}[h]
\center
\begin{forest}
  [\mbox{$[1]$}
    [\mbox{$[1,0]$},edge label={node[midway,above left,font=\scriptsize]{$x$}}
      [$\boxtimes$,edge label={node[midway,above left,font=\scriptsize]{$x$}}]
      [\mbox{$[1,0,0]$},edge label={node[midway,above right,font=\scriptsize]{$y$}}
        [\mbox{$[1,0,0]$},edge label={node[midway,above left,font=\scriptsize]{$x$}}
          [$\boxtimes$,edge label={node[midway,above left,font=\scriptsize]{$x$}}]
          [\mbox{$[1,0]$},edge label={node[midway,above right,font=\scriptsize]{$y$}}
            [$\boxtimes$,edge label={node[midway,above left,font=\scriptsize]{$x$}}]
            [\mbox{$[1]$},edge label={node[midway,above right,font=\scriptsize]{$y$}}
              [$\boxtimes$,edge label={node[midway,above left,font=\scriptsize]{$x$}}]
              [$\boxtimes$,edge label={node[midway,above right,font=\scriptsize]{$y$}}]
            ]]]
        [\mbox{$[0,1,0]$},edge label={node[midway,above right,font=\scriptsize]{$y$}}
          [\mbox{$[-1,0]$},edge label={node[midway,above left,font=\scriptsize]{$x$}}
            [$\boxtimes$,edge label={node[midway,above left,font=\scriptsize]{$x$}}]
            [\mbox{$[-1]$},edge label={node[midway,above right,font=\scriptsize]{$y$}}
              [$\boxtimes$,edge label={node[midway,above left,font=\scriptsize]{$x$}}]
              [$\boxtimes$,edge label={node[midway,above right,font=\scriptsize]{$y$}}]
          ]]
          [$\boxtimes$,edge label={node[midway,above right,font=\scriptsize]{$y$}}]]]]
    [\mbox{$[0,1]$},edge label={node[midway,above right,font=\scriptsize]{$y$}}
      [\mbox{$[0,1,0]$},edge label={node[midway,above left,font=\scriptsize]{$x$}}
        [$\boxtimes$,edge label={node[midway,above left,font=\scriptsize]{$x$}}]
        [\mbox{$[0,0,1]$},edge label={node[midway,above right,font=\scriptsize]{$y$}}
          [\mbox{$[1,0]$},edge label={node[midway,above left,font=\scriptsize]{$x$}}
            [$\boxtimes$,edge label={node[midway,above left,font=\scriptsize]{$x$}}]
            [\mbox{$[1]$},edge label={node[midway,above right,font=\scriptsize]{$y$}}
              [$\boxtimes$,edge label={node[midway,above left,font=\scriptsize]{$x$}}]
              [$\boxtimes$,edge label={node[midway,above right,font=\scriptsize]{$y$}}]
            ]]
          [\mbox{$[0,1]$},edge label={node[midway,above right,font=\scriptsize]{$y$}}
            [\mbox{$[-1]$},edge label={node[midway,above left,font=\scriptsize]{$x$}}
              [$\boxtimes$,edge label={node[midway,above left,font=\scriptsize]{$x$}}]
              [$\boxtimes$,edge label={node[midway,above right,font=\scriptsize]{$y$}}]
            ]
            [$\boxtimes$,edge label={node[midway,above right,font=\scriptsize]{$y$}}]]]]
      [\mbox{$[0,0,1]$},edge label={node[midway,above right,font=\scriptsize]{$y$}}
        [\mbox{$[-1,-1,-1]$},edge label={node[midway,above left,font=\scriptsize]{$x$}}
          [$\boxtimes$,edge label={node[midway,above left,font=\scriptsize]{$x$}}]
          [\mbox{$[-1,-1]$},edge label={node[midway,above right,font=\scriptsize]{$y$}}
            [\mbox{$[1]$},edge label={node[midway,above left,font=\scriptsize]{$x$}}
              [$\boxtimes$,edge label={node[midway,above left,font=\scriptsize]{$x$}}]
              [$\boxtimes$,edge label={node[midway,above right,font=\scriptsize]{$y$}}]
            ]
            [\mbox{$[-1]$},edge label={node[midway,above right,font=\scriptsize]{$y$}}
              [$\boxtimes$,edge label={node[midway,above left,font=\scriptsize]{$x$}}]
              [$\boxtimes$,edge label={node[midway,above right,font=\scriptsize]{$y$}}]
              ]]]
        [$\boxtimes$,edge label={node[midway,above right,font=\scriptsize]{$y$}}]]]
  ]
\end{forest}
  \caption{Presentation of elements of $R(\mathbb{E}_6)$}
  \label{fig:E6-pres}
\end{figure}
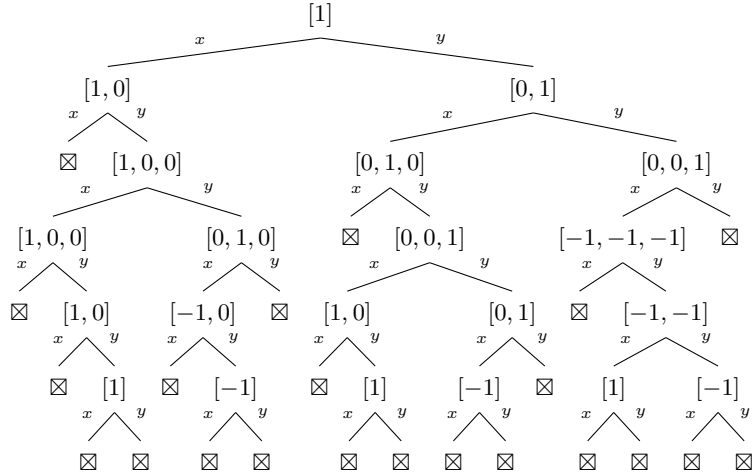

We note that both finding the basis $B$ of $R(\mathbb{E}_n)$
and the presentation of elements of $R(\mathbb{E}_n)$ in $B$,
even in the case of $n=8$, has 
a relatively low 
computational and memory complexity.
The basic properties of such trees has been presented in Table~\ref{t:REn}.
\begin{table}[htbp]\small
  \caption{Complexity of the trees for $R(\mathbb{E}_n)$}
  \label{t:REn}
  \begin{center}
\begin{tabular}{|l|c|c|c|}\hline
Dynkin type & $\mathbb{E}_6$ & $\mathbb{E}_7$ & $\mathbb{E}_8$  \\\hline
number of elements of basis $B$ & 12 & 24 & 60 \\\hline
depth of the tree & 6 & 9 & 15 \\\hline
number of nodes & 43 & 103 & 537  \\\hline
number of non-leaf and non-root nodes & 20 & 50 & 267  \\\hline
\end{tabular}
\end{center}
\end{table}

Much more computational complexity has the step
in which we compute $(x + y + f(x,y))^{n-3}$.
We recall that $f(x,y) = \sum_{b \in B'} \theta_b b$
for some coefficients $\theta_b \in K$.
Hence the powers of $x + y + f(x,y)$ are sums of the
elements $\omega$ of the form
\[
   \omega = k_{\omega} \prod_{b' \in B} \theta_{b'}^{\omega_{b'}} b
\]
for some integer $k_{\omega} \in \mathbb{Z}$,
natural numbers $\omega_{b'} \in \mathbb{N}$,
and $b \in B$.

Hence we may associate each of these summands with the
triple 
$(k_{\omega}, \underline{d}_{\omega}, b) \in \mathbb{Z} \times \mathbb{N}^{|B|} \times B$,
where $\underline{d}_{\omega} = [\omega_{b'}]_{b' \in B}$.
We note also that in fact in our calculation vector $\underline{d}_{\omega}$
can be treated as a sparse matrix, and detailed analysis show that 
even in the case of $\mathbb{E}_8$
it can be stored as a 128-bit vector.

On the other hand 
$b$ can be associated with natural number 
less then $|B|$ and stored as a $6$-bit value.
Observe that we can easily multiply such elements.
Indeed, for
$(k_1, \underline{d}_1, b_1), (k_2, \underline{d}_2, b_2) \in \mathbb{Z} \times \mathbb{N}^{|B|} \times B$,
with $b_1 b_2 = \sum_{b \in B} c_b b$ for some $c_b \in \mathbb{Z}$, $b \in B$.
We obtain then
\[
  (k_1, \underline{d}_1, b_1) (k_2, \underline{d}_2, b_2)
  = \sum_{b \in B} (c_b k_1 k_2, \underline{d}_1 + \underline{d}_2, b) .
\]
We note also that if $b_1 b_2 = 0$, then we may simply set 
$(k_1, \underline{d}_1, b_1) (k_2, \underline{d}_2, b_2) = 0$.
Moreover, in the above computations it is highly recommended 
to perform normalizations, 
that means, 
to replace pairs
of elements of the form $(k_1, \underline{d}, b)$, $(k_2, \underline{d}, b)$
by $(k_1+k_2, \underline{d}, b)$, 
and remove elements of the form $(0, \underline{d}, b)$.
Otherwise computational complexity of computing 
$(x + y + f(x,y))^{n-3}$
in the case of type  $\mathbb{E}_8$ would be very high.
We want to emphasize, that although keeping
elements of the form $(0, \underline{d}, b)$ may 
reduce amount of operations of allocating and freeing memory
and in the case of low number of such elements (i.e. in the case
of type $\bE_6$) can be useful, in the case of type $\bE_8$
it can results 
in allocating 
much
more elements
than necessary (even thousands times more), 
and in the consequence cause a big impact
on memory and time complexity of the algorithm.
We recall that 
the product of a two basis elements is a combination
of (often many) basis elements.

Dually 
(to the computation of the presentation $R(\mathbb{E}_n)$),
we compute 
the presentation of elements of $P(\mathbb{E}_n)$
as the forest of search threes,
where 
\begin{itemize}
\setlength{\itemsep}{0pt}
 \item
 roots corresponds to the (source) vertices,
 \item
 branches corresponds to arrows,
 \item
every node correspond to the path consisting of arrows
corresponding to branches on the path (in the tree) from the root to that node,
 \item
all leaves and only leaves correspond to zero elements,
 \item
with every node 
which is not a leaf we associate the vector
of coefficients of the corresponding element in the base 
of $P(\mathbb{E}_n)$.
\end{itemize}
We note that in general, trees in such a forest are not binary trees.
In particular the tree with root associated with the exceptional
vertex $3$ have three branches coming from root 
and is of the shape presented on Figure~\ref{fig:P-v3}.
\begin{figure}[h]
\center
\begin{forest}
    [$e_3$
      [$\bar{a}_2$,edge label={node[midway,above left,font=\scriptsize]{$\bar{a}_2$}}
         [$\bar{a}_2 \bar{a}_1$,edge label={node[midway,above left,font=\scriptsize]{$\bar{a}_1$}}
            [$\vdots$,edge label={node[midway,left,font=\scriptsize]{${a}_1$}}
            ]]
         [$\bar{a}_2 {a}_2$,edge label={node[midway,above right,font=\scriptsize]{${a}_2$}}
            [$\vdots$,edge label={node[midway,above left,font=\scriptsize]{$\bar{a}_2$}}]
            [$\vdots$,edge label={node[midway,below right,font=\scriptsize]{$\bar{a}_0$}}]
            [$\vdots$,edge label={node[midway,above right,font=\scriptsize]{${a}_3$}}
              ]]]
      [$\bar{a}_0$,edge label={node[midway,below left,font=\scriptsize]{$\bar{a}_0$}}     
          [$\qquad\vdots\qquad$,edge label={node[midway,right,font=\scriptsize]{${a}_0$}}]
      ]
      [${a}_3$,edge label={node[midway,above right,font=\scriptsize]{${a}_3$}}
         [$\ \ \vdots\ \ $,edge label={node[midway,above left,font=\scriptsize]{$\bar{a}_3$}}]
         [$\ \ \vdots\ \  $,edge label={node[midway,above right,font=\scriptsize]{${a}_4$}}]         
      ]]
\end{forest}
  \caption{Shape of a tree associated with vertex $3$ for $P(\mathbb{E}_n)$}
  \label{fig:P-v3}
\end{figure}
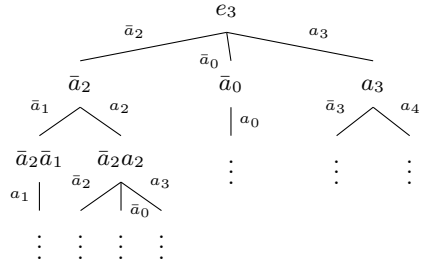

On Figure~\ref{fig:P-v0} we present also more detailed   shape of the tree 
for $P(\mathbb{E}_6)$ associated with the vertex $0$.
\begin{figure}[h]
\center
\begin{forest}
  [$e_0$
    [$a_0$,edge label={node[midway,left,font=\scriptsize]{$a_0$}}
      [$a_0 \bar{a}_2$,edge label={node[midway,above left,font=\scriptsize]{$\bar{a}_2$}}
         [$a_0 \bar{a}_2 \bar{a}_1$,edge label={node[midway,above left,font=\scriptsize]{$\bar{a}_1$}}
            [$a_0 \bar{a}_2 \bar{a}_1 a_1$,edge label={node[midway,left,font=\scriptsize]{${a}_1$}}
              [$\boxtimes$,edge label={node[midway,above left,font=\scriptsize]{$\bar{a}_1$}}]
              [$a_0 \bar{a}_2 \bar{a}_1 a_1 a_2$,edge label={node[midway,above right,font=\scriptsize]{${a}_2$}}
                  [$\vdots$,edge label={node[midway,above left,font=\scriptsize]{$\bar{a}_2$}}]
                  [$\boxtimes$,edge label={node[midway,below right,font=\scriptsize]{$\bar{a}_0$}}]
                  [$\vdots$,edge label={node[midway,above right,font=\scriptsize]{$\bar{a}_3$}}]
               ]]]
         [$a_0 \bar{a}_2 {a}_2$,edge label={node[midway,above right,font=\scriptsize]{${a}_2$}}
            [$a_0 \bar{a}_2 a_2 \bar{a}_2$,edge label={node[midway,above left,font=\scriptsize]{$\bar{a}_2$}}
              [$\boxtimes$,edge label={node[midway,above left,font=\scriptsize]{$\bar{a}_1$}}]
              [$a_0 \bar{a}_2  a_2 \bar{a}_2 a_2$,edge label={node[midway,above right,font=\scriptsize]{${a}_2$}}
                  [$\vdots$,edge label={node[midway,above left,font=\scriptsize]{$\bar{a}_2$}}]
                  [$\boxtimes$,edge label={node[midway,below right,font=\scriptsize]{$\bar{a}_0$}}]
                  [$\vdots$,edge label={node[midway,above right,font=\scriptsize]{$\bar{a}_3$}}]
                 ]]
            [$\boxtimes$,edge label={node[midway,below right,font=\scriptsize]{$\bar{a}_0$}}]
            [$a_0 \bar{a}_2 {a}_2 {a}_3$,edge label={node[midway,above right,font=\scriptsize]{${a}_3$}}
              [$a_0 \bar{a}_2  a_2 a_3 \bar{a}_3$,edge label={node[midway,above left,font=\scriptsize]{$\bar{a}_3$}}
                  [$\vdots$,edge label={node[midway,above left,font=\scriptsize]{$\bar{a}_2$}}]
                  [$\boxtimes$,edge label={node[midway,below right,font=\scriptsize]{$\bar{a}_0$}}]
                  [$\vdots$,edge label={node[midway,above right,font=\scriptsize]{$\bar{a}_3$}}]
              ]
              [$a_0 \bar{a}_2  a_2 a_3 a_4$,edge label={node[midway,above right,font=\scriptsize]{${a}_4$}}
                [$\vdots$,edge label={node[midway,right,font=\scriptsize]{$\bar{a}_4$}}]
                 ]
              ]]]
      [$\qquad\boxtimes\qquad$,edge label={node[midway,below right,font=\scriptsize]{$\bar{a}_0$}}]
      [$a_0 {a}_3$,edge label={node[midway,above right,font=\scriptsize]{${a}_3$}}
         [$\ \ \vdots\ \ $,edge label={node[midway,above left,font=\scriptsize]{$\bar{a}_3$}}]
         [$\ \ \vdots\ \  $,edge label={node[midway,above right,font=\scriptsize]{${a}_4$}}]         
      ]]
  ]
\end{forest}
  \caption{Shape of the tree associated with vertex $0$ for $P(\mathbb{E}_6)$}
  \label{fig:P-v0}
\end{figure}

For such a forest and the chosen basis $B$ of $P(\mathbb{E}_n)$
we construct the corresponding forest describing the presentation
of elements of $P(\mathbb{E}_n)$ in $B$ (consisting of the trees
of structure dual to the structure of trees from Figure~\ref{fig:E6-pres}).

We note that the structure of such a forest is much more complicated
than the structure of the tree for $R(\mathbb{E}_n)$ 
(see Table~\ref{t:PEn} for details and compare with Table~\ref{t:REn}).
\begin{table}[htbp]\small
  \caption{Complexity of the forests for $P(\mathbb{E}_n)$}
  \label{t:PEn}
  \begin{center}
\begin{tabular}{|l|c|c|c|}\hline
Dynkin type & $\mathbb{E}_6$ & $\mathbb{E}_7$ & $\mathbb{E}_8$  \\\hline
number elements of the basis $B$ of $P(\mathbb{E}_n)$ & 156 & 399 &1240 \\\hline
depth of the trees & 11 & 17 & 29 \\\hline
number of non-leaf and non-root nodes & 1551 &48887  & 47137364   \\\hline
number of tree nodes in total & 3395 & 108322 & 104753728  \\\hline
\end{tabular}
\end{center}
\end{table}

\section{Computation of admissibility condition}
\label{sec:adm}

We recall that we have the following computations
related to admissibility condition:
\begin{enumerate}[\ \  1.]
\setlength{\itemsep}{0pt}
 \item
Computation of a base of $R(\mathbb{E}_n)$ 
and presentation of elements of $R(\mathbb{E}_n)$ 
in that base.
 \item
Presentation of the function $f$ through the base elements.
 \item
Calculating the corresponding power of the function $f$.
 \item
Derivation of equations from the calculated power of $f$.
 \item
Reduction of the obtained system of equations.
 \item
Extraction of the substitutions from the reduced system of equation.
\end{enumerate}

We start with computation of a base of $R(\mathbb{E}_n)$ 
and the presentation of elements of $R(\mathbb{E}_n)$ in that base.
Outline of these computation is presented as Algorithm~\ref{alg:base-REn}.

\begin{algorithm}
  \caption{Finding the basis of $R(\mathbb{E}_n)$}
  \label{alg:base-REn}
\begin{algorithmic}[1]
\Require
 $S \gets \{x,y\}$,
 $R \gets relations$,
 \State  
 $i \gets 1$,
 $E_1 \gets S$,
 $B \gets S \cup \{ \mathbf{1}_{R(\bE_n)}\}$,
 $Pres[] \gets \{ (s,\delta_s) \,|\, s \in S\}$
 \While{$E_i \neq \emptyset$}
  \ForAll{$a \in E_{i}, s \in S$}  
    \State
     $Pres(a s) \gets \Box$ 
  \EndFor
  \State  $i \gets i+1$
  \While{$U \neq \emptyset$}
    \While{$\exists_{u \in U} \exists_{\sim_r \in R} \exists_{P \subseteq S^i\setminus U}
       \forall_{c \in P} \exists_{\alpha_c \in \{-1,+1\}}
   u\sim_r\sum_{c \in P} \alpha_c c$
       }
        \If {$\sum_{c \in P \cap E_i} \alpha_c Pres[c] \neq 0$}
          \State  
           $Pres[u] \gets \sum_{c \in P \cap E_i} \alpha_c Pres[c]$            
        \Else
          \State  
           $Pres[u] \gets \boxtimes$                   
        \EndIf
    \EndWhile
     \If {$U \neq \emptyset$}
        \State          
        $b \gets choose(U)$
        \State          
        $B \gets B \cup \{ b\}$
        \State          
        $Pres[b] \gets \delta_b$
      \EndIf
  \EndWhile
  \EndWhile
        \State          
  \Return $B$, $B \setminus (S \cup \{ \mathbf{1}_{R(\bE_n)}\})$, $Pres$, $i-1$
\end{algorithmic}
\end{algorithm}

In Algorithm~\ref{alg:base-REn} we denote by $B$ the computed 
set of basis elements of $R(\mathbb{E}_n)$
and by $Pres$ the tree of the structure presented
in Section~\ref{sec:data}.
We recall that every node of $Pres$ is associated to 
the element of $R(\mathbb{E}_n)$
corresponding to the path from the root to that node.
To every node of $Pres$ 
we assign either a mapping from $B$ to $K$
(the presentation of the associated element in the basis $B$)
or $\boxtimes$ indicating zero-path or $\Box$ indicating ``calculation in progress''.
If in $Pres$ there is the node induced by 
a path $a$, then we 
set $Pres[a]$ to the value assigned to that node
and we set $Pres[a]= \emptyset$ otherwise.

To simplify the notation we denote by $U$ the set of paths
corresponding to the nodes with value $\boxtimes$
($U = \{a \,|\, Pres[a] = \Box\}$)
and by $E_i \subseteq \{x,y \}^i$ 
the set of paths of length $i$ 
associated with nodes of $Pres$
of values different from $\boxtimes$ and $\Box$.
Moreover, for each element $b$ of $B$ we denote by
by $\delta_b$ the mapping $\delta_b : B \to K$
such that $\delta_b(b) = 1$ and $\delta_b(b') = 0$ 
for $b' \in B\setminus \{b\}$.

The function $choose()$ chooses some path from 
the given set of paths. 
Although this choice does not affect the correctness 
of the algorithm, it may have an impact
on the later calculations 
(and their computational complexity)
and 
we deal with the appropriate selection 
of this function in the next section.

We note that the statement
``$U \neq \emptyset$''
from Lines 7 and 15
is equivalent to the statement
``$\exists_{a} Pres[a] = \Box$''.
The statement from Line~8
means that there exist a path $u$
for which presentation still were not computed,
but which is in the relation $\sim_r$ with some
paths (forming the set $P$)
with known presentation in the base $B$.
Moreover, 
the statement from Line~9
means, that the combination of these
paths in the relation $\sim_r$ is non-zero.
We note, that we use the fact that all relations
are of a special form.
In particular in each of the relations 
all coefficients belong to the set $\{-1,+1\}$
and all paths are of the same length.

Observe also that the operation
$Pres[u] \gets \boxtimes$ 
from Line~12
acts as 
$U \gets U \setminus \{ u \}$,
and the operation from Line~10
acts as 
$U \gets U \setminus \{ u \}$ 
and 
$E_i \gets E_i \cup \{ u \}$.
Similar side-effects on the sets $U$ and $E_i$
are also performed in Lines 4 and 18.

We also note that above algorithm
contains some simplifications.
For example in the real implementation
$Pres[a]$ should contain only mappings on
the elements of $B$ of the same length
as $a$ (see Section~\ref{sec:data}).
We may also reduce a little the
time complexity of the calculations
when we replace
Lines 4--6
by setting 
$Pres[u] \gets \Box$
for 
$u \in \{a s\,|\, a \in E_{i-1}, s \in S \} 
   \cap \{s a\,|\, a \in E_{i-1}, s \in S \}$,
and 
$Pres[u] \gets \boxtimes$
for 
$u \in \{a s\,|\, a \in E_{i-1}, s \in S \} 
   \setminus \{s a\,|\, a \in E_{i-1}, s \in S \}$.

In the next step we 
compute $(x+y+f(x,y))^{n-3}$.
We start with some preparation.

Let $B$ be the basis of $R(\mathbb{E}_n)$,
$B' \subseteq B$ the basis of $\rad^2 R(\mathbb{E}_n)$,
and $f$ the fixed admissible element of $\rad^2 R(\mathbb{E}_n)$.
We recall that $f(x,y) = \sum_{b \in B'} \theta_b b$ 
for some fixed elements $\theta_b \in K$,
and hence 
$(x + y + f(x,y))^{n-3}$ is the sum of
elements $\omega$ of the form
\[
   \omega = k_{\omega} \prod_{b' \in B} \theta_{b'}^{\omega_{b'}} b
\]
for some integer $k \in \mathbb{Z}$, 
natural numbers $\omega_{b'} \in \mathbb{N}$,
and $b \in B'$.
We may divide these summands depending on $b$.
We denote by $\Omega_b$ the set of such summands 
indicated by the basis element $b$.
So we have 
\[
   \big(x + y + f(x,y)\big)^{n-3}
   = \sum_{b \in B} \sum_{\omega \in \Omega_b}  \omega
   = \sum_{b \in B} \bigg( \sum_{\omega \in \Omega_b}  
    k_{\omega} \prod_{b' \in B} \theta_{b'}^{\omega_{b'}}
    \bigg) b .
\]
Hence the admissible condition
$$\big(x + y + f(x,y)\big)^{n-3} = 0$$
is equivalent to the set of equations
\begin{align}
\label{eq:*}
   \bigg\{ \sum_{\omega \in \Omega_b}  
    k_{\omega} \prod_{b' \in B} \theta_{b'}^{\omega_{b'}}
 = 0
    \bigg\}_{b \in B} .
\end{align}
Clearly, the coefficients 
$k_{\omega}$, $\omega_{b'}$, for
$b' \in B$, $\omega \in \sum_{b \in B} \Omega_B$,
can be easily computed with the methods
introduced in Section~\ref{sec:data}.

Finally, we need to apply some changes 
to the admissibility condition
to reduce set of equations \eqref{eq:*} 
and modify it in such a way (see \eqref{eq:**}), 
so 
obtained
equations
could be easily used in the final computation.
We start with removing trivial equations.
Further, we sort these equations in the ascending order
of ``lengths'' of the base elements $b$.
We denote by $E_{start}$ the set of these equations.
We want to obtain a minimal set $Subst$ of  equivalent 
``substitutions'' of the form
\begin{align}
\label{eq:**}
    \bigg\{ 
   \theta_b =  \sum_{\omega \in \Omega'_b}  
    k'_{\omega} \prod_{b' \in B \setminus B''} \theta_{b'}^{\omega_{b'}}
    \bigg\}_{b \in B''} 
\end{align}
with $B'' \subseteq B$, 
$k'_{\omega}$ quotient numbers with denominator being a power of $2$,
and natural numbers $\omega_{b'} \in \mathbb{N}$.

\begin{algorithm}
  \caption{Finding ``substitutions'' equivalent to the admissibility condition}
  \label{alg:adm}
\begin{algorithmic}[1]
\Require
 $Eq \gets E_{start}$,
 \State  
 $Subst \gets \emptyset$,
 \While{$Eq \neq \emptyset$}
  \State $eq \gets min(Eq)$ 
  \State $Eq \gets Eq \setminus \{ eq \}$ 
  \If {$app(eq,Subst)$ is non-trivial}
     \State $eqsubst \gets app(eq,Subst)$
     \If {$\exists_{u \in max(eqsubst)} u_{base(u)} \in \{-1,1\}$}
        \State $Subst \gets Subst \cup \{ \mbox{substitution of $\theta_{base(u)}$ derived from } eqsubst \}$      
     \ElsIf {$\exists_{u \in max(eqsubst)} u_{base(u)} \in \{-2,2\}$}
        \State $Subst \gets Subst \cup \{ \mbox{substitution of $\theta_{base(u)}$ derived from } eqsubst \}$      
     \Else 
        \State \textbf{throw error}
     \EndIf
  \EndIf
  \EndWhile
        \State          
  \Return $Subst$
\end{algorithmic}
\end{algorithm}

We denote by $min(E)$ a function which 
choose the equation associated with the element
$b \in B$ of minimal length from some set of equations
$E \subseteq E_{start}$.
(We note that this function is not necessary unique.)
Moreover by $appl(eq,S)$ we 
mean the equation obtained from the equation $eq \in E_{start}$ 
by applying
all substitutions from the set $S \subseteq Subst$.
We also denote by $max(eq)$ 
the set of coefficients 
$u = \prod_{a \in B} \theta_{a}^{u_{a}}$ 
of the equation $eq$ 
\[
 \sum_{u \in \Omega} k_{u} 
  \prod_{a \in B} \theta_{a}^{u_{a}} = 0
\]
(with $k_u \in \mathbb{Q}$, $u_a \in \mathbb{N}$),
such that there
exists $a \in B$ satisfying 
$u_a \neq 0$,
$u_{a'} = 0$ for all $a' \in B \setminus \{a\}$,
and $a$ is of maximal length of that property.
By $base(u)$ we denote then such an element $a \in B$.
Then
Algorithm~\ref{alg:adm}
describes the procedure 
of construction of a sequence
of substitutions of selected coefficients
$\theta_b$
of $f$
which are equivalent to the admissibility condition. 

We note that in general the result of the above steps 
may depend on the characteristic of the field $K$.
In particular, in the case of type $\mathbb{F}_4$
we obtain the equation which is trivial 
if and only if the field $K$ is of characteristic $2$
(see the the proof of \cite[Lemma]{B:F4}).
We recall that in our computations 
(for the types $\mathbb{E}_7$ and $\mathbb{E}_8$)
we assume
that $K$ is not of characteristic $2$.
We note that in the case of the types 
$\mathbb{E}_7$ and $\mathbb{E}_8$
choosing element $u$ satisfying 
the condition from line $7$ 
or
the condition from line $9$ 
of the above algorithm
it always is possible 
(we refer for the details to
Sections
\ref{secA:adm}
and
\ref{secB:adm}).

We note that above computations in the case of Dynkin types
$\mathbb{E}_6$ and $\mathbb{F}_4$ are trivial,
because in these cases we have very few equations with very few coefficients,
and all of them being of simple form
(see \cite[Lemma]{B:E6} and \cite[Lemma]{B:F4} for details).
But on the other hand, it is much more complicated in the case
of type $\bE_8$.
In the 
Table~\ref{t:adm-cond}
we comparing the complexity of computations
for types 
$\mathbb{E}_6$, $\mathbb{E}_7$ and $\mathbb{E}_8$.
\begin{table}[htbp]\small
  \caption{Complexity of the computation of the admissibility condition}
  \label{t:adm-cond}
  \begin{center}
\begin{tabular}{|l|c|c|c|}\hline
Dynkin type & $\mathbb{E}_6$ & $\mathbb{E}_7$ & $\mathbb{E}_8$  \\\hline
number of elements of the basis $B$ & 12 & 24 & 60 \\\hline
number of equations in $E_{start}$ & 2 & 6 & 30 \\\hline
number of substitutions in $Subst$ & 2 & 3 & 10 \\\hline
maximal number of coefficients of single equation & 7 & 28 & 1333  \\\hline
maximal number of coefficients of single substitution & 5 & 10 & 824  \\\hline
\end{tabular}
\end{center}
\end{table}
The last row of 
this 
table contains numbers corresponding to the
choices of the calculations described in
\cite[Lemma]{B:E6},
\ref{secA:adm}
and
\ref{secB:adm},
respectively.
We note that the coefficients of equations 
for the type $\mathbb{E}_8$
are usually also of much more complicated form.

\section{Computation of a base of $P(\mathbb{E}_n)$}
\label{sec:baseEn}

In this section we construct  Algorithm~\ref{alg:base-PEn}  computing 
 a suitable  $K$-basis $B$ of the $K$-algebra $P(\bE_n)$ and  a $K$-linear presentation
of   elements of $P(\bE_n)$ in the basis $B$. 
We note that it is similar to the dual Algorithm~\ref{alg:base-REn}
presented in Section~\ref{sec:adm}.

Let $Q_{\bE_n}$ be the Gabriel quiver of $P(\bE_n)$ and 
$I$ be the set of relations of $K Q_{\bE_n}$ defined in
Section~\ref{sec:intro}. 
Then $P(\bE_n) \cong K Q_{\bE_n} / I$
for the ideal $I$ of $K Q_{\bE_n}$ generated by $R$. 
Following Section~\ref{sec:hom} we denote
$Q_{\bE_n}=(Q_0,Q_1,s,t)$, 
where $Q_0$ is the set of vertices of $Q$,
where $Q_1$ is the set of arrows of $Q$,
and $s, t : Q_1 \to Q_0$
assigns to every arrow its source and target, respectively.

In our algorithm $B$ denote the set of basis elements of $P(\mathbb{E}_n)$
and $Pres$ denote the forest of trees of the structure presented
in Section~\ref{sec:data}.
We recall that each node of $Pres$ is associated to 
some uniquely determined path of $K Q_{\bE_n}$
(corresponding to the path from root of the tree to that node).
To each node of $Pres$ we assign either a mapping from $B$ to $K$
(the presentation of the associated element in the basis $B$)
or $\boxtimes$ indicating zero-path or $\Box$ indicating ``calculation in progress''.
If in $Pres$ there is the node induced by 
a path $\omega$ of $K Q_{\bE_n}$, then we 
set $Pres[\omega]$ to the value assigned to that node
and otherwise we set $Pres[\omega]= \emptyset$.

To simplify the notation we denote  as before by $U$ the set of paths
corresponding to the nodes with value $\boxtimes$
($U = \{\omega \,|\, Pres[\omega] = \Box\}$)
and by $E_i \subseteq Q_1^i$ 
the set of paths of length $i$ 
associated with nodes of $Pres$ of values different from $\boxtimes$ and $\Box$.

Moreover, for each element $b$ of $B$, we denote by
by $\delta_b$ the mapping $\delta_b : B \to K$
such that $\delta_b(b) = 1$ and $\delta_b(b') = 0$ 
for $b' \in B\setminus \{b\}$.

\begin{algorithm}
  \caption{Finding the basis of $P(\mathbb{E}_n)$}
  \label{alg:base-PEn}
\begin{algorithmic}[1]
\Require
$Q_{\bE_n}=(Q_0,Q_1,s,t)$, 
$R$ 
such that
$K Q_{\bE_n} / \langle R \rangle \cong P(\bE_n)$
 \State  
 $i \gets 1$,
 $E_1 \gets Q_1$,
 $B \gets Q_1 \cup \{ \mathbf{1}_{KQ/\langle R \rangle}\}$,
 $Pres[] \gets \{ (\alpha,\delta_{\alpha}) \,|\, \alpha \in Q_1\}$
 \While{$E_i \neq \emptyset$}
  \ForAll{$\omega \in E_{i}, \alpha \in Q_1$ with $t(\omega) = s(\alpha)$}  
    \State
     $Pres(\omega \alpha) \gets \Box$ 
  \EndFor
  \State  $i \gets i+1$
  \While{$U \neq \emptyset$}
    \While{$\exists_{\omega \in U} \exists_{\sim \in R} \exists_{P \subseteq Q_1^i\setminus U}
       \forall_{c \in P} \exists_{\alpha_c \in \{-1,+1\}}
   u\sim \sum_{c \in P} \alpha_c c$
       }
        \If {$\sum_{c \in P \cap E_i} \alpha_c Pres[c] \neq 0$}
          \State  
           $Pres[\omega] \gets \sum_{c \in P \cap E_i} \alpha_c Pres[c]$            
        \Else
          \State  
           $Pres[\omega] \gets \boxtimes$                   
        \EndIf
    \EndWhile
     \If {$U \neq \emptyset$}
        \State          
        $b \gets choose(U)$
        \State          
        $B \gets B \cup \{ b\}$
        \State          
        $Pres[b] \gets \delta_b$
      \EndIf
  \EndWhile
  \EndWhile
        \State          
  \Return $B$, $Pres$, $i-1$
\end{algorithmic}
\end{algorithm}

We again made some simplifications.
We omitted details on the structure of $B$
and simplified the initialization of $Pres$, as before.
It also should be mention that 
in a practical implementation 
the calculations of the loop \emph{while} from Lines 7--20
should be divided in such a way,
so the calculations on the paths with different
source (respectively, target) could be performed
separately.

We note that the choice of the basis generators of $R(\bE_n)$
and $P(\bE_n)$ 
(performed by the function $chose()$ in 
Algorithms \ref{alg:base-REn} and~\ref{alg:base-PEn})
should be synchronized.
Moreover, this choice have a significant impact on the complexity
of generated equations 
and in the consequence on the computational complexity
of the final calculations.
In Tables~\ref{t:coefficients-ff} and \ref{t:coefficients-sort} 
we comparing number of the coefficients in the obtained equations
for different implementations of the function $chose()$.

In our implementation of the function $choice()$
we threat the paths as words 
over the alphabet consisting of the arrows
with the order defined by the following
sequence
$a_0, \bar{a}_0, a_2, \bar{a}_2, \bar{a}_3, a_3, a_1, \bar{a}_1$, 
$\bar{a}_4, a_4, \dots, \bar{a}_{n-2}, {a}_{n-2}$,
and we chose the first path in the lexicographical order.
In other words we use the rule
``closest to the exceptional vertex first''
and then the rule
``from the shortest branch first''.
We omit the detailed explanation of why this choice
is expected to be ``good''.
But in the next section we compare the complexity of the
formulas
obtained 
in this way
with the complexity of formulas
obtained
with the implementation of the ``first found'' approach
(see Tables \ref{t:coefficients-ff} and \ref{t:coefficients-sort}).

\section{Construction of the system of equations}
\label{sec:eq}

We keep the notation from Corollary~\ref{cor:hom2}.
We recall that by Corollary~\ref{cor:hom2} 
to prove that  for a given admissible $f \in \rad^2 R(\mathbb{E}_n)$ 
$K$-algebras $P^f(\mathbb{E}_n)$ and $P(\mathbb{E}_n)$ 
are isomorphic 
it suffices to 
find coefficients $\alpha(k,l,j) \in K$
such that 
there are satisfied the following $n$ equalities in $P(\mathbb{E}_n)$
\begin{gather*}
\delta(k_1,k_2) =0
\mbox{ for } (k_1,k_2) \in \{ (0,n-1), (1,n), (2n-3, n-2)\};
\\
\delta(k+n-2,k-1)+\delta(k,k+n-1) =0
\mbox{ for } k \in \{ 2, 4, \dots, n-2 \};
\\
\delta(n-1,0)+\delta(n+1,2)+\delta(2n-3,n-2) +
f\big(\delta(n-1,0),\delta(n+1,2)\big) = 0  
\end{gather*}
(one equality for each of the vertices)
where $\delta : \{0,\dots,n-3\}^2 \rightarrow P(\mathbb{E}_n)$
are defined as follows
\begin{align*}
  \delta(k_1,k_2) &=
   \sum_{l = 2}^m \sum_{j = 0}^{i(k_1,l) -1} \alpha(k_1,l,j)   a_{i(k_1,l,j,0)} \dots a_{i(k_1,l,j,l -1)} a_{k_2}
  + 
   \sum_{l = 2}^m \sum_{j = 0}^{i(k_2,l) -1} \alpha(k_2,l,j)   a_{k_1} a_{i(k_2,l,j,0)} \dots a_{i(k_2,l,j,l -1)}
 \\&\quad
  + 
   \bigg( \sum_{l = 2}^m \sum_{j = 0}^{i(k_1,l) -1} \alpha(k_1,l,j)  a_{i(k_1,l,j,0)} \dots a_{i(k_1,l,j,l -1)}\bigg)  
\,
   \bigg(\sum_{l = 2}^m \sum_{j = 0}^{i(k_2,l) -1} \alpha(k_2,l,j) a_{i(k_2,l,j,0)} \dots a_{i(k_2,l,j,l -1)}\bigg) .
\end{align*}

We recall also that the basis elements $P(\mathbb{E}_n)$
can be determined by Algorithm~\ref{alg:base-PEn} 
and the operations in $P(\mathbb{E}_n)$
were described in Section~\ref{sec:data} (with the use of the
presentation also provided by Algorithm~\ref{alg:base-PEn}).
Moreover, $f$ can be presented as a polynomial with
coefficients bound by the relations determined
by the admissibility condition, 
and hence satisfying the equations $Subst$
obtained by Algorithm~\ref{alg:adm}.
So in the consequence we should find
the coefficients  $\alpha(k,l,j) \in K$
for each set of coefficients of $f$
satisfying $Subst$.

For every non-exceptional vertex $k \in \{0,\dots,n-1\} \setminus \{3\}$
the corresponding equality in $P(\mathbb{E}_n)$
can be presented in the form of $\dim \rad^3 e_k P(\mathbb{E}_n) e_k$
equations in $K$ (one equation for each of the basis elements of 
$\rad^3 e_k P(\mathbb{E}_n) e_k$).
These equations can be easily obtained by the symbolic
calculations.

In the case of exceptional vertex $3$ we need
to calculate $f(\delta(n-1,0),\delta(n+1,2))$,
and apply to it the equations (substitutions) from $Subst$.
Then we may observe that the equation
\[\delta(n-1,0)+\delta(n+1,2)+\delta(2n-3,n-2) +
f\big(\delta(n-1,0),\delta(n+1,2)\big) = 0\]
can be presented as 
the set of
$\dim \rad^3 e_3 P(\mathbb{E}_n) e_3$
equations over $K$
with 
$\dim \rad^2 R(\mathbb{E}_n) - \# Subst$
invariables.
We also note that in some cases
(in particular in the case of $\mathbb{E}_8$)
it is convenient to 
refrain from applying equations from $Subst$
directly to $\dim \rad^3 e_3 P(\mathbb{E}_n) e_3$
(as it can significantly enlarge its notation)
and instead to apply them in the final stage
(see \ref {secB:solutions}  and \cite[\texttt{commands.txt}]{B:E8zip}
for details).

In Tables~\ref{t:coefficients-ff} and \ref{t:coefficients-sort} 
we compare the numbers of   coefficients in the obtained equations
for different implementations of the function $chose()$
(see the notes from the previous section)
before applying the 
equations implied by the admissibility condition
(see Algorithm~\ref{alg:adm}).

\begin{table}[htbp]\small
  \caption{Number of coefficients for ``first found'' approach}
  \label{t:coefficients-ff}
  \begin{center}
\begin{tabular}{|c|*{8}{c|}}\hline
Dynkin & \multicolumn{8}{c|}{number of the coefficients for the vertex} \\
type & 0 & 1 & 2 & 3 & 4 & 5 & 6 & 7 \\\hline\hline 
$\mathbb{E}_6$ & 11 & 2 & 25 & 351 & 25 & 2 & -- & -- \\\hline 
$\mathbb{E}_7$ & 42 & 11 & 116 & 12892 & 176 & 43 & 6  & --  \\\hline 
$\mathbb{E}_8$ & 243 & 48 & 590 & 7155925 & 1173 & 429 & 96  & 11  \\\hline 
\end{tabular}\end{center}
\end{table}

\begin{table}[htbp]\small
  \caption{Number of coefficients in equations for the choosen approach}
  \label{t:coefficients-sort}
  \begin{center}
\begin{tabular}{|c|*{8}{c|}}\hline
Dynkin & \multicolumn{8}{c|}{number of the coefficients for the vertex} \\
type & 0 & 1 & 2 & 3 & 4 & 5 & 6 & 7 \\\hline\hline 
$\mathbb{E}_6$ & 11 & 2 & 21 & 347 & 27 & 2 & -- & -- \\\hline 
$\mathbb{E}_7$ & 39 & 11 & 100 & 7173 & 168 & 43 & 6  & -- \\\hline 
$\mathbb{E}_8$ & 238 & 53 & 539 & 3790572 & 1190 & 455 & 92  & 11  \\\hline
\end{tabular}
\end{center}
\end{table}

We note that
the chosen approach 
let us
not only to reduce the total 
number of coefficients for the equations, but also 
significantly reduce the maximal number of coefficients peer equation.
Indeed, observe that 
$\dim \rad^3 e_3 P(\mathbb{E}_8) e_3 = \dim \rad^2 R(\mathbb{E}_8) = 57$,
so in the case of $\mathbb{E}_8$ we 
obtain $57$ equations associated with vertex $3$.
For the ``first found'' approach they have respectively 
\begin{align*}
&
6, 5, 7, 12, 13, 18, 17, 59, 85, 31, 34, 56, 126, 143, 154, 99, 
149, 110, 311, 
320, 
415, 581, 440, 545, 868, 
980, 
\\&
2050, 1309, 848, 684, 
2506, 
3986, 
4405, 
1888, 3842, 2536, 6485, 11025, 12346, 4985, 6581, 11260, 
32466,
\\&
34351, 
14875, 22074, 25469, 59074, 97983, 103449, 
57475, 
295418, 213548,
180430, 756019, 399349, 1738127
\end{align*}
coefficients.
In contrary, for the chosen approach they have
following number of coefficients:
\begin{align*}
&
5, 5, 5, 9, 11, 11, 11, 21, 22, 21, 23, 23, 83, 89, 94, 89, 95, 
90, 185, 
344, 
344, 184, 145, 149, 381, 
612, 
\\&
299, 576, 613, 600, 
1172, 1745,
1195, 1166, 
1182, 608, 2294, 2464, 2458, 2877, 3534, 2475, 
11010, 
\\&
12931, 9515,
9337, 
7187, 37259, 14598, 16878, 
14513, 
52420,  26975, 
51522, 127237, 86634,
295466.
\end{align*}
Moreover, in the latter approach we obtain more equations 
with ``small'' number of coefficients. 
We expect that this should simplify the substitutions
and in result reduce the complexity of solving the set of
equations.%
\footnote{We estimate, that in the case of $\mathbb{E}_7$,
where the difference in the complexity of equations is smaller
than in the case of $\mathbb{E}_8$,
solving the set of equations obtained with the chosen approach
(without applying the equations obtained from the admissibility condition)
is about 7--8 times faster than 
solving the set of equations obtained with the ``first found'' approach.
Unfortunately, the complexity of this problem in the case of $\mathbb{E}_8$
is to high to be able to make the similar estimations.}

\section{On solutions of the obtained system of equations}
\label{sec:solv}

We note that obtained systems of equations 
in the case of types $\mathbb{E}_7$ and $\mathbb{E}_8$ 
are to complicated to be solved manually.
So natural approach is to use
some dedicated software package
in order to try to solve them.

Simple comparing number of equations to number of variables
make us expect 
that if such a system of equations has a solution,
then the space of solutions is rather big.
On the other hand, we are not interested in finding
all solutions but 
we only want
\begin{itemize}
\setlength{\itemsep}{0pt}
 \item
  to determine if the obtained system of equations has a solution,
 \item
  if it is the case, then to find one, 
  possibly ``simple'' exemplary solution.
\end{itemize}

If problem of solving the obtained
system of equation is still of acceptable computational complexity,
then we may use the following approach:
\begin{itemize}
\setlength{\itemsep}{0pt}
 \item
  obtain the general form of the solution,
 \item
  replace some of the ``not-bounded'' variables by ``0'' 
  until we obtain a single solution (a 0-dimensional space of solutions).
\end{itemize}
It works fine in the case of type $\mathbb{E}_7$
(see \ref{secA:simpl}).

Unfortunately, in the case of type $\mathbb{E}_8$ 
computational complexity of the problem of
solving the original obtained system of equations 
seems to be a little too high.
In that case we may try to act differently:
\begin{enumerate}
\setlength{\itemsep}{0pt}
 \item
 \label{reduce:0}
  Try to replace some of the variables by ``0'' 
  in order to reduce the
  computational complexity 
  of the problem 
of solving such a system
  and check if the obtained system of equations
  is solvable and if we can find its solution.
 \item
 \label{reduce:01}
  If we obtain a system for which we can find 
  some non-empty space of solutions, then 
  replace some of the (remaining) ``not-bounded'' 
  variables by ``0'' and ``1''
  until we obtain a 0-dimensional space of solutions.
\end{enumerate}
This approach seems to work in the case of $\mathbb{E}_8$
(see  \ref{secB:simpl} and  \ref{secB:solutions}),
but it has some disadvantages:
\begin{itemize}
\setlength{\itemsep}{0pt}
 \item
  The process of finding the ``proper''
  variables to be replaced by ``0'' 
  in the step \eqref{reduce:0}
  is unintuitive and verifying 
  its effects takes a lot of time
  (but still less then trying to solve
   the initial system of equations).
 \item
  The process of finding the ``proper''
  variables to be replaced by ``0'' and ``1''
  in the step \eqref{reduce:01}
  requires prior analysis 
  of big amount of data
  and verifying its correctness 
  takes a lot of computation (and time).
 \item
  The structure of the obtained
  solution 
  seems to be 
  a little artificial and
  more complicated
  then necessary. 
  We note that it can be the result
  of not optimal choices made in the step 
  \eqref{reduce:0}.
\end{itemize}

\begin{appendix}
\section{Calculations for type $\mathbb{E}_7$}
\label{app:E7}

In this appendix 
we present 
construction of 
a $K$-algebra isomorphism 
$\varphi : P^f(\mathbb{E}_7) \to P(\mathbb{E}_7)$
for a given admissible deforming element $f$
according to the algorithm described in the main part of the article.
This construction is organized as follows.
In \ref{secA:adm} we derive the equations
for the coefficients of $f$ indicated by 
the admissibility condition of $f$.
In \ref{secA:hom} we present the general
form of the homomorphisms from
$P^f(\mathbb{E}_7)$ to $P(\mathbb{E}_7)$
(in particular we give the chosen basis elements
of $P(\mathbb{E}_7)$).
In \ref{secA:simpl} we describe the
chosen method 
(see Section~\ref{sec:solv} for details)
of simplifying the system of equations
for the coefficients of an isomorphism from
$P^f(\mathbb{E}_7)$ to $P(\mathbb{E}_7)$.
Finally, in \ref{secA:coeff} 
we present obtained coefficients of a (simplified) 
isomorphism from
$P^f(\mathbb{E}_7)$ to $P(\mathbb{E}_7)$.

We refer to the main part of the article for description of
applied algorithms
as well as for theoretical background.
We recall that 
considered case
is much more complicated than the case of type $\mathbb{E}_6$ 
(described in \cite{B:E6}) and it is not possible to perform
manually calculations described in 
the main part of the article.
But on the other hand the case of type is type $\mathbb{E}_8$ 
is even more complicated and need a little different approach
(see Section~\ref{sec:solv} and \ref{app:E8} for details).

\subsection{Equations derived from the admissibility condition}
\label{secA:adm}

We carry out the steps described in Section~\ref{sec:adm}.
First we compute a base of 
$R(\mathbb{E}_7)$
(see Algorithm~\ref{alg:adm} for details).
We 
take as a basis the set
$\{
1,
x, y, 
xy, yx, yy,
xyx, xyy, yxy, yyx,
xyxy,
xyyx,
yxyx,
\linebreak
yxyy,
xyxyx,
xyxyy,
yxyxy,
yxyyx,
xyxyxy,
xyxyyx,
yxyxyx,
xyxyxyx,
yxyxyxy,
xyxyxyxy
\}$. 
Hen\-ce we know that each element $f \in \rad R(\mathbb{E}_7)$ is of the form
\begin{align*}
f(x,y) &= 
\theta_1 xy + \theta_2 yx + \theta_3 yy
+ \theta_4 xyx + \theta_5 xyy + \theta_6 yxy + \theta_7 yyx
+ \theta_8 xyxy
+ \theta_9 xyyx
\\&\quad
+ \theta_{10} yxyx
+ \theta_{11} yxyy
+ \theta_{12} xyxyx
+ \theta_{13} xyxyy
+ \theta_{14} yxyxy
+ \theta_{15} yxyyx
+ \theta_{16} xyxyxy
\\&\quad
+ \theta_{17} xyxyyx
+ \theta_{18} yxyxyx
+ \theta_{19} xyxyxyx
+ \theta_{20} yxyxyxy
+ \theta_{21} xyxyxyxy
\end{align*}
for some $\theta_1, \dots, \theta_{21} \in K$.
We fix these coefficients.
Further, we compute 
\begin{align*}
\big(x+y&+f(x,y)\big)^4 = 
(\theta_{2}-2\theta_{3}+\theta_{1})yxyxy
+(3\theta_{4}-2\theta_{5}+\theta_{6}-2\theta_{7}+2\theta_{1}^2
-2\theta_{1}\theta_{3}+\theta_{2}^2-\theta_{3}^2)xyxyxy
\\&
+(3\theta_{4}-2\theta_{5}+\theta_{6}-2\theta_{7}+\theta_{1}^2
+2\theta_{2}^2-2\theta_{2}\theta_{3}-\theta_{3}^2)yxyxyx
+ (3\theta_{8}-2\theta_{9}+3\theta_{10}-2\theta_{11}
+3\theta_{1}\theta_{4}
\\&\ \quad
-2\theta_{1}\theta_{5}+3\theta_{2}\theta_{4}
-2\theta_{2}\theta_{7}
-2\theta_{3}\theta_{5}
+2\theta_{3}\theta_{6}
-2\theta_{3}\theta_{7}+\theta_{1}^2\theta_{2}+\theta_{1}\theta_{2}^2
-4\theta_{1}\theta_{2}\theta_{3}+\theta_{2}^3+\theta_{1}^3)xyxyxyx
\\&
+(3\theta_{8}-2\theta_{9}+3\theta_{10}-2\theta_{11}+2\theta_{1}\theta_{4}
-2\theta_{1}\theta_{5}+\theta_{1}\theta_{6}+2\theta_{2}\theta_{4}
+\theta_{2}\theta_{6}
-2\theta_{2}\theta_{7}
\\&\ \quad
+2\theta_{3}\theta_{4}
-2\theta_{3}\theta_{5}-2\theta_{3}\theta_{7}+\theta_{1}^2\theta_{2}
+\theta_{1}\theta_{2}^2-4\theta_{1}\theta_{2}\theta_{3}+\theta_{2}^3
+\theta_{1}^3)yxyxyxy
\\&
+ (4\theta_{1}\theta_{8}-2\theta_{1}\theta_{9}+2\theta_{1}\theta_{10}
-2\theta_{1}\theta_{11}-2\theta_{2}\theta_{8}+2\theta_{2}\theta_{9}
-4\theta_{2}\theta_{10}+2\theta_{2}\theta_{11}
-2\theta_{3}\theta_{8}
+2\theta_{3}\theta_{10}-2\theta_{1}\theta_{2}\theta_{7}
\\&\ \quad
+2\theta_{1}\theta_{3}\theta_{6}+2\theta_{1}^2\theta_{4}
-2\theta_{1}^2\theta_{5}-2\theta_{2}^2\theta_{4}+2\theta_{2}^2\theta_{7}
-2\theta_{2}\theta_{3}\theta_{4}+2\theta_{2}\theta_{3}\theta_{5}
+2\theta_{1}\theta_{2}\theta_{5}+2\theta_{1}\theta_{3}\theta_{4}
-2\theta_{1}\theta_{3}\theta_{5}
\\&\ \quad
-2\theta_{2}\theta_{3}\theta_{6}
-2\theta_{1}\theta_{3}\theta_{7}+2\theta_{2}\theta_{3}\theta_{7}
+\theta_{1}^4-4\theta_{1}^2\theta_{2}\theta_{3}
+4\theta_{1}\theta_{2}^2\theta_{3}
-\theta_{2}^4)xyxyxyxy
\end{align*}
(see Sections \ref{sec:adm} and \ref{sec:baseEn} for details).
Hence $f$ satisfy the admissibility condition if and only if the following
equalities are satisfied:
\begin{align*}
\theta_{2}-2\theta_{3}+\theta_{1}
&= 0, \\
3\theta_{4}-2\theta_{5}+\theta_{6}-2\theta_{7}+2\theta_{1}^2
-2\theta_{1}\theta_{3}+\theta_{2}^2-\theta_{3}^2
&= 0, \\
3\theta_{4}-2\theta_{5}+\theta_{6}-2\theta_{7}+\theta_{1}^2
+2\theta_{2}^2-2\theta_{2}\theta_{3}-\theta_{3}^2
&= 0, \\
3\theta_{8}-2\theta_{9}+3\theta_{10}-2\theta_{11}
+3\theta_{1}\theta_{4}-2\theta_{1}\theta_{5}+3\theta_{2}\theta_{4}
-2\theta_{2}\theta_{7}
\\
-2\theta_{3}\theta_{5}
+2\theta_{3}\theta_{6}
-2\theta_{3}\theta_{7}+\theta_{1}^2\theta_{2}+\theta_{1}\theta_{2}^2
-4\theta_{1}\theta_{2}\theta_{3}+\theta_{2}^3+\theta_{1}^3
&= 0, \\
3\theta_{8}-2\theta_{9}+3\theta_{10}-2\theta_{11}+2\theta_{1}\theta_{4}
-2\theta_{1}\theta_{5}+\theta_{1}\theta_{6}+2\theta_{2}\theta_{4}
+\theta_{2}\theta_{6}
\\
-2\theta_{2}\theta_{7}
+2\theta_{3}\theta_{4}
-2\theta_{3}\theta_{5}-2\theta_{3}\theta_{7}+\theta_{1}^2\theta_{2}
+\theta_{1}\theta_{2}^2-4\theta_{1}\theta_{2}\theta_{3}+\theta_{2}^3
+\theta_{1}^3
&= 0, \\
4\theta_{1}\theta_{8}-2\theta_{1}\theta_{9}+2\theta_{1}\theta_{10}
-2\theta_{1}\theta_{11}-2\theta_{2}\theta_{8}+2\theta_{2}\theta_{9}
-4\theta_{2}\theta_{10}+2\theta_{2}\theta_{11}
-2\theta_{3}\theta_{8}
+2\theta_{3}\theta_{10}
\\
-2\theta_{1}\theta_{2}\theta_{7}
+2\theta_{1}\theta_{3}\theta_{6}+2\theta_{1}^2\theta_{4}
-2\theta_{1}^2\theta_{5}-2\theta_{2}^2\theta_{4}
+2\theta_{2}^2\theta_{7}
-2\theta_{2}\theta_{3}\theta_{4}+2\theta_{2}\theta_{3}\theta_{5}
+2\theta_{1}\theta_{2}\theta_{5}
\\
+2\theta_{1}\theta_{3}\theta_{4}
-2\theta_{1}\theta_{3}\theta_{5}
-2\theta_{2}\theta_{3}\theta_{6}
-2\theta_{1}\theta_{3}\theta_{7}+2\theta_{2}\theta_{3}\theta_{7}
+\theta_{1}^4-4\theta_{1}^2\theta_{2}\theta_{3}
+4\theta_{1}\theta_{2}^2\theta_{3}
-\theta_{2}^4
&= 0
.
\end{align*}
We recall that following 
\cite[Theorem]{B:soc}
there is a non-isomorphic deformation of $P(\mathbb{E}_7)$
in characteristic $2$.
Hence we may assume that $K$ is of characteristic different from $2$.
Then, applying Algorithm~\ref{alg:adm} to these equations
we obtain the following equivalent set of equations:
\begin{align*}
\theta_2&=2\theta_3-\theta_1,
\\
\theta_6&=-3\theta_1^2+6\theta_1\theta_3-3\theta_3^2-3\theta_4+2\theta_5+2\theta_7,
\\
\theta_{11}&=\theta_1^2\theta_3-2\theta_1\theta_3^2+\theta_3^3
-\theta_1\theta_5+\theta_1\theta_7+\theta_3\theta_5
-\theta_3\theta_7-\theta_9+\tfrac{3}{2}(\theta_{10}+\theta_8)
.
\end{align*}

\subsection{General form of homomorphism}
\label{secA:hom}

Let $f$ be a given admissible element 
of the structure described in 
\ref{secA:adm}.
In order to construct an isomorphism 
$\varphi : P^f(\mathbb{E}_7) \to P(\mathbb{E}_7)$
we want to  
find the coefficients satisfying the assumptions
of Corollary~\ref{cor:hom2}.
We start with 
computing the base of $P(\mathbb{E}_7)$,
according to Algorithm~\ref{alg:base-PEn}.
In particular, for each arrow $\alpha \in Q_1$ we
compute a basis of 
$e_{s(\alpha)} P(\mathbb{E}_7) e_{t(\alpha)}$.
Then we conclude that our constructed isomorphism
should be given by the following formulas
\begin{footnotesize}
\setlength{\jot}{0pt}
\begin{align*}
\varphi(a_{0}) &= 
a_{0}
+ \alpha^{(0)}_{1} a_{0} \bar{a}_{2} a_{2}
+ \alpha^{(0)}_{2} a_{0} \bar{a}_{2} a_{2} \bar{a}_{0} a_{0}
+ \alpha^{(0)}_{3} a_{0} \bar{a}_{2} a_{2} \bar{a}_{2} a_{2}
+ \alpha^{(0)}_{4} a_{0} \bar{a}_{2} a_{2} \bar{a}_{0} a_{0} \bar{a}_{2} a_{2}
+ \alpha^{(0)}_{5} a_{0} \bar{a}_{2} a_{2} \bar{a}_{2} a_{2} \bar{a}_{0} a_{0}
\\&\quad
+ \alpha^{(0)}_{6} a_{0} \bar{a}_{2} a_{2} \bar{a}_{0} a_{0} \bar{a}_{2} a_{2} \bar{a}_{0} a_{0}
+ \alpha^{(0)}_{7} a_{0} \bar{a}_{2} a_{2} \bar{a}_{0} a_{0} \bar{a}_{2} a_{2} \bar{a}_{2} a_{2}
+ \alpha^{(0)}_{8} a_{0} \bar{a}_{2} a_{2} \bar{a}_{0} a_{0} \bar{a}_{2} a_{2} \bar{a}_{0} a_{0} \bar{a}_{2} a_{2}
\\&\quad
+ \alpha^{(0)}_{9} a_{0} \bar{a}_{2} a_{2} \bar{a}_{0} a_{0} \bar{a}_{2} a_{2} \bar{a}_{2} a_{2} \bar{a}_{0} a_{0}
+ \alpha^{(0)}_{10} a_{0} \bar{a}_{2} a_{2} \bar{a}_{0} a_{0} \bar{a}_{2} a_{2} \bar{a}_{0} a_{0} \bar{a}_{2} a_{2} \bar{a}_{0} a_{0}
+ \alpha^{(0)}_{11} a_{0} \bar{a}_{2} a_{2} \bar{a}_{0} a_{0} \bar{a}_{2} a_{2} \bar{a}_{0} a_{0} \bar{a}_{2} a_{2} \bar{a}_{0} a_{0} \bar{a}_{2} a_{2}
,
\\
\varphi(\bar{a}_{0}) &= 
\bar{a}_{0}
+ \bar{\alpha}^{(0)}_{1} \bar{a}_{2} a_{2} \bar{a}_{0}
+ \bar{\alpha}^{(0)}_{2} \bar{a}_{0} a_{0} \bar{a}_{2} a_{2} \bar{a}_{0}
+ \bar{\alpha}^{(0)}_{3} \bar{a}_{2} a_{2} \bar{a}_{2} a_{2} \bar{a}_{0}
+ \bar{\alpha}^{(0)}_{4} \bar{a}_{0} a_{0} \bar{a}_{2} a_{2} \bar{a}_{2} a_{2} \bar{a}_{0}
+ \bar{\alpha}^{(0)}_{5} \bar{a}_{2} a_{2} \bar{a}_{0} a_{0} \bar{a}_{2} a_{2} \bar{a}_{0}
\\&\quad
+ \bar{\alpha}^{(0)}_{6} \bar{a}_{0} a_{0} \bar{a}_{2} a_{2} \bar{a}_{0} a_{0} \bar{a}_{2} a_{2} \bar{a}_{0}
+ \bar{\alpha}^{(0)}_{7} \bar{a}_{2} a_{2} \bar{a}_{0} a_{0} \bar{a}_{2} a_{2} \bar{a}_{2} a_{2} \bar{a}_{0}
+ \bar{\alpha}^{(0)}_{8} \bar{a}_{0} a_{0} \bar{a}_{2} a_{2} \bar{a}_{0} a_{0} \bar{a}_{2} a_{2} \bar{a}_{2} a_{2} \bar{a}_{0}
\\&\quad
+ \bar{\alpha}^{(0)}_{9} \bar{a}_{2} a_{2} \bar{a}_{0} a_{0} \bar{a}_{2} a_{2} \bar{a}_{0} a_{0} \bar{a}_{2} a_{2} \bar{a}_{0}
+ \bar{\alpha}^{(0)}_{10} \bar{a}_{0} a_{0} \bar{a}_{2} a_{2} \bar{a}_{0} a_{0} \bar{a}_{2} a_{2} \bar{a}_{0} a_{0} \bar{a}_{2} a_{2} \bar{a}_{0}
+ \bar{\alpha}^{(0)}_{11} \bar{a}_{2} a_{2} \bar{a}_{0} a_{0} \bar{a}_{2} a_{2} \bar{a}_{0} a_{0} \bar{a}_{2} a_{2} \bar{a}_{0} a_{0} \bar{a}_{2} a_{2} \bar{a}_{0}
,
\\
\varphi(a_{1}) &= 
a_{1}
+ \alpha^{(1)}_{1} a_{1} a_{2} \bar{a}_{0} a_{0} \bar{a}_{2}
+ \alpha^{(1)}_{2} a_{1} a_{2} \bar{a}_{0} a_{0} \bar{a}_{2} a_{2} \bar{a}_{2}
+ \alpha^{(1)}_{3} a_{1} a_{2} \bar{a}_{0} a_{0} \bar{a}_{2} a_{2} \bar{a}_{0} a_{0} \bar{a}_{2}
+ \alpha^{(1)}_{4} a_{1} a_{2} \bar{a}_{0} a_{0} \bar{a}_{2} a_{2} \bar{a}_{0} a_{0} \bar{a}_{2} a_{2} \bar{a}_{2}
\\&\quad
+ \alpha^{(1)}_{5} a_{1} a_{2} \bar{a}_{0} a_{0} \bar{a}_{2} a_{2} \bar{a}_{0} a_{0} \bar{a}_{2} a_{2} \bar{a}_{2} a_{2} \bar{a}_{0} a_{0} \bar{a}_{2}
,
\\
\varphi(\bar{a}_{1}) &= 
\bar{a}_{1}
+ \bar{\alpha}^{(1)}_{1} a_{2} \bar{a}_{0} a_{0} \bar{a}_{2} \bar{a}_{1}
+ \bar{\alpha}^{(1)}_{2} a_{2} \bar{a}_{2} a_{2} \bar{a}_{0} a_{0} \bar{a}_{2} \bar{a}_{1}
+ \bar{\alpha}^{(1)}_{3} a_{2} \bar{a}_{0} a_{0} \bar{a}_{2} a_{2} \bar{a}_{0} a_{0} \bar{a}_{2} \bar{a}_{1}
+ \bar{\alpha}^{(1)}_{4} a_{2} \bar{a}_{0} a_{0} \bar{a}_{2} a_{2} \bar{a}_{2} a_{2} \bar{a}_{0} a_{0} \bar{a}_{2} \bar{a}_{1}
\\&\quad
+ \bar{\alpha}^{(1)}_{5} a_{2} \bar{a}_{0} a_{0} \bar{a}_{2} a_{2} \bar{a}_{0} a_{0} \bar{a}_{2} a_{2} \bar{a}_{2} a_{2} \bar{a}_{0} a_{0} \bar{a}_{2} \bar{a}_{1}
,
\\
\varphi(a_{2}) &= 
a_{2}
+ \alpha^{(2)}_{1} a_{2} \bar{a}_{0} a_{0}
+ \alpha^{(2)}_{2} a_{2} \bar{a}_{2} a_{2}
+ \alpha^{(2)}_{3} a_{2} \bar{a}_{0} a_{0} \bar{a}_{2} a_{2}
+ \alpha^{(2)}_{4} a_{2} \bar{a}_{2} a_{2} \bar{a}_{0} a_{0}
+ \alpha^{(2)}_{5} a_{2} \bar{a}_{0} a_{0} \bar{a}_{2} a_{2} \bar{a}_{0} a_{0}
+ \alpha^{(2)}_{6} a_{2} \bar{a}_{0} a_{0} \bar{a}_{2} a_{2} \bar{a}_{2} a_{2}
\\&\quad
+ \alpha^{(2)}_{7} a_{2} \bar{a}_{2} a_{2} \bar{a}_{0} a_{0} \bar{a}_{2} a_{2}
+ \alpha^{(2)}_{8} a_{2} \bar{a}_{0} a_{0} \bar{a}_{2} a_{2} \bar{a}_{0} a_{0} \bar{a}_{2} a_{2}
+ \alpha^{(2)}_{9} a_{2} \bar{a}_{0} a_{0} \bar{a}_{2} a_{2} \bar{a}_{2} a_{2} \bar{a}_{0} a_{0}
+ \alpha^{(2)}_{10} a_{2} \bar{a}_{2} a_{2} \bar{a}_{0} a_{0} \bar{a}_{2} a_{2} \bar{a}_{0} a_{0}
\\&\quad
+ \alpha^{(2)}_{11} a_{2} \bar{a}_{0} a_{0} \bar{a}_{2} a_{2} \bar{a}_{0} a_{0} \bar{a}_{2} a_{2} \bar{a}_{0} a_{0}
+ \alpha^{(2)}_{12} a_{2} \bar{a}_{0} a_{0} \bar{a}_{2} a_{2} \bar{a}_{0} a_{0} \bar{a}_{2} a_{2} \bar{a}_{2} a_{2}
+ \alpha^{(2)}_{13} a_{2} \bar{a}_{0} a_{0} \bar{a}_{2} a_{2} \bar{a}_{0} a_{0} \bar{a}_{2} a_{2} \bar{a}_{0} a_{0} \bar{a}_{2} a_{2}
\\&\quad
+ \alpha^{(2)}_{14} a_{2} \bar{a}_{0} a_{0} \bar{a}_{2} a_{2} \bar{a}_{0} a_{0} \bar{a}_{2} a_{2} \bar{a}_{2} a_{2} \bar{a}_{0} a_{0}
+ \alpha^{(2)}_{15} a_{2} \bar{a}_{0} a_{0} \bar{a}_{2} a_{2} \bar{a}_{0} a_{0} \bar{a}_{2} a_{2} \bar{a}_{0} a_{0} \bar{a}_{2} a_{2} \bar{a}_{0} a_{0}
,
\\
\varphi(\bar{a}_{2}) &= 
\bar{a}_{2}
+ \bar{\alpha}^{(2)}_{1} \bar{a}_{0} a_{0} \bar{a}_{2}
+ \bar{\alpha}^{(2)}_{2} \bar{a}_{2} a_{2} \bar{a}_{2}
+ \bar{\alpha}^{(2)}_{3} \bar{a}_{0} a_{0} \bar{a}_{2} a_{2} \bar{a}_{2}
+ \bar{\alpha}^{(2)}_{4} \bar{a}_{2} a_{2} \bar{a}_{0} a_{0} \bar{a}_{2}
+ \bar{\alpha}^{(2)}_{5} \bar{a}_{0} a_{0} \bar{a}_{2} a_{2} \bar{a}_{0} a_{0} \bar{a}_{2}
+ \bar{\alpha}^{(2)}_{6} \bar{a}_{2} a_{2} \bar{a}_{0} a_{0} \bar{a}_{2} a_{2} \bar{a}_{2}
\\&\quad
+ \bar{\alpha}^{(2)}_{7} \bar{a}_{2} a_{2} \bar{a}_{2} a_{2} \bar{a}_{0} a_{0} \bar{a}_{2}
+ \bar{\alpha}^{(2)}_{8} \bar{a}_{0} a_{0} \bar{a}_{2} a_{2} \bar{a}_{0} a_{0} \bar{a}_{2} a_{2} \bar{a}_{2}
+ \bar{\alpha}^{(2)}_{9} \bar{a}_{0} a_{0} \bar{a}_{2} a_{2} \bar{a}_{2} a_{2} \bar{a}_{0} a_{0} \bar{a}_{2}
+ \bar{\alpha}^{(2)}_{10} \bar{a}_{2} a_{2} \bar{a}_{0} a_{0} \bar{a}_{2} a_{2} \bar{a}_{0} a_{0} \bar{a}_{2}
\\&\quad
+ \bar{\alpha}^{(2)}_{11} \bar{a}_{0} a_{0} \bar{a}_{2} a_{2} \bar{a}_{0} a_{0} \bar{a}_{2} a_{2} \bar{a}_{0} a_{0} \bar{a}_{2}
+ \bar{\alpha}^{(2)}_{12} \bar{a}_{2} a_{2} \bar{a}_{0} a_{0} \bar{a}_{2} a_{2} \bar{a}_{2} a_{2} \bar{a}_{0} a_{0} \bar{a}_{2}
+ \bar{\alpha}^{(2)}_{13} \bar{a}_{0} a_{0} \bar{a}_{2} a_{2} \bar{a}_{0} a_{0} \bar{a}_{2} a_{2} \bar{a}_{2} a_{2} \bar{a}_{0} a_{0} \bar{a}_{2}
\\&\quad
+ \bar{\alpha}^{(2)}_{14} \bar{a}_{2} a_{2} \bar{a}_{0} a_{0} \bar{a}_{2} a_{2} \bar{a}_{0} a_{0} \bar{a}_{2} a_{2} \bar{a}_{0} a_{0} \bar{a}_{2}
+ \bar{\alpha}^{(2)}_{15} \bar{a}_{0} a_{0} \bar{a}_{2} a_{2} \bar{a}_{0} a_{0} \bar{a}_{2} a_{2} \bar{a}_{0} a_{0} \bar{a}_{2} a_{2} \bar{a}_{0} a_{0} \bar{a}_{2}
,
\\
\varphi(a_{3}) &= 
a_{3}
+ \alpha^{(3)}_{1} \bar{a}_{0} a_{0} a_{3}
+ \alpha^{(3)}_{2} \bar{a}_{2} a_{2} a_{3}
+ \alpha^{(3)}_{3} \bar{a}_{0} a_{0} \bar{a}_{2} a_{2} a_{3}
+ \alpha^{(3)}_{4} \bar{a}_{2} a_{2} \bar{a}_{0} a_{0} a_{3}
+ \alpha^{(3)}_{5} \bar{a}_{2} a_{2} \bar{a}_{2} a_{2} a_{3}
+ \alpha^{(3)}_{6} \bar{a}_{0} a_{0} \bar{a}_{2} a_{2} \bar{a}_{0} a_{0} a_{3}
\\&\quad
+ \alpha^{(3)}_{7} \bar{a}_{0} a_{0} \bar{a}_{2} a_{2} \bar{a}_{2} a_{2} a_{3}
+ \alpha^{(3)}_{8} \bar{a}_{2} a_{2} \bar{a}_{0} a_{0} \bar{a}_{2} a_{2} a_{3}
+ \alpha^{(3)}_{9} \bar{a}_{0} a_{0} \bar{a}_{2} a_{2} \bar{a}_{0} a_{0} \bar{a}_{2} a_{2} a_{3}
+ \alpha^{(3)}_{10} \bar{a}_{2} a_{2} \bar{a}_{0} a_{0} \bar{a}_{2} a_{2} \bar{a}_{0} a_{0} a_{3}
\\&\quad
+ \alpha^{(3)}_{11} \bar{a}_{2} a_{2} \bar{a}_{0} a_{0} \bar{a}_{2} a_{2} \bar{a}_{2} a_{2} a_{3}
+ \alpha^{(3)}_{12} \bar{a}_{0} a_{0} \bar{a}_{2} a_{2} \bar{a}_{0} a_{0} \bar{a}_{2} a_{2} \bar{a}_{0} a_{0} a_{3}
+ \alpha^{(3)}_{13} \bar{a}_{0} a_{0} \bar{a}_{2} a_{2} \bar{a}_{0} a_{0} \bar{a}_{2} a_{2} \bar{a}_{2} a_{2} a_{3}
+ \alpha^{(3)}_{14} \bar{a}_{2} a_{2} \bar{a}_{0} a_{0} \bar{a}_{2} a_{2} \bar{a}_{0} a_{0} \bar{a}_{2} a_{2} a_{3}
\\&\quad
+ \alpha^{(3)}_{15} \bar{a}_{0} a_{0} \bar{a}_{2} a_{2} \bar{a}_{0} a_{0} \bar{a}_{2} a_{2} \bar{a}_{0} a_{0} \bar{a}_{2} a_{2} a_{3}
+ \alpha^{(3)}_{16} \bar{a}_{2} a_{2} \bar{a}_{0} a_{0} \bar{a}_{2} a_{2} \bar{a}_{0} a_{0} \bar{a}_{2} a_{2} \bar{a}_{0} a_{0} a_{3}
+ \alpha^{(3)}_{17} \bar{a}_{0} a_{0} \bar{a}_{2} a_{2} \bar{a}_{0} a_{0} \bar{a}_{2} a_{2} \bar{a}_{0} a_{0} \bar{a}_{2} a_{2} \bar{a}_{0} a_{0} a_{3}
,
\\
\varphi(\bar{a}_{3}) &= 
\bar{a}_{3}
+ \bar{\alpha}^{(3)}_{1} \bar{a}_{3} \bar{a}_{0} a_{0}
+ \bar{\alpha}^{(3)}_{2} \bar{a}_{3} \bar{a}_{2} a_{2}
+ \bar{\alpha}^{(3)}_{3} \bar{a}_{3} \bar{a}_{0} a_{0} \bar{a}_{2} a_{2}
+ \bar{\alpha}^{(3)}_{4} \bar{a}_{3} \bar{a}_{2} a_{2} \bar{a}_{0} a_{0}
+ \bar{\alpha}^{(3)}_{5} \bar{a}_{3} \bar{a}_{2} a_{2} \bar{a}_{2} a_{2}
+ \bar{\alpha}^{(3)}_{6} \bar{a}_{3} \bar{a}_{0} a_{0} \bar{a}_{2} a_{2} \bar{a}_{0} a_{0}
\\&\quad
+ \bar{\alpha}^{(3)}_{7} \bar{a}_{3} \bar{a}_{0} a_{0} \bar{a}_{2} a_{2} \bar{a}_{2} a_{2}
+ \bar{\alpha}^{(3)}_{8} \bar{a}_{3} \bar{a}_{2} a_{2} \bar{a}_{0} a_{0} \bar{a}_{2} a_{2}
+ \bar{\alpha}^{(3)}_{9} \bar{a}_{3} \bar{a}_{0} a_{0} \bar{a}_{2} a_{2} \bar{a}_{0} a_{0} \bar{a}_{2} a_{2}
+ \bar{\alpha}^{(3)}_{10} \bar{a}_{3} \bar{a}_{0} a_{0} \bar{a}_{2} a_{2} \bar{a}_{2} a_{2} \bar{a}_{0} a_{0}
\\&\quad
+ \bar{\alpha}^{(3)}_{11} \bar{a}_{3} \bar{a}_{2} a_{2} \bar{a}_{0} a_{0} \bar{a}_{2} a_{2} \bar{a}_{2} a_{2}
+ \bar{\alpha}^{(3)}_{12} \bar{a}_{3} \bar{a}_{0} a_{0} \bar{a}_{2} a_{2} \bar{a}_{0} a_{0} \bar{a}_{2} a_{2} \bar{a}_{0} a_{0}
+ \bar{\alpha}^{(3)}_{13} \bar{a}_{3} \bar{a}_{0} a_{0} \bar{a}_{2} a_{2} \bar{a}_{0} a_{0} \bar{a}_{2} a_{2} \bar{a}_{2} a_{2}
+ \bar{\alpha}^{(3)}_{14} \bar{a}_{3} \bar{a}_{2} a_{2} \bar{a}_{0} a_{0} \bar{a}_{2} a_{2} \bar{a}_{2} a_{2} \bar{a}_{0} a_{0}
\\&\quad
+ \bar{\alpha}^{(3)}_{15} \bar{a}_{3} \bar{a}_{0} a_{0} \bar{a}_{2} a_{2} \bar{a}_{0} a_{0} \bar{a}_{2} a_{2} \bar{a}_{0} a_{0} \bar{a}_{2} a_{2}
+ \bar{\alpha}^{(3)}_{16} \bar{a}_{3} \bar{a}_{0} a_{0} \bar{a}_{2} a_{2} \bar{a}_{0} a_{0} \bar{a}_{2} a_{2} \bar{a}_{2} a_{2} \bar{a}_{0} a_{0}
+ \bar{\alpha}^{(3)}_{17} \bar{a}_{3} \bar{a}_{0} a_{0} \bar{a}_{2} a_{2} \bar{a}_{0} a_{0} \bar{a}_{2} a_{2} \bar{a}_{0} a_{0} \bar{a}_{2} a_{2} \bar{a}_{0} a_{0}
,
\\
\varphi(a_{4}) &= 
a_{4}
+ \alpha^{(4)}_{1} \bar{a}_{3} a_{3} a_{4}
+ \alpha^{(4)}_{2} \bar{a}_{3} \bar{a}_{0} a_{0} a_{3} a_{4}
+ \alpha^{(4)}_{3} \bar{a}_{3} \bar{a}_{0} a_{0} \bar{a}_{2} a_{2} a_{3} a_{4}
+ \alpha^{(4)}_{4} \bar{a}_{3} \bar{a}_{2} a_{2} \bar{a}_{0} a_{0} a_{3} a_{4}
+ \alpha^{(4)}_{5} \bar{a}_{3} \bar{a}_{0} a_{0} \bar{a}_{2} a_{2} \bar{a}_{0} a_{0} a_{3} a_{4}
\\&\quad
+ \alpha^{(4)}_{6} \bar{a}_{3} \bar{a}_{2} a_{2} \bar{a}_{0} a_{0} \bar{a}_{2} a_{2} a_{3} a_{4}
+ \alpha^{(4)}_{7} \bar{a}_{3} \bar{a}_{0} a_{0} \bar{a}_{2} a_{2} \bar{a}_{0} a_{0} \bar{a}_{2} a_{2} a_{3} a_{4}
+ \alpha^{(4)}_{8} \bar{a}_{3} \bar{a}_{0} a_{0} \bar{a}_{2} a_{2} \bar{a}_{0} a_{0} \bar{a}_{2} a_{2} \bar{a}_{0} a_{0} a_{3} a_{4}
\\&\quad
+ \alpha^{(4)}_{9} \bar{a}_{3} \bar{a}_{0} a_{0} \bar{a}_{2} a_{2} \bar{a}_{0} a_{0} \bar{a}_{2} a_{2} \bar{a}_{0} a_{0} \bar{a}_{2} a_{2} a_{3} a_{4}
,
\\
\varphi(\bar{a}_{4}) &= 
\bar{a}_{4}
+ \bar{\alpha}^{(4)}_{1} \bar{a}_{4} \bar{a}_{3} a_{3}
+ \bar{\alpha}^{(4)}_{2} \bar{a}_{4} \bar{a}_{3} \bar{a}_{0} a_{0} a_{3}
+ \bar{\alpha}^{(4)}_{3} \bar{a}_{4} \bar{a}_{3} \bar{a}_{0} a_{0} \bar{a}_{2} a_{2} a_{3}
+ \bar{\alpha}^{(4)}_{4} \bar{a}_{4} \bar{a}_{3} \bar{a}_{2} a_{2} \bar{a}_{0} a_{0} a_{3}
+ \bar{\alpha}^{(4)}_{5} \bar{a}_{4} \bar{a}_{3} \bar{a}_{0} a_{0} \bar{a}_{2} a_{2} \bar{a}_{0} a_{0} a_{3}
\\&\quad
+ \bar{\alpha}^{(4)}_{6} \bar{a}_{4} \bar{a}_{3} \bar{a}_{0} a_{0} \bar{a}_{2} a_{2} \bar{a}_{2} a_{2} a_{3}
+ \bar{\alpha}^{(4)}_{7} \bar{a}_{4} \bar{a}_{3} \bar{a}_{0} a_{0} \bar{a}_{2} a_{2} \bar{a}_{0} a_{0} \bar{a}_{2} a_{2} a_{3}
+ \bar{\alpha}^{(4)}_{8} \bar{a}_{4} \bar{a}_{3} \bar{a}_{0} a_{0} \bar{a}_{2} a_{2} \bar{a}_{0} a_{0} \bar{a}_{2} a_{2} \bar{a}_{0} a_{0} a_{3}
\\&\quad
+ \bar{\alpha}^{(4)}_{9} \bar{a}_{4} \bar{a}_{3} \bar{a}_{0} a_{0} \bar{a}_{2} a_{2} \bar{a}_{0} a_{0} \bar{a}_{2} a_{2} \bar{a}_{0} a_{0} \bar{a}_{2} a_{2} a_{3}
,
\\
\varphi(a_{5}) &= 
a_{5}
+ \alpha^{(5)}_{1} \bar{a}_{4} \bar{a}_{3} \bar{a}_{0} a_{0} a_{3} a_{4} a_{5}
+ \alpha^{(5)}_{2} \bar{a}_{4} \bar{a}_{3} \bar{a}_{2} a_{2} \bar{a}_{0} a_{0} a_{3} a_{4} a_{5}
+ \alpha^{(5)}_{3} \bar{a}_{4} \bar{a}_{3} \bar{a}_{0} a_{0} \bar{a}_{2} a_{2} \bar{a}_{0} a_{0} \bar{a}_{2} a_{2} \bar{a}_{0} a_{0} a_{3} a_{4} a_{5}
,
\\
\varphi(\bar{a}_{5}) &= 
\bar{a}_{5}
+ \bar{\alpha}^{(5)}_{1} \bar{a}_{5} \bar{a}_{4} \bar{a}_{3} \bar{a}_{0} a_{0} a_{3} a_{4}
+ \bar{\alpha}^{(5)}_{2} \bar{a}_{5} \bar{a}_{4} \bar{a}_{3} \bar{a}_{0} a_{0} \bar{a}_{2} a_{2} a_{3} a_{4}
+ \bar{\alpha}^{(5)}_{3} \bar{a}_{5} \bar{a}_{4} \bar{a}_{3} \bar{a}_{0} a_{0} \bar{a}_{2} a_{2} \bar{a}_{0} a_{0} \bar{a}_{2} a_{2} \bar{a}_{0} a_{0} a_{3} a_{4}
,
\end{align*}%
\end{footnotesize}%
with coefficients ${\alpha}^{(i)}_{j_i}, \bar{\alpha}^{(i)}_{j_i} \in K$,
for $i=0,\dots,5$, $j = 1, \dots,  j_i$, 
$j_0 = 11, j_1 = 5, j_2 = 15, j_3 = 17, j_4 = 9, j_5 = 3$.

\subsection{Simplified system of equations}
\label{secA:simpl}

In order to obtain the simplified system of equations
(see Section~\ref{sec:eq} for details)
we will assume that in the above definition of $\varphi$ the following
coefficients are equal zero:
\begin{small}
\setlength{\jot}{2pt}
\begin{gather*}
\alpha^{(0)}_{9},
\alpha^{(0)}_{10},
\alpha^{(0)}_{11},
\bar{\alpha}^{(0)}_{2},
\bar{\alpha}^{(0)}_{3},
\bar{\alpha}^{(0)}_{4},
\bar{\alpha}^{(0)}_{5},
\bar{\alpha}^{(0)}_{6},
\bar{\alpha}^{(0)}_{7},
\bar{\alpha}^{(0)}_{8},
\bar{\alpha}^{(0)}_{9},
\bar{\alpha}^{(0)}_{10},
\bar{\alpha}^{(0)}_{11},
\alpha^{(1)}_{1},
\alpha^{(1)}_{3},
\alpha^{(1)}_{5},
\bar{\alpha}^{(1)}_{1},
\bar{\alpha}^{(1)}_{2},
\bar{\alpha}^{(1)}_{3},
\bar{\alpha}^{(1)}_{4},
\bar{\alpha}^{(1)}_{5},
\alpha^{(2)}_{1},
\alpha^{(2)}_{2},
\\
\alpha^{(2)}_{5},
\alpha^{(2)}_{6},
\alpha^{(2)}_{8},
\alpha^{(2)}_{9},
\alpha^{(2)}_{11},
\alpha^{(2)}_{12},
\bar{\alpha}^{(2)}_{1},
\bar{\alpha}^{(2)}_{2},
\bar{\alpha}^{(2)}_{5},
\bar{\alpha}^{(2)}_{7},
\bar{\alpha}^{(2)}_{9},
\bar{\alpha}^{(2)}_{10},
\bar{\alpha}^{(2)}_{12},
\bar{\alpha}^{(2)}_{14},
\alpha^{(3)}_{2},
\alpha^{(3)}_{3},
\alpha^{(3)}_{4},
\alpha^{(3)}_{5},
\alpha^{(3)}_{6},
\alpha^{(3)}_{7},
\alpha^{(3)}_{8},
\alpha^{(3)}_{9},
\alpha^{(3)}_{10},
\\
\alpha^{(3)}_{11},
\alpha^{(3)}_{12},
\alpha^{(3)}_{13},
\alpha^{(3)}_{14},
\alpha^{(3)}_{15},
\alpha^{(3)}_{16},
\alpha^{(3)}_{17},
\bar{\alpha}^{(3)}_{1},
\bar{\alpha}^{(3)}_{6},
\bar{\alpha}^{(3)}_{15},
\alpha^{(4)}_{2},
\alpha^{(4)}_{3},
\alpha^{(4)}_{4},
\alpha^{(4)}_{5},
\alpha^{(4)}_{6},
\alpha^{(4)}_{7},
\alpha^{(4)}_{8},
\alpha^{(4)}_{9},
\bar{\alpha}^{(4)}_{1},
\bar{\alpha}^{(4)}_{2},
\alpha^{(5)}_{2},
\alpha^{(5)}_{3}.
\end{gather*}%
\end{small}%
We note that this choice is crucial
to obtain the system of equation with 
the solution of ``acceptable'' size.
Then, 
applying the relations of $P(\mathbb{E}_7)$
we obtain the following equations
\begin{footnotesize}
\setlength{\jot}{0pt}
\begin{align*}
-\bar{\alpha}^{(3)}_{2}
-\alpha^{(3)}_{1}
+\alpha^{(0)}_{1}
+\theta_{1}
&= 0
,\\
\bar{\alpha}^{(0)}_{1}
+\theta_{2}
&= 0
,\\
-\bar{\alpha}^{(3)}_{2}
+\theta_{3}
&= 0
,\\
-\bar{\alpha}^{(3)}_{4}
+\alpha^{(0)}_{2}
+\theta_{4}
&= 0
,\end{align*}
\begin{align*}
-\bar{\alpha}^{(3)}_{5}
-\alpha^{(3)}_{1}\bar{\alpha}^{(3)}_{2}
+\bar{\alpha}^{(2)}_{3}
+\alpha^{(0)}_{3}
+\alpha^{(0)}_{1}\theta_{1}
+\theta_{5}
&= 0
,
\\
-\bar{\alpha}^{(3)}_{3}
+\alpha^{(2)}_{3}
+\bar{\alpha}^{(2)}_{4}
+\bar{\alpha}^{(0)}_{1}\alpha^{(0)}_{1}
+\bar{\alpha}^{(0)}_{1}\theta_{1}
+\alpha^{(0)}_{1}\theta_{2}
+\theta_{6}
&= 0
,\\
-\bar{\alpha}^{(3)}_{4}
+\alpha^{(2)}_{4}
+\bar{\alpha}^{(0)}_{1}\theta_{2}
+\theta_{7}
&= 0
,
\\
-\alpha^{(3)}_{1}\bar{\alpha}^{(3)}_{3}
-\alpha^{(2)}_{7}
+\alpha^{(0)}_{4}
+\alpha^{(2)}_{3}\theta_{1}
+\bar{\alpha}^{(2)}_{4}\theta_{1}
+\alpha^{(0)}_{2}\theta_{1}
-\bar{\alpha}^{(0)}_{1}\alpha^{(0)}_{1}\theta_{2}
-\alpha^{(2)}_{3}\theta_{3}
-\bar{\alpha}^{(2)}_{4}\theta_{3}
-\alpha^{(2)}_{4}\theta_{3}
+\alpha^{(0)}_{1}\theta_{4}
-\bar{\alpha}^{(0)}_{1}\theta_{6}
-\alpha^{(0)}_{1}\theta_{7}
+\theta_{8}
&= 0
,
\\
\bar{\alpha}^{(3)}_{8}
-\alpha^{(3)}_{1}\bar{\alpha}^{(3)}_{4}
-\alpha^{(2)}_{7}
+\alpha^{(0)}_{5}
+\alpha^{(2)}_{4}\theta_{1}
-\bar{\alpha}^{(0)}_{1}\alpha^{(0)}_{1}\theta_{2}
+\bar{\alpha}^{(2)}_{3}\theta_{2}
-\alpha^{(2)}_{3}\theta_{3}
-\bar{\alpha}^{(2)}_{4}\theta_{3}
-\alpha^{(2)}_{4}\theta_{3}
+\bar{\alpha}^{(0)}_{1}\theta_{4}
+\alpha^{(0)}_{1}\theta_{4}
-\bar{\alpha}^{(0)}_{1}\theta_{6}
-\alpha^{(0)}_{1}\theta_{7}
+\theta_{9}
&= 0
,
\\
\bar{\alpha}^{(3)}_{8}
-\alpha^{(2)}_{7}
+\bar{\alpha}^{(0)}_{1}\alpha^{(0)}_{2}
+\alpha^{(0)}_{2}\theta_{2}
-\bar{\alpha}^{(0)}_{1}\alpha^{(0)}_{1}\theta_{2}
+\alpha^{(2)}_{3}\theta_{2}
+\bar{\alpha}^{(2)}_{4}\theta_{2}
-\alpha^{(2)}_{3}\theta_{3}
-\bar{\alpha}^{(2)}_{4}\theta_{3}
-\alpha^{(2)}_{4}\theta_{3}
+\bar{\alpha}^{(0)}_{1}\theta_{4}
-\bar{\alpha}^{(0)}_{1}\theta_{6}
-\alpha^{(0)}_{1}\theta_{7}
+\theta_{10}
&= 0
,
\\
-\bar{\alpha}^{(3)}_{7}
+\bar{\alpha}^{(3)}_{8}
-\alpha^{(2)}_{7}
+\bar{\alpha}^{(2)}_{6}
+\bar{\alpha}^{(0)}_{1}\alpha^{(0)}_{3}
+\bar{\alpha}^{(0)}_{1}\alpha^{(0)}_{1}\theta_{1}
+\alpha^{(0)}_{3}\theta_{2}
-\bar{\alpha}^{(0)}_{1}\alpha^{(0)}_{1}\theta_{2}
-\alpha^{(2)}_{3}\theta_{3}
+\bar{\alpha}^{(2)}_{3}\theta_{3}
-\bar{\alpha}^{(2)}_{4}\theta_{3}
+\alpha^{(2)}_{3}\theta_{3}
-\alpha^{(2)}_{4}\theta_{3}
\\
+\bar{\alpha}^{(2)}_{4}\theta_{3}
+\bar{\alpha}^{(0)}_{1}\theta_{5}
+\alpha^{(0)}_{1}\theta_{6}
-\bar{\alpha}^{(0)}_{1}\theta_{6}
-\alpha^{(0)}_{1}\theta_{7}
+\theta_{11}
&= 0
,
\\
-\bar{\alpha}^{(3)}_{11}
+\alpha^{(3)}_{1}\bar{\alpha}^{(3)}_{8}
-\alpha^{(2)}_{10}
-\bar{\alpha}^{(2)}_{3}\alpha^{(2)}_{3}
+\alpha^{(0)}_{6}
-\alpha^{(2)}_{7}\theta_{1}
-\alpha^{(0)}_{1}\alpha^{(2)}_{3}\theta_{1}
-\alpha^{(0)}_{1}\bar{\alpha}^{(2)}_{4}\theta_{1}
-\alpha^{(0)}_{5}\theta_{1}
-\bar{\alpha}^{(0)}_{1}\alpha^{(0)}_{2}\theta_{2}
+\bar{\alpha}^{(0)}_{1}\alpha^{(0)}_{3}\theta_{2}
-\alpha^{(2)}_{4}\bar{\alpha}^{(0)}_{1}\theta_{2}
\\
-\alpha^{(2)}_{7}\theta_{2}
-\bar{\alpha}^{(2)}_{3}\alpha^{(0)}_{1}\theta_{2}
+\bar{\alpha}^{(2)}_{6}\theta_{3}
+\alpha^{(2)}_{7}\theta_{3}
+\alpha^{(0)}_{2}\theta_{4}
-\bar{\alpha}^{(0)}_{1}\alpha^{(0)}_{1}\theta_{4}
+\alpha^{(2)}_{3}\theta_{4}
+\bar{\alpha}^{(2)}_{4}\theta_{4}
-\alpha^{(0)}_{1}\alpha^{(0)}_{1}\theta_{4}
+\alpha^{(0)}_{2}\theta_{4}
-\alpha^{(2)}_{3}\theta_{5}
-\bar{\alpha}^{(2)}_{4}\theta_{5}
\\
-\alpha^{(2)}_{4}\theta_{5}
+\bar{\alpha}^{(0)}_{1}\alpha^{(0)}_{1}\theta_{6}
-\bar{\alpha}^{(2)}_{3}\theta_{6}
-\alpha^{(0)}_{2}\theta_{7}
+\alpha^{(0)}_{3}\theta_{7}
-\alpha^{(2)}_{3}\theta_{7}
-\bar{\alpha}^{(2)}_{4}\theta_{7}
-\alpha^{(2)}_{4}\theta_{7}
-\bar{\alpha}^{(0)}_{1}\theta_{8}
-\alpha^{(0)}_{1}\theta_{8}
-\alpha^{(0)}_{1}\theta_{9}
-\bar{\alpha}^{(0)}_{1}\theta_{10}
+\bar{\alpha}^{(0)}_{1}\theta_{11}
+\theta_{12}
&= 0
,
\\
-\bar{\alpha}^{(3)}_{11}
-\alpha^{(3)}_{1}\bar{\alpha}^{(3)}_{7}
+\alpha^{(3)}_{1}\bar{\alpha}^{(3)}_{8}
-\bar{\alpha}^{(2)}_{3}\alpha^{(2)}_{3}
+\bar{\alpha}^{(2)}_{8}
+\alpha^{(0)}_{7}
-\alpha^{(2)}_{7}\theta_{1}
+\bar{\alpha}^{(2)}_{6}\theta_{1}
-\alpha^{(0)}_{1}\alpha^{(2)}_{3}\theta_{1}
+\alpha^{(0)}_{1}\bar{\alpha}^{(2)}_{3}\theta_{1}
-\alpha^{(0)}_{1}\bar{\alpha}^{(2)}_{4}\theta_{1}
+\alpha^{(0)}_{4}\theta_{1}
\\
-\alpha^{(0)}_{5}\theta_{1}
-\bar{\alpha}^{(2)}_{3}\alpha^{(0)}_{1}\theta_{2}
+\alpha^{(0)}_{3}\theta_{4}
-\bar{\alpha}^{(0)}_{1}\alpha^{(0)}_{1}\theta_{4}
-\alpha^{(0)}_{1}\alpha^{(0)}_{1}\theta_{4}
-\alpha^{(2)}_{3}\theta_{5}
+\bar{\alpha}^{(2)}_{3}\theta_{5}
-\bar{\alpha}^{(2)}_{4}\theta_{5}
+\alpha^{(2)}_{3}\theta_{5}
-\alpha^{(2)}_{4}\theta_{5}
+\bar{\alpha}^{(2)}_{4}\theta_{5}
+\alpha^{(0)}_{2}\theta_{5}
\\
-\bar{\alpha}^{(2)}_{3}\theta_{6}
+\alpha^{(0)}_{1}\theta_{8}
-\bar{\alpha}^{(0)}_{1}\theta_{8}
-\alpha^{(0)}_{1}\theta_{8}
-\alpha^{(0)}_{1}\theta_{9}
+\theta_{13}
&= 0
,
\\
-\bar{\alpha}^{(3)}_{9}
+\bar{\alpha}^{(3)}_{11}
+\bar{\alpha}^{(2)}_{4}\alpha^{(2)}_{3}
+\bar{\alpha}^{(0)}_{1}\alpha^{(0)}_{4}
+\bar{\alpha}^{(0)}_{1}\alpha^{(2)}_{3}\theta_{1}
+\bar{\alpha}^{(0)}_{1}\bar{\alpha}^{(2)}_{4}\theta_{1}
+\bar{\alpha}^{(0)}_{1}\alpha^{(0)}_{2}\theta_{1}
+\alpha^{(0)}_{4}\theta_{2}
-\bar{\alpha}^{(0)}_{1}\alpha^{(0)}_{3}\theta_{2}
+\alpha^{(2)}_{3}\alpha^{(0)}_{1}\theta_{2}
+\bar{\alpha}^{(2)}_{4}\alpha^{(0)}_{1}\theta_{2}
-\bar{\alpha}^{(2)}_{6}\theta_{3}
\\
-\alpha^{(2)}_{7}\theta_{3}
+\bar{\alpha}^{(0)}_{1}\alpha^{(0)}_{1}\theta_{4}
+\alpha^{(2)}_{3}\theta_{6}
+\bar{\alpha}^{(2)}_{4}\theta_{6}
+\alpha^{(0)}_{2}\theta_{6}
-\bar{\alpha}^{(0)}_{1}\alpha^{(0)}_{1}\theta_{6}
+\alpha^{(2)}_{3}\theta_{6}
+\bar{\alpha}^{(2)}_{4}\theta_{6}
-\alpha^{(0)}_{3}\theta_{7}
+\bar{\alpha}^{(0)}_{1}\theta_{8}
+\alpha^{(0)}_{1}\theta_{10}
-\bar{\alpha}^{(0)}_{1}\theta_{11}
+\theta_{14}
&= 0
,
\\
-\bar{\alpha}^{(3)}_{10}
-\alpha^{(2)}_{10}
+\bar{\alpha}^{(2)}_{4}\alpha^{(2)}_{4}
+\bar{\alpha}^{(0)}_{1}\alpha^{(0)}_{5}
+\bar{\alpha}^{(0)}_{1}\alpha^{(2)}_{4}\theta_{1}
+\alpha^{(0)}_{5}\theta_{2}
-\bar{\alpha}^{(0)}_{1}\alpha^{(0)}_{2}\theta_{2}
+\alpha^{(2)}_{3}\bar{\alpha}^{(0)}_{1}\theta_{2}
-\alpha^{(2)}_{4}\bar{\alpha}^{(0)}_{1}\theta_{2}
-\alpha^{(2)}_{7}\theta_{2}
+\bar{\alpha}^{(2)}_{4}\bar{\alpha}^{(0)}_{1}\theta_{2}
+\bar{\alpha}^{(2)}_{6}\theta_{2}
\\
+\bar{\alpha}^{(0)}_{1}\bar{\alpha}^{(0)}_{1}\theta_{4}
+\bar{\alpha}^{(0)}_{1}\alpha^{(0)}_{1}\theta_{4}
+\alpha^{(2)}_{4}\theta_{6}
-\alpha^{(0)}_{2}\theta_{7}
-\alpha^{(2)}_{3}\theta_{7}
+\bar{\alpha}^{(2)}_{3}\theta_{7}
-\bar{\alpha}^{(2)}_{4}\theta_{7}
+\alpha^{(2)}_{3}\theta_{7}
-\alpha^{(2)}_{4}\theta_{7}
+\bar{\alpha}^{(2)}_{4}\theta_{7}
+\bar{\alpha}^{(0)}_{1}\theta_{9}
+\bar{\alpha}^{(0)}_{1}\theta_{10}
+\alpha^{(0)}_{1}\theta_{10}
\\
-\bar{\alpha}^{(0)}_{1}\theta_{10}
+\theta_{15}
&= 0
,
\\
-\bar{\alpha}^{(3)}_{13}
-\alpha^{(3)}_{1}\bar{\alpha}^{(3)}_{9}
+\alpha^{(3)}_{1}\bar{\alpha}^{(3)}_{11}
-\bar{\alpha}^{(2)}_{4}\alpha^{(2)}_{7}
-\bar{\alpha}^{(2)}_{6}\alpha^{(2)}_{3}
+\bar{\alpha}^{(2)}_{11}
+\alpha^{(0)}_{8}
+\bar{\alpha}^{(0)}_{1}\alpha^{(0)}_{7}
+\bar{\alpha}^{(2)}_{4}\alpha^{(2)}_{3}\theta_{1}
-\alpha^{(0)}_{1}\bar{\alpha}^{(2)}_{6}\theta_{1}
+\alpha^{(0)}_{2}\alpha^{(2)}_{3}\theta_{1}
+\alpha^{(0)}_{2}\bar{\alpha}^{(2)}_{4}\theta_{1}
\\
-\alpha^{(0)}_{3}\bar{\alpha}^{(2)}_{3}\theta_{1}
+\alpha^{(0)}_{6}\theta_{1}
-\bar{\alpha}^{(0)}_{1}\alpha^{(2)}_{7}\theta_{1}
+\bar{\alpha}^{(0)}_{1}\bar{\alpha}^{(2)}_{6}\theta_{1}
-\bar{\alpha}^{(0)}_{1}\alpha^{(0)}_{1}\alpha^{(2)}_{3}\theta_{1}
+\bar{\alpha}^{(0)}_{1}\alpha^{(0)}_{1}\bar{\alpha}^{(2)}_{3}\theta_{1}
-\bar{\alpha}^{(0)}_{1}\alpha^{(0)}_{1}\bar{\alpha}^{(2)}_{4}\theta_{1}
+\bar{\alpha}^{(0)}_{1}\alpha^{(0)}_{4}\theta_{1}
-\bar{\alpha}^{(0)}_{1}\alpha^{(0)}_{5}\theta_{1}
+\alpha^{(0)}_{7}\theta_{2}
\\
+\alpha^{(2)}_{3}\alpha^{(0)}_{3}\theta_{2}
-\alpha^{(2)}_{3}\bar{\alpha}^{(0)}_{1}\alpha^{(0)}_{1}\theta_{2}
-\bar{\alpha}^{(2)}_{3}\alpha^{(0)}_{3}\theta_{2}
+\bar{\alpha}^{(2)}_{4}\alpha^{(0)}_{3}\theta_{2}
-\bar{\alpha}^{(2)}_{4}\bar{\alpha}^{(0)}_{1}\alpha^{(0)}_{1}\theta_{2}
-\bar{\alpha}^{(2)}_{6}\alpha^{(0)}_{1}\theta_{2}
-\bar{\alpha}^{(2)}_{3}\alpha^{(2)}_{3}\theta_{3}
+\bar{\alpha}^{(2)}_{8}\theta_{3}
-\alpha^{(2)}_{3}\alpha^{(2)}_{3}\theta_{3}
+\alpha^{(2)}_{3}\bar{\alpha}^{(2)}_{3}\theta_{3}
\\
-\alpha^{(2)}_{3}\bar{\alpha}^{(2)}_{4}\theta_{3}
-\bar{\alpha}^{(2)}_{3}\bar{\alpha}^{(2)}_{3}\theta_{3}
-\bar{\alpha}^{(2)}_{3}\alpha^{(2)}_{3}\theta_{3}
-\bar{\alpha}^{(2)}_{4}\alpha^{(2)}_{3}\theta_{3}
+\bar{\alpha}^{(2)}_{4}\bar{\alpha}^{(2)}_{3}\theta_{3}
-\bar{\alpha}^{(2)}_{4}\bar{\alpha}^{(2)}_{4}\theta_{3}
+\bar{\alpha}^{(2)}_{4}\alpha^{(2)}_{3}\theta_{3}
-\bar{\alpha}^{(2)}_{4}\alpha^{(2)}_{4}\theta_{3}
+\alpha^{(0)}_{4}\theta_{4}
-\bar{\alpha}^{(0)}_{1}\alpha^{(0)}_{3}\theta_{4}
\\
+\alpha^{(2)}_{3}\alpha^{(0)}_{1}\theta_{4}
+\bar{\alpha}^{(2)}_{4}\alpha^{(0)}_{1}\theta_{4}
-\alpha^{(0)}_{1}\alpha^{(0)}_{3}\theta_{4}
+\alpha^{(0)}_{2}\alpha^{(0)}_{1}\theta_{4}
+\bar{\alpha}^{(0)}_{1}\alpha^{(0)}_{3}\theta_{4}
-\bar{\alpha}^{(0)}_{1}\bar{\alpha}^{(0)}_{1}\alpha^{(0)}_{1}\theta_{4}
-\bar{\alpha}^{(0)}_{1}\alpha^{(0)}_{1}\alpha^{(0)}_{1}\theta_{4}
-\bar{\alpha}^{(2)}_{6}\theta_{5}
-\alpha^{(2)}_{7}\theta_{5}
-\alpha^{(0)}_{1}\bar{\alpha}^{(2)}_{3}\theta_{5}
\\
-\alpha^{(0)}_{1}\alpha^{(2)}_{3}\theta_{5}
-\alpha^{(0)}_{1}\bar{\alpha}^{(2)}_{4}\theta_{5}
-\alpha^{(0)}_{5}\theta_{5}
-\bar{\alpha}^{(0)}_{1}\alpha^{(2)}_{3}\theta_{5}
+\bar{\alpha}^{(0)}_{1}\bar{\alpha}^{(2)}_{3}\theta_{5}
-\bar{\alpha}^{(0)}_{1}\bar{\alpha}^{(2)}_{4}\theta_{5}
+\bar{\alpha}^{(0)}_{1}\alpha^{(2)}_{3}\theta_{5}
-\bar{\alpha}^{(0)}_{1}\alpha^{(2)}_{4}\theta_{5}
+\bar{\alpha}^{(0)}_{1}\bar{\alpha}^{(2)}_{4}\theta_{5}
+\bar{\alpha}^{(0)}_{1}\alpha^{(0)}_{2}\theta_{5}
-\alpha^{(2)}_{7}\theta_{6}
\\
+\bar{\alpha}^{(2)}_{6}\theta_{6}
-\alpha^{(0)}_{1}\alpha^{(2)}_{3}\theta_{6}
+\alpha^{(0)}_{1}\bar{\alpha}^{(2)}_{3}\theta_{6}
-\alpha^{(0)}_{1}\bar{\alpha}^{(2)}_{4}\theta_{6}
+\alpha^{(0)}_{4}\theta_{6}
-\alpha^{(0)}_{5}\theta_{6}
+\alpha^{(2)}_{3}\alpha^{(0)}_{1}\theta_{6}
-\alpha^{(2)}_{3}\bar{\alpha}^{(0)}_{1}\theta_{6}
-\bar{\alpha}^{(2)}_{3}\alpha^{(0)}_{1}\theta_{6}
+\bar{\alpha}^{(2)}_{4}\alpha^{(0)}_{1}\theta_{6}
-\bar{\alpha}^{(2)}_{4}\bar{\alpha}^{(0)}_{1}\theta_{6}
\\
-\bar{\alpha}^{(2)}_{6}\theta_{6}
-\bar{\alpha}^{(2)}_{3}\alpha^{(0)}_{1}\theta_{7}
-\alpha^{(2)}_{3}\alpha^{(0)}_{1}\theta_{7}
-\bar{\alpha}^{(2)}_{4}\alpha^{(0)}_{1}\theta_{7}
+\alpha^{(2)}_{3}\theta_{8}
+\bar{\alpha}^{(2)}_{4}\theta_{8}
+\alpha^{(0)}_{2}\theta_{8}
-\bar{\alpha}^{(0)}_{1}\alpha^{(0)}_{1}\theta_{8}
+\alpha^{(2)}_{3}\theta_{8}
+\bar{\alpha}^{(2)}_{4}\theta_{8}
-\alpha^{(0)}_{1}\alpha^{(0)}_{1}\theta_{8}
+\alpha^{(0)}_{2}\theta_{8}
\\
+\bar{\alpha}^{(0)}_{1}\alpha^{(0)}_{1}\theta_{8}
-\bar{\alpha}^{(0)}_{1}\bar{\alpha}^{(0)}_{1}\theta_{8}
-\bar{\alpha}^{(0)}_{1}\alpha^{(0)}_{1}\theta_{8}
-\alpha^{(0)}_{3}\theta_{9}
-\bar{\alpha}^{(0)}_{1}\alpha^{(0)}_{1}\theta_{9}
+\alpha^{(0)}_{3}\theta_{10}
-\bar{\alpha}^{(0)}_{1}\alpha^{(0)}_{1}\theta_{10}
-\alpha^{(0)}_{1}\alpha^{(0)}_{1}\theta_{10}
-\alpha^{(2)}_{3}\theta_{11}
+\bar{\alpha}^{(2)}_{3}\theta_{11}
-\bar{\alpha}^{(2)}_{4}\theta_{11}
\\
+\alpha^{(2)}_{3}\theta_{11}
-\alpha^{(2)}_{4}\theta_{11}
+\bar{\alpha}^{(2)}_{4}\theta_{11}
+\alpha^{(0)}_{2}\theta_{11}
+\alpha^{(2)}_{3}\theta_{11}
-\bar{\alpha}^{(2)}_{3}\theta_{11}
+\bar{\alpha}^{(2)}_{4}\theta_{11}
+\alpha^{(0)}_{1}\theta_{12}
-\bar{\alpha}^{(0)}_{1}\theta_{13}
-\alpha^{(0)}_{1}\theta_{13}
+\bar{\alpha}^{(0)}_{1}\theta_{13}
+\alpha^{(0)}_{1}\theta_{14}
-\bar{\alpha}^{(0)}_{1}\theta_{14}
\\
-\alpha^{(0)}_{1}\theta_{14}
-\alpha^{(0)}_{1}\theta_{15}
+\theta_{16}
&= 0
,
\\
-\bar{\alpha}^{(3)}_{14}
-\alpha^{(3)}_{1}\bar{\alpha}^{(3)}_{10}
-\alpha^{(2)}_{10}\theta_{1}
+\bar{\alpha}^{(2)}_{4}\alpha^{(2)}_{4}\theta_{1}
+\alpha^{(0)}_{2}\alpha^{(2)}_{4}\theta_{1}
-\bar{\alpha}^{(2)}_{3}\alpha^{(0)}_{2}\theta_{2}
-\bar{\alpha}^{(2)}_{3}\alpha^{(2)}_{3}\theta_{2}
+\bar{\alpha}^{(2)}_{8}\theta_{2}
+\alpha^{(0)}_{5}\theta_{4}
-\bar{\alpha}^{(0)}_{1}\alpha^{(0)}_{2}\theta_{4}
+\alpha^{(2)}_{3}\bar{\alpha}^{(0)}_{1}\theta_{4}
-\alpha^{(2)}_{4}\bar{\alpha}^{(0)}_{1}\theta_{4}
\\
-\alpha^{(2)}_{7}\theta_{4}
+\bar{\alpha}^{(2)}_{4}\bar{\alpha}^{(0)}_{1}\theta_{4}
+\bar{\alpha}^{(2)}_{6}\theta_{4}
-\alpha^{(0)}_{1}\alpha^{(0)}_{2}\theta_{4}
-\alpha^{(0)}_{1}\alpha^{(2)}_{3}\theta_{4}
+\alpha^{(0)}_{1}\bar{\alpha}^{(2)}_{3}\theta_{4}
-\alpha^{(0)}_{1}\bar{\alpha}^{(2)}_{4}\theta_{4}
+\alpha^{(0)}_{2}\bar{\alpha}^{(0)}_{1}\theta_{4}
+\alpha^{(0)}_{4}\theta_{4}
-\alpha^{(0)}_{5}\theta_{4}
+\alpha^{(2)}_{4}\theta_{8}
\\
-\alpha^{(0)}_{2}\theta_{9}
-\alpha^{(2)}_{3}\theta_{9}
+\bar{\alpha}^{(2)}_{3}\theta_{9}
-\bar{\alpha}^{(2)}_{4}\theta_{9}
+\alpha^{(2)}_{3}\theta_{9}
-\alpha^{(2)}_{4}\theta_{9}
+\bar{\alpha}^{(2)}_{4}\theta_{9}
+\alpha^{(0)}_{2}\theta_{9}
-\bar{\alpha}^{(2)}_{3}\theta_{10}
+\bar{\alpha}^{(0)}_{1}\theta_{12}
+\alpha^{(0)}_{1}\theta_{12}
-\bar{\alpha}^{(0)}_{1}\theta_{12}
-\alpha^{(0)}_{1}\theta_{12}
+\theta_{17}
&= 0
,
\\
-\bar{\alpha}^{(3)}_{12}
+\bar{\alpha}^{(3)}_{14}
-\bar{\alpha}^{(2)}_{4}\alpha^{(2)}_{7}
-\bar{\alpha}^{(2)}_{6}\alpha^{(2)}_{3}
+\bar{\alpha}^{(0)}_{1}\alpha^{(0)}_{6}
-\bar{\alpha}^{(0)}_{1}\alpha^{(2)}_{7}\theta_{1}
-\bar{\alpha}^{(0)}_{1}\alpha^{(0)}_{1}\alpha^{(2)}_{3}\theta_{1}
-\bar{\alpha}^{(0)}_{1}\alpha^{(0)}_{1}\bar{\alpha}^{(2)}_{4}\theta_{1}
-\bar{\alpha}^{(0)}_{1}\alpha^{(0)}_{5}\theta_{1}
+\alpha^{(0)}_{6}\theta_{2}
+\bar{\alpha}^{(0)}_{1}\alpha^{(0)}_{4}\theta_{2}
\\
-\bar{\alpha}^{(0)}_{1}\alpha^{(0)}_{5}\theta_{2}
+\alpha^{(2)}_{3}\alpha^{(0)}_{2}\theta_{2}
-\alpha^{(2)}_{3}\bar{\alpha}^{(0)}_{1}\alpha^{(0)}_{1}\theta_{2}
+\alpha^{(2)}_{4}\bar{\alpha}^{(0)}_{1}\alpha^{(0)}_{1}\theta_{2}
+\alpha^{(2)}_{7}\alpha^{(0)}_{1}\theta_{2}
-\alpha^{(2)}_{7}\bar{\alpha}^{(0)}_{1}\theta_{2}
+\bar{\alpha}^{(2)}_{4}\alpha^{(0)}_{2}\theta_{2}
-\bar{\alpha}^{(2)}_{4}\bar{\alpha}^{(0)}_{1}\alpha^{(0)}_{1}\theta_{2}
+\bar{\alpha}^{(2)}_{4}\alpha^{(2)}_{3}\theta_{2}
\\
-\bar{\alpha}^{(2)}_{6}\alpha^{(0)}_{1}\theta_{2}
-\bar{\alpha}^{(2)}_{3}\alpha^{(2)}_{3}\theta_{3}
+\bar{\alpha}^{(2)}_{4}\alpha^{(2)}_{3}\theta_{3}
-\bar{\alpha}^{(2)}_{4}\alpha^{(2)}_{4}\theta_{3}
-\alpha^{(2)}_{3}\alpha^{(2)}_{3}\theta_{3}
-\alpha^{(2)}_{3}\bar{\alpha}^{(2)}_{4}\theta_{3}
+\alpha^{(2)}_{4}\alpha^{(2)}_{3}\theta_{3}
-\alpha^{(2)}_{4}\alpha^{(2)}_{4}\theta_{3}
+\alpha^{(2)}_{4}\bar{\alpha}^{(2)}_{4}\theta_{3}
+\alpha^{(2)}_{10}\theta_{3}
\\
-\bar{\alpha}^{(2)}_{4}\alpha^{(2)}_{3}\theta_{3}
-\bar{\alpha}^{(2)}_{4}\bar{\alpha}^{(2)}_{4}\theta_{3}
-\bar{\alpha}^{(2)}_{4}\alpha^{(2)}_{4}\theta_{3}
+\bar{\alpha}^{(0)}_{1}\alpha^{(0)}_{2}\theta_{4}
-\bar{\alpha}^{(0)}_{1}\bar{\alpha}^{(0)}_{1}\alpha^{(0)}_{1}\theta_{4}
+\bar{\alpha}^{(0)}_{1}\alpha^{(2)}_{3}\theta_{4}
+\bar{\alpha}^{(0)}_{1}\bar{\alpha}^{(2)}_{4}\theta_{4}
-\bar{\alpha}^{(0)}_{1}\alpha^{(0)}_{1}\alpha^{(0)}_{1}\theta_{4}
+\bar{\alpha}^{(0)}_{1}\alpha^{(0)}_{2}\theta_{4}
\\
-\bar{\alpha}^{(0)}_{1}\alpha^{(2)}_{3}\theta_{5}
-\bar{\alpha}^{(0)}_{1}\bar{\alpha}^{(2)}_{4}\theta_{5}
-\bar{\alpha}^{(0)}_{1}\alpha^{(2)}_{4}\theta_{5}
-\alpha^{(2)}_{7}\theta_{6}
-\alpha^{(0)}_{1}\alpha^{(2)}_{3}\theta_{6}
-\alpha^{(0)}_{1}\bar{\alpha}^{(2)}_{4}\theta_{6}
-\alpha^{(0)}_{5}\theta_{6}
+\bar{\alpha}^{(0)}_{1}\alpha^{(2)}_{3}\theta_{6}
-\bar{\alpha}^{(0)}_{1}\alpha^{(2)}_{4}\theta_{6}
+\bar{\alpha}^{(0)}_{1}\bar{\alpha}^{(2)}_{4}\theta_{6}
\\
+\bar{\alpha}^{(0)}_{1}\alpha^{(0)}_{2}\theta_{6}
-\alpha^{(2)}_{3}\bar{\alpha}^{(0)}_{1}\theta_{6}
+\alpha^{(2)}_{4}\bar{\alpha}^{(0)}_{1}\theta_{6}
+\alpha^{(2)}_{7}\theta_{6}
-\bar{\alpha}^{(2)}_{4}\bar{\alpha}^{(0)}_{1}\theta_{6}
-\bar{\alpha}^{(2)}_{6}\theta_{6}
+\alpha^{(0)}_{4}\theta_{7}
-\alpha^{(0)}_{5}\theta_{7}
+\alpha^{(2)}_{3}\alpha^{(0)}_{1}\theta_{7}
-\alpha^{(2)}_{3}\bar{\alpha}^{(0)}_{1}\theta_{7}
-\bar{\alpha}^{(2)}_{3}\alpha^{(0)}_{1}\theta_{7}
\\
+\bar{\alpha}^{(2)}_{4}\alpha^{(0)}_{1}\theta_{7}
-\bar{\alpha}^{(2)}_{4}\bar{\alpha}^{(0)}_{1}\theta_{7}
-\bar{\alpha}^{(2)}_{6}\theta_{7}
-\alpha^{(2)}_{3}\alpha^{(0)}_{1}\theta_{7}
+\alpha^{(2)}_{4}\alpha^{(0)}_{1}\theta_{7}
-\alpha^{(2)}_{4}\bar{\alpha}^{(0)}_{1}\theta_{7}
-\alpha^{(2)}_{7}\theta_{7}
-\bar{\alpha}^{(2)}_{4}\alpha^{(0)}_{1}\theta_{7}
-\bar{\alpha}^{(0)}_{1}\bar{\alpha}^{(0)}_{1}\theta_{8}
-\bar{\alpha}^{(0)}_{1}\alpha^{(0)}_{1}\theta_{8}
\\
-\bar{\alpha}^{(0)}_{1}\alpha^{(0)}_{1}\theta_{9}
+\alpha^{(0)}_{2}\theta_{10}
-\bar{\alpha}^{(0)}_{1}\alpha^{(0)}_{1}\theta_{10}
+\alpha^{(2)}_{3}\theta_{10}
+\bar{\alpha}^{(2)}_{4}\theta_{10}
-\alpha^{(0)}_{1}\alpha^{(0)}_{1}\theta_{10}
+\alpha^{(0)}_{2}\theta_{10}
+\bar{\alpha}^{(0)}_{1}\alpha^{(0)}_{1}\theta_{10}
-\bar{\alpha}^{(0)}_{1}\bar{\alpha}^{(0)}_{1}\theta_{10}
-\bar{\alpha}^{(0)}_{1}\alpha^{(0)}_{1}\theta_{10}
\\
+\alpha^{(2)}_{3}\theta_{10}
+\bar{\alpha}^{(2)}_{4}\theta_{10}
-\alpha^{(2)}_{3}\theta_{11}
-\bar{\alpha}^{(2)}_{4}\theta_{11}
-\alpha^{(2)}_{4}\theta_{11}
+\bar{\alpha}^{(0)}_{1}\theta_{12}
-\bar{\alpha}^{(0)}_{1}\theta_{14}
-\alpha^{(0)}_{1}\theta_{14}
+\bar{\alpha}^{(0)}_{1}\theta_{14}
-\alpha^{(0)}_{1}\theta_{15}
-\bar{\alpha}^{(0)}_{1}\theta_{15}
+\theta_{18}
&= 0
,
\\
-\bar{\alpha}^{(3)}_{16}
-\alpha^{(3)}_{1}\bar{\alpha}^{(3)}_{12}
+\alpha^{(3)}_{1}\bar{\alpha}^{(3)}_{14}
+\alpha^{(2)}_{14}
-\bar{\alpha}^{(2)}_{4}\alpha^{(2)}_{10}
-\bar{\alpha}^{(2)}_{8}\alpha^{(2)}_{3}
-\bar{\alpha}^{(2)}_{13}
-\bar{\alpha}^{(2)}_{4}\alpha^{(2)}_{7}\theta_{1}
-\bar{\alpha}^{(2)}_{6}\alpha^{(2)}_{3}\theta_{1}
-\alpha^{(0)}_{1}\bar{\alpha}^{(2)}_{3}\alpha^{(2)}_{3}\theta_{1}
+\alpha^{(0)}_{1}\bar{\alpha}^{(2)}_{4}\alpha^{(2)}_{3}\theta_{1}
\\
-\alpha^{(0)}_{1}\bar{\alpha}^{(2)}_{4}\alpha^{(2)}_{4}\theta_{1}
-\alpha^{(0)}_{2}\alpha^{(2)}_{7}\theta_{1}
-\alpha^{(0)}_{4}\alpha^{(2)}_{3}\theta_{1}
-\alpha^{(0)}_{4}\bar{\alpha}^{(2)}_{4}\theta_{1}
+\alpha^{(0)}_{5}\alpha^{(2)}_{3}\theta_{1}
-\alpha^{(0)}_{5}\alpha^{(2)}_{4}\theta_{1}
+\alpha^{(0)}_{5}\bar{\alpha}^{(2)}_{4}\theta_{1}
-\bar{\alpha}^{(0)}_{1}\alpha^{(2)}_{10}\theta_{1}
+\bar{\alpha}^{(0)}_{1}\bar{\alpha}^{(2)}_{4}\alpha^{(2)}_{4}\theta_{1}
\\
+\bar{\alpha}^{(0)}_{1}\alpha^{(0)}_{2}\alpha^{(2)}_{4}\theta_{1}
+\alpha^{(2)}_{3}\alpha^{(0)}_{5}\theta_{2}
-\alpha^{(2)}_{3}\bar{\alpha}^{(0)}_{1}\alpha^{(0)}_{2}\theta_{2}
+\bar{\alpha}^{(2)}_{3}\alpha^{(0)}_{4}\theta_{2}
-\bar{\alpha}^{(2)}_{3}\alpha^{(0)}_{5}\theta_{2}
+\bar{\alpha}^{(2)}_{3}\alpha^{(2)}_{3}\alpha^{(0)}_{1}\theta_{2}
-\bar{\alpha}^{(2)}_{3}\alpha^{(2)}_{3}\bar{\alpha}^{(0)}_{1}\theta_{2}
+\bar{\alpha}^{(2)}_{4}\alpha^{(0)}_{5}\theta_{2}
-\bar{\alpha}^{(2)}_{4}\bar{\alpha}^{(0)}_{1}\alpha^{(0)}_{2}\theta_{2}
\\
+\bar{\alpha}^{(2)}_{4}\alpha^{(2)}_{3}\bar{\alpha}^{(0)}_{1}\theta_{2}
-\bar{\alpha}^{(2)}_{4}\alpha^{(2)}_{4}\bar{\alpha}^{(0)}_{1}\theta_{2}
-\bar{\alpha}^{(2)}_{4}\alpha^{(2)}_{7}\theta_{2}
-\bar{\alpha}^{(2)}_{6}\alpha^{(0)}_{2}\theta_{2}
-\bar{\alpha}^{(2)}_{6}\alpha^{(2)}_{3}\theta_{2}
-\bar{\alpha}^{(2)}_{8}\alpha^{(0)}_{1}\theta_{2}
+\bar{\alpha}^{(2)}_{11}\theta_{2}
+\alpha^{(0)}_{6}\theta_{4}
+\bar{\alpha}^{(0)}_{1}\alpha^{(0)}_{4}\theta_{4}
-\bar{\alpha}^{(0)}_{1}\alpha^{(0)}_{5}\theta_{4}
\\
+\alpha^{(2)}_{3}\alpha^{(0)}_{2}\theta_{4}
-\alpha^{(2)}_{3}\bar{\alpha}^{(0)}_{1}\alpha^{(0)}_{1}\theta_{4}
+\alpha^{(2)}_{4}\bar{\alpha}^{(0)}_{1}\alpha^{(0)}_{1}\theta_{4}
+\alpha^{(2)}_{7}\alpha^{(0)}_{1}\theta_{4}
-\alpha^{(2)}_{7}\bar{\alpha}^{(0)}_{1}\theta_{4}
+\bar{\alpha}^{(2)}_{4}\alpha^{(0)}_{2}\theta_{4}
-\bar{\alpha}^{(2)}_{4}\bar{\alpha}^{(0)}_{1}\alpha^{(0)}_{1}\theta_{4}
+\bar{\alpha}^{(2)}_{4}\alpha^{(2)}_{3}\theta_{4}
-\bar{\alpha}^{(2)}_{6}\alpha^{(0)}_{1}\theta_{4}
\\
+\alpha^{(0)}_{1}\alpha^{(0)}_{4}\theta_{4}
-\alpha^{(0)}_{1}\alpha^{(0)}_{5}\theta_{4}
+\alpha^{(0)}_{1}\alpha^{(2)}_{3}\alpha^{(0)}_{1}\theta_{4}
-\alpha^{(0)}_{1}\alpha^{(2)}_{3}\bar{\alpha}^{(0)}_{1}\theta_{4}
-\alpha^{(0)}_{1}\bar{\alpha}^{(2)}_{3}\alpha^{(0)}_{1}\theta_{4}
+\alpha^{(0)}_{1}\bar{\alpha}^{(2)}_{4}\alpha^{(0)}_{1}\theta_{4}
-\alpha^{(0)}_{1}\bar{\alpha}^{(2)}_{4}\bar{\alpha}^{(0)}_{1}\theta_{4}
-\alpha^{(0)}_{1}\bar{\alpha}^{(2)}_{6}\theta_{4}
+\alpha^{(0)}_{2}\alpha^{(0)}_{2}\theta_{4}
\\
-\alpha^{(0)}_{2}\bar{\alpha}^{(0)}_{1}\alpha^{(0)}_{1}\theta_{4}
+\alpha^{(0)}_{2}\alpha^{(2)}_{3}\theta_{4}
+\alpha^{(0)}_{2}\bar{\alpha}^{(2)}_{4}\theta_{4}
-\alpha^{(0)}_{3}\bar{\alpha}^{(2)}_{3}\theta_{4}
-\alpha^{(0)}_{4}\alpha^{(0)}_{1}\theta_{4}
+\alpha^{(0)}_{5}\alpha^{(0)}_{1}\theta_{4}
-\alpha^{(0)}_{5}\bar{\alpha}^{(0)}_{1}\theta_{4}
+\alpha^{(0)}_{6}\theta_{4}
+\bar{\alpha}^{(0)}_{1}\alpha^{(0)}_{5}\theta_{4}
-\bar{\alpha}^{(0)}_{1}\bar{\alpha}^{(0)}_{1}\alpha^{(0)}_{2}\theta_{4}
\\
+\bar{\alpha}^{(0)}_{1}\alpha^{(2)}_{3}\bar{\alpha}^{(0)}_{1}\theta_{4}
-\bar{\alpha}^{(0)}_{1}\alpha^{(2)}_{4}\bar{\alpha}^{(0)}_{1}\theta_{4}
-\bar{\alpha}^{(0)}_{1}\alpha^{(2)}_{7}\theta_{4}
+\bar{\alpha}^{(0)}_{1}\bar{\alpha}^{(2)}_{4}\bar{\alpha}^{(0)}_{1}\theta_{4}
+\bar{\alpha}^{(0)}_{1}\bar{\alpha}^{(2)}_{6}\theta_{4}
-\bar{\alpha}^{(0)}_{1}\alpha^{(0)}_{1}\alpha^{(0)}_{2}\theta_{4}
-\bar{\alpha}^{(0)}_{1}\alpha^{(0)}_{1}\alpha^{(2)}_{3}\theta_{4}
+\bar{\alpha}^{(0)}_{1}\alpha^{(0)}_{1}\bar{\alpha}^{(2)}_{3}\theta_{4}
\\
-\bar{\alpha}^{(0)}_{1}\alpha^{(0)}_{1}\bar{\alpha}^{(2)}_{4}\theta_{4}
+\bar{\alpha}^{(0)}_{1}\alpha^{(0)}_{2}\bar{\alpha}^{(0)}_{1}\theta_{4}
+\bar{\alpha}^{(0)}_{1}\alpha^{(0)}_{4}\theta_{4}
-\bar{\alpha}^{(0)}_{1}\alpha^{(0)}_{5}\theta_{4}
-\bar{\alpha}^{(2)}_{3}\alpha^{(2)}_{3}\theta_{5}
+\bar{\alpha}^{(2)}_{4}\alpha^{(2)}_{3}\theta_{5}
-\bar{\alpha}^{(2)}_{4}\alpha^{(2)}_{4}\theta_{5}
-\alpha^{(2)}_{3}\alpha^{(2)}_{3}\theta_{5}
-\alpha^{(2)}_{3}\bar{\alpha}^{(2)}_{4}\theta_{5}
\\
+\alpha^{(2)}_{4}\alpha^{(2)}_{3}\theta_{5}
-\alpha^{(2)}_{4}\alpha^{(2)}_{4}\theta_{5}
+\alpha^{(2)}_{4}\bar{\alpha}^{(2)}_{4}\theta_{5}
+\alpha^{(2)}_{10}\theta_{5}
-\bar{\alpha}^{(2)}_{4}\alpha^{(2)}_{3}\theta_{5}
-\bar{\alpha}^{(2)}_{4}\bar{\alpha}^{(2)}_{4}\theta_{5}
-\bar{\alpha}^{(2)}_{4}\alpha^{(2)}_{4}\theta_{5}
-\alpha^{(0)}_{2}\alpha^{(2)}_{3}\theta_{5}
-\alpha^{(0)}_{2}\bar{\alpha}^{(2)}_{4}\theta_{5}
-\alpha^{(0)}_{2}\alpha^{(2)}_{4}\theta_{5}
-\alpha^{(2)}_{10}\theta_{6}
\\
+\bar{\alpha}^{(2)}_{4}\alpha^{(2)}_{4}\theta_{6}
+\alpha^{(0)}_{2}\alpha^{(2)}_{4}\theta_{6}
+\alpha^{(2)}_{3}\alpha^{(2)}_{4}\theta_{6}
+\bar{\alpha}^{(2)}_{3}\alpha^{(2)}_{3}\theta_{6}
-\bar{\alpha}^{(2)}_{3}\alpha^{(2)}_{4}\theta_{6}
+\bar{\alpha}^{(2)}_{3}\bar{\alpha}^{(2)}_{4}\theta_{6}
+\bar{\alpha}^{(2)}_{3}\alpha^{(0)}_{2}\theta_{6}
+\bar{\alpha}^{(2)}_{3}\alpha^{(2)}_{3}\theta_{6}
+\bar{\alpha}^{(2)}_{4}\alpha^{(2)}_{4}\theta_{6}
-\bar{\alpha}^{(2)}_{8}\theta_{6}
\\
-\bar{\alpha}^{(2)}_{3}\alpha^{(0)}_{2}\theta_{7}
-\bar{\alpha}^{(2)}_{3}\alpha^{(2)}_{3}\theta_{7}
+\bar{\alpha}^{(2)}_{8}\theta_{7}
-\alpha^{(2)}_{3}\alpha^{(0)}_{2}\theta_{7}
-\alpha^{(2)}_{3}\alpha^{(2)}_{3}\theta_{7}
+\alpha^{(2)}_{3}\bar{\alpha}^{(2)}_{3}\theta_{7}
-\alpha^{(2)}_{3}\bar{\alpha}^{(2)}_{4}\theta_{7}
-\bar{\alpha}^{(2)}_{3}\bar{\alpha}^{(2)}_{3}\theta_{7}
-\bar{\alpha}^{(2)}_{3}\alpha^{(2)}_{3}\theta_{7}
-\bar{\alpha}^{(2)}_{4}\alpha^{(0)}_{2}\theta_{7}
\\
-\bar{\alpha}^{(2)}_{4}\alpha^{(2)}_{3}\theta_{7}
+\bar{\alpha}^{(2)}_{4}\bar{\alpha}^{(2)}_{3}\theta_{7}
-\bar{\alpha}^{(2)}_{4}\bar{\alpha}^{(2)}_{4}\theta_{7}
+\bar{\alpha}^{(2)}_{4}\alpha^{(2)}_{3}\theta_{7}
-\bar{\alpha}^{(2)}_{4}\alpha^{(2)}_{4}\theta_{7}
-\alpha^{(2)}_{7}\theta_{8}
-\alpha^{(0)}_{1}\alpha^{(2)}_{3}\theta_{8}
-\alpha^{(0)}_{1}\bar{\alpha}^{(2)}_{4}\theta_{8}
-\alpha^{(0)}_{5}\theta_{8}
+\bar{\alpha}^{(0)}_{1}\alpha^{(2)}_{3}\theta_{8}
-\bar{\alpha}^{(0)}_{1}\alpha^{(2)}_{4}\theta_{8}
\\
+\bar{\alpha}^{(0)}_{1}\bar{\alpha}^{(2)}_{4}\theta_{8}
+\bar{\alpha}^{(0)}_{1}\alpha^{(0)}_{2}\theta_{8}
-\alpha^{(2)}_{3}\bar{\alpha}^{(0)}_{1}\theta_{8}
+\alpha^{(2)}_{4}\bar{\alpha}^{(0)}_{1}\theta_{8}
+\alpha^{(2)}_{7}\theta_{8}
-\bar{\alpha}^{(2)}_{4}\bar{\alpha}^{(0)}_{1}\theta_{8}
-\bar{\alpha}^{(2)}_{6}\theta_{8}
+\alpha^{(0)}_{1}\alpha^{(2)}_{3}\theta_{8}
-\alpha^{(0)}_{1}\alpha^{(2)}_{4}\theta_{8}
+\alpha^{(0)}_{1}\bar{\alpha}^{(2)}_{4}\theta_{8}
+\alpha^{(0)}_{1}\alpha^{(0)}_{2}\theta_{8}
\\
+\alpha^{(0)}_{1}\alpha^{(2)}_{3}\theta_{8}
-\alpha^{(0)}_{1}\bar{\alpha}^{(2)}_{3}\theta_{8}
+\alpha^{(0)}_{1}\bar{\alpha}^{(2)}_{4}\theta_{8}
-\alpha^{(0)}_{2}\bar{\alpha}^{(0)}_{1}\theta_{8}
-\alpha^{(0)}_{4}\theta_{8}
+\alpha^{(0)}_{5}\theta_{8}
+\bar{\alpha}^{(0)}_{1}\alpha^{(2)}_{4}\theta_{8}
+\alpha^{(0)}_{4}\theta_{9}
-\alpha^{(0)}_{5}\theta_{9}
+\alpha^{(2)}_{3}\alpha^{(0)}_{1}\theta_{9}
-\alpha^{(2)}_{3}\bar{\alpha}^{(0)}_{1}\theta_{9}
\\
-\bar{\alpha}^{(2)}_{3}\alpha^{(0)}_{1}\theta_{9}
+\bar{\alpha}^{(2)}_{4}\alpha^{(0)}_{1}\theta_{9}
-\bar{\alpha}^{(2)}_{4}\bar{\alpha}^{(0)}_{1}\theta_{9}
-\bar{\alpha}^{(2)}_{6}\theta_{9}
-\alpha^{(2)}_{3}\alpha^{(0)}_{1}\theta_{9}
+\alpha^{(2)}_{4}\alpha^{(0)}_{1}\theta_{9}
-\alpha^{(2)}_{4}\bar{\alpha}^{(0)}_{1}\theta_{9}
-\alpha^{(2)}_{7}\theta_{9}
-\bar{\alpha}^{(2)}_{4}\alpha^{(0)}_{1}\theta_{9}
-\alpha^{(0)}_{1}\bar{\alpha}^{(2)}_{3}\theta_{9}
-\alpha^{(0)}_{1}\alpha^{(2)}_{3}\theta_{9}
\\
-\alpha^{(0)}_{1}\bar{\alpha}^{(2)}_{4}\theta_{9}
-\alpha^{(0)}_{2}\alpha^{(0)}_{1}\theta_{9}
-\alpha^{(0)}_{5}\theta_{9}
-\bar{\alpha}^{(0)}_{1}\alpha^{(0)}_{2}\theta_{9}
-\bar{\alpha}^{(0)}_{1}\alpha^{(2)}_{3}\theta_{9}
+\bar{\alpha}^{(0)}_{1}\bar{\alpha}^{(2)}_{3}\theta_{9}
-\bar{\alpha}^{(0)}_{1}\bar{\alpha}^{(2)}_{4}\theta_{9}
+\bar{\alpha}^{(0)}_{1}\alpha^{(2)}_{3}\theta_{9}
-\bar{\alpha}^{(0)}_{1}\alpha^{(2)}_{4}\theta_{9}
+\bar{\alpha}^{(0)}_{1}\bar{\alpha}^{(2)}_{4}\theta_{9}
\\
+\bar{\alpha}^{(0)}_{1}\alpha^{(0)}_{2}\theta_{9}
+\alpha^{(0)}_{5}\theta_{10}
-\bar{\alpha}^{(0)}_{1}\alpha^{(0)}_{2}\theta_{10}
+\alpha^{(2)}_{3}\bar{\alpha}^{(0)}_{1}\theta_{10}
-\alpha^{(2)}_{4}\bar{\alpha}^{(0)}_{1}\theta_{10}
-\alpha^{(2)}_{7}\theta_{10}
+\bar{\alpha}^{(2)}_{4}\bar{\alpha}^{(0)}_{1}\theta_{10}
+\bar{\alpha}^{(2)}_{6}\theta_{10}
-\alpha^{(0)}_{1}\alpha^{(0)}_{2}\theta_{10}
-\alpha^{(0)}_{1}\alpha^{(2)}_{3}\theta_{10}
\\
+\alpha^{(0)}_{1}\bar{\alpha}^{(2)}_{3}\theta_{10}
-\alpha^{(0)}_{1}\bar{\alpha}^{(2)}_{4}\theta_{10}
+\alpha^{(0)}_{2}\bar{\alpha}^{(0)}_{1}\theta_{10}
+\alpha^{(0)}_{4}\theta_{10}
-\alpha^{(0)}_{5}\theta_{10}
+\alpha^{(2)}_{3}\bar{\alpha}^{(0)}_{1}\theta_{10}
+\alpha^{(2)}_{3}\alpha^{(0)}_{1}\theta_{10}
-\alpha^{(2)}_{3}\bar{\alpha}^{(0)}_{1}\theta_{10}
+\bar{\alpha}^{(2)}_{3}\alpha^{(0)}_{1}\theta_{10}
-\bar{\alpha}^{(2)}_{3}\bar{\alpha}^{(0)}_{1}\theta_{10}
\end{align*}
\begin{align*}
-\bar{\alpha}^{(2)}_{3}\alpha^{(0)}_{1}\theta_{10}
+\bar{\alpha}^{(2)}_{4}\bar{\alpha}^{(0)}_{1}\theta_{10}
+\bar{\alpha}^{(2)}_{4}\alpha^{(0)}_{1}\theta_{10}
-\bar{\alpha}^{(2)}_{4}\bar{\alpha}^{(0)}_{1}\theta_{10}
-\bar{\alpha}^{(2)}_{6}\theta_{10}
+\alpha^{(0)}_{2}\theta_{12}
-\bar{\alpha}^{(0)}_{1}\alpha^{(0)}_{1}\theta_{12}
+\alpha^{(2)}_{3}\theta_{12}
+\bar{\alpha}^{(2)}_{4}\theta_{12}
-\alpha^{(0)}_{1}\alpha^{(0)}_{1}\theta_{12}
+\alpha^{(0)}_{2}\theta_{12}
\\
+\bar{\alpha}^{(0)}_{1}\alpha^{(0)}_{1}\theta_{12}
-\bar{\alpha}^{(0)}_{1}\bar{\alpha}^{(0)}_{1}\theta_{12}
-\bar{\alpha}^{(0)}_{1}\alpha^{(0)}_{1}\theta_{12}
+\alpha^{(2)}_{3}\theta_{12}
+\bar{\alpha}^{(2)}_{4}\theta_{12}
+\alpha^{(0)}_{1}\alpha^{(0)}_{1}\theta_{12}
-\alpha^{(0)}_{1}\bar{\alpha}^{(0)}_{1}\theta_{12}
-\alpha^{(0)}_{1}\alpha^{(0)}_{1}\theta_{12}
+\alpha^{(0)}_{2}\theta_{12}
+\bar{\alpha}^{(0)}_{1}\bar{\alpha}^{(0)}_{1}\theta_{12}
\\
+\bar{\alpha}^{(0)}_{1}\alpha^{(0)}_{1}\theta_{12}
-\bar{\alpha}^{(0)}_{1}\bar{\alpha}^{(0)}_{1}\theta_{12}
-\bar{\alpha}^{(0)}_{1}\alpha^{(0)}_{1}\theta_{12}
-\alpha^{(2)}_{3}\theta_{13}
-\bar{\alpha}^{(2)}_{4}\theta_{13}
-\alpha^{(2)}_{4}\theta_{13}
+\alpha^{(2)}_{4}\theta_{14}
+\bar{\alpha}^{(2)}_{3}\theta_{14}
-\alpha^{(0)}_{2}\theta_{15}
-\alpha^{(2)}_{3}\theta_{15}
+\bar{\alpha}^{(2)}_{3}\theta_{15}
-\bar{\alpha}^{(2)}_{4}\theta_{15}
\\
+\alpha^{(2)}_{3}\theta_{15}
-\alpha^{(2)}_{4}\theta_{15}
+\bar{\alpha}^{(2)}_{4}\theta_{15}
+\alpha^{(0)}_{2}\theta_{15}
+\alpha^{(2)}_{3}\theta_{15}
-\bar{\alpha}^{(2)}_{3}\theta_{15}
+\bar{\alpha}^{(2)}_{4}\theta_{15}
-\bar{\alpha}^{(0)}_{1}\theta_{16}
-\alpha^{(0)}_{1}\theta_{16}
+\bar{\alpha}^{(0)}_{1}\theta_{16}
+\alpha^{(0)}_{1}\theta_{16}
-\alpha^{(0)}_{1}\theta_{17}
-\bar{\alpha}^{(0)}_{1}\theta_{17}
\\
-\alpha^{(0)}_{1}\theta_{17}
+\bar{\alpha}^{(0)}_{1}\theta_{17}
+\bar{\alpha}^{(0)}_{1}\theta_{18}
+\alpha^{(0)}_{1}\theta_{18}
-\bar{\alpha}^{(0)}_{1}\theta_{18}
-\alpha^{(0)}_{1}\theta_{18}
+\theta_{19}
&= 0
,
\\
\alpha^{(2)}_{13}
+\bar{\alpha}^{(0)}_{1}\alpha^{(0)}_{8}
+\bar{\alpha}^{(0)}_{1}\bar{\alpha}^{(2)}_{4}\alpha^{(2)}_{3}\theta_{1}
-\bar{\alpha}^{(0)}_{1}\alpha^{(0)}_{1}\bar{\alpha}^{(2)}_{6}\theta_{1}
+\bar{\alpha}^{(0)}_{1}\alpha^{(0)}_{2}\alpha^{(2)}_{3}\theta_{1}
+\bar{\alpha}^{(0)}_{1}\alpha^{(0)}_{2}\bar{\alpha}^{(2)}_{4}\theta_{1}
-\bar{\alpha}^{(0)}_{1}\alpha^{(0)}_{3}\bar{\alpha}^{(2)}_{3}\theta_{1}
+\bar{\alpha}^{(0)}_{1}\alpha^{(0)}_{6}\theta_{1}
+\alpha^{(0)}_{8}\theta_{2}
+\bar{\alpha}^{(0)}_{1}\alpha^{(0)}_{7}\theta_{2}
\\
+\alpha^{(2)}_{3}\alpha^{(0)}_{4}\theta_{2}
-\alpha^{(2)}_{3}\bar{\alpha}^{(0)}_{1}\alpha^{(0)}_{3}\theta_{2}
+\alpha^{(2)}_{4}\bar{\alpha}^{(0)}_{1}\alpha^{(0)}_{3}\theta_{2}
+\alpha^{(2)}_{7}\alpha^{(0)}_{3}\theta_{2}
-\alpha^{(2)}_{7}\bar{\alpha}^{(0)}_{1}\alpha^{(0)}_{1}\theta_{2}
+\bar{\alpha}^{(2)}_{4}\alpha^{(0)}_{4}\theta_{2}
-\bar{\alpha}^{(2)}_{4}\bar{\alpha}^{(0)}_{1}\alpha^{(0)}_{3}\theta_{2}
+\bar{\alpha}^{(2)}_{4}\alpha^{(2)}_{3}\alpha^{(0)}_{1}\theta_{2}
-\bar{\alpha}^{(2)}_{6}\alpha^{(0)}_{3}\theta_{2}
\\
-\bar{\alpha}^{(2)}_{4}\alpha^{(2)}_{7}\theta_{3}
-\bar{\alpha}^{(2)}_{6}\alpha^{(2)}_{3}\theta_{3}
+\bar{\alpha}^{(2)}_{11}\theta_{3}
-\alpha^{(2)}_{3}\bar{\alpha}^{(2)}_{6}\theta_{3}
-\alpha^{(2)}_{4}\alpha^{(2)}_{7}\theta_{3}
+\alpha^{(2)}_{4}\bar{\alpha}^{(2)}_{6}\theta_{3}
-\alpha^{(2)}_{7}\alpha^{(2)}_{3}\theta_{3}
+\alpha^{(2)}_{7}\bar{\alpha}^{(2)}_{3}\theta_{3}
-\alpha^{(2)}_{7}\bar{\alpha}^{(2)}_{4}\theta_{3}
-\bar{\alpha}^{(2)}_{4}\bar{\alpha}^{(2)}_{6}\theta_{3}
\\
-\bar{\alpha}^{(2)}_{4}\alpha^{(2)}_{7}\theta_{3}
-\bar{\alpha}^{(2)}_{6}\bar{\alpha}^{(2)}_{3}\theta_{3}
-\bar{\alpha}^{(2)}_{6}\alpha^{(2)}_{3}\theta_{3}
+\bar{\alpha}^{(0)}_{1}\alpha^{(0)}_{4}\theta_{4}
-\bar{\alpha}^{(0)}_{1}\bar{\alpha}^{(0)}_{1}\alpha^{(0)}_{3}\theta_{4}
+\bar{\alpha}^{(0)}_{1}\alpha^{(2)}_{3}\alpha^{(0)}_{1}\theta_{4}
+\bar{\alpha}^{(0)}_{1}\bar{\alpha}^{(2)}_{4}\alpha^{(0)}_{1}\theta_{4}
-\bar{\alpha}^{(0)}_{1}\alpha^{(0)}_{1}\alpha^{(0)}_{3}\theta_{4}
+\bar{\alpha}^{(0)}_{1}\alpha^{(0)}_{2}\alpha^{(0)}_{1}\theta_{4}
\\
-\bar{\alpha}^{(0)}_{1}\bar{\alpha}^{(2)}_{6}\theta_{5}
-\bar{\alpha}^{(0)}_{1}\alpha^{(2)}_{7}\theta_{5}
-\bar{\alpha}^{(0)}_{1}\alpha^{(0)}_{1}\bar{\alpha}^{(2)}_{3}\theta_{5}
-\bar{\alpha}^{(0)}_{1}\alpha^{(0)}_{1}\alpha^{(2)}_{3}\theta_{5}
-\bar{\alpha}^{(0)}_{1}\alpha^{(0)}_{1}\bar{\alpha}^{(2)}_{4}\theta_{5}
-\bar{\alpha}^{(0)}_{1}\alpha^{(0)}_{5}\theta_{5}
+\bar{\alpha}^{(2)}_{4}\alpha^{(2)}_{3}\theta_{6}
-\alpha^{(0)}_{1}\bar{\alpha}^{(2)}_{6}\theta_{6}
+\alpha^{(0)}_{2}\alpha^{(2)}_{3}\theta_{6}
\\
+\alpha^{(0)}_{2}\bar{\alpha}^{(2)}_{4}\theta_{6}
-\alpha^{(0)}_{3}\bar{\alpha}^{(2)}_{3}\theta_{6}
+\alpha^{(0)}_{6}\theta_{6}
-\bar{\alpha}^{(0)}_{1}\alpha^{(2)}_{7}\theta_{6}
+\bar{\alpha}^{(0)}_{1}\bar{\alpha}^{(2)}_{6}\theta_{6}
-\bar{\alpha}^{(0)}_{1}\alpha^{(0)}_{1}\alpha^{(2)}_{3}\theta_{6}
+\bar{\alpha}^{(0)}_{1}\alpha^{(0)}_{1}\bar{\alpha}^{(2)}_{3}\theta_{6}
-\bar{\alpha}^{(0)}_{1}\alpha^{(0)}_{1}\bar{\alpha}^{(2)}_{4}\theta_{6}
+\bar{\alpha}^{(0)}_{1}\alpha^{(0)}_{4}\theta_{6}
\\
-\bar{\alpha}^{(0)}_{1}\alpha^{(0)}_{5}\theta_{6}
+\alpha^{(2)}_{3}\alpha^{(2)}_{3}\theta_{6}
+\alpha^{(2)}_{3}\bar{\alpha}^{(2)}_{4}\theta_{6}
+\alpha^{(2)}_{3}\alpha^{(0)}_{2}\theta_{6}
-\alpha^{(2)}_{3}\bar{\alpha}^{(0)}_{1}\alpha^{(0)}_{1}\theta_{6}
+\alpha^{(2)}_{4}\bar{\alpha}^{(0)}_{1}\alpha^{(0)}_{1}\theta_{6}
+\alpha^{(2)}_{7}\alpha^{(0)}_{1}\theta_{6}
-\alpha^{(2)}_{7}\bar{\alpha}^{(0)}_{1}\theta_{6}
+\bar{\alpha}^{(2)}_{4}\alpha^{(2)}_{3}\theta_{6}
\\
+\bar{\alpha}^{(2)}_{4}\bar{\alpha}^{(2)}_{4}\theta_{6}
+\bar{\alpha}^{(2)}_{4}\alpha^{(0)}_{2}\theta_{6}
-\bar{\alpha}^{(2)}_{4}\bar{\alpha}^{(0)}_{1}\alpha^{(0)}_{1}\theta_{6}
+\bar{\alpha}^{(2)}_{4}\alpha^{(2)}_{3}\theta_{6}
-\bar{\alpha}^{(2)}_{6}\alpha^{(0)}_{1}\theta_{6}
+\alpha^{(0)}_{7}\theta_{7}
+\alpha^{(2)}_{3}\alpha^{(0)}_{3}\theta_{7}
-\alpha^{(2)}_{3}\bar{\alpha}^{(0)}_{1}\alpha^{(0)}_{1}\theta_{7}
-\bar{\alpha}^{(2)}_{3}\alpha^{(0)}_{3}\theta_{7}
+\bar{\alpha}^{(2)}_{4}\alpha^{(0)}_{3}\theta_{7}
\\
-\bar{\alpha}^{(2)}_{4}\bar{\alpha}^{(0)}_{1}\alpha^{(0)}_{1}\theta_{7}
-\bar{\alpha}^{(2)}_{6}\alpha^{(0)}_{1}\theta_{7}
-\alpha^{(2)}_{3}\alpha^{(0)}_{3}\theta_{7}
+\alpha^{(2)}_{4}\alpha^{(0)}_{3}\theta_{7}
-\alpha^{(2)}_{4}\bar{\alpha}^{(0)}_{1}\alpha^{(0)}_{1}\theta_{7}
-\alpha^{(2)}_{7}\alpha^{(0)}_{1}\theta_{7}
-\bar{\alpha}^{(2)}_{4}\alpha^{(0)}_{3}\theta_{7}
+\bar{\alpha}^{(0)}_{1}\alpha^{(2)}_{3}\theta_{8}
+\bar{\alpha}^{(0)}_{1}\bar{\alpha}^{(2)}_{4}\theta_{8}
\\
+\bar{\alpha}^{(0)}_{1}\alpha^{(0)}_{2}\theta_{8}
-\bar{\alpha}^{(0)}_{1}\bar{\alpha}^{(0)}_{1}\alpha^{(0)}_{1}\theta_{8}
+\bar{\alpha}^{(0)}_{1}\alpha^{(2)}_{3}\theta_{8}
+\bar{\alpha}^{(0)}_{1}\bar{\alpha}^{(2)}_{4}\theta_{8}
-\bar{\alpha}^{(0)}_{1}\alpha^{(0)}_{1}\alpha^{(0)}_{1}\theta_{8}
+\bar{\alpha}^{(0)}_{1}\alpha^{(0)}_{2}\theta_{8}
-\bar{\alpha}^{(0)}_{1}\alpha^{(0)}_{3}\theta_{9}
+\alpha^{(0)}_{4}\theta_{10}
-\bar{\alpha}^{(0)}_{1}\alpha^{(0)}_{3}\theta_{10}
\\
+\alpha^{(2)}_{3}\alpha^{(0)}_{1}\theta_{10}
+\bar{\alpha}^{(2)}_{4}\alpha^{(0)}_{1}\theta_{10}
-\alpha^{(0)}_{1}\alpha^{(0)}_{3}\theta_{10}
+\alpha^{(0)}_{2}\alpha^{(0)}_{1}\theta_{10}
+\bar{\alpha}^{(0)}_{1}\alpha^{(0)}_{3}\theta_{10}
-\bar{\alpha}^{(0)}_{1}\bar{\alpha}^{(0)}_{1}\alpha^{(0)}_{1}\theta_{10}
-\bar{\alpha}^{(0)}_{1}\alpha^{(0)}_{1}\alpha^{(0)}_{1}\theta_{10}
+\alpha^{(2)}_{3}\alpha^{(0)}_{1}\theta_{10}
+\bar{\alpha}^{(2)}_{4}\alpha^{(0)}_{1}\theta_{10}
\\
-\bar{\alpha}^{(2)}_{6}\theta_{11}
-\alpha^{(2)}_{7}\theta_{11}
-\alpha^{(0)}_{1}\bar{\alpha}^{(2)}_{3}\theta_{11}
-\alpha^{(0)}_{1}\alpha^{(2)}_{3}\theta_{11}
-\alpha^{(0)}_{1}\bar{\alpha}^{(2)}_{4}\theta_{11}
-\alpha^{(0)}_{5}\theta_{11}
-\bar{\alpha}^{(0)}_{1}\alpha^{(2)}_{3}\theta_{11}
+\bar{\alpha}^{(0)}_{1}\bar{\alpha}^{(2)}_{3}\theta_{11}
-\bar{\alpha}^{(0)}_{1}\bar{\alpha}^{(2)}_{4}\theta_{11}
+\bar{\alpha}^{(0)}_{1}\alpha^{(2)}_{3}\theta_{11}
\\
-\bar{\alpha}^{(0)}_{1}\alpha^{(2)}_{4}\theta_{11}
+\bar{\alpha}^{(0)}_{1}\bar{\alpha}^{(2)}_{4}\theta_{11}
+\bar{\alpha}^{(0)}_{1}\alpha^{(0)}_{2}\theta_{11}
-\alpha^{(2)}_{3}\bar{\alpha}^{(0)}_{1}\theta_{11}
+\alpha^{(2)}_{4}\bar{\alpha}^{(0)}_{1}\theta_{11}
+\alpha^{(2)}_{7}\theta_{11}
-\bar{\alpha}^{(2)}_{4}\bar{\alpha}^{(0)}_{1}\theta_{11}
-\bar{\alpha}^{(2)}_{6}\theta_{11}
+\bar{\alpha}^{(0)}_{1}\alpha^{(0)}_{1}\theta_{12}
-\bar{\alpha}^{(0)}_{1}\bar{\alpha}^{(0)}_{1}\theta_{13}
\\
-\bar{\alpha}^{(0)}_{1}\alpha^{(0)}_{1}\theta_{13}
+\alpha^{(2)}_{3}\theta_{14}
+\bar{\alpha}^{(2)}_{4}\theta_{14}
+\alpha^{(0)}_{2}\theta_{14}
-\bar{\alpha}^{(0)}_{1}\alpha^{(0)}_{1}\theta_{14}
+\alpha^{(2)}_{3}\theta_{14}
+\bar{\alpha}^{(2)}_{4}\theta_{14}
-\alpha^{(0)}_{1}\alpha^{(0)}_{1}\theta_{14}
+\alpha^{(0)}_{2}\theta_{14}
+\bar{\alpha}^{(0)}_{1}\alpha^{(0)}_{1}\theta_{14}
-\bar{\alpha}^{(0)}_{1}\bar{\alpha}^{(0)}_{1}\theta_{14}
\\
-\bar{\alpha}^{(0)}_{1}\alpha^{(0)}_{1}\theta_{14}
+\alpha^{(2)}_{3}\theta_{14}
+\bar{\alpha}^{(2)}_{4}\theta_{14}
-\alpha^{(0)}_{3}\theta_{15}
-\bar{\alpha}^{(0)}_{1}\alpha^{(0)}_{1}\theta_{15}
+\bar{\alpha}^{(0)}_{1}\theta_{16}
+\alpha^{(0)}_{1}\theta_{18}
+\theta_{20}
&= 0
,
\\
\bar{\alpha}^{(3)}_{17}
-\alpha^{(2)}_{15}
+\bar{\alpha}^{(2)}_{11}\alpha^{(2)}_{3}
+\bar{\alpha}^{(2)}_{15}
+\alpha^{(2)}_{13}\theta_{1}
-\alpha^{(0)}_{1}\bar{\alpha}^{(2)}_{4}\alpha^{(2)}_{7}\theta_{1}
-\alpha^{(0)}_{1}\bar{\alpha}^{(2)}_{6}\alpha^{(2)}_{3}\theta_{1}
+\alpha^{(0)}_{1}\bar{\alpha}^{(2)}_{11}\theta_{1}
+\alpha^{(0)}_{2}\bar{\alpha}^{(2)}_{4}\alpha^{(2)}_{3}\theta_{1}
-\alpha^{(0)}_{3}\bar{\alpha}^{(2)}_{3}\alpha^{(2)}_{3}\theta_{1}
+\alpha^{(0)}_{3}\bar{\alpha}^{(2)}_{8}\theta_{1}
\\
-\alpha^{(0)}_{4}\bar{\alpha}^{(2)}_{6}\theta_{1}
-\alpha^{(0)}_{5}\alpha^{(2)}_{7}\theta_{1}
+\alpha^{(0)}_{5}\bar{\alpha}^{(2)}_{6}\theta_{1}
+\alpha^{(0)}_{6}\alpha^{(2)}_{3}\theta_{1}
+\alpha^{(0)}_{6}\bar{\alpha}^{(2)}_{4}\theta_{1}
-\alpha^{(0)}_{7}\bar{\alpha}^{(2)}_{3}\theta_{1}
+\bar{\alpha}^{(0)}_{1}\bar{\alpha}^{(2)}_{4}\alpha^{(2)}_{7}\theta_{1}
+\bar{\alpha}^{(0)}_{1}\bar{\alpha}^{(2)}_{6}\alpha^{(2)}_{3}\theta_{1}
+\bar{\alpha}^{(0)}_{1}\alpha^{(0)}_{1}\bar{\alpha}^{(2)}_{3}\alpha^{(2)}_{3}\theta_{1}
\\
-\bar{\alpha}^{(0)}_{1}\alpha^{(0)}_{1}\bar{\alpha}^{(2)}_{4}\alpha^{(2)}_{3}\theta_{1}
+\bar{\alpha}^{(0)}_{1}\alpha^{(0)}_{1}\bar{\alpha}^{(2)}_{4}\alpha^{(2)}_{4}\theta_{1}
+\bar{\alpha}^{(0)}_{1}\alpha^{(0)}_{2}\alpha^{(2)}_{7}\theta_{1}
+\bar{\alpha}^{(0)}_{1}\alpha^{(0)}_{4}\alpha^{(2)}_{3}\theta_{1}
+\bar{\alpha}^{(0)}_{1}\alpha^{(0)}_{4}\bar{\alpha}^{(2)}_{4}\theta_{1}
-\bar{\alpha}^{(0)}_{1}\alpha^{(0)}_{5}\alpha^{(2)}_{3}\theta_{1}
+\bar{\alpha}^{(0)}_{1}\alpha^{(0)}_{5}\alpha^{(2)}_{4}\theta_{1}
\\
-\bar{\alpha}^{(0)}_{1}\alpha^{(0)}_{5}\bar{\alpha}^{(2)}_{4}\theta_{1}
-\alpha^{(2)}_{3}\alpha^{(0)}_{6}\theta_{2}
-\alpha^{(2)}_{3}\bar{\alpha}^{(0)}_{1}\alpha^{(0)}_{4}\theta_{2}
+\alpha^{(2)}_{3}\bar{\alpha}^{(0)}_{1}\alpha^{(0)}_{5}\theta_{2}
-\alpha^{(2)}_{4}\bar{\alpha}^{(0)}_{1}\alpha^{(0)}_{5}\theta_{2}
-\alpha^{(2)}_{7}\alpha^{(0)}_{5}\theta_{2}
+\alpha^{(2)}_{7}\bar{\alpha}^{(0)}_{1}\alpha^{(0)}_{2}\theta_{2}
-\alpha^{(2)}_{13}\theta_{2}
+\bar{\alpha}^{(2)}_{3}\alpha^{(0)}_{7}\theta_{2}
\\
+\bar{\alpha}^{(2)}_{3}\alpha^{(2)}_{3}\alpha^{(0)}_{3}\theta_{2}
-\bar{\alpha}^{(2)}_{3}\alpha^{(2)}_{3}\bar{\alpha}^{(0)}_{1}\alpha^{(0)}_{1}\theta_{2}
-\bar{\alpha}^{(2)}_{4}\alpha^{(0)}_{6}\theta_{2}
-\bar{\alpha}^{(2)}_{4}\bar{\alpha}^{(0)}_{1}\alpha^{(0)}_{4}\theta_{2}
+\bar{\alpha}^{(2)}_{4}\bar{\alpha}^{(0)}_{1}\alpha^{(0)}_{5}\theta_{2}
-\bar{\alpha}^{(2)}_{4}\alpha^{(2)}_{3}\alpha^{(0)}_{2}\theta_{2}
+\bar{\alpha}^{(2)}_{4}\alpha^{(2)}_{3}\bar{\alpha}^{(0)}_{1}\alpha^{(0)}_{1}\theta_{2}
\\
-\bar{\alpha}^{(2)}_{4}\alpha^{(2)}_{4}\bar{\alpha}^{(0)}_{1}\alpha^{(0)}_{1}\theta_{2}
-\bar{\alpha}^{(2)}_{4}\alpha^{(2)}_{7}\alpha^{(0)}_{1}\theta_{2}
+\bar{\alpha}^{(2)}_{4}\alpha^{(2)}_{7}\bar{\alpha}^{(0)}_{1}\theta_{2}
-\bar{\alpha}^{(2)}_{6}\alpha^{(0)}_{4}\theta_{2}
+\bar{\alpha}^{(2)}_{6}\alpha^{(0)}_{5}\theta_{2}
-\bar{\alpha}^{(2)}_{6}\alpha^{(2)}_{3}\alpha^{(0)}_{1}\theta_{2}
+\bar{\alpha}^{(2)}_{6}\alpha^{(2)}_{3}\bar{\alpha}^{(0)}_{1}\theta_{2}
-\bar{\alpha}^{(2)}_{8}\alpha^{(0)}_{3}\theta_{2}
\\
+\bar{\alpha}^{(2)}_{11}\alpha^{(0)}_{1}\theta_{2}
-\alpha^{(2)}_{14}\theta_{3}
+\bar{\alpha}^{(2)}_{4}\alpha^{(2)}_{10}\theta_{3}
+\bar{\alpha}^{(2)}_{8}\alpha^{(2)}_{3}\theta_{3}
+\bar{\alpha}^{(2)}_{13}\theta_{3}
+\alpha^{(2)}_{3}\bar{\alpha}^{(2)}_{3}\alpha^{(2)}_{3}\theta_{3}
-\alpha^{(2)}_{3}\bar{\alpha}^{(2)}_{4}\alpha^{(2)}_{3}\theta_{3}
+\alpha^{(2)}_{3}\bar{\alpha}^{(2)}_{4}\alpha^{(2)}_{4}\theta_{3}
+\alpha^{(2)}_{4}\alpha^{(2)}_{10}\theta_{3}
\\
-\alpha^{(2)}_{4}\bar{\alpha}^{(2)}_{4}\alpha^{(2)}_{4}\theta_{3}
-\alpha^{(2)}_{10}\alpha^{(2)}_{4}\theta_{3}
+\alpha^{(2)}_{14}\theta_{3}
-\bar{\alpha}^{(2)}_{3}\bar{\alpha}^{(2)}_{3}\alpha^{(2)}_{3}\theta_{3}
+\bar{\alpha}^{(2)}_{3}\bar{\alpha}^{(2)}_{8}\theta_{3}
-\bar{\alpha}^{(2)}_{3}\alpha^{(2)}_{3}\alpha^{(2)}_{3}\theta_{3}
-\bar{\alpha}^{(2)}_{3}\bar{\alpha}^{(2)}_{3}\alpha^{(2)}_{3}\theta_{3}
+\bar{\alpha}^{(2)}_{3}\bar{\alpha}^{(2)}_{8}\theta_{3}
-\bar{\alpha}^{(2)}_{3}\alpha^{(2)}_{3}\alpha^{(2)}_{3}\theta_{3}
\\
+\bar{\alpha}^{(2)}_{3}\alpha^{(2)}_{3}\bar{\alpha}^{(2)}_{3}\theta_{3}
-\bar{\alpha}^{(2)}_{3}\alpha^{(2)}_{3}\bar{\alpha}^{(2)}_{4}\theta_{3}
+\bar{\alpha}^{(2)}_{4}\bar{\alpha}^{(2)}_{3}\alpha^{(2)}_{3}\theta_{3}
-\bar{\alpha}^{(2)}_{4}\bar{\alpha}^{(2)}_{4}\alpha^{(2)}_{3}\theta_{3}
+\bar{\alpha}^{(2)}_{4}\bar{\alpha}^{(2)}_{4}\alpha^{(2)}_{4}\theta_{3}
+\bar{\alpha}^{(2)}_{4}\alpha^{(2)}_{3}\alpha^{(2)}_{3}\theta_{3}
+\bar{\alpha}^{(2)}_{4}\alpha^{(2)}_{3}\bar{\alpha}^{(2)}_{4}\theta_{3}
-\bar{\alpha}^{(2)}_{4}\alpha^{(2)}_{4}\alpha^{(2)}_{3}\theta_{3}
\\
+\bar{\alpha}^{(2)}_{4}\alpha^{(2)}_{4}\alpha^{(2)}_{4}\theta_{3}
-\bar{\alpha}^{(2)}_{4}\alpha^{(2)}_{4}\bar{\alpha}^{(2)}_{4}\theta_{3}
-\bar{\alpha}^{(2)}_{4}\alpha^{(2)}_{10}\theta_{3}
-\bar{\alpha}^{(2)}_{8}\bar{\alpha}^{(2)}_{3}\theta_{3}
-\bar{\alpha}^{(2)}_{8}\alpha^{(2)}_{3}\theta_{3}
-\bar{\alpha}^{(2)}_{13}\theta_{3}
+\alpha^{(0)}_{8}\theta_{4}
+\bar{\alpha}^{(0)}_{1}\alpha^{(0)}_{7}\theta_{4}
+\alpha^{(2)}_{3}\alpha^{(0)}_{4}\theta_{4}
-\alpha^{(2)}_{3}\bar{\alpha}^{(0)}_{1}\alpha^{(0)}_{3}\theta_{4}
\\
+\alpha^{(2)}_{4}\bar{\alpha}^{(0)}_{1}\alpha^{(0)}_{3}\theta_{4}
+\alpha^{(2)}_{7}\alpha^{(0)}_{3}\theta_{4}
-\alpha^{(2)}_{7}\bar{\alpha}^{(0)}_{1}\alpha^{(0)}_{1}\theta_{4}
+\bar{\alpha}^{(2)}_{4}\alpha^{(0)}_{4}\theta_{4}
-\bar{\alpha}^{(2)}_{4}\bar{\alpha}^{(0)}_{1}\alpha^{(0)}_{3}\theta_{4}
+\bar{\alpha}^{(2)}_{4}\alpha^{(2)}_{3}\alpha^{(0)}_{1}\theta_{4}
-\bar{\alpha}^{(2)}_{6}\alpha^{(0)}_{3}\theta_{4}
+\alpha^{(0)}_{1}\alpha^{(0)}_{7}\theta_{4}
\\
+\alpha^{(0)}_{1}\alpha^{(2)}_{3}\alpha^{(0)}_{3}\theta_{4}
-\alpha^{(0)}_{1}\alpha^{(2)}_{3}\bar{\alpha}^{(0)}_{1}\alpha^{(0)}_{1}\theta_{4}
-\alpha^{(0)}_{1}\bar{\alpha}^{(2)}_{3}\alpha^{(0)}_{3}\theta_{4}
+\alpha^{(0)}_{1}\bar{\alpha}^{(2)}_{4}\alpha^{(0)}_{3}\theta_{4}
-\alpha^{(0)}_{1}\bar{\alpha}^{(2)}_{4}\bar{\alpha}^{(0)}_{1}\alpha^{(0)}_{1}\theta_{4}
-\alpha^{(0)}_{1}\bar{\alpha}^{(2)}_{6}\alpha^{(0)}_{1}\theta_{4}
+\alpha^{(0)}_{2}\alpha^{(0)}_{4}\theta_{4}
\\
-\alpha^{(0)}_{2}\bar{\alpha}^{(0)}_{1}\alpha^{(0)}_{3}\theta_{4}
+\alpha^{(0)}_{2}\alpha^{(2)}_{3}\alpha^{(0)}_{1}\theta_{4}
+\alpha^{(0)}_{2}\bar{\alpha}^{(2)}_{4}\alpha^{(0)}_{1}\theta_{4}
-\alpha^{(0)}_{3}\bar{\alpha}^{(2)}_{3}\alpha^{(0)}_{1}\theta_{4}
-\alpha^{(0)}_{4}\alpha^{(0)}_{3}\theta_{4}
+\alpha^{(0)}_{5}\alpha^{(0)}_{3}\theta_{4}
-\alpha^{(0)}_{5}\bar{\alpha}^{(0)}_{1}\alpha^{(0)}_{1}\theta_{4}
+\alpha^{(0)}_{6}\alpha^{(0)}_{1}\theta_{4}
\\
-\bar{\alpha}^{(0)}_{1}\alpha^{(0)}_{6}\theta_{4}
-\bar{\alpha}^{(0)}_{1}\bar{\alpha}^{(0)}_{1}\alpha^{(0)}_{4}\theta_{4}
+\bar{\alpha}^{(0)}_{1}\bar{\alpha}^{(0)}_{1}\alpha^{(0)}_{5}\theta_{4}
-\bar{\alpha}^{(0)}_{1}\alpha^{(2)}_{3}\alpha^{(0)}_{2}\theta_{4}
+\bar{\alpha}^{(0)}_{1}\alpha^{(2)}_{3}\bar{\alpha}^{(0)}_{1}\alpha^{(0)}_{1}\theta_{4}
-\bar{\alpha}^{(0)}_{1}\alpha^{(2)}_{4}\bar{\alpha}^{(0)}_{1}\alpha^{(0)}_{1}\theta_{4}
-\bar{\alpha}^{(0)}_{1}\alpha^{(2)}_{7}\alpha^{(0)}_{1}\theta_{4}
\\
+\bar{\alpha}^{(0)}_{1}\alpha^{(2)}_{7}\bar{\alpha}^{(0)}_{1}\theta_{4}
-\bar{\alpha}^{(0)}_{1}\bar{\alpha}^{(2)}_{4}\alpha^{(0)}_{2}\theta_{4}
+\bar{\alpha}^{(0)}_{1}\bar{\alpha}^{(2)}_{4}\bar{\alpha}^{(0)}_{1}\alpha^{(0)}_{1}\theta_{4}
-\bar{\alpha}^{(0)}_{1}\bar{\alpha}^{(2)}_{4}\alpha^{(2)}_{3}\theta_{4}
+\bar{\alpha}^{(0)}_{1}\bar{\alpha}^{(2)}_{6}\alpha^{(0)}_{1}\theta_{4}
-\bar{\alpha}^{(0)}_{1}\alpha^{(0)}_{1}\alpha^{(0)}_{4}\theta_{4}
+\bar{\alpha}^{(0)}_{1}\alpha^{(0)}_{1}\alpha^{(0)}_{5}\theta_{4}
\\
-\bar{\alpha}^{(0)}_{1}\alpha^{(0)}_{1}\alpha^{(2)}_{3}\alpha^{(0)}_{1}\theta_{4}
+\bar{\alpha}^{(0)}_{1}\alpha^{(0)}_{1}\alpha^{(2)}_{3}\bar{\alpha}^{(0)}_{1}\theta_{4}
+\bar{\alpha}^{(0)}_{1}\alpha^{(0)}_{1}\bar{\alpha}^{(2)}_{3}\alpha^{(0)}_{1}\theta_{4}
-\bar{\alpha}^{(0)}_{1}\alpha^{(0)}_{1}\bar{\alpha}^{(2)}_{4}\alpha^{(0)}_{1}\theta_{4}
+\bar{\alpha}^{(0)}_{1}\alpha^{(0)}_{1}\bar{\alpha}^{(2)}_{4}\bar{\alpha}^{(0)}_{1}\theta_{4}
+\bar{\alpha}^{(0)}_{1}\alpha^{(0)}_{1}\bar{\alpha}^{(2)}_{6}\theta_{4}
\\
-\bar{\alpha}^{(0)}_{1}\alpha^{(0)}_{2}\alpha^{(0)}_{2}\theta_{4}
+\bar{\alpha}^{(0)}_{1}\alpha^{(0)}_{2}\bar{\alpha}^{(0)}_{1}\alpha^{(0)}_{1}\theta_{4}
-\bar{\alpha}^{(0)}_{1}\alpha^{(0)}_{2}\alpha^{(2)}_{3}\theta_{4}
-\bar{\alpha}^{(0)}_{1}\alpha^{(0)}_{2}\bar{\alpha}^{(2)}_{4}\theta_{4}
+\bar{\alpha}^{(0)}_{1}\alpha^{(0)}_{3}\bar{\alpha}^{(2)}_{3}\theta_{4}
+\bar{\alpha}^{(0)}_{1}\alpha^{(0)}_{4}\alpha^{(0)}_{1}\theta_{4}
-\bar{\alpha}^{(0)}_{1}\alpha^{(0)}_{5}\alpha^{(0)}_{1}\theta_{4}
\\
+\bar{\alpha}^{(0)}_{1}\alpha^{(0)}_{5}\bar{\alpha}^{(0)}_{1}\theta_{4}
-\bar{\alpha}^{(0)}_{1}\alpha^{(0)}_{6}\theta_{4}
-\bar{\alpha}^{(2)}_{4}\alpha^{(2)}_{7}\theta_{5}
-\bar{\alpha}^{(2)}_{6}\alpha^{(2)}_{3}\theta_{5}
+\bar{\alpha}^{(2)}_{11}\theta_{5}
-\alpha^{(2)}_{3}\bar{\alpha}^{(2)}_{6}\theta_{5}
-\alpha^{(2)}_{4}\alpha^{(2)}_{7}\theta_{5}
+\alpha^{(2)}_{4}\bar{\alpha}^{(2)}_{6}\theta_{5}
-\alpha^{(2)}_{7}\alpha^{(2)}_{3}\theta_{5}
+\alpha^{(2)}_{7}\bar{\alpha}^{(2)}_{3}\theta_{5}
\\
-\alpha^{(2)}_{7}\bar{\alpha}^{(2)}_{4}\theta_{5}
-\bar{\alpha}^{(2)}_{4}\bar{\alpha}^{(2)}_{6}\theta_{5}
-\bar{\alpha}^{(2)}_{4}\alpha^{(2)}_{7}\theta_{5}
-\bar{\alpha}^{(2)}_{6}\bar{\alpha}^{(2)}_{3}\theta_{5}
-\bar{\alpha}^{(2)}_{6}\alpha^{(2)}_{3}\theta_{5}
-\alpha^{(0)}_{1}\bar{\alpha}^{(2)}_{3}\alpha^{(2)}_{3}\theta_{5}
+\alpha^{(0)}_{1}\bar{\alpha}^{(2)}_{8}\theta_{5}
-\alpha^{(0)}_{1}\alpha^{(2)}_{3}\alpha^{(2)}_{3}\theta_{5}
+\alpha^{(0)}_{1}\alpha^{(2)}_{3}\bar{\alpha}^{(2)}_{3}\theta_{5}
\\
-\alpha^{(0)}_{1}\alpha^{(2)}_{3}\bar{\alpha}^{(2)}_{4}\theta_{5}
-\alpha^{(0)}_{1}\bar{\alpha}^{(2)}_{3}\bar{\alpha}^{(2)}_{3}\theta_{5}
-\alpha^{(0)}_{1}\bar{\alpha}^{(2)}_{3}\alpha^{(2)}_{3}\theta_{5}
-\alpha^{(0)}_{1}\bar{\alpha}^{(2)}_{4}\alpha^{(2)}_{3}\theta_{5}
+\alpha^{(0)}_{1}\bar{\alpha}^{(2)}_{4}\bar{\alpha}^{(2)}_{3}\theta_{5}
-\alpha^{(0)}_{1}\bar{\alpha}^{(2)}_{4}\bar{\alpha}^{(2)}_{4}\theta_{5}
+\alpha^{(0)}_{1}\bar{\alpha}^{(2)}_{4}\alpha^{(2)}_{3}\theta_{5}
-\alpha^{(0)}_{1}\bar{\alpha}^{(2)}_{4}\alpha^{(2)}_{4}\theta_{5}
\\
-\alpha^{(0)}_{2}\bar{\alpha}^{(2)}_{6}\theta_{5}
-\alpha^{(0)}_{2}\alpha^{(2)}_{7}\theta_{5}
-\alpha^{(0)}_{4}\bar{\alpha}^{(2)}_{3}\theta_{5}
-\alpha^{(0)}_{4}\alpha^{(2)}_{3}\theta_{5}
-\alpha^{(0)}_{4}\bar{\alpha}^{(2)}_{4}\theta_{5}
-\alpha^{(0)}_{5}\alpha^{(2)}_{3}\theta_{5}
+\alpha^{(0)}_{5}\bar{\alpha}^{(2)}_{3}\theta_{5}
-\alpha^{(0)}_{5}\bar{\alpha}^{(2)}_{4}\theta_{5}
+\alpha^{(0)}_{5}\alpha^{(2)}_{3}\theta_{5}
-\alpha^{(0)}_{5}\alpha^{(2)}_{4}\theta_{5}
\\
+\alpha^{(0)}_{5}\bar{\alpha}^{(2)}_{4}\theta_{5}
+\bar{\alpha}^{(0)}_{1}\bar{\alpha}^{(2)}_{3}\alpha^{(2)}_{3}\theta_{5}
-\bar{\alpha}^{(0)}_{1}\bar{\alpha}^{(2)}_{4}\alpha^{(2)}_{3}\theta_{5}
+\bar{\alpha}^{(0)}_{1}\bar{\alpha}^{(2)}_{4}\alpha^{(2)}_{4}\theta_{5}
+\bar{\alpha}^{(0)}_{1}\alpha^{(2)}_{3}\alpha^{(2)}_{3}\theta_{5}
+\bar{\alpha}^{(0)}_{1}\alpha^{(2)}_{3}\bar{\alpha}^{(2)}_{4}\theta_{5}
-\bar{\alpha}^{(0)}_{1}\alpha^{(2)}_{4}\alpha^{(2)}_{3}\theta_{5}
+\bar{\alpha}^{(0)}_{1}\alpha^{(2)}_{4}\alpha^{(2)}_{4}\theta_{5}
\\
-\bar{\alpha}^{(0)}_{1}\alpha^{(2)}_{4}\bar{\alpha}^{(2)}_{4}\theta_{5}
-\bar{\alpha}^{(0)}_{1}\alpha^{(2)}_{10}\theta_{5}
+\bar{\alpha}^{(0)}_{1}\bar{\alpha}^{(2)}_{4}\alpha^{(2)}_{3}\theta_{5}
+\bar{\alpha}^{(0)}_{1}\bar{\alpha}^{(2)}_{4}\bar{\alpha}^{(2)}_{4}\theta_{5}
+\bar{\alpha}^{(0)}_{1}\bar{\alpha}^{(2)}_{4}\alpha^{(2)}_{4}\theta_{5}
+\bar{\alpha}^{(0)}_{1}\alpha^{(0)}_{2}\alpha^{(2)}_{3}\theta_{5}
+\bar{\alpha}^{(0)}_{1}\alpha^{(0)}_{2}\bar{\alpha}^{(2)}_{4}\theta_{5}
+\bar{\alpha}^{(0)}_{1}\alpha^{(0)}_{2}\alpha^{(2)}_{4}\theta_{5}
\\
+\bar{\alpha}^{(2)}_{4}\alpha^{(2)}_{7}\theta_{6}
+\bar{\alpha}^{(2)}_{6}\alpha^{(2)}_{3}\theta_{6}
+\alpha^{(0)}_{1}\bar{\alpha}^{(2)}_{3}\alpha^{(2)}_{3}\theta_{6}
-\alpha^{(0)}_{1}\bar{\alpha}^{(2)}_{4}\alpha^{(2)}_{3}\theta_{6}
+\alpha^{(0)}_{1}\bar{\alpha}^{(2)}_{4}\alpha^{(2)}_{4}\theta_{6}
+\alpha^{(0)}_{2}\alpha^{(2)}_{7}\theta_{6}
+\alpha^{(0)}_{4}\alpha^{(2)}_{3}\theta_{6}
+\alpha^{(0)}_{4}\bar{\alpha}^{(2)}_{4}\theta_{6}
-\alpha^{(0)}_{5}\alpha^{(2)}_{3}\theta_{6}
\\
+\alpha^{(0)}_{5}\alpha^{(2)}_{4}\theta_{6}
-\alpha^{(0)}_{5}\bar{\alpha}^{(2)}_{4}\theta_{6}
+\bar{\alpha}^{(0)}_{1}\alpha^{(2)}_{10}\theta_{6}
-\bar{\alpha}^{(0)}_{1}\bar{\alpha}^{(2)}_{4}\alpha^{(2)}_{4}\theta_{6}
-\bar{\alpha}^{(0)}_{1}\alpha^{(0)}_{2}\alpha^{(2)}_{4}\theta_{6}
+\alpha^{(2)}_{3}\alpha^{(2)}_{7}\theta_{6}
+\alpha^{(2)}_{3}\alpha^{(0)}_{1}\alpha^{(2)}_{3}\theta_{6}
+\alpha^{(2)}_{3}\alpha^{(0)}_{1}\bar{\alpha}^{(2)}_{4}\theta_{6}
+\alpha^{(2)}_{3}\alpha^{(0)}_{5}\theta_{6}
\\
-\alpha^{(2)}_{3}\bar{\alpha}^{(0)}_{1}\alpha^{(2)}_{3}\theta_{6}
+\alpha^{(2)}_{3}\bar{\alpha}^{(0)}_{1}\alpha^{(2)}_{4}\theta_{6}
-\alpha^{(2)}_{3}\bar{\alpha}^{(0)}_{1}\bar{\alpha}^{(2)}_{4}\theta_{6}
-\alpha^{(2)}_{3}\bar{\alpha}^{(0)}_{1}\alpha^{(0)}_{2}\theta_{6}
-\alpha^{(2)}_{4}\bar{\alpha}^{(0)}_{1}\alpha^{(2)}_{4}\theta_{6}
-\alpha^{(2)}_{7}\alpha^{(2)}_{4}\theta_{6}
-\bar{\alpha}^{(2)}_{3}\alpha^{(2)}_{7}\theta_{6}
+\bar{\alpha}^{(2)}_{3}\bar{\alpha}^{(2)}_{6}\theta_{6}
\\
-\bar{\alpha}^{(2)}_{3}\alpha^{(0)}_{1}\alpha^{(2)}_{3}\theta_{6}
+\bar{\alpha}^{(2)}_{3}\alpha^{(0)}_{1}\bar{\alpha}^{(2)}_{3}\theta_{6}
-\bar{\alpha}^{(2)}_{3}\alpha^{(0)}_{1}\bar{\alpha}^{(2)}_{4}\theta_{6}
+\bar{\alpha}^{(2)}_{3}\alpha^{(0)}_{4}\theta_{6}
-\bar{\alpha}^{(2)}_{3}\alpha^{(0)}_{5}\theta_{6}
+\bar{\alpha}^{(2)}_{3}\alpha^{(2)}_{3}\alpha^{(0)}_{1}\theta_{6}
-\bar{\alpha}^{(2)}_{3}\alpha^{(2)}_{3}\bar{\alpha}^{(0)}_{1}\theta_{6}
+\bar{\alpha}^{(2)}_{4}\alpha^{(2)}_{7}\theta_{6}
\\
+\bar{\alpha}^{(2)}_{4}\alpha^{(0)}_{1}\alpha^{(2)}_{3}\theta_{6}
+\bar{\alpha}^{(2)}_{4}\alpha^{(0)}_{1}\bar{\alpha}^{(2)}_{4}\theta_{6}
+\bar{\alpha}^{(2)}_{4}\alpha^{(0)}_{5}\theta_{6}
-\bar{\alpha}^{(2)}_{4}\bar{\alpha}^{(0)}_{1}\alpha^{(2)}_{3}\theta_{6}
+\bar{\alpha}^{(2)}_{4}\bar{\alpha}^{(0)}_{1}\alpha^{(2)}_{4}\theta_{6}
-\bar{\alpha}^{(2)}_{4}\bar{\alpha}^{(0)}_{1}\bar{\alpha}^{(2)}_{4}\theta_{6}
-\bar{\alpha}^{(2)}_{4}\bar{\alpha}^{(0)}_{1}\alpha^{(0)}_{2}\theta_{6}
+\bar{\alpha}^{(2)}_{4}\alpha^{(2)}_{3}\bar{\alpha}^{(0)}_{1}\theta_{6}
\\
-\bar{\alpha}^{(2)}_{4}\alpha^{(2)}_{4}\bar{\alpha}^{(0)}_{1}\theta_{6}
-\bar{\alpha}^{(2)}_{4}\alpha^{(2)}_{7}\theta_{6}
-\bar{\alpha}^{(2)}_{6}\alpha^{(2)}_{3}\theta_{6}
+\bar{\alpha}^{(2)}_{6}\alpha^{(2)}_{4}\theta_{6}
-\bar{\alpha}^{(2)}_{6}\bar{\alpha}^{(2)}_{4}\theta_{6}
-\bar{\alpha}^{(2)}_{6}\alpha^{(0)}_{2}\theta_{6}
-\bar{\alpha}^{(2)}_{6}\alpha^{(2)}_{3}\theta_{6}
-\bar{\alpha}^{(2)}_{8}\alpha^{(0)}_{1}\theta_{6}
+\bar{\alpha}^{(2)}_{11}\theta_{6}
-\alpha^{(2)}_{3}\alpha^{(0)}_{5}\theta_{7}
\\
+\alpha^{(2)}_{3}\bar{\alpha}^{(0)}_{1}\alpha^{(0)}_{2}\theta_{7}
-\bar{\alpha}^{(2)}_{3}\alpha^{(0)}_{4}\theta_{7}
+\bar{\alpha}^{(2)}_{3}\alpha^{(0)}_{5}\theta_{7}
-\bar{\alpha}^{(2)}_{3}\alpha^{(2)}_{3}\alpha^{(0)}_{1}\theta_{7}
+\bar{\alpha}^{(2)}_{3}\alpha^{(2)}_{3}\bar{\alpha}^{(0)}_{1}\theta_{7}
-\bar{\alpha}^{(2)}_{4}\alpha^{(0)}_{5}\theta_{7}
+\bar{\alpha}^{(2)}_{4}\bar{\alpha}^{(0)}_{1}\alpha^{(0)}_{2}\theta_{7}
-\bar{\alpha}^{(2)}_{4}\alpha^{(2)}_{3}\bar{\alpha}^{(0)}_{1}\theta_{7}
\\
+\bar{\alpha}^{(2)}_{4}\alpha^{(2)}_{4}\bar{\alpha}^{(0)}_{1}\theta_{7}
+\bar{\alpha}^{(2)}_{4}\alpha^{(2)}_{7}\theta_{7}
+\bar{\alpha}^{(2)}_{6}\alpha^{(0)}_{2}\theta_{7}
+\bar{\alpha}^{(2)}_{6}\alpha^{(2)}_{3}\theta_{7}
+\bar{\alpha}^{(2)}_{8}\alpha^{(0)}_{1}\theta_{7}
-\bar{\alpha}^{(2)}_{11}\theta_{7}
-\alpha^{(2)}_{3}\alpha^{(0)}_{4}\theta_{7}
+\alpha^{(2)}_{3}\alpha^{(0)}_{5}\theta_{7}
-\alpha^{(2)}_{3}\alpha^{(2)}_{3}\alpha^{(0)}_{1}\theta_{7}
\\
+\alpha^{(2)}_{3}\alpha^{(2)}_{3}\bar{\alpha}^{(0)}_{1}\theta_{7}
+\alpha^{(2)}_{3}\bar{\alpha}^{(2)}_{3}\alpha^{(0)}_{1}\theta_{7}
-\alpha^{(2)}_{3}\bar{\alpha}^{(2)}_{4}\alpha^{(0)}_{1}\theta_{7}
+\alpha^{(2)}_{3}\bar{\alpha}^{(2)}_{4}\bar{\alpha}^{(0)}_{1}\theta_{7}
+\alpha^{(2)}_{3}\bar{\alpha}^{(2)}_{6}\theta_{7}
-\alpha^{(2)}_{4}\alpha^{(0)}_{5}\theta_{7}
+\alpha^{(2)}_{4}\bar{\alpha}^{(0)}_{1}\alpha^{(0)}_{2}\theta_{7}
-\alpha^{(2)}_{4}\alpha^{(2)}_{3}\bar{\alpha}^{(0)}_{1}\theta_{7}
\\
+\alpha^{(2)}_{4}\alpha^{(2)}_{4}\bar{\alpha}^{(0)}_{1}\theta_{7}
+\alpha^{(2)}_{4}\alpha^{(2)}_{7}\theta_{7}
-\alpha^{(2)}_{4}\bar{\alpha}^{(2)}_{4}\bar{\alpha}^{(0)}_{1}\theta_{7}
-\alpha^{(2)}_{4}\bar{\alpha}^{(2)}_{6}\theta_{7}
+\alpha^{(2)}_{7}\alpha^{(0)}_{2}\theta_{7}
+\alpha^{(2)}_{7}\alpha^{(2)}_{3}\theta_{7}
-\alpha^{(2)}_{7}\bar{\alpha}^{(2)}_{3}\theta_{7}
+\alpha^{(2)}_{7}\bar{\alpha}^{(2)}_{4}\theta_{7}
-\alpha^{(2)}_{10}\bar{\alpha}^{(0)}_{1}\theta_{7}
\\
-\bar{\alpha}^{(2)}_{3}\bar{\alpha}^{(2)}_{3}\alpha^{(0)}_{1}\theta_{7}
-\bar{\alpha}^{(2)}_{3}\alpha^{(2)}_{3}\alpha^{(0)}_{1}\theta_{7}
-\bar{\alpha}^{(2)}_{4}\alpha^{(0)}_{4}\theta_{7}
+\bar{\alpha}^{(2)}_{4}\alpha^{(0)}_{5}\theta_{7}
-\bar{\alpha}^{(2)}_{4}\alpha^{(2)}_{3}\alpha^{(0)}_{1}\theta_{7}
+\bar{\alpha}^{(2)}_{4}\alpha^{(2)}_{3}\bar{\alpha}^{(0)}_{1}\theta_{7}
+\bar{\alpha}^{(2)}_{4}\bar{\alpha}^{(2)}_{3}\alpha^{(0)}_{1}\theta_{7}
-\bar{\alpha}^{(2)}_{4}\bar{\alpha}^{(2)}_{4}\alpha^{(0)}_{1}\theta_{7}
\\
+\bar{\alpha}^{(2)}_{4}\bar{\alpha}^{(2)}_{4}\bar{\alpha}^{(0)}_{1}\theta_{7}
+\bar{\alpha}^{(2)}_{4}\bar{\alpha}^{(2)}_{6}\theta_{7}
+\bar{\alpha}^{(2)}_{4}\alpha^{(2)}_{3}\alpha^{(0)}_{1}\theta_{7}
-\bar{\alpha}^{(2)}_{4}\alpha^{(2)}_{4}\alpha^{(0)}_{1}\theta_{7}
+\bar{\alpha}^{(2)}_{4}\alpha^{(2)}_{4}\bar{\alpha}^{(0)}_{1}\theta_{7}
+\bar{\alpha}^{(2)}_{4}\alpha^{(2)}_{7}\theta_{7}
+\bar{\alpha}^{(2)}_{6}\bar{\alpha}^{(2)}_{3}\theta_{7}
+\bar{\alpha}^{(2)}_{6}\alpha^{(2)}_{3}\theta_{7}
+\bar{\alpha}^{(2)}_{4}\alpha^{(2)}_{3}\theta_{8}
\\
-\alpha^{(0)}_{1}\bar{\alpha}^{(2)}_{6}\theta_{8}
+\alpha^{(0)}_{2}\alpha^{(2)}_{3}\theta_{8}
+\alpha^{(0)}_{2}\bar{\alpha}^{(2)}_{4}\theta_{8}
-\alpha^{(0)}_{3}\bar{\alpha}^{(2)}_{3}\theta_{8}
+\alpha^{(0)}_{6}\theta_{8}
-\bar{\alpha}^{(0)}_{1}\alpha^{(2)}_{7}\theta_{8}
+\bar{\alpha}^{(0)}_{1}\bar{\alpha}^{(2)}_{6}\theta_{8}
-\bar{\alpha}^{(0)}_{1}\alpha^{(0)}_{1}\alpha^{(2)}_{3}\theta_{8}
+\bar{\alpha}^{(0)}_{1}\alpha^{(0)}_{1}\bar{\alpha}^{(2)}_{3}\theta_{8}
\\
-\bar{\alpha}^{(0)}_{1}\alpha^{(0)}_{1}\bar{\alpha}^{(2)}_{4}\theta_{8}
+\bar{\alpha}^{(0)}_{1}\alpha^{(0)}_{4}\theta_{8}
-\bar{\alpha}^{(0)}_{1}\alpha^{(0)}_{5}\theta_{8}
+\alpha^{(2)}_{3}\alpha^{(2)}_{3}\theta_{8}
+\alpha^{(2)}_{3}\bar{\alpha}^{(2)}_{4}\theta_{8}
+\alpha^{(2)}_{3}\alpha^{(0)}_{2}\theta_{8}
-\alpha^{(2)}_{3}\bar{\alpha}^{(0)}_{1}\alpha^{(0)}_{1}\theta_{8}
+\alpha^{(2)}_{4}\bar{\alpha}^{(0)}_{1}\alpha^{(0)}_{1}\theta_{8}
+\alpha^{(2)}_{7}\alpha^{(0)}_{1}\theta_{8}
\\
-\alpha^{(2)}_{7}\bar{\alpha}^{(0)}_{1}\theta_{8}
+\bar{\alpha}^{(2)}_{4}\alpha^{(2)}_{3}\theta_{8}
+\bar{\alpha}^{(2)}_{4}\bar{\alpha}^{(2)}_{4}\theta_{8}
+\bar{\alpha}^{(2)}_{4}\alpha^{(0)}_{2}\theta_{8}
-\bar{\alpha}^{(2)}_{4}\bar{\alpha}^{(0)}_{1}\alpha^{(0)}_{1}\theta_{8}
+\bar{\alpha}^{(2)}_{4}\alpha^{(2)}_{3}\theta_{8}
-\bar{\alpha}^{(2)}_{6}\alpha^{(0)}_{1}\theta_{8}
-\alpha^{(0)}_{1}\alpha^{(2)}_{7}\theta_{8}
+\alpha^{(0)}_{1}\bar{\alpha}^{(2)}_{6}\theta_{8}
\\
-\alpha^{(0)}_{1}\alpha^{(0)}_{1}\alpha^{(2)}_{3}\theta_{8}
+\alpha^{(0)}_{1}\alpha^{(0)}_{1}\bar{\alpha}^{(2)}_{3}\theta_{8}
-\alpha^{(0)}_{1}\alpha^{(0)}_{1}\bar{\alpha}^{(2)}_{4}\theta_{8}
+\alpha^{(0)}_{1}\alpha^{(0)}_{4}\theta_{8}
-\alpha^{(0)}_{1}\alpha^{(0)}_{5}\theta_{8}
+\alpha^{(0)}_{1}\alpha^{(2)}_{3}\alpha^{(0)}_{1}\theta_{8}
-\alpha^{(0)}_{1}\alpha^{(2)}_{3}\bar{\alpha}^{(0)}_{1}\theta_{8}
-\alpha^{(0)}_{1}\bar{\alpha}^{(2)}_{3}\alpha^{(0)}_{1}\theta_{8}
\\
+\alpha^{(0)}_{1}\bar{\alpha}^{(2)}_{4}\alpha^{(0)}_{1}\theta_{8}
-\alpha^{(0)}_{1}\bar{\alpha}^{(2)}_{4}\bar{\alpha}^{(0)}_{1}\theta_{8}
-\alpha^{(0)}_{1}\bar{\alpha}^{(2)}_{6}\theta_{8}
+\alpha^{(0)}_{2}\alpha^{(2)}_{3}\theta_{8}
+\alpha^{(0)}_{2}\bar{\alpha}^{(2)}_{4}\theta_{8}
+\alpha^{(0)}_{2}\alpha^{(0)}_{2}\theta_{8}
-\alpha^{(0)}_{2}\bar{\alpha}^{(0)}_{1}\alpha^{(0)}_{1}\theta_{8}
+\alpha^{(0)}_{2}\alpha^{(2)}_{3}\theta_{8}
+\alpha^{(0)}_{2}\bar{\alpha}^{(2)}_{4}\theta_{8}
\end{align*}
\begin{align*}
-\alpha^{(0)}_{3}\bar{\alpha}^{(2)}_{3}\theta_{8}
-\alpha^{(0)}_{4}\alpha^{(0)}_{1}\theta_{8}
+\alpha^{(0)}_{5}\alpha^{(0)}_{1}\theta_{8}
-\alpha^{(0)}_{5}\bar{\alpha}^{(0)}_{1}\theta_{8}
+\alpha^{(0)}_{6}\theta_{8}
+\bar{\alpha}^{(0)}_{1}\alpha^{(2)}_{7}\theta_{8}
+\bar{\alpha}^{(0)}_{1}\alpha^{(0)}_{1}\alpha^{(2)}_{3}\theta_{8}
+\bar{\alpha}^{(0)}_{1}\alpha^{(0)}_{1}\bar{\alpha}^{(2)}_{4}\theta_{8}
+\bar{\alpha}^{(0)}_{1}\alpha^{(0)}_{5}\theta_{8}
\\
-\bar{\alpha}^{(0)}_{1}\bar{\alpha}^{(0)}_{1}\alpha^{(2)}_{3}\theta_{8}
+\bar{\alpha}^{(0)}_{1}\bar{\alpha}^{(0)}_{1}\alpha^{(2)}_{4}\theta_{8}
-\bar{\alpha}^{(0)}_{1}\bar{\alpha}^{(0)}_{1}\bar{\alpha}^{(2)}_{4}\theta_{8}
-\bar{\alpha}^{(0)}_{1}\bar{\alpha}^{(0)}_{1}\alpha^{(0)}_{2}\theta_{8}
+\bar{\alpha}^{(0)}_{1}\alpha^{(2)}_{3}\bar{\alpha}^{(0)}_{1}\theta_{8}
-\bar{\alpha}^{(0)}_{1}\alpha^{(2)}_{4}\bar{\alpha}^{(0)}_{1}\theta_{8}
-\bar{\alpha}^{(0)}_{1}\alpha^{(2)}_{7}\theta_{8}
+\bar{\alpha}^{(0)}_{1}\bar{\alpha}^{(2)}_{4}\bar{\alpha}^{(0)}_{1}\theta_{8}
\\
+\bar{\alpha}^{(0)}_{1}\bar{\alpha}^{(2)}_{6}\theta_{8}
-\bar{\alpha}^{(0)}_{1}\alpha^{(0)}_{1}\alpha^{(2)}_{3}\theta_{8}
+\bar{\alpha}^{(0)}_{1}\alpha^{(0)}_{1}\alpha^{(2)}_{4}\theta_{8}
-\bar{\alpha}^{(0)}_{1}\alpha^{(0)}_{1}\bar{\alpha}^{(2)}_{4}\theta_{8}
-\bar{\alpha}^{(0)}_{1}\alpha^{(0)}_{1}\alpha^{(0)}_{2}\theta_{8}
-\bar{\alpha}^{(0)}_{1}\alpha^{(0)}_{1}\alpha^{(2)}_{3}\theta_{8}
+\bar{\alpha}^{(0)}_{1}\alpha^{(0)}_{1}\bar{\alpha}^{(2)}_{3}\theta_{8}
-\bar{\alpha}^{(0)}_{1}\alpha^{(0)}_{1}\bar{\alpha}^{(2)}_{4}\theta_{8}
\\
+\bar{\alpha}^{(0)}_{1}\alpha^{(0)}_{2}\bar{\alpha}^{(0)}_{1}\theta_{8}
+\bar{\alpha}^{(0)}_{1}\alpha^{(0)}_{4}\theta_{8}
-\bar{\alpha}^{(0)}_{1}\alpha^{(0)}_{5}\theta_{8}
+\alpha^{(0)}_{7}\theta_{9}
+\alpha^{(2)}_{3}\alpha^{(0)}_{3}\theta_{9}
-\alpha^{(2)}_{3}\bar{\alpha}^{(0)}_{1}\alpha^{(0)}_{1}\theta_{9}
-\bar{\alpha}^{(2)}_{3}\alpha^{(0)}_{3}\theta_{9}
+\bar{\alpha}^{(2)}_{4}\alpha^{(0)}_{3}\theta_{9}
-\bar{\alpha}^{(2)}_{4}\bar{\alpha}^{(0)}_{1}\alpha^{(0)}_{1}\theta_{9}
\\
-\bar{\alpha}^{(2)}_{6}\alpha^{(0)}_{1}\theta_{9}
-\alpha^{(2)}_{3}\alpha^{(0)}_{3}\theta_{9}
+\alpha^{(2)}_{4}\alpha^{(0)}_{3}\theta_{9}
-\alpha^{(2)}_{4}\bar{\alpha}^{(0)}_{1}\alpha^{(0)}_{1}\theta_{9}
-\alpha^{(2)}_{7}\alpha^{(0)}_{1}\theta_{9}
-\bar{\alpha}^{(2)}_{4}\alpha^{(0)}_{3}\theta_{9}
-\alpha^{(0)}_{1}\bar{\alpha}^{(2)}_{3}\alpha^{(0)}_{1}\theta_{9}
-\alpha^{(0)}_{1}\alpha^{(2)}_{3}\alpha^{(0)}_{1}\theta_{9}
-\alpha^{(0)}_{1}\bar{\alpha}^{(2)}_{4}\alpha^{(0)}_{1}\theta_{9}
\\
-\alpha^{(0)}_{2}\alpha^{(0)}_{3}\theta_{9}
-\alpha^{(0)}_{5}\alpha^{(0)}_{1}\theta_{9}
-\bar{\alpha}^{(0)}_{1}\alpha^{(0)}_{4}\theta_{9}
+\bar{\alpha}^{(0)}_{1}\alpha^{(0)}_{5}\theta_{9}
-\bar{\alpha}^{(0)}_{1}\alpha^{(2)}_{3}\alpha^{(0)}_{1}\theta_{9}
+\bar{\alpha}^{(0)}_{1}\alpha^{(2)}_{3}\bar{\alpha}^{(0)}_{1}\theta_{9}
+\bar{\alpha}^{(0)}_{1}\bar{\alpha}^{(2)}_{3}\alpha^{(0)}_{1}\theta_{9}
-\bar{\alpha}^{(0)}_{1}\bar{\alpha}^{(2)}_{4}\alpha^{(0)}_{1}\theta_{9}
\\
+\bar{\alpha}^{(0)}_{1}\bar{\alpha}^{(2)}_{4}\bar{\alpha}^{(0)}_{1}\theta_{9}
+\bar{\alpha}^{(0)}_{1}\bar{\alpha}^{(2)}_{6}\theta_{9}
+\bar{\alpha}^{(0)}_{1}\alpha^{(2)}_{3}\alpha^{(0)}_{1}\theta_{9}
-\bar{\alpha}^{(0)}_{1}\alpha^{(2)}_{4}\alpha^{(0)}_{1}\theta_{9}
+\bar{\alpha}^{(0)}_{1}\alpha^{(2)}_{4}\bar{\alpha}^{(0)}_{1}\theta_{9}
+\bar{\alpha}^{(0)}_{1}\alpha^{(2)}_{7}\theta_{9}
+\bar{\alpha}^{(0)}_{1}\bar{\alpha}^{(2)}_{4}\alpha^{(0)}_{1}\theta_{9}
+\bar{\alpha}^{(0)}_{1}\alpha^{(0)}_{1}\bar{\alpha}^{(2)}_{3}\theta_{9}
\\
+\bar{\alpha}^{(0)}_{1}\alpha^{(0)}_{1}\alpha^{(2)}_{3}\theta_{9}
+\bar{\alpha}^{(0)}_{1}\alpha^{(0)}_{1}\bar{\alpha}^{(2)}_{4}\theta_{9}
+\bar{\alpha}^{(0)}_{1}\alpha^{(0)}_{2}\alpha^{(0)}_{1}\theta_{9}
+\bar{\alpha}^{(0)}_{1}\alpha^{(0)}_{5}\theta_{9}
-\alpha^{(0)}_{6}\theta_{10}
-\bar{\alpha}^{(0)}_{1}\alpha^{(0)}_{4}\theta_{10}
+\bar{\alpha}^{(0)}_{1}\alpha^{(0)}_{5}\theta_{10}
-\alpha^{(2)}_{3}\alpha^{(0)}_{2}\theta_{10}
+\alpha^{(2)}_{3}\bar{\alpha}^{(0)}_{1}\alpha^{(0)}_{1}\theta_{10}
\\
-\alpha^{(2)}_{4}\bar{\alpha}^{(0)}_{1}\alpha^{(0)}_{1}\theta_{10}
-\alpha^{(2)}_{7}\alpha^{(0)}_{1}\theta_{10}
+\alpha^{(2)}_{7}\bar{\alpha}^{(0)}_{1}\theta_{10}
-\bar{\alpha}^{(2)}_{4}\alpha^{(0)}_{2}\theta_{10}
+\bar{\alpha}^{(2)}_{4}\bar{\alpha}^{(0)}_{1}\alpha^{(0)}_{1}\theta_{10}
-\bar{\alpha}^{(2)}_{4}\alpha^{(2)}_{3}\theta_{10}
+\bar{\alpha}^{(2)}_{6}\alpha^{(0)}_{1}\theta_{10}
-\alpha^{(0)}_{1}\alpha^{(0)}_{4}\theta_{10}
+\alpha^{(0)}_{1}\alpha^{(0)}_{5}\theta_{10}
\\
-\alpha^{(0)}_{1}\alpha^{(2)}_{3}\alpha^{(0)}_{1}\theta_{10}
+\alpha^{(0)}_{1}\alpha^{(2)}_{3}\bar{\alpha}^{(0)}_{1}\theta_{10}
+\alpha^{(0)}_{1}\bar{\alpha}^{(2)}_{3}\alpha^{(0)}_{1}\theta_{10}
-\alpha^{(0)}_{1}\bar{\alpha}^{(2)}_{4}\alpha^{(0)}_{1}\theta_{10}
+\alpha^{(0)}_{1}\bar{\alpha}^{(2)}_{4}\bar{\alpha}^{(0)}_{1}\theta_{10}
+\alpha^{(0)}_{1}\bar{\alpha}^{(2)}_{6}\theta_{10}
-\alpha^{(0)}_{2}\alpha^{(0)}_{2}\theta_{10}
+\alpha^{(0)}_{2}\bar{\alpha}^{(0)}_{1}\alpha^{(0)}_{1}\theta_{10}
\\
-\alpha^{(0)}_{2}\alpha^{(2)}_{3}\theta_{10}
-\alpha^{(0)}_{2}\bar{\alpha}^{(2)}_{4}\theta_{10}
+\alpha^{(0)}_{3}\bar{\alpha}^{(2)}_{3}\theta_{10}
+\alpha^{(0)}_{4}\alpha^{(0)}_{1}\theta_{10}
-\alpha^{(0)}_{5}\alpha^{(0)}_{1}\theta_{10}
+\alpha^{(0)}_{5}\bar{\alpha}^{(0)}_{1}\theta_{10}
-\alpha^{(0)}_{6}\theta_{10}
-\bar{\alpha}^{(0)}_{1}\alpha^{(0)}_{5}\theta_{10}
+\bar{\alpha}^{(0)}_{1}\bar{\alpha}^{(0)}_{1}\alpha^{(0)}_{2}\theta_{10}
\\
-\bar{\alpha}^{(0)}_{1}\alpha^{(2)}_{3}\bar{\alpha}^{(0)}_{1}\theta_{10}
+\bar{\alpha}^{(0)}_{1}\alpha^{(2)}_{4}\bar{\alpha}^{(0)}_{1}\theta_{10}
+\bar{\alpha}^{(0)}_{1}\alpha^{(2)}_{7}\theta_{10}
-\bar{\alpha}^{(0)}_{1}\bar{\alpha}^{(2)}_{4}\bar{\alpha}^{(0)}_{1}\theta_{10}
-\bar{\alpha}^{(0)}_{1}\bar{\alpha}^{(2)}_{6}\theta_{10}
+\bar{\alpha}^{(0)}_{1}\alpha^{(0)}_{1}\alpha^{(0)}_{2}\theta_{10}
+\bar{\alpha}^{(0)}_{1}\alpha^{(0)}_{1}\alpha^{(2)}_{3}\theta_{10}
-\bar{\alpha}^{(0)}_{1}\alpha^{(0)}_{1}\bar{\alpha}^{(2)}_{3}\theta_{10}
\\
+\bar{\alpha}^{(0)}_{1}\alpha^{(0)}_{1}\bar{\alpha}^{(2)}_{4}\theta_{10}
-\bar{\alpha}^{(0)}_{1}\alpha^{(0)}_{2}\bar{\alpha}^{(0)}_{1}\theta_{10}
-\bar{\alpha}^{(0)}_{1}\alpha^{(0)}_{4}\theta_{10}
+\bar{\alpha}^{(0)}_{1}\alpha^{(0)}_{5}\theta_{10}
-\alpha^{(2)}_{3}\alpha^{(0)}_{2}\theta_{10}
+\alpha^{(2)}_{3}\bar{\alpha}^{(0)}_{1}\alpha^{(0)}_{1}\theta_{10}
-\alpha^{(2)}_{3}\alpha^{(2)}_{3}\theta_{10}
-\alpha^{(2)}_{3}\bar{\alpha}^{(2)}_{4}\theta_{10}
\\
+\alpha^{(2)}_{3}\alpha^{(0)}_{1}\alpha^{(0)}_{1}\theta_{10}
-\alpha^{(2)}_{3}\alpha^{(0)}_{2}\theta_{10}
-\alpha^{(2)}_{3}\bar{\alpha}^{(0)}_{1}\alpha^{(0)}_{1}\theta_{10}
+\alpha^{(2)}_{3}\bar{\alpha}^{(0)}_{1}\bar{\alpha}^{(0)}_{1}\theta_{10}
+\alpha^{(2)}_{3}\bar{\alpha}^{(0)}_{1}\alpha^{(0)}_{1}\theta_{10}
-\alpha^{(2)}_{4}\bar{\alpha}^{(0)}_{1}\bar{\alpha}^{(0)}_{1}\theta_{10}
-\alpha^{(2)}_{4}\bar{\alpha}^{(0)}_{1}\alpha^{(0)}_{1}\theta_{10}
-\alpha^{(2)}_{7}\bar{\alpha}^{(0)}_{1}\theta_{10}
\\
-\alpha^{(2)}_{7}\alpha^{(0)}_{1}\theta_{10}
+\alpha^{(2)}_{7}\bar{\alpha}^{(0)}_{1}\theta_{10}
+\bar{\alpha}^{(2)}_{3}\alpha^{(0)}_{3}\theta_{10}
-\bar{\alpha}^{(2)}_{3}\bar{\alpha}^{(0)}_{1}\alpha^{(0)}_{1}\theta_{10}
-\bar{\alpha}^{(2)}_{3}\alpha^{(0)}_{1}\alpha^{(0)}_{1}\theta_{10}
-\bar{\alpha}^{(2)}_{4}\alpha^{(0)}_{2}\theta_{10}
+\bar{\alpha}^{(2)}_{4}\bar{\alpha}^{(0)}_{1}\alpha^{(0)}_{1}\theta_{10}
-\bar{\alpha}^{(2)}_{4}\alpha^{(2)}_{3}\theta_{10}
\\
-\bar{\alpha}^{(2)}_{4}\bar{\alpha}^{(2)}_{4}\theta_{10}
+\bar{\alpha}^{(2)}_{4}\alpha^{(0)}_{1}\alpha^{(0)}_{1}\theta_{10}
-\bar{\alpha}^{(2)}_{4}\alpha^{(0)}_{2}\theta_{10}
-\bar{\alpha}^{(2)}_{4}\bar{\alpha}^{(0)}_{1}\alpha^{(0)}_{1}\theta_{10}
+\bar{\alpha}^{(2)}_{4}\bar{\alpha}^{(0)}_{1}\bar{\alpha}^{(0)}_{1}\theta_{10}
+\bar{\alpha}^{(2)}_{4}\bar{\alpha}^{(0)}_{1}\alpha^{(0)}_{1}\theta_{10}
-\bar{\alpha}^{(2)}_{4}\alpha^{(2)}_{3}\theta_{10}
-\bar{\alpha}^{(2)}_{6}\alpha^{(0)}_{1}\theta_{10}
\\
+\bar{\alpha}^{(2)}_{6}\bar{\alpha}^{(0)}_{1}\theta_{10}
+\bar{\alpha}^{(2)}_{6}\alpha^{(0)}_{1}\theta_{10}
+\bar{\alpha}^{(2)}_{3}\alpha^{(2)}_{3}\theta_{11}
-\bar{\alpha}^{(2)}_{4}\alpha^{(2)}_{3}\theta_{11}
+\bar{\alpha}^{(2)}_{4}\alpha^{(2)}_{4}\theta_{11}
+\alpha^{(2)}_{3}\alpha^{(2)}_{3}\theta_{11}
+\alpha^{(2)}_{3}\bar{\alpha}^{(2)}_{4}\theta_{11}
-\alpha^{(2)}_{4}\alpha^{(2)}_{3}\theta_{11}
+\alpha^{(2)}_{4}\alpha^{(2)}_{4}\theta_{11}
\\
-\alpha^{(2)}_{4}\bar{\alpha}^{(2)}_{4}\theta_{11}
-\alpha^{(2)}_{10}\theta_{11}
+\bar{\alpha}^{(2)}_{4}\alpha^{(2)}_{3}\theta_{11}
+\bar{\alpha}^{(2)}_{4}\bar{\alpha}^{(2)}_{4}\theta_{11}
+\bar{\alpha}^{(2)}_{4}\alpha^{(2)}_{4}\theta_{11}
+\alpha^{(0)}_{2}\alpha^{(2)}_{3}\theta_{11}
+\alpha^{(0)}_{2}\bar{\alpha}^{(2)}_{4}\theta_{11}
+\alpha^{(0)}_{2}\alpha^{(2)}_{4}\theta_{11}
+\alpha^{(2)}_{3}\alpha^{(2)}_{3}\theta_{11}
+\alpha^{(2)}_{3}\bar{\alpha}^{(2)}_{4}\theta_{11}
\\
+\alpha^{(2)}_{3}\alpha^{(2)}_{4}\theta_{11}
-\bar{\alpha}^{(2)}_{3}\alpha^{(2)}_{3}\theta_{11}
+\bar{\alpha}^{(2)}_{3}\bar{\alpha}^{(2)}_{3}\theta_{11}
-\bar{\alpha}^{(2)}_{3}\bar{\alpha}^{(2)}_{4}\theta_{11}
+\bar{\alpha}^{(2)}_{3}\alpha^{(2)}_{3}\theta_{11}
-\bar{\alpha}^{(2)}_{3}\alpha^{(2)}_{4}\theta_{11}
+\bar{\alpha}^{(2)}_{3}\bar{\alpha}^{(2)}_{4}\theta_{11}
+\bar{\alpha}^{(2)}_{3}\alpha^{(0)}_{2}\theta_{11}
+\bar{\alpha}^{(2)}_{3}\alpha^{(2)}_{3}\theta_{11}
\\
+\bar{\alpha}^{(2)}_{4}\alpha^{(2)}_{3}\theta_{11}
+\bar{\alpha}^{(2)}_{4}\bar{\alpha}^{(2)}_{4}\theta_{11}
+\bar{\alpha}^{(2)}_{4}\alpha^{(2)}_{4}\theta_{11}
-\bar{\alpha}^{(2)}_{8}\theta_{11}
+\alpha^{(0)}_{4}\theta_{12}
-\bar{\alpha}^{(0)}_{1}\alpha^{(0)}_{3}\theta_{12}
+\alpha^{(2)}_{3}\alpha^{(0)}_{1}\theta_{12}
+\bar{\alpha}^{(2)}_{4}\alpha^{(0)}_{1}\theta_{12}
-\alpha^{(0)}_{1}\alpha^{(0)}_{3}\theta_{12}
+\alpha^{(0)}_{2}\alpha^{(0)}_{1}\theta_{12}
\\
+\bar{\alpha}^{(0)}_{1}\alpha^{(0)}_{3}\theta_{12}
-\bar{\alpha}^{(0)}_{1}\bar{\alpha}^{(0)}_{1}\alpha^{(0)}_{1}\theta_{12}
-\bar{\alpha}^{(0)}_{1}\alpha^{(0)}_{1}\alpha^{(0)}_{1}\theta_{12}
+\alpha^{(2)}_{3}\alpha^{(0)}_{1}\theta_{12}
+\bar{\alpha}^{(2)}_{4}\alpha^{(0)}_{1}\theta_{12}
+\alpha^{(0)}_{1}\alpha^{(0)}_{3}\theta_{12}
-\alpha^{(0)}_{1}\bar{\alpha}^{(0)}_{1}\alpha^{(0)}_{1}\theta_{12}
-\alpha^{(0)}_{1}\alpha^{(0)}_{1}\alpha^{(0)}_{1}\theta_{12}
\\
+\alpha^{(0)}_{2}\alpha^{(0)}_{1}\theta_{12}
-\bar{\alpha}^{(0)}_{1}\alpha^{(0)}_{2}\theta_{12}
+\bar{\alpha}^{(0)}_{1}\bar{\alpha}^{(0)}_{1}\alpha^{(0)}_{1}\theta_{12}
-\bar{\alpha}^{(0)}_{1}\alpha^{(2)}_{3}\theta_{12}
-\bar{\alpha}^{(0)}_{1}\bar{\alpha}^{(2)}_{4}\theta_{12}
+\bar{\alpha}^{(0)}_{1}\alpha^{(0)}_{1}\alpha^{(0)}_{1}\theta_{12}
-\bar{\alpha}^{(0)}_{1}\alpha^{(0)}_{2}\theta_{12}
-\bar{\alpha}^{(0)}_{1}\bar{\alpha}^{(0)}_{1}\alpha^{(0)}_{1}\theta_{12}
\\
+\bar{\alpha}^{(0)}_{1}\bar{\alpha}^{(0)}_{1}\bar{\alpha}^{(0)}_{1}\theta_{12}
+\bar{\alpha}^{(0)}_{1}\bar{\alpha}^{(0)}_{1}\alpha^{(0)}_{1}\theta_{12}
-\bar{\alpha}^{(0)}_{1}\alpha^{(2)}_{3}\theta_{12}
-\bar{\alpha}^{(0)}_{1}\bar{\alpha}^{(2)}_{4}\theta_{12}
-\bar{\alpha}^{(0)}_{1}\alpha^{(0)}_{1}\alpha^{(0)}_{1}\theta_{12}
+\bar{\alpha}^{(0)}_{1}\alpha^{(0)}_{1}\bar{\alpha}^{(0)}_{1}\theta_{12}
+\bar{\alpha}^{(0)}_{1}\alpha^{(0)}_{1}\alpha^{(0)}_{1}\theta_{12}
-\bar{\alpha}^{(0)}_{1}\alpha^{(0)}_{2}\theta_{12}
\\
-\bar{\alpha}^{(2)}_{6}\theta_{13}
-\alpha^{(2)}_{7}\theta_{13}
-\alpha^{(0)}_{1}\bar{\alpha}^{(2)}_{3}\theta_{13}
-\alpha^{(0)}_{1}\alpha^{(2)}_{3}\theta_{13}
-\alpha^{(0)}_{1}\bar{\alpha}^{(2)}_{4}\theta_{13}
-\alpha^{(0)}_{5}\theta_{13}
-\bar{\alpha}^{(0)}_{1}\alpha^{(2)}_{3}\theta_{13}
+\bar{\alpha}^{(0)}_{1}\bar{\alpha}^{(2)}_{3}\theta_{13}
-\bar{\alpha}^{(0)}_{1}\bar{\alpha}^{(2)}_{4}\theta_{13}
+\bar{\alpha}^{(0)}_{1}\alpha^{(2)}_{3}\theta_{13}
\\
-\bar{\alpha}^{(0)}_{1}\alpha^{(2)}_{4}\theta_{13}
+\bar{\alpha}^{(0)}_{1}\bar{\alpha}^{(2)}_{4}\theta_{13}
+\bar{\alpha}^{(0)}_{1}\alpha^{(0)}_{2}\theta_{13}
-\alpha^{(2)}_{3}\bar{\alpha}^{(0)}_{1}\theta_{13}
+\alpha^{(2)}_{4}\bar{\alpha}^{(0)}_{1}\theta_{13}
+\alpha^{(2)}_{7}\theta_{13}
-\bar{\alpha}^{(2)}_{4}\bar{\alpha}^{(0)}_{1}\theta_{13}
-\bar{\alpha}^{(2)}_{6}\theta_{13}
-\alpha^{(0)}_{1}\alpha^{(2)}_{3}\theta_{13}
+\alpha^{(0)}_{1}\bar{\alpha}^{(2)}_{3}\theta_{13}
\\
-\alpha^{(0)}_{1}\bar{\alpha}^{(2)}_{4}\theta_{13}
+\alpha^{(0)}_{1}\alpha^{(2)}_{3}\theta_{13}
-\alpha^{(0)}_{1}\alpha^{(2)}_{4}\theta_{13}
+\alpha^{(0)}_{1}\bar{\alpha}^{(2)}_{4}\theta_{13}
+\alpha^{(0)}_{1}\alpha^{(0)}_{2}\theta_{13}
+\alpha^{(0)}_{1}\alpha^{(2)}_{3}\theta_{13}
-\alpha^{(0)}_{1}\bar{\alpha}^{(2)}_{3}\theta_{13}
+\alpha^{(0)}_{1}\bar{\alpha}^{(2)}_{4}\theta_{13}
-\alpha^{(0)}_{2}\bar{\alpha}^{(0)}_{1}\theta_{13}
-\alpha^{(0)}_{4}\theta_{13}
\\
+\alpha^{(0)}_{5}\theta_{13}
+\bar{\alpha}^{(0)}_{1}\alpha^{(2)}_{3}\theta_{13}
+\bar{\alpha}^{(0)}_{1}\bar{\alpha}^{(2)}_{4}\theta_{13}
+\bar{\alpha}^{(0)}_{1}\alpha^{(2)}_{4}\theta_{13}
+\alpha^{(2)}_{7}\theta_{14}
+\alpha^{(0)}_{1}\alpha^{(2)}_{3}\theta_{14}
+\alpha^{(0)}_{1}\bar{\alpha}^{(2)}_{4}\theta_{14}
+\alpha^{(0)}_{5}\theta_{14}
-\bar{\alpha}^{(0)}_{1}\alpha^{(2)}_{3}\theta_{14}
+\bar{\alpha}^{(0)}_{1}\alpha^{(2)}_{4}\theta_{14}
\\
-\bar{\alpha}^{(0)}_{1}\bar{\alpha}^{(2)}_{4}\theta_{14}
-\bar{\alpha}^{(0)}_{1}\alpha^{(0)}_{2}\theta_{14}
+\alpha^{(2)}_{3}\bar{\alpha}^{(0)}_{1}\theta_{14}
-\alpha^{(2)}_{4}\bar{\alpha}^{(0)}_{1}\theta_{14}
-\alpha^{(2)}_{7}\theta_{14}
+\bar{\alpha}^{(2)}_{4}\bar{\alpha}^{(0)}_{1}\theta_{14}
+\bar{\alpha}^{(2)}_{6}\theta_{14}
-\alpha^{(0)}_{1}\alpha^{(2)}_{3}\theta_{14}
+\alpha^{(0)}_{1}\alpha^{(2)}_{4}\theta_{14}
-\alpha^{(0)}_{1}\bar{\alpha}^{(2)}_{4}\theta_{14}
\\
-\alpha^{(0)}_{1}\alpha^{(0)}_{2}\theta_{14}
-\alpha^{(0)}_{1}\alpha^{(2)}_{3}\theta_{14}
+\alpha^{(0)}_{1}\bar{\alpha}^{(2)}_{3}\theta_{14}
-\alpha^{(0)}_{1}\bar{\alpha}^{(2)}_{4}\theta_{14}
+\alpha^{(0)}_{2}\bar{\alpha}^{(0)}_{1}\theta_{14}
+\alpha^{(0)}_{4}\theta_{14}
-\alpha^{(0)}_{5}\theta_{14}
-\bar{\alpha}^{(0)}_{1}\alpha^{(2)}_{4}\theta_{14}
+\alpha^{(2)}_{3}\bar{\alpha}^{(0)}_{1}\theta_{14}
+\alpha^{(2)}_{3}\alpha^{(0)}_{1}\theta_{14}
\\
-\alpha^{(2)}_{3}\bar{\alpha}^{(0)}_{1}\theta_{14}
+\bar{\alpha}^{(2)}_{3}\alpha^{(0)}_{1}\theta_{14}
-\bar{\alpha}^{(2)}_{3}\bar{\alpha}^{(0)}_{1}\theta_{14}
-\bar{\alpha}^{(2)}_{3}\alpha^{(0)}_{1}\theta_{14}
+\bar{\alpha}^{(2)}_{4}\bar{\alpha}^{(0)}_{1}\theta_{14}
+\bar{\alpha}^{(2)}_{4}\alpha^{(0)}_{1}\theta_{14}
-\bar{\alpha}^{(2)}_{4}\bar{\alpha}^{(0)}_{1}\theta_{14}
-\bar{\alpha}^{(2)}_{6}\theta_{14}
-\alpha^{(0)}_{4}\theta_{15}
+\alpha^{(0)}_{5}\theta_{15}
\\
-\alpha^{(2)}_{3}\alpha^{(0)}_{1}\theta_{15}
+\alpha^{(2)}_{3}\bar{\alpha}^{(0)}_{1}\theta_{15}
+\bar{\alpha}^{(2)}_{3}\alpha^{(0)}_{1}\theta_{15}
-\bar{\alpha}^{(2)}_{4}\alpha^{(0)}_{1}\theta_{15}
+\bar{\alpha}^{(2)}_{4}\bar{\alpha}^{(0)}_{1}\theta_{15}
+\bar{\alpha}^{(2)}_{6}\theta_{15}
+\alpha^{(2)}_{3}\alpha^{(0)}_{1}\theta_{15}
-\alpha^{(2)}_{4}\alpha^{(0)}_{1}\theta_{15}
+\alpha^{(2)}_{4}\bar{\alpha}^{(0)}_{1}\theta_{15}
+\alpha^{(2)}_{7}\theta_{15}
\\
+\bar{\alpha}^{(2)}_{4}\alpha^{(0)}_{1}\theta_{15}
+\alpha^{(0)}_{1}\bar{\alpha}^{(2)}_{3}\theta_{15}
+\alpha^{(0)}_{1}\alpha^{(2)}_{3}\theta_{15}
+\alpha^{(0)}_{1}\bar{\alpha}^{(2)}_{4}\theta_{15}
+\alpha^{(0)}_{2}\alpha^{(0)}_{1}\theta_{15}
+\alpha^{(0)}_{5}\theta_{15}
+\bar{\alpha}^{(0)}_{1}\alpha^{(0)}_{2}\theta_{15}
+\bar{\alpha}^{(0)}_{1}\alpha^{(2)}_{3}\theta_{15}
-\bar{\alpha}^{(0)}_{1}\bar{\alpha}^{(2)}_{3}\theta_{15}
+\bar{\alpha}^{(0)}_{1}\bar{\alpha}^{(2)}_{4}\theta_{15}
\\
-\bar{\alpha}^{(0)}_{1}\alpha^{(2)}_{3}\theta_{15}
+\bar{\alpha}^{(0)}_{1}\alpha^{(2)}_{4}\theta_{15}
-\bar{\alpha}^{(0)}_{1}\bar{\alpha}^{(2)}_{4}\theta_{15}
-\bar{\alpha}^{(0)}_{1}\alpha^{(0)}_{2}\theta_{15}
+\alpha^{(2)}_{3}\alpha^{(0)}_{1}\theta_{15}
+\alpha^{(2)}_{3}\bar{\alpha}^{(0)}_{1}\theta_{15}
-\alpha^{(2)}_{4}\bar{\alpha}^{(0)}_{1}\theta_{15}
-\alpha^{(2)}_{7}\theta_{15}
-\bar{\alpha}^{(2)}_{3}\alpha^{(0)}_{1}\theta_{15}
+\bar{\alpha}^{(2)}_{4}\alpha^{(0)}_{1}\theta_{15}
\\
+\bar{\alpha}^{(2)}_{4}\bar{\alpha}^{(0)}_{1}\theta_{15}
+\bar{\alpha}^{(2)}_{6}\theta_{15}
+\alpha^{(2)}_{3}\theta_{16}
+\bar{\alpha}^{(2)}_{4}\theta_{16}
+\alpha^{(0)}_{2}\theta_{16}
-\bar{\alpha}^{(0)}_{1}\alpha^{(0)}_{1}\theta_{16}
+\alpha^{(2)}_{3}\theta_{16}
+\bar{\alpha}^{(2)}_{4}\theta_{16}
-\alpha^{(0)}_{1}\alpha^{(0)}_{1}\theta_{16}
+\alpha^{(0)}_{2}\theta_{16}
+\bar{\alpha}^{(0)}_{1}\alpha^{(0)}_{1}\theta_{16}
\\
-\bar{\alpha}^{(0)}_{1}\bar{\alpha}^{(0)}_{1}\theta_{16}
-\bar{\alpha}^{(0)}_{1}\alpha^{(0)}_{1}\theta_{16}
+\alpha^{(2)}_{3}\theta_{16}
+\bar{\alpha}^{(2)}_{4}\theta_{16}
+\alpha^{(0)}_{1}\alpha^{(0)}_{1}\theta_{16}
-\alpha^{(0)}_{1}\bar{\alpha}^{(0)}_{1}\theta_{16}
-\alpha^{(0)}_{1}\alpha^{(0)}_{1}\theta_{16}
+\alpha^{(0)}_{2}\theta_{16}
+\bar{\alpha}^{(0)}_{1}\bar{\alpha}^{(0)}_{1}\theta_{16}
+\bar{\alpha}^{(0)}_{1}\alpha^{(0)}_{1}\theta_{16}
\\
-\bar{\alpha}^{(0)}_{1}\bar{\alpha}^{(0)}_{1}\theta_{16}
-\bar{\alpha}^{(0)}_{1}\alpha^{(0)}_{1}\theta_{16}
-\alpha^{(0)}_{3}\theta_{17}
-\bar{\alpha}^{(0)}_{1}\alpha^{(0)}_{1}\theta_{17}
-\alpha^{(0)}_{1}\alpha^{(0)}_{1}\theta_{17}
+\bar{\alpha}^{(0)}_{1}\alpha^{(0)}_{1}\theta_{17}
+\bar{\alpha}^{(0)}_{1}\bar{\alpha}^{(0)}_{1}\theta_{17}
+\bar{\alpha}^{(0)}_{1}\alpha^{(0)}_{1}\theta_{17}
-\alpha^{(0)}_{2}\theta_{18}
+\bar{\alpha}^{(0)}_{1}\alpha^{(0)}_{1}\theta_{18}
\\
-\alpha^{(2)}_{3}\theta_{18}
-\bar{\alpha}^{(2)}_{4}\theta_{18}
+\alpha^{(0)}_{1}\alpha^{(0)}_{1}\theta_{18}
-\alpha^{(0)}_{2}\theta_{18}
-\bar{\alpha}^{(0)}_{1}\alpha^{(0)}_{1}\theta_{18}
+\bar{\alpha}^{(0)}_{1}\bar{\alpha}^{(0)}_{1}\theta_{18}
+\bar{\alpha}^{(0)}_{1}\alpha^{(0)}_{1}\theta_{18}
-\alpha^{(2)}_{3}\theta_{18}
-\bar{\alpha}^{(2)}_{4}\theta_{18}
-\alpha^{(0)}_{1}\alpha^{(0)}_{1}\theta_{18}
+\alpha^{(0)}_{1}\bar{\alpha}^{(0)}_{1}\theta_{18}
\\
+\alpha^{(0)}_{1}\alpha^{(0)}_{1}\theta_{18}
-\alpha^{(0)}_{2}\theta_{18}
-\bar{\alpha}^{(0)}_{1}\bar{\alpha}^{(0)}_{1}\theta_{18}
-\bar{\alpha}^{(0)}_{1}\alpha^{(0)}_{1}\theta_{18}
+\bar{\alpha}^{(0)}_{1}\bar{\alpha}^{(0)}_{1}\theta_{18}
+\bar{\alpha}^{(0)}_{1}\alpha^{(0)}_{1}\theta_{18}
-\alpha^{(2)}_{3}\theta_{18}
-\bar{\alpha}^{(2)}_{4}\theta_{18}
+\alpha^{(0)}_{1}\theta_{19}
-\bar{\alpha}^{(0)}_{1}\theta_{19}
+\bar{\alpha}^{(0)}_{1}\theta_{20}
\\
+\alpha^{(0)}_{1}\theta_{20}
-\bar{\alpha}^{(0)}_{1}\theta_{20}
-\alpha^{(0)}_{1}\theta_{20}
+\theta_{21}
&= 0
,
\\
\bar{\alpha}^{(0)}_{1}
+\alpha^{(0)}_{1}
&= 0
,
\\
\alpha^{(0)}_{1}\bar{\alpha}^{(0)}_{1}
+\alpha^{(0)}_{3}
&= 0
,\\
\alpha^{(0)}_{2}\bar{\alpha}^{(0)}_{1}
+\alpha^{(0)}_{4}
&= 0
,\\
\alpha^{(0)}_{4}\bar{\alpha}^{(0)}_{1}
-\alpha^{(0)}_{5}\bar{\alpha}^{(0)}_{1}
+\alpha^{(0)}_{7}
&= 0
,
\\
 \alpha^{(0)}_{6}\bar{\alpha}^{(0)}_{1}
+\alpha^{(0)}_{8}
&= 0
,\\
\bar{\alpha}^{(2)}_{3}
+\alpha^{(2)}_{3}
&= 0
,\\
\bar{\alpha}^{(2)}_{4}
+\alpha^{(2)}_{4}
&= 0
,\\
-\alpha^{(1)}_{2}
+\bar{\alpha}^{(2)}_{6}
+\alpha^{(2)}_{7}
&= 0
,\\
\bar{\alpha}^{(2)}_{8}
+\alpha^{(2)}_{3}\bar{\alpha}^{(2)}_{3}
-\alpha^{(2)}_{4}\bar{\alpha}^{(2)}_{4}
-\alpha^{(2)}_{10}
&= 0
,\\
\alpha^{(2)}_{3}\bar{\alpha}^{(2)}_{4}
-\alpha^{(2)}_{4}\bar{\alpha}^{(2)}_{4}
-\alpha^{(2)}_{10}
&= 0
,
\\
-\alpha^{(1)}_{4}
+\bar{\alpha}^{(2)}_{11}
-\alpha^{(2)}_{3}\bar{\alpha}^{(2)}_{6}
+\alpha^{(2)}_{4}\bar{\alpha}^{(2)}_{6}
+\alpha^{(2)}_{7}\bar{\alpha}^{(2)}_{3}
-\alpha^{(2)}_{7}\bar{\alpha}^{(2)}_{4}
&= 0
,\\
\bar{\alpha}^{(2)}_{13}
+\alpha^{(2)}_{14}
&= 0
,\\
\bar{\alpha}^{(2)}_{15}
+\alpha^{(2)}_{3}\bar{\alpha}^{(2)}_{11}
+\alpha^{(2)}_{15}
&= 0
,
\\
\alpha^{(4)}_{1}
+\alpha^{(3)}_{1}
&= 0
,\\
\alpha^{(4)}_{1}
+\bar{\alpha}^{(3)}_{2}
&= 0
,\\
\bar{\alpha}^{(3)}_{3}
-\bar{\alpha}^{(3)}_{5}
&= 0
,\\
\bar{\alpha}^{(3)}_{2}\alpha^{(3)}_{1}
+\bar{\alpha}^{(3)}_{4}
-\bar{\alpha}^{(3)}_{5}
&= 0
,\\
\bar{\alpha}^{(3)}_{3}\alpha^{(3)}_{1}
-\bar{\alpha}^{(3)}_{5}\alpha^{(3)}_{1}
&= 0
,\\
-\bar{\alpha}^{(4)}_{4}
-\bar{\alpha}^{(3)}_{5}\alpha^{(3)}_{1}
+\bar{\alpha}^{(3)}_{7}
&= 0
,\\
\bar{\alpha}^{(4)}_{3}
-\bar{\alpha}^{(4)}_{4}
-\bar{\alpha}^{(3)}_{5}\alpha^{(3)}_{1}
+\bar{\alpha}^{(3)}_{8}
&= 0
,\\
\bar{\alpha}^{(4)}_{5}
-\bar{\alpha}^{(3)}_{7}\alpha^{(3)}_{1}
+\bar{\alpha}^{(3)}_{8}\alpha^{(3)}_{1}
+\bar{\alpha}^{(3)}_{9}
-\bar{\alpha}^{(3)}_{10}
&= 0
,\end{align*}
\begin{align*}
\bar{\alpha}^{(4)}_{6}
+\alpha^{(4)}_{1}\bar{\alpha}^{(4)}_{3}
+\bar{\alpha}^{(3)}_{11}
&= 0
,\\
\bar{\alpha}^{(4)}_{7}
-\alpha^{(4)}_{1}\bar{\alpha}^{(4)}_{5}
+\bar{\alpha}^{(3)}_{9}\alpha^{(3)}_{1}
-\bar{\alpha}^{(3)}_{11}\alpha^{(3)}_{1}
+\bar{\alpha}^{(3)}_{12}
-\bar{\alpha}^{(3)}_{14}
&= 0
,\\
\bar{\alpha}^{(4)}_{7}
-\alpha^{(4)}_{1}\bar{\alpha}^{(4)}_{5}
-\bar{\alpha}^{(3)}_{11}\alpha^{(3)}_{1}
+\bar{\alpha}^{(3)}_{13}
-\bar{\alpha}^{(3)}_{14}
&= 0
,
\\
-\bar{\alpha}^{(4)}_{8}
-\bar{\alpha}^{(3)}_{13}\alpha^{(3)}_{1}
-\bar{\alpha}^{(3)}_{16}
&= 0
,\\
-\bar{\alpha}^{(4)}_{9}
-\alpha^{(4)}_{1}\bar{\alpha}^{(4)}_{8}
+\bar{\alpha}^{(3)}_{17}
&= 0
,
\\
-\alpha^{(5)}_{1}
+\bar{\alpha}^{(4)}_{3}
&= 0
,\\
-\bar{\alpha}^{(5)}_{1}
+\bar{\alpha}^{(4)}_{4}
&= 0
,\\
-\bar{\alpha}^{(5)}_{2}
-\bar{\alpha}^{(4)}_{4}\alpha^{(4)}_{1}
+\bar{\alpha}^{(4)}_{5}
-\bar{\alpha}^{(4)}_{6}
&= 0
,\\
\alpha^{(5)}_{1}\bar{\alpha}^{(5)}_{1}
+\bar{\alpha}^{(4)}_{8}
&= 0
,\\
\bar{\alpha}^{(5)}_{3}
+\alpha^{(5)}_{1}\bar{\alpha}^{(5)}_{2}
-\bar{\alpha}^{(4)}_{8}\alpha^{(4)}_{1}
+\bar{\alpha}^{(4)}_{9}
&= 0
,\\
\alpha^{(5)}_{1}
+\bar{\alpha}^{(5)}_{1}
&= 0
,\\
-\bar{\alpha}^{(5)}_{2}\alpha^{(5)}_{1}
+\bar{\alpha}^{(5)}_{3}
&= 0 .
\end{align*}%
\end{footnotesize}%
Further,
we apply 
formulas for 
$\theta_2$, $\theta_6$ and $\theta_{11}$
(obtained in \ref{secA:adm}), 
to the above equations 
and solve obtained system of equations 
for the remaining variables 
\begin{small}
\setlength{\jot}{2pt}
\begin{gather*}
\alpha^{(0)}_{1},
\alpha^{(0)}_{2},
\alpha^{(0)}_{3},
\alpha^{(0)}_{4},
\alpha^{(0)}_{5},
\alpha^{(0)}_{6},
\alpha^{(0)}_{7},
\alpha^{(0)}_{8},
\bar{\alpha}^{(0)}_{1},
\alpha^{(1)}_{2},
\alpha^{(1)}_{4},
\alpha^{(2)}_{3},
\alpha^{(2)}_{4},
\alpha^{(2)}_{7},
\alpha^{(2)}_{10},
\alpha^{(2)}_{13},
\alpha^{(2)}_{14},
\alpha^{(2)}_{15},
\bar{\alpha}^{(2)}_{3},
\bar{\alpha}^{(2)}_{4},
\bar{\alpha}^{(2)}_{6},
\bar{\alpha}^{(2)}_{8},
\bar{\alpha}^{(2)}_{11},
\bar{\alpha}^{(2)}_{13},
\bar{\alpha}^{(2)}_{15},
\alpha^{(3)}_{1},
\\
\bar{\alpha}^{(3)}_{2},
\bar{\alpha}^{(3)}_{3},
\bar{\alpha}^{(3)}_{4},
\bar{\alpha}^{(3)}_{5},
\bar{\alpha}^{(3)}_{7},
\bar{\alpha}^{(3)}_{8},
\bar{\alpha}^{(3)}_{9},
\bar{\alpha}^{(3)}_{10},
\bar{\alpha}^{(3)}_{11},
\bar{\alpha}^{(3)}_{12},
\bar{\alpha}^{(3)}_{13},
\bar{\alpha}^{(3)}_{14},
\bar{\alpha}^{(3)}_{16},
\bar{\alpha}^{(3)}_{17},
\alpha^{(4)}_{1},
\bar{\alpha}^{(4)}_{3},
\bar{\alpha}^{(4)}_{4},
\bar{\alpha}^{(4)}_{5},
\bar{\alpha}^{(4)}_{6},
\bar{\alpha}^{(4)}_{7},
\bar{\alpha}^{(4)}_{8},
\bar{\alpha}^{(4)}_{9},
\alpha^{(5)}_{1},
\bar{\alpha}^{(5)}_{1},
\bar{\alpha}^{(5)}_{2},
\bar{\alpha}^{(5)}_{3}.
\end{gather*}%
\end{small}%

\subsection{Coefficients for an exemplary homomorphism}
\label{secA:coeff}

Finally, we obtain the following 
solution of the above system of equation
\begin{footnotesize}
\setlength{\jot}{0pt}
\begin{align*}
\alpha^{(0)}_{1} &=  2 \theta_{3}
- \theta_{1}
,
\\
\alpha^{(0)}_{2} &= 
- \theta_{1}^{2} 
+ 2 \theta_{1} \theta_{3}
- 2 \theta_{3}^{2} 
- 2 \theta_{4}
+ \theta_{5}
+ \theta_{7}
,
\\
\alpha^{(0)}_{3} &=  \theta_{1}^{2} 
- 4 \theta_{1} \theta_{3}
+ 4 \theta_{3}^{2} 
,
\\
\alpha^{(0)}_{4} &=  \theta_{1}^{3} 
- 4 \theta_{1}^{2}  \theta_{3}
+ 6 \theta_{1} \theta_{3}^{2} 
- 4 \theta_{3}^{3} 
+ 2 \theta_{1} \theta_{4}
- \theta_{1} \theta_{5}
- \theta_{1} \theta_{7}
- 4 \theta_{3} \theta_{4}
+ 2 \theta_{3} \theta_{5}
+ 2 \theta_{3} \theta_{7}
,
\\
\alpha^{(0)}_{5} &= 
- 2 \theta_{1}^{3} 
+ 11 \theta_{1}^{2}  \theta_{3}
- 18 \theta_{1} \theta_{3}^{2} 
+ 10 \theta_{3}^{3} 
- \theta_{1} \theta_{4}
+ 2 \theta_{1} \theta_{7}
+ 3 \theta_{3} \theta_{4}
- \theta_{3} \theta_{5}
- 3 \theta_{3} \theta_{7}
+ \theta_{10}
- \theta_{9}
,
\\
\alpha^{(0)}_{6} &=  4 \theta_{7} \theta_{5}
+ \theta_{5}^{2} 
- 6 \theta_{4} \theta_{5}
- 6 \theta_{4} \theta_{7}
- 9 \theta_{1} \theta_{5} \theta_{3}
- 11 \theta_{1} \theta_{7} \theta_{3}
+ 12 \theta_{1} \theta_{3} \theta_{4}
- 26 \theta_{3}^{4} 
- 6 \theta_{1}^{4} 
- \tfrac{1}{2} \theta_{10} \theta_{1}
+ \tfrac{1}{2} \theta_{10} \theta_{3}
+ 33 \theta_{1}^{3}  \theta_{3}
- 70 \theta_{1}^{2}  \theta_{3}^{2} 
\\&\ \quad
+ 69 \theta_{1} \theta_{3}^{3} 
+ 3 \theta_{1}^{2}  \theta_{5}
+ \theta_{1}^{2}  \theta_{7}
+ 5 \theta_{3}^{2}  \theta_{5}
+ 11 \theta_{3}^{2}  \theta_{7}
- \theta_{1}^{2}  \theta_{4}
- 11 \theta_{3}^{2}  \theta_{4}
+ \theta_{9} \theta_{1}
- \theta_{9} \theta_{3}
- \tfrac{9}{2} \theta_{8} \theta_{1}
+ \tfrac{13}{2} \theta_{8} \theta_{3}
+ \theta_{7}^{2} 
+ 6 \theta_{4}^{2} 
- \theta_{12}
+ \theta_{13}
,
\\
\alpha^{(0)}_{7} &= 
- 3 \theta_{1}^{4} 
+ 21 \theta_{1}^{3}  \theta_{3}
- 54 \theta_{1}^{2}  \theta_{3}^{2} 
+ 62 \theta_{1} \theta_{3}^{3} 
- 28 \theta_{3}^{4} 
- 3 \theta_{1}^{2}  \theta_{4}
+ \theta_{1}^{2}  \theta_{5}
+ 3 \theta_{1}^{2}  \theta_{7}
+ 13 \theta_{1} \theta_{3} \theta_{4}
- 5 \theta_{1} \theta_{3} \theta_{5}
- 11 \theta_{1} \theta_{3} \theta_{7}
- 14 \theta_{3}^{2}  \theta_{4}
\\&\ \quad
+ 6 \theta_{3}^{2}  \theta_{5}
+ 10 \theta_{3}^{2}  \theta_{7}
+ \theta_{1} \theta_{10}
- \theta_{1} \theta_{9}
- 2 \theta_{10} \theta_{3}
+ 2 \theta_{3} \theta_{9}
,
\\
\alpha^{(0)}_{8} &= 
- \theta_{1} \theta_{5}^{2} 
+ 2 \theta_{3} \theta_{5}^{2} 
- 4 \theta_{1} \theta_{7} \theta_{5}
+ 8 \theta_{3} \theta_{7} \theta_{5}
+ 13 \theta_{7} \theta_{1}^{2}  \theta_{3}
+ \theta_{1}^{3}  \theta_{4}
+ 35 \theta_{3}^{2}  \theta_{1} \theta_{4}
- 23 \theta_{3}^{2}  \theta_{1} \theta_{5}
- 33 \theta_{3}^{2}  \theta_{1} \theta_{7}
+ 15 \theta_{5} \theta_{1}^{2}  \theta_{3}
\\&\ \quad
- 52 \theta_{3}^{5} 
+ 136 \theta_{3}^{2}  \theta_{1}^{3} 
- 209 \theta_{3}^{3}  \theta_{1}^{2} 
+ 164 \theta_{3}^{4}  \theta_{1}
- 22 \theta_{3}^{3}  \theta_{4}
+ 10 \theta_{3}^{3}  \theta_{5}
+ 22 \theta_{3}^{3}  \theta_{7}
- 6 \theta_{1} \theta_{4}^{2} 
+ 12 \theta_{3} \theta_{4}^{2} 
- \theta_{1} \theta_{7}^{2} 
+ 2 \theta_{3} \theta_{7}^{2} 
\\&\ \quad
- \theta_{13} \theta_{1}
+ 2 \theta_{13} \theta_{3}
+ \theta_{12} \theta_{1}
- 2 \theta_{12} \theta_{3}
+ \tfrac{9}{2} \theta_{1}^{2}  \theta_{8}
+ 13 \theta_{3}^{2}  \theta_{8}
- \theta_{1}^{2}  \theta_{9}
- 2 \theta_{3}^{2}  \theta_{9}
+ \tfrac{1}{2} \theta_{1}^{2}  \theta_{10}
+ \theta_{10} \theta_{3}^{2} 
- \tfrac{31}{2} \theta_{1} \theta_{3} \theta_{8}
+ 3 \theta_{1} \theta_{3} \theta_{9}
\\&\ \quad
- \tfrac{3}{2} \theta_{1} \theta_{10} \theta_{3}
+ 6 \theta_{1} \theta_{4} \theta_{5}
+ 6 \theta_{1} \theta_{4} \theta_{7}
- 12 \theta_{3} \theta_{4} \theta_{5}
- 12 \theta_{3} \theta_{4} \theta_{7}
- 45 \theta_{1}^{4}  \theta_{3}
- 3 \theta_{1}^{3}  \theta_{5}
- \theta_{1}^{3}  \theta_{7}
+ 6 \theta_{1}^{5} 
- 14 \theta_{4} \theta_{1}^{2}  \theta_{3}
,
\\
\bar{\alpha}^{(0)}_{1} &= 
- 2 \theta_{3}
+ \theta_{1}
,
\\
\alpha^{(1)}_{2} &=  \theta_{1}^{3} 
- 3 \theta_{1}^{2}  \theta_{3}
+ 2 \theta_{3}^{3} 
+ \theta_{1} \theta_{4}
- \theta_{1} \theta_{5}
- \theta_{1} \theta_{7}
- 3 \theta_{3} \theta_{4}
+ 3 \theta_{3} \theta_{5}
+ \theta_{3} \theta_{7}
- \theta_{10}
+ \theta_{9}
,
\\
\alpha^{(1)}_{4} &=  \tfrac{5}{2} \theta_{10} \theta_{5}
+ \tfrac{13}{2} \theta_{8} \theta_{5}
- 11 \theta_{1} \theta_{5}^{2} 
+ 10 \theta_{3} \theta_{5}^{2} 
- 16 \theta_{1} \theta_{7} \theta_{5}
+ 18 \theta_{3} \theta_{7} \theta_{5}
- 3 \theta_{9} \theta_{5}
- 46 \theta_{7} \theta_{1}^{2}  \theta_{3}
- 31 \theta_{1}^{3}  \theta_{4}
- 70 \theta_{3}^{2}  \theta_{1} \theta_{4}
+ 36 \theta_{3}^{2}  \theta_{1} \theta_{5}
\\&\ \quad
+ 46 \theta_{3}^{2}  \theta_{1} \theta_{7}
- 42 \theta_{5} \theta_{1}^{2}  \theta_{3}
+ 25 \theta_{3}^{5} 
- 94 \theta_{3}^{2}  \theta_{1}^{3} 
+ 122 \theta_{3}^{3}  \theta_{1}^{2} 
- 86 \theta_{3}^{4}  \theta_{1}
+ 22 \theta_{3}^{3}  \theta_{4}
- 12 \theta_{3}^{3}  \theta_{5}
- 16 \theta_{3}^{3}  \theta_{7}
- 12 \theta_{4} \theta_{8}
- 4 \theta_{10} \theta_{4}
\\&\ \quad
- 24 \theta_{1} \theta_{4}^{2} 
+ 28 \theta_{3} \theta_{4}^{2} 
+ 4 \theta_{4} \theta_{9}
+ \tfrac{3}{2} \theta_{10} \theta_{7}
+ \tfrac{11}{2} \theta_{7} \theta_{8}
- \theta_{7} \theta_{9}
- 5 \theta_{1} \theta_{7}^{2} 
+ 8 \theta_{3} \theta_{7}^{2} 
- \theta_{13} \theta_{1}
+ \theta_{13} \theta_{3}
+ 2 \theta_{12} \theta_{1}
- 2 \theta_{12} \theta_{3}
+ 2 \theta_{14} \theta_{1}
\\&\ \quad
- 2 \theta_{14} \theta_{3}
- \theta_{15} \theta_{1}
+ \theta_{15} \theta_{3}
- \tfrac{5}{2} \theta_{1}^{2}  \theta_{8}
+ \tfrac{1}{2} \theta_{3}^{2}  \theta_{8}
- 7 \theta_{3}^{2}  \theta_{9}
- \tfrac{3}{2} \theta_{1}^{2}  \theta_{10}
+ \tfrac{9}{2} \theta_{10} \theta_{3}^{2} 
+ 6 \theta_{1} \theta_{3} \theta_{9}
- 2 \theta_{1} \theta_{10} \theta_{3}
+ 34 \theta_{1} \theta_{4} \theta_{5}
\\&\ \quad
+ 22 \theta_{1} \theta_{4} \theta_{7}
- 34 \theta_{3} \theta_{4} \theta_{5}
- 30 \theta_{3} \theta_{4} \theta_{7}
+ 42 \theta_{1}^{4}  \theta_{3}
+ 18 \theta_{1}^{3}  \theta_{5}
+ 16 \theta_{1}^{3}  \theta_{7}
- 9 \theta_{1}^{5} 
+ 79 \theta_{4} \theta_{1}^{2}  \theta_{3}
- \theta_{16}
+ \theta_{17}
+ \theta_{18}
,
\\
\alpha^{(2)}_{3} &=  \theta_{1}^{2} 
- 4 \theta_{1} \theta_{3}
+ 4 \theta_{3}^{2} 
+ \theta_{4}
- \theta_{7}
,
\\
\alpha^{(2)}_{4} &= 
- 2 \theta_{1} \theta_{3}
+ 2 \theta_{3}^{2} 
- \theta_{4}
+ \theta_{5}
,
\\
\alpha^{(2)}_{7} &=  3 \theta_{1}^{3} 
- 10 \theta_{1}^{2}  \theta_{3}
+ 11 \theta_{1} \theta_{3}^{2} 
- 5 \theta_{3}^{3} 
+ 4 \theta_{1} \theta_{4}
- 3 \theta_{1} \theta_{5}
- 2 \theta_{1} \theta_{7}
- 8 \theta_{3} \theta_{4}
+ 5 \theta_{3} \theta_{5}
+ 4 \theta_{3} \theta_{7}
+ \theta_{8}
,
\\
\alpha^{(2)}_{10} &=  2 \theta_{1}^{3}  \theta_{3}
- 6 \theta_{1}^{2}  \theta_{3}^{2} 
+ 8 \theta_{1} \theta_{3}^{3} 
- 4 \theta_{3}^{4} 
+ \theta_{1}^{2}  \theta_{4}
- \theta_{1}^{2}  \theta_{5}
+ 2 \theta_{1} \theta_{3} \theta_{4}
- 2 \theta_{1} \theta_{3} \theta_{7}
- 2 \theta_{3}^{2}  \theta_{4}
+ 2 \theta_{3}^{2}  \theta_{7}
+ 2 \theta_{4}^{2} 
- 3 \theta_{4} \theta_{5}
- \theta_{4} \theta_{7}
+ \theta_{5}^{2} 
+ \theta_{5} \theta_{7}
,
\\
\alpha^{(2)}_{13} &=  12 \theta_{4}^{3} 
- \tfrac{3}{2} \theta_{10}^{2} 
- 3 \theta_{4} \theta_{12}
- 2 \theta_{4} \theta_{14}
- 8 \theta_{7} \theta_{5}^{2} 
+ 201 \theta_{3}^{2}  \theta_{4} \theta_{5}
+ 221 \theta_{3}^{2}  \theta_{4} \theta_{7}
- 126 \theta_{3}^{2}  \theta_{5} \theta_{7}
+ 95 \theta_{1}^{2}  \theta_{3} \theta_{8}
- \tfrac{201}{2} \theta_{1} \theta_{3}^{2}  \theta_{8}
- 51 \theta_{1}^{2}  \theta_{3} \theta_{9}
\\&\ \quad
+ 61 \theta_{1}^{2}  \theta_{4} \theta_{5}
+ 30 \theta_{1}^{2}  \theta_{4} \theta_{7}
- 26 \theta_{1}^{2}  \theta_{5} \theta_{7}
+ 51 \theta_{1} \theta_{3}^{2}  \theta_{9}
+ 172 \theta_{1} \theta_{3} \theta_{4}^{2} 
+ 92 \theta_{1} \theta_{3} \theta_{5}^{2} 
+ 60 \theta_{1} \theta_{3} \theta_{7}^{2} 
+ 2 \theta_{5} \theta_{12}
+ \theta_{5} \theta_{14}
+ 2 \theta_{1}^{2}  \theta_{13}
\\&\ \quad
- 2 \theta_{7} \theta_{13}
+ 2 \theta_{3}^{2}  \theta_{13}
- 2 \theta_{5} \theta_{13}
- 3 \theta_{8}^{2} 
+ 18 \theta_{1} \theta_{4} \theta_{9}
- 14 \theta_{1} \theta_{5} \theta_{9}
- 10 \theta_{1} \theta_{7} \theta_{9}
- 34 \theta_{3} \theta_{4} \theta_{9}
+ 24 \theta_{3} \theta_{5} \theta_{9}
+ 20 \theta_{3} \theta_{7} \theta_{9}
\\&\ \quad
- 36 \theta_{1} \theta_{4} \theta_{8}
+ \tfrac{55}{2} \theta_{1} \theta_{5} \theta_{8}
+ \tfrac{41}{2} \theta_{1} \theta_{7} \theta_{8}
+ 75 \theta_{3} \theta_{4} \theta_{8}
- 50 \theta_{3} \theta_{5} \theta_{8}
- 45 \theta_{3} \theta_{7} \theta_{8}
- 3 \theta_{1} \theta_{3} \theta_{13}
- 24 \theta_{3} \theta_{5} \theta_{10}
- 21 \theta_{3} \theta_{7} \theta_{10}
\\&\ \quad
+ \tfrac{29}{2} \theta_{1} \theta_{5} \theta_{10}
+ \tfrac{19}{2} \theta_{1} \theta_{7} \theta_{10}
+ 35 \theta_{3} \theta_{4} \theta_{10}
- 18 \theta_{1} \theta_{4} \theta_{10}
+ 28 \theta_{4} \theta_{5} \theta_{7}
+ 4 \theta_{14} \theta_{3}^{2} 
- \theta_{15} \theta_{3}^{2} 
- \theta_{1}^{2}  \theta_{12}
+ 3 \theta_{12} \theta_{3}^{2} 
+ \theta_{18} \theta_{1}
- 3 \theta_{18} \theta_{3}
\\&\ \quad
- \theta_{16} \theta_{1}
+ 3 \theta_{16} \theta_{3}
+ 8 \theta_{4} \theta_{7}^{2} 
- 20 \theta_{4}^{2}  \theta_{7}
+ 8 \theta_{4} \theta_{5}^{2} 
- 20 \theta_{4}^{2}  \theta_{5}
+ \theta_{1} \theta_{15} \theta_{3}
- 2 \theta_{1} \theta_{12} \theta_{3}
- 4 \theta_{1} \theta_{14} \theta_{3}
- \tfrac{113}{2} \theta_{1} \theta_{3}^{2}  \theta_{10}
+ 54 \theta_{1}^{2}  \theta_{3} \theta_{10}
\\&\ \quad
+ 2 \theta_{7} \theta_{12}
+ \theta_{7} \theta_{14}
- \tfrac{31}{2} \theta_{1}^{3}  \theta_{10}
+ 16 \theta_{3}^{3}  \theta_{10}
- \tfrac{9}{2} \theta_{8} \theta_{10}
+ \tfrac{7}{2} \theta_{8} \theta_{9}
+ 3 \theta_{4} \theta_{13}
- 244 \theta_{1} \theta_{3} \theta_{4} \theta_{5}
- 200 \theta_{1} \theta_{3} \theta_{4} \theta_{7}
+ 134 \theta_{1} \theta_{3} \theta_{5} \theta_{7}
\\&\ \quad
+ \tfrac{5}{2} \theta_{9} \theta_{10}
+ 421 \theta_{1}^{3}  \theta_{3} \theta_{4}
- 306 \theta_{1}^{3}  \theta_{3} \theta_{5}
- 228 \theta_{1}^{3}  \theta_{3} \theta_{7}
- 839 \theta_{1}^{2}  \theta_{3}^{2}  \theta_{4}
+ 546 \theta_{1}^{2}  \theta_{3}^{2}  \theta_{5}
+ 487 \theta_{1}^{2}  \theta_{3}^{2}  \theta_{7}
+ 730 \theta_{1} \theta_{3}^{3}  \theta_{4}
- 436 \theta_{1} \theta_{3}^{3}  \theta_{5}
\\&\ \quad
- 444 \theta_{1} \theta_{3}^{3}  \theta_{7}
- \theta_{17} \theta_{3}
- \tfrac{55}{2} \theta_{1}^{3}  \theta_{8}
+ 35 \theta_{3}^{3}  \theta_{8}
+ 220 \theta_{1}^{5}  \theta_{3}
- 629 \theta_{1}^{4}  \theta_{3}^{2} 
+ 980 \theta_{1}^{3}  \theta_{3}^{3} 
- 889 \theta_{1}^{2}  \theta_{3}^{4} 
+ 445 \theta_{1} \theta_{3}^{5} 
- 74 \theta_{1}^{4}  \theta_{4}
+ 62 \theta_{1}^{4}  \theta_{5}
\\&\ \quad
+ 36 \theta_{1}^{4}  \theta_{7}
- 242 \theta_{3}^{4}  \theta_{4}
+ 136 \theta_{3}^{4}  \theta_{5}
+ 150 \theta_{3}^{4}  \theta_{7}
+ 15 \theta_{1}^{3}  \theta_{9}
- 32 \theta_{1}^{2}  \theta_{4}^{2} 
- 30 \theta_{1}^{2}  \theta_{5}^{2} 
- 8 \theta_{1}^{2}  \theta_{7}^{2} 
- 13 \theta_{3}^{3}  \theta_{9}
- 171 \theta_{3}^{2}  \theta_{4}^{2} 
- 62 \theta_{3}^{2}  \theta_{5}^{2} 
\\&\ \quad
- 72 \theta_{3}^{2}  \theta_{7}^{2} 
- 32 \theta_{1}^{6} 
- 95 \theta_{3}^{6} 
- \theta_{9}^{2} 
- \theta_{20}
- 8 \theta_{7}^{2}  \theta_{5}
,
\\
\alpha^{(2)}_{14} &= 
- 13 \theta_{4}^{3} 
- \tfrac{1}{8} \theta_{10}^{2} 
+ 2 \theta_{4} \theta_{12}
+ \theta_{4} \theta_{14}
- \tfrac{3}{2} \theta_{4} \theta_{15}
+ 4 \theta_{7} \theta_{5}^{2} 
+ \tfrac{105}{2} \theta_{3}^{2}  \theta_{4} \theta_{5}
+ \tfrac{93}{2} \theta_{3}^{2}  \theta_{4} \theta_{7}
- 35 \theta_{3}^{2}  \theta_{5} \theta_{7}
+ \theta_{5} \theta_{15}
+ \tfrac{13}{2} \theta_{1}^{2}  \theta_{3} \theta_{8}
+ \tfrac{3}{2} \theta_{1} \theta_{3}^{2}  \theta_{8}
\\&\ \quad
+ 3 \theta_{1}^{2}  \theta_{3} \theta_{9}
+ 28 \theta_{1}^{2}  \theta_{4} \theta_{5}
+ \tfrac{57}{2} \theta_{1}^{2}  \theta_{4} \theta_{7}
- \tfrac{35}{2} \theta_{1}^{2}  \theta_{5} \theta_{7}
- 12 \theta_{1} \theta_{3}^{2}  \theta_{9}
+ 50 \theta_{1} \theta_{3} \theta_{4}^{2} 
+ 27 \theta_{1} \theta_{3} \theta_{5}^{2} 
+ 21 \theta_{1} \theta_{3} \theta_{7}^{2} 
- \tfrac{1}{2} \theta_{5} \theta_{12}
- \tfrac{1}{2} \theta_{5} \theta_{14}
\\&\ \quad
- \tfrac{1}{2} \theta_{7} \theta_{13}
+ \tfrac{1}{2} \theta_{3}^{2}  \theta_{13}
- \tfrac{5}{8} \theta_{8}^{2} 
+ \tfrac{1}{2} \theta_{7} \theta_{15}
- 2 \theta_{1} \theta_{5} \theta_{9}
- \theta_{1} \theta_{7} \theta_{9}
+ 3 \theta_{3} \theta_{4} \theta_{9}
+ 2 \theta_{3} \theta_{5} \theta_{9}
+ \tfrac{17}{4} \theta_{1} \theta_{5} \theta_{8}
+ \tfrac{15}{4} \theta_{1} \theta_{7} \theta_{8}
+ \theta_{3} \theta_{4} \theta_{8}
\\&\ \quad
- \tfrac{15}{2} \theta_{3} \theta_{5} \theta_{8}
- \tfrac{11}{2} \theta_{3} \theta_{7} \theta_{8}
- \tfrac{1}{2} \theta_{1} \theta_{3} \theta_{13}
- \tfrac{5}{2} \theta_{3} \theta_{5} \theta_{10}
- \tfrac{1}{2} \theta_{3} \theta_{7} \theta_{10}
+ \tfrac{7}{4} \theta_{1} \theta_{5} \theta_{10}
+ \tfrac{5}{4} \theta_{1} \theta_{7} \theta_{10}
- \theta_{3} \theta_{4} \theta_{10}
- \theta_{1} \theta_{4} \theta_{10}
- 18 \theta_{4} \theta_{5} \theta_{7}
\\&\ \quad
+ \tfrac{3}{2} \theta_{1}^{2}  \theta_{14}
+ \tfrac{5}{2} \theta_{14} \theta_{3}^{2} 
- \theta_{1}^{2}  \theta_{15}
- \tfrac{1}{2} \theta_{15} \theta_{3}^{2} 
+ \tfrac{3}{2} \theta_{1}^{2}  \theta_{12}
+ 3 \theta_{12} \theta_{3}^{2} 
+ \tfrac{1}{2} \theta_{18} \theta_{1}
- \theta_{18} \theta_{3}
- \tfrac{1}{2} \theta_{16} \theta_{1}
+ \theta_{16} \theta_{3}
- \tfrac{17}{2} \theta_{4} \theta_{7}^{2} 
+ \tfrac{37}{2} \theta_{4}^{2}  \theta_{7}
- \tfrac{17}{2} \theta_{4} \theta_{5}^{2} 
\\&\ \quad
+ \tfrac{37}{2} \theta_{4}^{2}  \theta_{5}
+ \tfrac{3}{2} \theta_{1} \theta_{15} \theta_{3}
- 4 \theta_{1} \theta_{12} \theta_{3}
- 4 \theta_{1} \theta_{14} \theta_{3}
+ \tfrac{21}{2} \theta_{1} \theta_{3}^{2}  \theta_{10}
- \tfrac{3}{2} \theta_{1}^{2}  \theta_{3} \theta_{10}
- \tfrac{1}{2} \theta_{7} \theta_{12}
- \tfrac{1}{2} \theta_{7} \theta_{14}
- \tfrac{3}{4} \theta_{1}^{3}  \theta_{10}
- \tfrac{41}{4} \theta_{3}^{3}  \theta_{10}
+ \tfrac{1}{4} \theta_{8} \theta_{10}
\end{align*}
\begin{align*}
&\ \quad
+ \tfrac{1}{2} \theta_{8} \theta_{9}
- \tfrac{1}{2} \theta_{4} \theta_{13}
- 67 \theta_{1} \theta_{3} \theta_{4} \theta_{5}
- 65 \theta_{1} \theta_{3} \theta_{4} \theta_{7}
+ 44 \theta_{1} \theta_{3} \theta_{5} \theta_{7}
+ \tfrac{1}{2} \theta_{9} \theta_{10}
+ \tfrac{167}{2} \theta_{1}^{3}  \theta_{3} \theta_{4}
- 64 \theta_{1}^{3}  \theta_{3} \theta_{5}
- 72 \theta_{1}^{3}  \theta_{3} \theta_{7}
- 133 \theta_{1}^{2}  \theta_{3}^{2}  \theta_{4}
\\&\ \quad
+ \tfrac{193}{2} \theta_{1}^{2}  \theta_{3}^{2}  \theta_{5}
+ 125 \theta_{1}^{2}  \theta_{3}^{2}  \theta_{7}
+ 83 \theta_{1} \theta_{3}^{3}  \theta_{4}
- \tfrac{127}{2} \theta_{1} \theta_{3}^{3}  \theta_{5}
- \tfrac{173}{2} \theta_{1} \theta_{3}^{3}  \theta_{7}
+ \tfrac{3}{2} \theta_{7}^{3} 
- \tfrac{1}{2} \theta_{17} \theta_{1}
+ \tfrac{1}{2} \theta_{17} \theta_{3}
- \tfrac{11}{4} \theta_{1}^{3}  \theta_{8}
- \tfrac{25}{4} \theta_{3}^{3}  \theta_{8}
+ \tfrac{3}{2} \theta_{5}^{3} 
\\&\ \quad
+ 53 \theta_{1}^{5}  \theta_{3}
- \tfrac{301}{2} \theta_{1}^{4}  \theta_{3}^{2} 
+ \tfrac{471}{2} \theta_{1}^{3}  \theta_{3}^{3} 
- \tfrac{423}{2} \theta_{1}^{2}  \theta_{3}^{4} 
+ 99 \theta_{1} \theta_{3}^{5} 
- \tfrac{39}{2} \theta_{1}^{4}  \theta_{4}
+ \tfrac{31}{2} \theta_{1}^{4}  \theta_{5}
+ \tfrac{31}{2} \theta_{1}^{4}  \theta_{7}
- 16 \theta_{3}^{4}  \theta_{4}
+ 18 \theta_{3}^{4}  \theta_{5}
+ 17 \theta_{3}^{4}  \theta_{7}
\\&\ \quad
+ \tfrac{1}{2} \theta_{1}^{3}  \theta_{9}
- 24 \theta_{1}^{2}  \theta_{4}^{2} 
- 11 \theta_{1}^{2}  \theta_{5}^{2} 
- 8 \theta_{1}^{2}  \theta_{7}^{2} 
+ \tfrac{21}{2} \theta_{3}^{3}  \theta_{9}
- 36 \theta_{3}^{2}  \theta_{4}^{2} 
- \tfrac{37}{2} \theta_{3}^{2}  \theta_{5}^{2} 
- \tfrac{31}{2} \theta_{3}^{2}  \theta_{7}^{2} 
- 8 \theta_{1}^{6} 
- 18 \theta_{3}^{6} 
- \tfrac{1}{2} \theta_{9}^{2} 
- \tfrac{1}{2} \theta_{19}
+ 4 \theta_{7}^{2}  \theta_{5}
,
\\
\alpha^{(2)}_{15} &= 
- 13 \theta_{1} \theta_{5}^{3} 
- 3 \theta_{7} \theta_{9} \theta_{5}
- \theta_{1} \theta_{20}
+ \tfrac{3}{2} \theta_{4} \theta_{16}
- \tfrac{3}{2} \theta_{4} \theta_{17}
- \tfrac{3}{2} \theta_{4} \theta_{18}
- 17 \theta_{1} \theta_{7}^{2}  \theta_{5}
+ \theta_{9} \theta_{3} \theta_{10}
+ \tfrac{13}{2} \theta_{3} \theta_{5}^{3} 
- \tfrac{107}{4} \theta_{4} \theta_{5} \theta_{8}
- \tfrac{51}{4} \theta_{4} \theta_{7} \theta_{8}
\\&\ \quad
- 30 \theta_{3} \theta_{4} \theta_{5} \theta_{7}
+ 68 \theta_{1} \theta_{4} \theta_{5} \theta_{7}
- 73 \theta_{3} \theta_{1}^{3}  \theta_{9}
+ 64 \theta_{3} \theta_{1}^{2}  \theta_{4}^{2} 
+ \tfrac{155}{2} \theta_{3} \theta_{1}^{2}  \theta_{5}^{2} 
+ \tfrac{79}{2} \theta_{3} \theta_{1}^{2}  \theta_{7}^{2} 
+ 4 \theta_{1}^{3}  \theta_{7}^{2} 
+ \tfrac{9}{2} \theta_{3} \theta_{5} \theta_{13}
- 28 \theta_{1} \theta_{7} \theta_{5}^{2} 
\\&\ \quad
- \tfrac{25}{4} \theta_{1}^{2}  \theta_{4} \theta_{10}
- 28 \theta_{3} \theta_{4} \theta_{5}^{2} 
- \tfrac{59}{4} \theta_{3}^{2}  \theta_{4} \theta_{10}
+ \tfrac{11}{2} \theta_{1} \theta_{4} \theta_{13}
+ \theta_{3} \theta_{20}
+ \tfrac{9}{2} \theta_{3} \theta_{7} \theta_{13}
- \tfrac{37}{2} \theta_{3} \theta_{1} \theta_{5} \theta_{10}
- \tfrac{37}{2} \theta_{3} \theta_{1} \theta_{7} \theta_{10}
+ \tfrac{45}{2} \theta_{3} \theta_{1} \theta_{4} \theta_{10}
\\&\ \quad
+ \tfrac{11}{2} \theta_{3} \theta_{4} \theta_{15}
- \tfrac{37}{4} \theta_{8}^{2}  \theta_{1}
- \tfrac{1}{2} \theta_{8} \theta_{14}
+ \tfrac{31}{2} \theta_{1} \theta_{3}^{2}  \theta_{13}
+ \tfrac{87}{2} \theta_{3}^{7} 
+ \tfrac{1}{2} \theta_{10} \theta_{14}
+ \tfrac{41}{2} \theta_{3}^{2}  \theta_{5} \theta_{10}
+ 9 \theta_{3}^{2}  \theta_{7} \theta_{10}
- 22 \theta_{1}^{7} 
+ \tfrac{9}{2} \theta_{15} \theta_{1}^{2}  \theta_{3}
\\&\ \quad
- \tfrac{15}{2} \theta_{15} \theta_{1} \theta_{3}^{2} 
+ \tfrac{7}{2} \theta_{15} \theta_{3}^{3} 
- \tfrac{243}{2} \theta_{3} \theta_{1}^{2}  \theta_{4} \theta_{5}
- \tfrac{197}{2} \theta_{3} \theta_{1}^{2}  \theta_{4} \theta_{7}
+ 86 \theta_{3} \theta_{1}^{2}  \theta_{5} \theta_{7}
+ \tfrac{37}{2} \theta_{1} \theta_{4} \theta_{7}^{2} 
- \tfrac{1}{2} \theta_{7} \theta_{10} \theta_{5}
+ \tfrac{23}{2} \theta_{7} \theta_{8} \theta_{5}
+ \tfrac{35}{2} \theta_{3} \theta_{7} \theta_{5}^{2} 
\\&\ \quad
- 777 \theta_{1}^{3}  \theta_{3}^{2}  \theta_{4}
+ 725 \theta_{1}^{3}  \theta_{3}^{2}  \theta_{5}
+ 558 \theta_{1}^{3}  \theta_{3}^{2}  \theta_{7}
+ 1003 \theta_{1}^{2}  \theta_{3}^{3}  \theta_{4}
- 863 \theta_{1}^{2}  \theta_{3}^{3}  \theta_{5}
- 786 \theta_{1}^{2}  \theta_{3}^{3}  \theta_{7}
- 565 \theta_{1} \theta_{3}^{4}  \theta_{4}
+ 485 \theta_{1} \theta_{3}^{4}  \theta_{5}
\\&\ \quad
+ 496 \theta_{1} \theta_{3}^{4}  \theta_{7}
- 5 \theta_{9} \theta_{5}^{2} 
- \tfrac{495}{2} \theta_{3}^{3}  \theta_{4} \theta_{5}
- \tfrac{429}{2} \theta_{3}^{3}  \theta_{4} \theta_{7}
+ 160 \theta_{3}^{3}  \theta_{5} \theta_{7}
+ 115 \theta_{1}^{2}  \theta_{3}^{2}  \theta_{9}
- 59 \theta_{1} \theta_{3}^{3}  \theta_{9}
- 235 \theta_{1} \theta_{3}^{2}  \theta_{4}^{2} 
- 162 \theta_{1} \theta_{3}^{2}  \theta_{5}^{2} 
\\&\ \quad
- 115 \theta_{1} \theta_{3}^{2}  \theta_{7}^{2} 
+ \tfrac{29}{2} \theta_{1} \theta_{3}^{3}  \theta_{10}
- \tfrac{115}{2} \theta_{1}^{2}  \theta_{3}^{2}  \theta_{10}
+ \tfrac{75}{2} \theta_{3} \theta_{1}^{3}  \theta_{10}
+ \tfrac{747}{2} \theta_{1} \theta_{3}^{2}  \theta_{4} \theta_{5}
+ \tfrac{621}{2} \theta_{1} \theta_{3}^{2}  \theta_{4} \theta_{7}
- 240 \theta_{1} \theta_{3}^{2}  \theta_{5} \theta_{7}
- \tfrac{7}{2} \theta_{3} \theta_{4} \theta_{13}
\\&\ \quad
- \tfrac{5}{2} \theta_{3} \theta_{7} \theta_{15}
+ \tfrac{505}{2} \theta_{3} \theta_{1}^{4}  \theta_{4}
- \tfrac{553}{2} \theta_{3} \theta_{1}^{4}  \theta_{5}
- \tfrac{343}{2} \theta_{3} \theta_{1}^{4}  \theta_{7}
- 2 \theta_{3} \theta_{7} \theta_{12}
+ 35 \theta_{1}^{5}  \theta_{5}
+ 14 \theta_{1}^{5}  \theta_{7}
- 3 \theta_{1}^{3}  \theta_{14}
- \theta_{9} \theta_{13}
+ 13 \theta_{8}^{2}  \theta_{3}
\\&\ \quad
- \tfrac{65}{2} \theta_{1}^{3}  \theta_{4} \theta_{5}
- \tfrac{45}{2} \theta_{1}^{3}  \theta_{4} \theta_{7}
+ \tfrac{19}{2} \theta_{1}^{2}  \theta_{4} \theta_{9}
- \tfrac{65}{4} \theta_{1}^{2}  \theta_{4} \theta_{8}
- 4 \theta_{1} \theta_{4} \theta_{14}
- \tfrac{21}{4} \theta_{10} \theta_{4} \theta_{5}
+ \tfrac{11}{4} \theta_{10} \theta_{4} \theta_{7}
+ \tfrac{115}{2} \theta_{1} \theta_{4} \theta_{5}^{2} 
+ 12 \theta_{4} \theta_{5} \theta_{9}
\\&\ \quad
+ 2 \theta_{4} \theta_{7} \theta_{9}
+ \tfrac{45}{2} \theta_{3} \theta_{4}^{2}  \theta_{5}
+ \tfrac{5}{2} \theta_{3} \theta_{4}^{2}  \theta_{7}
+ 3 \theta_{1} \theta_{7} \theta_{12}
- \tfrac{7}{2} \theta_{1} \theta_{7} \theta_{13}
+ \tfrac{1}{2} \theta_{1} \theta_{7} \theta_{15}
- \tfrac{5}{2} \theta_{3} \theta_{5} \theta_{15}
- \tfrac{145}{2} \theta_{1} \theta_{4}^{2}  \theta_{5}
- \tfrac{77}{2} \theta_{1} \theta_{4}^{2}  \theta_{7}
\\&\ \quad
+ 3 \theta_{10} \theta_{1}^{2}  \theta_{5}
+ 7 \theta_{10} \theta_{1}^{2}  \theta_{7}
+ 2 \theta_{3} \theta_{4} \theta_{14}
- 6 \theta_{14} \theta_{3}^{3} 
- 6 \theta_{12} \theta_{3}^{3} 
- 4 \theta_{1}^{3}  \theta_{12}
+ \tfrac{11}{2} \theta_{1}^{3}  \theta_{13}
- \tfrac{1}{2} \theta_{1}^{3}  \theta_{15}
+ \tfrac{7}{2} \theta_{18} \theta_{3}^{2} 
- \tfrac{7}{2} \theta_{16} \theta_{3}^{2} 
- \tfrac{29}{2} \theta_{1}^{2}  \theta_{3} \theta_{13}
\\&\ \quad
+ 4 \theta_{1} \theta_{12} \theta_{3}^{2} 
+ 7 \theta_{1} \theta_{14} \theta_{3}^{2} 
+ 2 \theta_{1} \theta_{7} \theta_{14}
+ 4 \theta_{1} \theta_{5} \theta_{14}
+ 12 \theta_{1}^{3}  \theta_{5} \theta_{7}
+ 5 \theta_{1} \theta_{5} \theta_{12}
- \tfrac{11}{2} \theta_{1} \theta_{5} \theta_{13}
+ \tfrac{1}{2} \theta_{1} \theta_{5} \theta_{15}
- \tfrac{149}{2} \theta_{1} \theta_{7} \theta_{3} \theta_{8}
\\&\ \quad
+ \tfrac{21}{2} \theta_{3} \theta_{7}^{2}  \theta_{5}
- \tfrac{179}{4} \theta_{3}^{4}  \theta_{8}
- 2 \theta_{9}^{2}  \theta_{1}
- \tfrac{1}{2} \theta_{8} \theta_{10} \theta_{1}
+ \tfrac{49}{2} \theta_{8} \theta_{1}^{2}  \theta_{5}
+ \tfrac{33}{2} \theta_{8} \theta_{1}^{2}  \theta_{7}
+ \tfrac{31}{2} \theta_{9} \theta_{1}^{4} 
- 666 \theta_{1}^{5}  \theta_{3}^{2} 
+ \tfrac{2515}{2} \theta_{1}^{4}  \theta_{3}^{3} 
- 1409 \theta_{1}^{3}  \theta_{3}^{4} 
\\&\ \quad
+ \tfrac{1871}{2} \theta_{1}^{2}  \theta_{3}^{5} 
- 328 \theta_{1} \theta_{3}^{6} 
+ \tfrac{221}{2} \theta_{3}^{5}  \theta_{4}
- \tfrac{221}{2} \theta_{3}^{5}  \theta_{5}
- \tfrac{217}{2} \theta_{3}^{5}  \theta_{7}
- \tfrac{5}{2} \theta_{3}^{4}  \theta_{9}
+ 165 \theta_{3}^{3}  \theta_{4}^{2} 
+ \tfrac{179}{2} \theta_{3}^{3}  \theta_{5}^{2} 
+ \tfrac{157}{2} \theta_{3}^{3}  \theta_{7}^{2} 
+ \tfrac{381}{2} \theta_{3} \theta_{1}^{6} 
+ 2 \theta_{3} \theta_{9}^{2} 
\\&\ \quad
+ \tfrac{161}{2} \theta_{1} \theta_{3} \theta_{4} \theta_{8}
- \tfrac{177}{2} \theta_{1} \theta_{5} \theta_{3} \theta_{8}
+ \tfrac{73}{4} \theta_{3}^{4}  \theta_{10}
- \tfrac{29}{4} \theta_{10} \theta_{1}^{4} 
- \tfrac{1}{4} \theta_{10}^{2}  \theta_{1}
- 4 \theta_{8} \theta_{3} \theta_{10}
- 4 \theta_{3} \theta_{4} \theta_{7}^{2} 
- 4 \theta_{1} \theta_{4} \theta_{12}
- \tfrac{3}{2} \theta_{1} \theta_{4} \theta_{15}
+ 2 \theta_{8} \theta_{13}
\\&\ \quad
- 4 \theta_{3} \theta_{5} \theta_{12}
+ \tfrac{35}{2} \theta_{4}^{2}  \theta_{8}
+ \tfrac{1}{2} \theta_{10} \theta_{4}^{2} 
- 6 \theta_{4}^{2}  \theta_{9}
+ 27 \theta_{1} \theta_{4}^{3} 
+ 3 \theta_{3} \theta_{4}^{3} 
- 2 \theta_{1} \theta_{7}^{3} 
- \tfrac{1}{2} \theta_{3} \theta_{7}^{3} 
+ \tfrac{7}{4} \theta_{7}^{2}  \theta_{8}
- \tfrac{3}{4} \theta_{10} \theta_{7}^{2} 
- \theta_{3} \theta_{10}^{2} 
- \theta_{8} \theta_{12}
\\&\ \quad
- \tfrac{1}{2} \theta_{8} \theta_{15}
- \tfrac{13}{2} \theta_{3}^{3}  \theta_{13}
+ \tfrac{35}{4} \theta_{8} \theta_{5}^{2} 
+ \tfrac{53}{2} \theta_{3}^{2}  \theta_{4} \theta_{9}
- \tfrac{63}{2} \theta_{3}^{2}  \theta_{5} \theta_{9}
- \tfrac{41}{2} \theta_{3}^{2}  \theta_{7} \theta_{9}
- 37 \theta_{3} \theta_{1} \theta_{4} \theta_{9}
+ 37 \theta_{3} \theta_{1} \theta_{5} \theta_{9}
+ 31 \theta_{3} \theta_{1} \theta_{7} \theta_{9}
\\&\ \quad
- \theta_{3} \theta_{7} \theta_{14}
- 3 \theta_{3} \theta_{5} \theta_{14}
\tfrac{1}{2} \theta_{1}^{2}  \theta_{16}
- \tfrac{1}{2} \theta_{1}^{2}  \theta_{18}
- \tfrac{5}{2} \theta_{1}^{2}  \theta_{17}
- \tfrac{5}{2} \theta_{17} \theta_{3}^{2} 
- \theta_{19} \theta_{1}
+ 2 \theta_{19} \theta_{3}
+ 2 \theta_{3} \theta_{1}^{2}  \theta_{14}
+ 5 \theta_{3} \theta_{1}^{2}  \theta_{12}
- 2 \theta_{3} \theta_{18} \theta_{1}
\\&\ \quad
+ 2 \theta_{3} \theta_{16} \theta_{1}
+ 5 \theta_{1} \theta_{17} \theta_{3}
+ \tfrac{3}{2} \theta_{5} \theta_{17}
+ \tfrac{3}{2} \theta_{5} \theta_{18}
- \tfrac{3}{2} \theta_{5} \theta_{16}
- \tfrac{1}{2} \theta_{7} \theta_{16}
+ \tfrac{1}{2} \theta_{7} \theta_{17}
+ \tfrac{1}{2} \theta_{7} \theta_{18}
- 20 \theta_{1}^{5}  \theta_{4}
+ 29 \theta_{1}^{3}  \theta_{4}^{2} 
- \tfrac{133}{4} \theta_{8} \theta_{1}^{4} 
\\&\ \quad
+ \tfrac{319}{2} \theta_{1}^{3}  \theta_{3} \theta_{8}
- \tfrac{547}{2} \theta_{1}^{2}  \theta_{3}^{2}  \theta_{8}
+ \tfrac{385}{2} \theta_{1} \theta_{3}^{3}  \theta_{8}
+ 72 \theta_{3}^{2}  \theta_{5} \theta_{8}
+ \tfrac{133}{2} \theta_{3}^{2}  \theta_{7} \theta_{8}
- \tfrac{291}{4} \theta_{3}^{2}  \theta_{4} \theta_{8}
- 7 \theta_{9} \theta_{3} \theta_{8}
+ \theta_{9} \theta_{10} \theta_{1}
- \tfrac{19}{2} \theta_{9} \theta_{1}^{2}  \theta_{5}
\\&\ \quad
- \tfrac{19}{2} \theta_{9} \theta_{1}^{2}  \theta_{7}
+ \tfrac{13}{4} \theta_{10} \theta_{5}^{2} 
+ \tfrac{1}{2} \theta_{21}
+ 7 \theta_{9} \theta_{8} \theta_{1}
+ \theta_{10} \theta_{12}
- \tfrac{1}{2} \theta_{10} \theta_{15}
,
\\
\bar{\alpha}^{(2)}_{3} &= 
- \theta_{1}^{2} 
+ 4 \theta_{1} \theta_{3}
- 4 \theta_{3}^{2} 
- \theta_{4}
+ \theta_{7}
,
\\
\bar{\alpha}^{(2)}_{4} &=  2 \theta_{1} \theta_{3}
- 2 \theta_{3}^{2} 
+ \theta_{4}
- \theta_{5}
,
\\
\bar{\alpha}^{(2)}_{6} &= 
- 2 \theta_{1}^{3} 
+ 7 \theta_{1}^{2}  \theta_{3}
- 11 \theta_{1} \theta_{3}^{2} 
+ 7 \theta_{3}^{3} 
- 3 \theta_{1} \theta_{4}
+ 2 \theta_{1} \theta_{5}
+ \theta_{1} \theta_{7}
+ 5 \theta_{3} \theta_{4}
- 2 \theta_{3} \theta_{5}
- 3 \theta_{3} \theta_{7}
- \theta_{10}
- \theta_{8}
+ \theta_{9}
,
\\
\bar{\alpha}^{(2)}_{8} &=  \theta_{1}^{4} 
- 6 \theta_{1}^{3}  \theta_{3}
+ 14 \theta_{1}^{2}  \theta_{3}^{2} 
- 16 \theta_{1} \theta_{3}^{3} 
+ 8 \theta_{3}^{4} 
+ 3 \theta_{1}^{2}  \theta_{4}
- \theta_{1}^{2}  \theta_{5}
- 2 \theta_{1}^{2}  \theta_{7}
- 10 \theta_{1} \theta_{3} \theta_{4}
+ 4 \theta_{1} \theta_{3} \theta_{5}
+ 6 \theta_{1} \theta_{3} \theta_{7}
+ 10 \theta_{3}^{2}  \theta_{4}
- 4 \theta_{3}^{2}  \theta_{5}
\\&\ \quad
- 6 \theta_{3}^{2}  \theta_{7}
+ 2 \theta_{4}^{2} 
- \theta_{4} \theta_{5}
- 3 \theta_{4} \theta_{7}
+ \theta_{5} \theta_{7}
+ \theta_{7}^{2} 
,
\\
\bar{\alpha}^{(2)}_{11} &=  \tfrac{7}{2} \theta_{10} \theta_{5}
+ \tfrac{13}{2} \theta_{8} \theta_{5}
- 10 \theta_{1} \theta_{5}^{2} 
+ 7 \theta_{3} \theta_{5}^{2} 
- 14 \theta_{1} \theta_{7} \theta_{5}
+ 14 \theta_{3} \theta_{7} \theta_{5}
- 4 \theta_{9} \theta_{5}
- 40 \theta_{7} \theta_{1}^{2}  \theta_{3}
- 28 \theta_{1}^{3}  \theta_{4}
- 62 \theta_{3}^{2}  \theta_{1} \theta_{4}
+ 28 \theta_{3}^{2}  \theta_{1} \theta_{5}
\\&\ \quad
+ 42 \theta_{3}^{2}  \theta_{1} \theta_{7}
- 34 \theta_{5} \theta_{1}^{2}  \theta_{3}
+ 29 \theta_{3}^{5} 
- 86 \theta_{3}^{2}  \theta_{1}^{3} 
+ 118 \theta_{3}^{3}  \theta_{1}^{2} 
- 90 \theta_{3}^{4}  \theta_{1}
+ 20 \theta_{3}^{3}  \theta_{4}
- 8 \theta_{3}^{3}  \theta_{5}
- 16 \theta_{3}^{3}  \theta_{7}
- 12 \theta_{4} \theta_{8}
- 6 \theta_{10} \theta_{4}
\\&\ \quad
- 22 \theta_{1} \theta_{4}^{2} 
+ 22 \theta_{3} \theta_{4}^{2} 
+ 6 \theta_{4} \theta_{9}
+ \tfrac{5}{2} \theta_{10} \theta_{7}
+ \tfrac{11}{2} \theta_{7} \theta_{8}
- 2 \theta_{7} \theta_{9}
- 4 \theta_{1} \theta_{7}^{2} 
+ 7 \theta_{3} \theta_{7}^{2} 
- \theta_{13} \theta_{1}
+ \theta_{13} \theta_{3}
+ 2 \theta_{12} \theta_{1}
- 2 \theta_{12} \theta_{3}
\\&\ \quad
+ 2 \theta_{14} \theta_{1}
- 2 \theta_{14} \theta_{3}
- \theta_{15} \theta_{1}
+ \theta_{15} \theta_{3}
- \tfrac{5}{2} \theta_{1}^{2}  \theta_{8}
+ \tfrac{1}{2} \theta_{3}^{2}  \theta_{8}
+ \theta_{1}^{2}  \theta_{9}
- 5 \theta_{3}^{2}  \theta_{9}
- \tfrac{5}{2} \theta_{1}^{2}  \theta_{10}
+ \tfrac{5}{2} \theta_{10} \theta_{3}^{2} 
+ 4 \theta_{1} \theta_{3} \theta_{9}
+ 31 \theta_{1} \theta_{4} \theta_{5}
\\&\ \quad
+ 19 \theta_{1} \theta_{4} \theta_{7}
- 25 \theta_{3} \theta_{4} \theta_{5}
- 25 \theta_{3} \theta_{4} \theta_{7}
+ 37 \theta_{1}^{4}  \theta_{3}
+ 16 \theta_{1}^{3}  \theta_{5}
+ 14 \theta_{1}^{3}  \theta_{7}
- 8 \theta_{1}^{5} 
+ 68 \theta_{4} \theta_{1}^{2}  \theta_{3}
- \theta_{16}
+ \theta_{17}
+ \theta_{18}
,
\\
\bar{\alpha}^{(2)}_{13} &=  13 \theta_{4}^{3} 
+ \tfrac{1}{8} \theta_{10}^{2} 
- 2 \theta_{4} \theta_{12}
- \theta_{4} \theta_{14}
+ \tfrac{3}{2} \theta_{4} \theta_{15}
- 4 \theta_{7} \theta_{5}^{2} 
- \tfrac{105}{2} \theta_{3}^{2}  \theta_{4} \theta_{5}
- \tfrac{93}{2} \theta_{3}^{2}  \theta_{4} \theta_{7}
+ 35 \theta_{3}^{2}  \theta_{5} \theta_{7}
- \theta_{5} \theta_{15}
- \tfrac{13}{2} \theta_{1}^{2}  \theta_{3} \theta_{8}
- \tfrac{3}{2} \theta_{1} \theta_{3}^{2}  \theta_{8}
\\&\ \quad
- 3 \theta_{1}^{2}  \theta_{3} \theta_{9}
- 28 \theta_{1}^{2}  \theta_{4} \theta_{5}
- \tfrac{57}{2} \theta_{1}^{2}  \theta_{4} \theta_{7}
+ \tfrac{35}{2} \theta_{1}^{2}  \theta_{5} \theta_{7}
+ 12 \theta_{1} \theta_{3}^{2}  \theta_{9}
- 50 \theta_{1} \theta_{3} \theta_{4}^{2} 
- 27 \theta_{1} \theta_{3} \theta_{5}^{2} 
- 21 \theta_{1} \theta_{3} \theta_{7}^{2} 
+ \tfrac{1}{2} \theta_{5} \theta_{12}
+ \tfrac{1}{2} \theta_{5} \theta_{14}
\\&\ \quad
+ \tfrac{1}{2} \theta_{7} \theta_{13}
- \tfrac{1}{2} \theta_{3}^{2}  \theta_{13}
+ \tfrac{5}{8} \theta_{8}^{2} 
- \tfrac{1}{2} \theta_{7} \theta_{15}
+ 2 \theta_{1} \theta_{5} \theta_{9}
+ \theta_{1} \theta_{7} \theta_{9}
- 3 \theta_{3} \theta_{4} \theta_{9}
- 2 \theta_{3} \theta_{5} \theta_{9}
- \tfrac{17}{4} \theta_{1} \theta_{5} \theta_{8}
- \tfrac{15}{4} \theta_{1} \theta_{7} \theta_{8}
- \theta_{3} \theta_{4} \theta_{8}
\\&\ \quad
+ \tfrac{15}{2} \theta_{3} \theta_{5} \theta_{8}
+ \tfrac{11}{2} \theta_{3} \theta_{7} \theta_{8}
+ \tfrac{1}{2} \theta_{1} \theta_{3} \theta_{13}
+ \tfrac{5}{2} \theta_{3} \theta_{5} \theta_{10}
+ \tfrac{1}{2} \theta_{3} \theta_{7} \theta_{10}
- \tfrac{7}{4} \theta_{1} \theta_{5} \theta_{10}
- \tfrac{5}{4} \theta_{1} \theta_{7} \theta_{10}
+ \theta_{3} \theta_{4} \theta_{10}
+ \theta_{1} \theta_{4} \theta_{10}
+ 18 \theta_{4} \theta_{5} \theta_{7}
\\&\ \quad
- \tfrac{3}{2} \theta_{1}^{2}  \theta_{14}
- \tfrac{5}{2} \theta_{14} \theta_{3}^{2} 
+ \theta_{1}^{2}  \theta_{15}
+ \tfrac{1}{2} \theta_{15} \theta_{3}^{2} 
- \tfrac{3}{2} \theta_{1}^{2}  \theta_{12}
- 3 \theta_{12} \theta_{3}^{2} 
- \tfrac{1}{2} \theta_{18} \theta_{1}
+ \theta_{18} \theta_{3}
+ \tfrac{1}{2} \theta_{16} \theta_{1}
- \theta_{16} \theta_{3}
+ \tfrac{17}{2} \theta_{4} \theta_{7}^{2} 
- \tfrac{37}{2} \theta_{4}^{2}  \theta_{7}
+ \tfrac{17}{2} \theta_{4} \theta_{5}^{2} 
\\&\ \quad
- \tfrac{37}{2} \theta_{4}^{2}  \theta_{5}
- \tfrac{3}{2} \theta_{1} \theta_{15} \theta_{3}
+ 4 \theta_{1} \theta_{12} \theta_{3}
+ 4 \theta_{1} \theta_{14} \theta_{3}
- \tfrac{21}{2} \theta_{1} \theta_{3}^{2}  \theta_{10}
+ \tfrac{3}{2} \theta_{1}^{2}  \theta_{3} \theta_{10}
+ \tfrac{1}{2} \theta_{7} \theta_{12}
+ \tfrac{1}{2} \theta_{7} \theta_{14}
+ \tfrac{3}{4} \theta_{1}^{3}  \theta_{10}
+ \tfrac{41}{4} \theta_{3}^{3}  \theta_{10}
- \tfrac{1}{4} \theta_{8} \theta_{10}
\\&\ \quad
- \tfrac{1}{2} \theta_{8} \theta_{9}
+ \tfrac{1}{2} \theta_{4} \theta_{13}
+ 67 \theta_{1} \theta_{3} \theta_{4} \theta_{5}
+ 65 \theta_{1} \theta_{3} \theta_{4} \theta_{7}
- 44 \theta_{1} \theta_{3} \theta_{5} \theta_{7}
- \tfrac{1}{2} \theta_{9} \theta_{10}
- \tfrac{167}{2} \theta_{1}^{3}  \theta_{3} \theta_{4}
+ 64 \theta_{1}^{3}  \theta_{3} \theta_{5}
+ 72 \theta_{1}^{3}  \theta_{3} \theta_{7}
+ 133 \theta_{1}^{2}  \theta_{3}^{2}  \theta_{4}
\\&\ \quad
- \tfrac{193}{2} \theta_{1}^{2}  \theta_{3}^{2}  \theta_{5}
- 125 \theta_{1}^{2}  \theta_{3}^{2}  \theta_{7}
- 83 \theta_{1} \theta_{3}^{3}  \theta_{4}
+ \tfrac{127}{2} \theta_{1} \theta_{3}^{3}  \theta_{5}
+ \tfrac{173}{2} \theta_{1} \theta_{3}^{3}  \theta_{7}
- \tfrac{3}{2} \theta_{7}^{3} 
+ \tfrac{1}{2} \theta_{17} \theta_{1}
- \tfrac{1}{2} \theta_{17} \theta_{3}
+ \tfrac{11}{4} \theta_{1}^{3}  \theta_{8}
+ \tfrac{25}{4} \theta_{3}^{3}  \theta_{8}
- \tfrac{3}{2} \theta_{5}^{3} 
\\&\ \quad
- 53 \theta_{1}^{5}  \theta_{3}
+ \tfrac{301}{2} \theta_{1}^{4}  \theta_{3}^{2} 
- \tfrac{471}{2} \theta_{1}^{3}  \theta_{3}^{3} 
+ \tfrac{423}{2} \theta_{1}^{2}  \theta_{3}^{4} 
- 99 \theta_{1} \theta_{3}^{5} 
+ \tfrac{39}{2} \theta_{1}^{4}  \theta_{4}
- \tfrac{31}{2} \theta_{1}^{4}  \theta_{5}
- \tfrac{31}{2} \theta_{1}^{4}  \theta_{7}
+ 16 \theta_{3}^{4}  \theta_{4}
- 18 \theta_{3}^{4}  \theta_{5}
- 17 \theta_{3}^{4}  \theta_{7}
\\&\ \quad
- \tfrac{1}{2} \theta_{1}^{3}  \theta_{9}
+ 24 \theta_{1}^{2}  \theta_{4}^{2} 
+ 11 \theta_{1}^{2}  \theta_{5}^{2} 
+ 8 \theta_{1}^{2}  \theta_{7}^{2} 
- \tfrac{21}{2} \theta_{3}^{3}  \theta_{9}
+ 36 \theta_{3}^{2}  \theta_{4}^{2} 
+ \tfrac{37}{2} \theta_{3}^{2}  \theta_{5}^{2} 
+ \tfrac{31}{2} \theta_{3}^{2}  \theta_{7}^{2} 
+ 8 \theta_{1}^{6} 
+ 18 \theta_{3}^{6} 
+ \tfrac{1}{2} \theta_{9}^{2} 
+ \tfrac{1}{2} \theta_{19}
- 4 \theta_{7}^{2}  \theta_{5}
,
\\
\bar{\alpha}^{(2)}_{15} &=  13 \theta_{1} \theta_{5}^{3} 
- \theta_{7} \theta_{9} \theta_{5}
+ \theta_{1} \theta_{20}
- \tfrac{1}{2} \theta_{4} \theta_{16}
+ \tfrac{1}{2} \theta_{4} \theta_{17}
+ \tfrac{1}{2} \theta_{4} \theta_{18}
+ 3 \theta_{1} \theta_{7}^{2}  \theta_{5}
- \theta_{9} \theta_{3} \theta_{10}
- \tfrac{13}{2} \theta_{3} \theta_{5}^{3} 
+ \tfrac{81}{4} \theta_{4} \theta_{5} \theta_{8}
- \tfrac{19}{4} \theta_{4} \theta_{7} \theta_{8}
\\&\ \quad
- 9 \theta_{3} \theta_{4} \theta_{5} \theta_{7}
- 23 \theta_{1} \theta_{4} \theta_{5} \theta_{7}
+ 73 \theta_{3} \theta_{1}^{3}  \theta_{9}
- 242 \theta_{3} \theta_{1}^{2}  \theta_{4}^{2} 
- \tfrac{249}{2} \theta_{3} \theta_{1}^{2}  \theta_{5}^{2} 
- \tfrac{205}{2} \theta_{3} \theta_{1}^{2}  \theta_{7}^{2} 
+ 10 \theta_{1}^{3}  \theta_{5}^{2} 
+ 14 \theta_{1}^{3}  \theta_{7}^{2} 
- \tfrac{9}{2} \theta_{3} \theta_{5} \theta_{13}
\\&\ \quad
+ 18 \theta_{1} \theta_{7} \theta_{5}^{2} 
+ \tfrac{59}{4} \theta_{1}^{2}  \theta_{4} \theta_{10}
+ 21 \theta_{3} \theta_{4} \theta_{5}^{2} 
+ \tfrac{145}{4} \theta_{3}^{2}  \theta_{4} \theta_{10}
- \tfrac{9}{2} \theta_{1} \theta_{4} \theta_{13}
- \theta_{3} \theta_{20}
- \tfrac{7}{2} \theta_{3} \theta_{7} \theta_{13}
+ \tfrac{65}{2} \theta_{3} \theta_{1} \theta_{5} \theta_{10}
+ \tfrac{57}{2} \theta_{3} \theta_{1} \theta_{7} \theta_{10}
\\&\ \quad
- \tfrac{93}{2} \theta_{3} \theta_{1} \theta_{4} \theta_{10}
+ 2 \theta_{3} \theta_{4} \theta_{12}
- \tfrac{13}{2} \theta_{3} \theta_{4} \theta_{15}
+ \tfrac{37}{4} \theta_{8}^{2}  \theta_{1}
+ \tfrac{1}{2} \theta_{8} \theta_{14}
- \tfrac{15}{2} \theta_{1} \theta_{3}^{2}  \theta_{13}
- \tfrac{319}{2} \theta_{3}^{7} 
- \tfrac{1}{2} \theta_{10} \theta_{14}
- \tfrac{69}{2} \theta_{3}^{2}  \theta_{5} \theta_{10}
- \tfrac{33}{2} \theta_{3}^{2}  \theta_{7} \theta_{10}
\\&\ \quad
+ 30 \theta_{1}^{7} 
- \tfrac{19}{2} \theta_{15} \theta_{1}^{2}  \theta_{3}
+ \tfrac{31}{2} \theta_{15} \theta_{1} \theta_{3}^{2} 
- \tfrac{15}{2} \theta_{15} \theta_{3}^{3} 
+ \tfrac{609}{2} \theta_{3} \theta_{1}^{2}  \theta_{4} \theta_{5}
+ \tfrac{615}{2} \theta_{3} \theta_{1}^{2}  \theta_{4} \theta_{7}
- 190 \theta_{3} \theta_{1}^{2}  \theta_{5} \theta_{7}
+ \tfrac{9}{2} \theta_{1} \theta_{4} \theta_{7}^{2} 
+ 4 \theta_{7} \theta_{10} \theta_{5}
\\&\ \quad
- 5 \theta_{7} \theta_{8} \theta_{5}
- \tfrac{21}{2} \theta_{3} \theta_{7} \theta_{5}^{2} 
+ 1309 \theta_{1}^{3}  \theta_{3}^{2}  \theta_{4}
- 953 \theta_{1}^{3}  \theta_{3}^{2}  \theta_{5}
- 902 \theta_{1}^{3}  \theta_{3}^{2}  \theta_{7}
- 1661 \theta_{1}^{2}  \theta_{3}^{3}  \theta_{4}
+ 1119 \theta_{1}^{2}  \theta_{3}^{3}  \theta_{5}
+ 1248 \theta_{1}^{2}  \theta_{3}^{3}  \theta_{7}
\\&\ \quad
+ 983 \theta_{1} \theta_{3}^{4}  \theta_{4}
- 629 \theta_{1} \theta_{3}^{4}  \theta_{5}
- 818 \theta_{1} \theta_{3}^{4}  \theta_{7}
+ 5 \theta_{9} \theta_{5}^{2} 
+ \tfrac{711}{2} \theta_{3}^{3}  \theta_{4} \theta_{5}
+ \tfrac{701}{2} \theta_{3}^{3}  \theta_{4} \theta_{7}
- 224 \theta_{3}^{3}  \theta_{5} \theta_{7}
- 98 \theta_{1}^{2}  \theta_{3}^{2}  \theta_{9}
+ 23 \theta_{1} \theta_{3}^{3}  \theta_{9}
\\&\ \quad
+ 473 \theta_{1} \theta_{3}^{2}  \theta_{4}^{2} 
+ 230 \theta_{1} \theta_{3}^{2}  \theta_{5}^{2} 
+ 201 \theta_{1} \theta_{3}^{2}  \theta_{7}^{2} 
- \tfrac{9}{2} \theta_{1} \theta_{3}^{3}  \theta_{10}
+ 65 \theta_{1}^{2}  \theta_{3}^{2}  \theta_{10}
- \tfrac{95}{2} \theta_{3} \theta_{1}^{3}  \theta_{10}
- \tfrac{1251}{2} \theta_{1} \theta_{3}^{2}  \theta_{4} \theta_{5}
- \tfrac{1181}{2} \theta_{1} \theta_{3}^{2}  \theta_{4} \theta_{7}
\\&\ \quad
+ 380 \theta_{1} \theta_{3}^{2}  \theta_{5} \theta_{7}
+ \tfrac{5}{2} \theta_{3} \theta_{4} \theta_{13}
+ \tfrac{7}{2} \theta_{3} \theta_{7} \theta_{15}
- \tfrac{939}{2} \theta_{3} \theta_{1}^{4}  \theta_{4}
+ \tfrac{749}{2} \theta_{3} \theta_{1}^{4}  \theta_{5}
+ \tfrac{609}{2} \theta_{3} \theta_{1}^{4}  \theta_{7}
- 51 \theta_{1}^{5}  \theta_{5}
- 36 \theta_{1}^{5}  \theta_{7}
+ \theta_{1}^{3}  \theta_{14}
+ \theta_{9} \theta_{13}
\\&\ \quad
- 13 \theta_{8}^{2}  \theta_{3}
- \tfrac{29}{2} \theta_{1}^{3}  \theta_{4} \theta_{5}
- \tfrac{77}{2} \theta_{1}^{3}  \theta_{4} \theta_{7}
- \tfrac{33}{2} \theta_{1}^{2}  \theta_{4} \theta_{9}
+ \tfrac{123}{4} \theta_{1}^{2}  \theta_{4} \theta_{8}
+ 2 \theta_{1} \theta_{4} \theta_{14}
+ \tfrac{7}{4} \theta_{10} \theta_{4} \theta_{5}
- \tfrac{45}{4} \theta_{10} \theta_{4} \theta_{7}
- \tfrac{95}{2} \theta_{1} \theta_{4} \theta_{5}^{2} 
- 8 \theta_{4} \theta_{5} \theta_{9}
\\&\ \quad
+ 6 \theta_{4} \theta_{7} \theta_{9}
+ \tfrac{5}{2} \theta_{3} \theta_{4}^{2}  \theta_{5}
+ \tfrac{89}{2} \theta_{3} \theta_{4}^{2}  \theta_{7}
- \theta_{1} \theta_{7} \theta_{12}
+ \tfrac{5}{2} \theta_{1} \theta_{7} \theta_{13}
- \tfrac{3}{2} \theta_{1} \theta_{7} \theta_{15}
+ \tfrac{5}{2} \theta_{3} \theta_{5} \theta_{15}
+ \tfrac{83}{2} \theta_{1} \theta_{4}^{2}  \theta_{5}
- \tfrac{5}{2} \theta_{1} \theta_{4}^{2}  \theta_{7}
- \tfrac{13}{2} \theta_{10} \theta_{1}^{2}  \theta_{5}
\\&\ \quad
- 12 \theta_{10} \theta_{1}^{2}  \theta_{7}
+ 14 \theta_{14} \theta_{3}^{3} 
+ 14 \theta_{12} \theta_{3}^{3} 
+ 2 \theta_{1}^{3}  \theta_{12}
- \tfrac{9}{2} \theta_{1}^{3}  \theta_{13}
+ \tfrac{3}{2} \theta_{1}^{3}  \theta_{15}
- \tfrac{15}{2} \theta_{18} \theta_{3}^{2} 
+ \tfrac{15}{2} \theta_{16} \theta_{3}^{2} 
+ \tfrac{19}{2} \theta_{1}^{2}  \theta_{3} \theta_{13}
- 20 \theta_{1} \theta_{12} \theta_{3}^{2} 
\\&\ \quad
- 23 \theta_{1} \theta_{14} \theta_{3}^{2} 
- 4 \theta_{1} \theta_{5} \theta_{14}
+ 18 \theta_{1}^{3}  \theta_{5} \theta_{7}
- 5 \theta_{1} \theta_{5} \theta_{12}
+ \tfrac{11}{2} \theta_{1} \theta_{5} \theta_{13}
- \tfrac{1}{2} \theta_{1} \theta_{5} \theta_{15}
+ \tfrac{193}{2} \theta_{1} \theta_{7} \theta_{3} \theta_{8}
+ \tfrac{7}{2} \theta_{3} \theta_{7}^{2}  \theta_{5}
+ \tfrac{171}{4} \theta_{3}^{4}  \theta_{8}
+ 2 \theta_{9}^{2}  \theta_{1}
\\&\ \quad
+ \tfrac{1}{2} \theta_{8} \theta_{10} \theta_{1}
- 31 \theta_{8} \theta_{1}^{2}  \theta_{5}
- \tfrac{49}{2} \theta_{8} \theta_{1}^{2}  \theta_{7}
- \tfrac{33}{2} \theta_{9} \theta_{1}^{4} 
+ 932 \theta_{1}^{5}  \theta_{3}^{2} 
- \tfrac{3735}{2} \theta_{1}^{4}  \theta_{3}^{3} 
+ 2315 \theta_{1}^{3}  \theta_{3}^{4} 
- \tfrac{3593}{2} \theta_{1}^{2}  \theta_{3}^{5} 
+ 804 \theta_{1} \theta_{3}^{6} 
- \tfrac{439}{2} \theta_{3}^{5}  \theta_{4}
\\&\ \quad
+ \tfrac{285}{2} \theta_{3}^{5}  \theta_{5}
+ \tfrac{403}{2} \theta_{3}^{5}  \theta_{7}
+ \tfrac{45}{2} \theta_{3}^{4}  \theta_{9}
- 273 \theta_{3}^{3}  \theta_{4}^{2} 
- \tfrac{235}{2} \theta_{3}^{3}  \theta_{5}^{2} 
- \tfrac{245}{2} \theta_{3}^{3}  \theta_{7}^{2} 
- \tfrac{519}{2} \theta_{3} \theta_{1}^{6} 
- 2 \theta_{3} \theta_{9}^{2} 
- \tfrac{257}{2} \theta_{1} \theta_{3} \theta_{4} \theta_{8}
+ \tfrac{229}{2} \theta_{1} \theta_{5} \theta_{3} \theta_{8}
\\&\ \quad
- \tfrac{113}{4} \theta_{3}^{4}  \theta_{10}
+ \tfrac{39}{4} \theta_{10} \theta_{1}^{4} 
+ \tfrac{1}{4} \theta_{10}^{2}  \theta_{1}
+ 4 \theta_{8} \theta_{3} \theta_{10}
- 28 \theta_{3} \theta_{4} \theta_{7}^{2} 
+ 2 \theta_{1} \theta_{4} \theta_{12}
+ \tfrac{5}{2} \theta_{1} \theta_{4} \theta_{15}
- 2 \theta_{8} \theta_{13}
+ 4 \theta_{3} \theta_{5} \theta_{12}
- \tfrac{11}{2} \theta_{4}^{2}  \theta_{8}
\\&\ \quad
+ \tfrac{11}{2} \theta_{10} \theta_{4}^{2} 
- 5 \theta_{1} \theta_{4}^{3} 
- 25 \theta_{3} \theta_{4}^{3} 
- 2 \theta_{1} \theta_{7}^{3} 
+ \tfrac{15}{2} \theta_{3} \theta_{7}^{3} 
- 2 \theta_{7}^{2}  \theta_{9}
+ \tfrac{15}{4} \theta_{7}^{2}  \theta_{8}
+ \tfrac{13}{4} \theta_{10} \theta_{7}^{2} 
+ \theta_{3} \theta_{10}^{2} 
+ \theta_{8} \theta_{12}
+ \tfrac{1}{2} \theta_{8} \theta_{15}
+ \tfrac{5}{2} \theta_{3}^{3}  \theta_{13}
\\
&\ \quad
- \tfrac{35}{4} \theta_{8} \theta_{5}^{2} 
- \tfrac{91}{2} \theta_{3}^{2}  \theta_{4} \theta_{9}
+ \tfrac{95}{2} \theta_{3}^{2}  \theta_{5} \theta_{9}
+ \tfrac{47}{2} \theta_{3}^{2}  \theta_{7} \theta_{9}
+ 57 \theta_{3} \theta_{1} \theta_{4} \theta_{9}
- 53 \theta_{3} \theta_{1} \theta_{5} \theta_{9}
- 35 \theta_{3} \theta_{1} \theta_{7} \theta_{9}
- \theta_{3} \theta_{7} \theta_{14}
+ 3 \theta_{3} \theta_{5} \theta_{14}
\end{align*}
\begin{align*}
&\ \quad
+ \tfrac{1}{2} \theta_{1}^{2}  \theta_{16}
- \tfrac{1}{2} \theta_{1}^{2}  \theta_{18}
+ \tfrac{3}{2} \theta_{1}^{2}  \theta_{17}
- \tfrac{3}{2} \theta_{17} \theta_{3}^{2} 
+ \theta_{19} \theta_{1}
- 2 \theta_{19} \theta_{3}
+ 8 \theta_{3} \theta_{1}^{2}  \theta_{14}
+ 5 \theta_{3} \theta_{1}^{2}  \theta_{12}
+ 6 \theta_{3} \theta_{18} \theta_{1}
- 6 \theta_{3} \theta_{16} \theta_{1}
- \theta_{1} \theta_{17} \theta_{3}
\\&\ \quad
- \tfrac{3}{2} \theta_{5} \theta_{17}
- \tfrac{3}{2} \theta_{5} \theta_{18}
+ \tfrac{3}{2} \theta_{5} \theta_{16}
- \tfrac{1}{2} \theta_{7} \theta_{16}
+ \tfrac{1}{2} \theta_{7} \theta_{17}
+ \tfrac{1}{2} \theta_{7} \theta_{18}
+ 56 \theta_{1}^{5}  \theta_{4}
+ 21 \theta_{1}^{3}  \theta_{4}^{2} 
+ \tfrac{143}{4} \theta_{8} \theta_{1}^{4} 
- \tfrac{339}{2} \theta_{1}^{3}  \theta_{3} \theta_{8}
+ 283 \theta_{1}^{2}  \theta_{3}^{2}  \theta_{8}
\\&\ \quad
- \tfrac{381}{2} \theta_{1} \theta_{3}^{3}  \theta_{8}
- 98 \theta_{3}^{2}  \theta_{5} \theta_{8}
- 88 \theta_{3}^{2}  \theta_{7} \theta_{8}
+ \tfrac{481}{4} \theta_{3}^{2}  \theta_{4} \theta_{8}
+ 7 \theta_{9} \theta_{3} \theta_{8}
- \theta_{9} \theta_{10} \theta_{1}
+ \tfrac{27}{2} \theta_{9} \theta_{1}^{2}  \theta_{5}
+ \tfrac{25}{2} \theta_{9} \theta_{1}^{2}  \theta_{7}
- \tfrac{13}{4} \theta_{10} \theta_{5}^{2} 
- \tfrac{1}{2} \theta_{21}
\\&\ \quad
- 7 \theta_{9} \theta_{8} \theta_{1}
- \theta_{10} \theta_{12}
+ \tfrac{1}{2} \theta_{10} \theta_{15}
,
\\
\alpha^{(3)}_{1} &=  \theta_{3}
,
\\
\bar{\alpha}^{(3)}_{2} &=  \theta_{3}
,
\\
\bar{\alpha}^{(3)}_{3} &= 
- \theta_{1}^{2} 
+ 2 \theta_{1} \theta_{3}
- \theta_{3}^{2} 
- \theta_{4}
+ \theta_{5}
+ \theta_{7}
,
\\
\bar{\alpha}^{(3)}_{4} &= 
- \theta_{1}^{2} 
+ 2 \theta_{1} \theta_{3}
- 2 \theta_{3}^{2} 
- \theta_{4}
+ \theta_{5}
+ \theta_{7}
,
\\
\bar{\alpha}^{(3)}_{5} &= 
- \theta_{1}^{2} 
+ 2 \theta_{1} \theta_{3}
- \theta_{3}^{2} 
- \theta_{4}
+ \theta_{5}
+ \theta_{7}
,
\\
\bar{\alpha}^{(3)}_{7} &= 
- 4 \theta_{1}^{2}  \theta_{3}
+ 6 \theta_{1} \theta_{3}^{2} 
- 4 \theta_{3}^{3} 
- 2 \theta_{3} \theta_{4}
+ 2 \theta_{3} \theta_{5}
+ 2 \theta_{3} \theta_{7}
+ \theta_{1}^{3} 
+ \theta_{1} \theta_{4}
- \theta_{1} \theta_{5}
- \theta_{1} \theta_{7}
- \tfrac{1}{2} \theta_{10}
+ \tfrac{1}{2} \theta_{8}
,
\\
\bar{\alpha}^{(3)}_{8} &=  2 \theta_{1}^{3} 
- 7 \theta_{1}^{2}  \theta_{3}
+ 10 \theta_{1} \theta_{3}^{2} 
- 7 \theta_{3}^{3} 
+ 2 \theta_{1} \theta_{4}
- 2 \theta_{1} \theta_{5}
- 2 \theta_{1} \theta_{7}
- 3 \theta_{3} \theta_{4}
+ 3 \theta_{3} \theta_{5}
+ 3 \theta_{3} \theta_{7}
- \theta_{10}
+ \theta_{8}
,
\\
\bar{\alpha}^{(3)}_{9} &= 
- 2 \theta_{7} \theta_{5}
- 2 \theta_{5}^{2} 
+ 5 \theta_{4} \theta_{5}
+ 6 \theta_{4} \theta_{7}
- 16 \theta_{1} \theta_{5} \theta_{3}
- 14 \theta_{1} \theta_{7} \theta_{3}
+ 29 \theta_{1} \theta_{3} \theta_{4}
- 5 \theta_{3}^{4} 
- 7 \theta_{1}^{4} 
- \tfrac{7}{2} \theta_{10} \theta_{1}
+ \tfrac{7}{2} \theta_{10} \theta_{3}
+ 26 \theta_{1}^{3}  \theta_{3}
- 39 \theta_{1}^{2}  \theta_{3}^{2} 
\\&\ \quad
+ 25 \theta_{1} \theta_{3}^{3} 
+ 9 \theta_{1}^{2}  \theta_{5}
+ 5 \theta_{1}^{2}  \theta_{7}
+ 9 \theta_{3}^{2}  \theta_{5}
+ 10 \theta_{3}^{2}  \theta_{7}
- 13 \theta_{1}^{2}  \theta_{4}
- 19 \theta_{3}^{2}  \theta_{4}
+ 3 \theta_{9} \theta_{1}
- 3 \theta_{9} \theta_{3}
- \tfrac{7}{2} \theta_{8} \theta_{1}
+ \tfrac{7}{2} \theta_{8} \theta_{3}
- 2 \theta_{7}^{2} 
- 5 \theta_{4}^{2} 
+ \theta_{13}
+ \theta_{14}
,
\\
\bar{\alpha}^{(3)}_{10} &=  3 \theta_{1}^{4} 
- 15 \theta_{1}^{3}  \theta_{3}
+ 28 \theta_{1}^{2}  \theta_{3}^{2} 
- 26 \theta_{1} \theta_{3}^{3} 
+ 10 \theta_{3}^{4} 
+ 3 \theta_{1}^{2}  \theta_{4}
- 3 \theta_{1}^{2}  \theta_{5}
- \theta_{1}^{2}  \theta_{7}
- 11 \theta_{1} \theta_{3} \theta_{4}
+ 9 \theta_{1} \theta_{3} \theta_{5}
+ 7 \theta_{1} \theta_{3} \theta_{7}
+ 9 \theta_{3}^{2}  \theta_{4}
- 5 \theta_{3}^{2}  \theta_{5}
\\&\ \quad
- 8 \theta_{3}^{2}  \theta_{7}
+ 2 \theta_{1} \theta_{8}
- 4 \theta_{3} \theta_{8}
+ \theta_{4} \theta_{7}
- \theta_{5} \theta_{7}
+ \theta_{15}
,
\\
\bar{\alpha}^{(3)}_{11} &=  \theta_{7} \theta_{5}
- \theta_{4} \theta_{5}
- 4 \theta_{1} \theta_{5} \theta_{3}
- 6 \theta_{1} \theta_{7} \theta_{3}
+ 6 \theta_{1} \theta_{3} \theta_{4}
- 3 \theta_{3}^{4} 
- 3 \theta_{1}^{4} 
- \theta_{10} \theta_{1}
- \tfrac{5}{2} \theta_{10} \theta_{3}
+ 10 \theta_{1}^{3}  \theta_{3}
- 14 \theta_{1}^{2}  \theta_{3}^{2} 
+ 8 \theta_{1} \theta_{3}^{3} 
+ 3 \theta_{1}^{2}  \theta_{5}
+ \theta_{1}^{2}  \theta_{7}
\\&\ \quad
+ 3 \theta_{3}^{2}  \theta_{5}
+ 4 \theta_{3}^{2}  \theta_{7}
- 3 \theta_{1}^{2}  \theta_{4}
- 4 \theta_{3}^{2}  \theta_{4}
+ 2 \theta_{9} \theta_{1}
- 3 \theta_{8} \theta_{1}
+ \tfrac{5}{2} \theta_{8} \theta_{3}
+ \theta_{13}
,
\\
\bar{\alpha}^{(3)}_{12} &=  4 \theta_{8} \theta_{5}
- 7 \theta_{1} \theta_{5}^{2} 
+ 9 \theta_{3} \theta_{5}^{2} 
- 10 \theta_{1} \theta_{7} \theta_{5}
+ 19 \theta_{3} \theta_{7} \theta_{5}
- \theta_{9} \theta_{5}
- 36 \theta_{7} \theta_{1}^{2}  \theta_{3}
- 20 \theta_{1}^{3}  \theta_{4}
- 38 \theta_{3}^{2}  \theta_{1} \theta_{4}
+ 4 \theta_{3}^{2}  \theta_{1} \theta_{5}
+ 27 \theta_{3}^{2}  \theta_{1} \theta_{7}
- 21 \theta_{5} \theta_{1}^{2}  \theta_{3}
\\&\ \quad
+ 3 \theta_{3}^{5} 
- 41 \theta_{3}^{2}  \theta_{1}^{3} 
+ 41 \theta_{3}^{3}  \theta_{1}^{2} 
- 21 \theta_{3}^{4}  \theta_{1}
+ \theta_{3}^{3}  \theta_{4}
+ 3 \theta_{3}^{3}  \theta_{5}
+ 2 \theta_{3}^{3}  \theta_{7}
- \tfrac{15}{2} \theta_{4} \theta_{8}
- \tfrac{1}{2} \theta_{10} \theta_{4}
- 15 \theta_{1} \theta_{4}^{2} 
+ 27 \theta_{3} \theta_{4}^{2} 
+ \theta_{4} \theta_{9}
- \tfrac{1}{2} \theta_{10} \theta_{7}
\\&\ \quad
+ \tfrac{5}{2} \theta_{7} \theta_{8}
+ \theta_{7} \theta_{9}
- 4 \theta_{1} \theta_{7}^{2} 
+ 8 \theta_{3} \theta_{7}^{2} 
+ \theta_{12} \theta_{1}
- 2 \theta_{12} \theta_{3}
+ \theta_{14} \theta_{1}
- 2 \theta_{14} \theta_{3}
- \theta_{15} \theta_{3}
- \tfrac{3}{2} \theta_{1}^{2}  \theta_{8}
+ 15 \theta_{3}^{2}  \theta_{8}
- \theta_{1}^{2}  \theta_{9}
- 10 \theta_{3}^{2}  \theta_{9}
\\&\ \quad
+ \tfrac{1}{2} \theta_{1}^{2}  \theta_{10}
+ 6 \theta_{10} \theta_{3}^{2} 
- 8 \theta_{1} \theta_{3} \theta_{8}
+ 8 \theta_{1} \theta_{3} \theta_{9}
- 4 \theta_{1} \theta_{10} \theta_{3}
+ 21 \theta_{1} \theta_{4} \theta_{5}
+ 15 \theta_{1} \theta_{4} \theta_{7}
- 31 \theta_{3} \theta_{4} \theta_{5}
- 32 \theta_{3} \theta_{4} \theta_{7}
+ 24 \theta_{1}^{4}  \theta_{3}
+ 11 \theta_{1}^{3}  \theta_{5}
\\&\ \quad
+ 11 \theta_{1}^{3}  \theta_{7}
- 6 \theta_{1}^{5} 
+ 55 \theta_{4} \theta_{1}^{2}  \theta_{3}
+ \theta_{17}
+ \theta_{18}
,
\\
\bar{\alpha}^{(3)}_{13} &=  4 \theta_{8} \theta_{5}
- 7 \theta_{1} \theta_{5}^{2} 
+ 7 \theta_{3} \theta_{5}^{2} 
- 10 \theta_{1} \theta_{7} \theta_{5}
+ 17 \theta_{3} \theta_{7} \theta_{5}
- \theta_{9} \theta_{5}
- 31 \theta_{7} \theta_{1}^{2}  \theta_{3}
- 20 \theta_{1}^{3}  \theta_{4}
- 9 \theta_{3}^{2}  \theta_{1} \theta_{4}
- 12 \theta_{3}^{2}  \theta_{1} \theta_{5}
\\&\ \quad
+ 13 \theta_{3}^{2}  \theta_{1} \theta_{7}
- 12 \theta_{5} \theta_{1}^{2}  \theta_{3}
- 2 \theta_{3}^{5} 
- 15 \theta_{3}^{2}  \theta_{1}^{3} 
+ 2 \theta_{3}^{3}  \theta_{1}^{2} 
+ 4 \theta_{3}^{4}  \theta_{1}
- 18 \theta_{3}^{3}  \theta_{4}
+ 12 \theta_{3}^{3}  \theta_{5}
+ 12 \theta_{3}^{3}  \theta_{7}
- \tfrac{15}{2} \theta_{4} \theta_{8}
- \tfrac{1}{2} \theta_{10} \theta_{4}
\\&\ \quad
- 15 \theta_{1} \theta_{4}^{2} 
+ 22 \theta_{3} \theta_{4}^{2} 
+ \theta_{4} \theta_{9}
- \tfrac{1}{2} \theta_{10} \theta_{7}
+ \tfrac{5}{2} \theta_{7} \theta_{8}
+ \theta_{7} \theta_{9}
- 4 \theta_{1} \theta_{7}^{2} 
+ 6 \theta_{3} \theta_{7}^{2} 
+ \theta_{13} \theta_{3}
+ \theta_{12} \theta_{1}
- 2 \theta_{12} \theta_{3}
+ \theta_{14} \theta_{1}
\\&\ \quad
- \theta_{14} \theta_{3}
- \theta_{15} \theta_{3}
- \tfrac{3}{2} \theta_{1}^{2}  \theta_{8}
+ \tfrac{37}{2} \theta_{3}^{2}  \theta_{8}
- \theta_{1}^{2}  \theta_{9}
- 13 \theta_{3}^{2}  \theta_{9}
+ \tfrac{1}{2} \theta_{1}^{2}  \theta_{10}
+ \tfrac{19}{2} \theta_{10} \theta_{3}^{2} 
- \tfrac{23}{2} \theta_{1} \theta_{3} \theta_{8}
+ 11 \theta_{1} \theta_{3} \theta_{9}
- \tfrac{15}{2} \theta_{1} \theta_{10} \theta_{3}
\\&\ \quad
+ 21 \theta_{1} \theta_{4} \theta_{5}
+ 15 \theta_{1} \theta_{4} \theta_{7}
- 26 \theta_{3} \theta_{4} \theta_{5}
- 26 \theta_{3} \theta_{4} \theta_{7}
+ 17 \theta_{1}^{4}  \theta_{3}
+ 11 \theta_{1}^{3}  \theta_{5}
+ 11 \theta_{1}^{3}  \theta_{7}
- 6 \theta_{1}^{5} 
+ 42 \theta_{4} \theta_{1}^{2}  \theta_{3}
+ \theta_{17}
+ \theta_{18}
,
\\
\bar{\alpha}^{(3)}_{14} &= 
- \theta_{1}^{5} 
+ 7 \theta_{1}^{4}  \theta_{3}
- 29 \theta_{1}^{3}  \theta_{3}^{2} 
+ 60 \theta_{1}^{2}  \theta_{3}^{3} 
- 58 \theta_{1} \theta_{3}^{4} 
+ 22 \theta_{3}^{5} 
- 4 \theta_{1}^{3}  \theta_{4}
+ 2 \theta_{1}^{3}  \theta_{7}
+ 10 \theta_{1}^{2}  \theta_{3} \theta_{4}
+ 7 \theta_{1}^{2}  \theta_{3} \theta_{5}
- 11 \theta_{1}^{2}  \theta_{3} \theta_{7}
\\&\ \quad
- 13 \theta_{1} \theta_{3}^{2}  \theta_{4}
- 13 \theta_{1} \theta_{3}^{2}  \theta_{5}
+ 17 \theta_{1} \theta_{3}^{2}  \theta_{7}
+ 7 \theta_{3}^{3}  \theta_{4}
+ 5 \theta_{3}^{3}  \theta_{5}
- 8 \theta_{3}^{3}  \theta_{7}
+ \theta_{1}^{2}  \theta_{10}
- \theta_{1}^{2}  \theta_{9}
- 4 \theta_{1} \theta_{10} \theta_{3}
- 4 \theta_{1} \theta_{3} \theta_{8}
\\&\ \quad
+ 6 \theta_{1} \theta_{3} \theta_{9}
- 3 \theta_{1} \theta_{4}^{2} 
+ 4 \theta_{1} \theta_{4} \theta_{5}
+ 2 \theta_{1} \theta_{4} \theta_{7}
- \theta_{1} \theta_{5}^{2} 
- \theta_{1} \theta_{7}^{2} 
+ 4 \theta_{10} \theta_{3}^{2} 
+ 6 \theta_{3}^{2}  \theta_{8}
- 6 \theta_{3}^{2}  \theta_{9}
+ 3 \theta_{3} \theta_{4}^{2} 
- \theta_{3} \theta_{4} \theta_{5}
\\&\ \quad
- 6 \theta_{3} \theta_{4} \theta_{7}
+ \theta_{3} \theta_{5} \theta_{7}
+ 2 \theta_{3} \theta_{7}^{2} 
- \theta_{10} \theta_{7}
- \theta_{15} \theta_{3}
- 3 \theta_{4} \theta_{8}
+ \theta_{4} \theta_{9}
+ \theta_{5} \theta_{8}
- \theta_{5} \theta_{9}
+ \theta_{7} \theta_{9}
+ \theta_{17}
,
\\
\bar{\alpha}^{(3)}_{16} &= 
- \tfrac{1}{4} \theta_{10}^{2} 
+ 28 \theta_{3}^{2}  \theta_{4} \theta_{5}
+ 28 \theta_{3}^{2}  \theta_{4} \theta_{7}
- 19 \theta_{3}^{2}  \theta_{5} \theta_{7}
+ \tfrac{9}{2} \theta_{1}^{2}  \theta_{3} \theta_{8}
+ \tfrac{15}{2} \theta_{1} \theta_{3}^{2}  \theta_{8}
+ \theta_{1}^{2}  \theta_{3} \theta_{9}
+ 2 \theta_{1}^{2}  \theta_{4} \theta_{5}
+ 2 \theta_{1}^{2}  \theta_{4} \theta_{7}
- 2 \theta_{1}^{2}  \theta_{5} \theta_{7}
- 11 \theta_{1} \theta_{3}^{2}  \theta_{9}
\\&\ \quad
+ 17 \theta_{1} \theta_{3} \theta_{4}^{2} 
+ 9 \theta_{1} \theta_{3} \theta_{5}^{2} 
+ 6 \theta_{1} \theta_{3} \theta_{7}^{2} 
- \theta_{3}^{2}  \theta_{13}
- \tfrac{1}{4} \theta_{8}^{2} 
- \theta_{3} \theta_{4} \theta_{9}
+ \theta_{3} \theta_{5} \theta_{9}
- \theta_{3} \theta_{7} \theta_{9}
- \theta_{1} \theta_{4} \theta_{8}
+ \theta_{1} \theta_{5} \theta_{8}
+ \theta_{1} \theta_{7} \theta_{8}
+ \tfrac{17}{2} \theta_{3} \theta_{4} \theta_{8}
\\&\ \quad
- 5 \theta_{3} \theta_{5} \theta_{8}
- \tfrac{7}{2} \theta_{3} \theta_{7} \theta_{8}
+ \theta_{3} \theta_{5} \theta_{10}
+ \tfrac{3}{2} \theta_{3} \theta_{7} \theta_{10}
- \theta_{1} \theta_{5} \theta_{10}
- \theta_{1} \theta_{7} \theta_{10}
- \tfrac{1}{2} \theta_{3} \theta_{4} \theta_{10}
+ \theta_{1} \theta_{4} \theta_{10}
+ \theta_{14} \theta_{3}^{2} 
+ \theta_{15} \theta_{3}^{2} 
+ 2 \theta_{12} \theta_{3}^{2} 
\\&\ \quad
- \theta_{18} \theta_{3}
- \theta_{1} \theta_{12} \theta_{3}
- \theta_{1} \theta_{14} \theta_{3}
+ \tfrac{23}{2} \theta_{1} \theta_{3}^{2}  \theta_{10}
- \tfrac{7}{2} \theta_{1}^{2}  \theta_{3} \theta_{10}
+ \theta_{1}^{3}  \theta_{10}
- \tfrac{25}{2} \theta_{3}^{3}  \theta_{10}
+ \tfrac{1}{2} \theta_{8} \theta_{10}
- 25 \theta_{1} \theta_{3} \theta_{4} \theta_{5}
- 19 \theta_{1} \theta_{3} \theta_{4} \theta_{7}
\\&\ \quad
+ 14 \theta_{1} \theta_{3} \theta_{5} \theta_{7}
+ 28 \theta_{1}^{3}  \theta_{3} \theta_{4}
- 19 \theta_{1}^{3}  \theta_{3} \theta_{5}
- 19 \theta_{1}^{3}  \theta_{3} \theta_{7}
- 56 \theta_{1}^{2}  \theta_{3}^{2}  \theta_{4}
+ 26 \theta_{1}^{2}  \theta_{3}^{2}  \theta_{5}
+ 45 \theta_{1}^{2}  \theta_{3}^{2}  \theta_{7}
+ 23 \theta_{1} \theta_{3}^{3}  \theta_{4}
- 2 \theta_{1} \theta_{3}^{3}  \theta_{5}
- 27 \theta_{1} \theta_{3}^{3}  \theta_{7}
\\&\ \quad
- \theta_{17} \theta_{3}
- \theta_{1}^{3}  \theta_{8}
- \tfrac{31}{2} \theta_{3}^{3}  \theta_{8}
+ 12 \theta_{1}^{5}  \theta_{3}
- 34 \theta_{1}^{4}  \theta_{3}^{2} 
+ 45 \theta_{1}^{3}  \theta_{3}^{3} 
- 36 \theta_{1}^{2}  \theta_{3}^{4} 
+ 20 \theta_{1} \theta_{3}^{5} 
- 2 \theta_{1}^{4}  \theta_{4}
+ 2 \theta_{1}^{4}  \theta_{5}
+ 2 \theta_{1}^{4}  \theta_{7}
+ 12 \theta_{3}^{4}  \theta_{4}
- 6 \theta_{3}^{4}  \theta_{5}
\\&\ \quad
- 6 \theta_{3}^{4}  \theta_{7}
- \theta_{1}^{2}  \theta_{4}^{2} 
- \theta_{1}^{2}  \theta_{5}^{2} 
- \theta_{1}^{2}  \theta_{7}^{2} 
+ 13 \theta_{3}^{3}  \theta_{9}
- 23 \theta_{3}^{2}  \theta_{4}^{2} 
- 8 \theta_{3}^{2}  \theta_{5}^{2} 
- 7 \theta_{3}^{2}  \theta_{7}^{2} 
- \theta_{1}^{6} 
- 7 \theta_{3}^{6} 
,
\\
\bar{\alpha}^{(3)}_{17} &= 
- 4 \theta_{1} \theta_{5}^{3} 
- 8 \theta_{1} \theta_{7}^{2}  \theta_{5}
- 3 \theta_{9} \theta_{3} \theta_{10}
+ 4 \theta_{3} \theta_{5}^{3} 
- 6 \theta_{4} \theta_{5} \theta_{8}
- 5 \theta_{4} \theta_{7} \theta_{8}
- 26 \theta_{3} \theta_{4} \theta_{5} \theta_{7}
+ 26 \theta_{1} \theta_{4} \theta_{5} \theta_{7}
+ 12 \theta_{3} \theta_{1}^{3}  \theta_{9}
- 124 \theta_{3} \theta_{1}^{2}  \theta_{4}^{2} 
- 72 \theta_{3} \theta_{1}^{2}  \theta_{5}^{2} 
\\&\ \quad
- 52 \theta_{3} \theta_{1}^{2}  \theta_{7}^{2} 
+ 22 \theta_{1}^{3}  \theta_{5}^{2} 
+ 14 \theta_{1}^{3}  \theta_{7}^{2} 
- 8 \theta_{1} \theta_{7} \theta_{5}^{2} 
- 8 \theta_{1}^{2}  \theta_{4} \theta_{10}
- 16 \theta_{3} \theta_{4} \theta_{5}^{2} 
- 12 \theta_{3}^{2}  \theta_{4} \theta_{10}
- 4 \theta_{3} \theta_{1} \theta_{5} \theta_{10}
+ 2 \theta_{3} \theta_{1} \theta_{7} \theta_{10}
+ 17 \theta_{3} \theta_{1} \theta_{4} \theta_{10}
\\&\ \quad
- 2 \theta_{3} \theta_{4} \theta_{15}
+ \tfrac{5}{2} \theta_{8}^{2}  \theta_{1}
- \theta_{8} \theta_{14}
- 72 \theta_{3}^{7} 
+ \theta_{10} \theta_{14}
- \theta_{3}^{2}  \theta_{5} \theta_{10}
+ 2 \theta_{3}^{2}  \theta_{7} \theta_{10}
+ 14 \theta_{1}^{7} 
- 6 \theta_{15} \theta_{1}^{2}  \theta_{3}
+ 8 \theta_{15} \theta_{1} \theta_{3}^{2} 
- 6 \theta_{15} \theta_{3}^{3} 
+ 190 \theta_{3} \theta_{1}^{2}  \theta_{4} \theta_{5}
\\&\ \quad
+ 164 \theta_{3} \theta_{1}^{2}  \theta_{4} \theta_{7}
- 112 \theta_{3} \theta_{1}^{2}  \theta_{5} \theta_{7}
+ 14 \theta_{1} \theta_{4} \theta_{7}^{2} 
- 2 \theta_{7} \theta_{10} \theta_{5}
+ 2 \theta_{7} \theta_{8} \theta_{5}
+ 8 \theta_{3} \theta_{7} \theta_{5}^{2} 
+ 524 \theta_{1}^{3}  \theta_{3}^{2}  \theta_{4}
- 388 \theta_{1}^{3}  \theta_{3}^{2}  \theta_{5}
- 326 \theta_{1}^{3}  \theta_{3}^{2}  \theta_{7}
\\&\ \quad
- 686 \theta_{1}^{2}  \theta_{3}^{3}  \theta_{4}
+ 480 \theta_{1}^{2}  \theta_{3}^{3}  \theta_{5}
+ 426 \theta_{1}^{2}  \theta_{3}^{3}  \theta_{7}
+ 506 \theta_{1} \theta_{3}^{4}  \theta_{4}
- 324 \theta_{1} \theta_{3}^{4}  \theta_{5}
- 312 \theta_{1} \theta_{3}^{4}  \theta_{7}
+ 106 \theta_{3}^{3}  \theta_{4} \theta_{5}
+ 106 \theta_{3}^{3}  \theta_{4} \theta_{7}
- 62 \theta_{3}^{3}  \theta_{5} \theta_{7}
\\&\ \quad
- 26 \theta_{1}^{2}  \theta_{3}^{2}  \theta_{9}
+ 30 \theta_{1} \theta_{3}^{3}  \theta_{9}
+ 156 \theta_{1} \theta_{3}^{2}  \theta_{4}^{2} 
+ 80 \theta_{1} \theta_{3}^{2}  \theta_{5}^{2} 
+ 74 \theta_{1} \theta_{3}^{2}  \theta_{7}^{2} 
- 20 \theta_{1} \theta_{3}^{3}  \theta_{10}
+ 3 \theta_{1}^{2}  \theta_{3}^{2}  \theta_{10}
+ 4 \theta_{3} \theta_{1}^{3}  \theta_{10}
- 228 \theta_{1} \theta_{3}^{2}  \theta_{4} \theta_{5}
\\&\ \quad
- 214 \theta_{1} \theta_{3}^{2}  \theta_{4} \theta_{7}
+ 138 \theta_{1} \theta_{3}^{2}  \theta_{5} \theta_{7}
+ 2 \theta_{3} \theta_{7} \theta_{15}
- 222 \theta_{3} \theta_{1}^{4}  \theta_{4}
+ 172 \theta_{3} \theta_{1}^{4}  \theta_{5}
+ 136 \theta_{3} \theta_{1}^{4}  \theta_{7}
- 32 \theta_{1}^{5}  \theta_{5}
- 24 \theta_{1}^{5}  \theta_{7}
- 2 \theta_{1}^{3}  \theta_{14}
- 5 \theta_{8}^{2}  \theta_{3}
\\&\ \quad
- 56 \theta_{1}^{3}  \theta_{4} \theta_{5}
- 46 \theta_{1}^{3}  \theta_{4} \theta_{7}
- 2 \theta_{1}^{2}  \theta_{4} \theta_{9}
+ 18 \theta_{1}^{2}  \theta_{4} \theta_{8}
- 2 \theta_{1} \theta_{4} \theta_{14}
+ 6 \theta_{10} \theta_{4} \theta_{5}
+ 5 \theta_{10} \theta_{4} \theta_{7}
+ 16 \theta_{1} \theta_{4} \theta_{5}^{2} 
+ 22 \theta_{3} \theta_{4}^{2}  \theta_{5}
+ 20 \theta_{3} \theta_{4}^{2}  \theta_{7}
\\&\ \quad
- 2 \theta_{1} \theta_{7} \theta_{15}
+ 2 \theta_{3} \theta_{5} \theta_{15}
- 22 \theta_{1} \theta_{4}^{2}  \theta_{5}
- 20 \theta_{1} \theta_{4}^{2}  \theta_{7}
+ 4 \theta_{10} \theta_{1}^{2}  \theta_{5}
+ 2 \theta_{3} \theta_{4} \theta_{14}
+ 6 \theta_{14} \theta_{3}^{3} 
+ 2 \theta_{1}^{3}  \theta_{15}
- 8 \theta_{1} \theta_{14} \theta_{3}^{2} 
+ 2 \theta_{1} \theta_{7} \theta_{14}
+ 2 \theta_{1} \theta_{5} \theta_{14}
\\&\ \quad
+ 32 \theta_{1}^{3}  \theta_{5} \theta_{7}
- 2 \theta_{1} \theta_{5} \theta_{15}
+ 30 \theta_{1} \theta_{7} \theta_{3} \theta_{8}
+ 8 \theta_{3} \theta_{7}^{2}  \theta_{5}
+ 42 \theta_{3}^{4}  \theta_{8}
- 14 \theta_{8} \theta_{1}^{2}  \theta_{5}
- 10 \theta_{8} \theta_{1}^{2}  \theta_{7}
- 2 \theta_{9} \theta_{1}^{4} 
+ 348 \theta_{1}^{5}  \theta_{3}^{2} 
- 694 \theta_{1}^{4}  \theta_{3}^{3} 
+ 892 \theta_{1}^{3}  \theta_{3}^{4} 
\\&\ \quad
- 734 \theta_{1}^{2}  \theta_{3}^{5} 
+ 354 \theta_{1} \theta_{3}^{6} 
- 168 \theta_{3}^{5}  \theta_{4}
+ 90 \theta_{3}^{5}  \theta_{5}
+ 108 \theta_{3}^{5}  \theta_{7}
- 18 \theta_{3}^{4}  \theta_{9}
- 78 \theta_{3}^{3}  \theta_{4}^{2} 
- 34 \theta_{3}^{3}  \theta_{5}^{2} 
- 40 \theta_{3}^{3}  \theta_{7}^{2} 
- 104 \theta_{3} \theta_{1}^{6} 
- 49 \theta_{1} \theta_{3} \theta_{4} \theta_{8}
\\&\ \quad
+ 36 \theta_{1} \theta_{5} \theta_{3} \theta_{8}
+ 24 \theta_{3}^{4}  \theta_{10}
- 2 \theta_{10} \theta_{1}^{4} 
- \tfrac{5}{2} \theta_{10}^{2}  \theta_{1}
- \theta_{8} \theta_{3} \theta_{10}
- 14 \theta_{3} \theta_{4} \theta_{7}^{2} 
+ 2 \theta_{1} \theta_{4} \theta_{15}
+ 5 \theta_{4}^{2}  \theta_{8}
- 5 \theta_{10} \theta_{4}^{2} 
+ 10 \theta_{1} \theta_{4}^{3} 
- 10 \theta_{3} \theta_{4}^{3} 
- 4 \theta_{1} \theta_{7}^{3} 
\\&\ \quad
+ 4 \theta_{3} \theta_{7}^{3} 
+ 2 \theta_{7}^{2}  \theta_{8}
- 2 \theta_{10} \theta_{7}^{2} 
+ 6 \theta_{3} \theta_{10}^{2} 
+ \theta_{8} \theta_{15}
+ 2 \theta_{8} \theta_{5}^{2} 
- 6 \theta_{3}^{2}  \theta_{4} \theta_{9}
+ 6 \theta_{3}^{2}  \theta_{5} \theta_{9}
+ 6 \theta_{3}^{2}  \theta_{7} \theta_{9}
+ 8 \theta_{3} \theta_{1} \theta_{4} \theta_{9}
- 8 \theta_{3} \theta_{1} \theta_{5} \theta_{9}
- 8 \theta_{3} \theta_{1} \theta_{7} \theta_{9}
\\&\ \quad
- 2 \theta_{3} \theta_{7} \theta_{14}
- 2 \theta_{3} \theta_{5} \theta_{14}
+ 6 \theta_{3} \theta_{1}^{2}  \theta_{14}
+ 40 \theta_{1}^{5}  \theta_{4}
+ 36 \theta_{1}^{3}  \theta_{4}^{2} 
+ 12 \theta_{8} \theta_{1}^{4} 
- 56 \theta_{1}^{3}  \theta_{3} \theta_{8}
+ 103 \theta_{1}^{2}  \theta_{3}^{2}  \theta_{8}
- 98 \theta_{1} \theta_{3}^{3}  \theta_{8}
- 21 \theta_{3}^{2}  \theta_{5} \theta_{8}
\\&\ \quad
- 24 \theta_{3}^{2}  \theta_{7} \theta_{8}
+ 34 \theta_{3}^{2}  \theta_{4} \theta_{8}
+ 3 \theta_{9} \theta_{3} \theta_{8}
+ \theta_{9} \theta_{10} \theta_{1}
+ 2 \theta_{9} \theta_{1}^{2}  \theta_{5}
+ 2 \theta_{9} \theta_{1}^{2}  \theta_{7}
- 2 \theta_{10} \theta_{5}^{2} 
- \theta_{9} \theta_{8} \theta_{1}
- \theta_{10} \theta_{15}
,
\\
\alpha^{(4)}_{1} &= 
- \theta_{3}
,
\\
\bar{\alpha}^{(4)}_{3} &= 
- \theta_{1}^{3} 
+ 3 \theta_{1}^{2}  \theta_{3}
- 4 \theta_{1} \theta_{3}^{2} 
+ 3 \theta_{3}^{3} 
- \theta_{1} \theta_{4}
+ \theta_{1} \theta_{5}
+ \theta_{1} \theta_{7}
+ \theta_{3} \theta_{4}
- \theta_{3} \theta_{5}
- \theta_{3} \theta_{7}
+ \tfrac{1}{2} \theta_{10}
- \tfrac{1}{2} \theta_{8}
\\
\bar{\alpha}^{(4)}_{4} &=  \theta_{1}^{3} 
- 3 \theta_{1}^{2}  \theta_{3}
+ 4 \theta_{1} \theta_{3}^{2} 
- 3 \theta_{3}^{3} 
+ \theta_{1} \theta_{4}
- \theta_{1} \theta_{5}
- \theta_{1} \theta_{7}
- \theta_{3} \theta_{4}
+ \theta_{3} \theta_{5}
+ \theta_{3} \theta_{7}
- \tfrac{1}{2} \theta_{10}
+ \tfrac{1}{2} \theta_{8}
,
\\
\bar{\alpha}^{(4)}_{5} &=  \theta_{7} \theta_{5}
+ 2 \theta_{5}^{2} 
- 5 \theta_{4} \theta_{5}
- 5 \theta_{4} \theta_{7}
+ 26 \theta_{1} \theta_{5} \theta_{3}
+ 22 \theta_{1} \theta_{7} \theta_{3}
- 41 \theta_{1} \theta_{3} \theta_{4}
+ 18 \theta_{3}^{4} 
+ 10 \theta_{1}^{4} 
+ \tfrac{7}{2} \theta_{10} \theta_{1}
- 3 \theta_{10} \theta_{3}
- 42 \theta_{1}^{3}  \theta_{3}
+ 70 \theta_{1}^{2}  \theta_{3}^{2} 
\\&\ \quad
- 55 \theta_{1} \theta_{3}^{3} 
- 12 \theta_{1}^{2}  \theta_{5}
- 6 \theta_{1}^{2}  \theta_{7}
- 15 \theta_{3}^{2}  \theta_{5}
- 19 \theta_{3}^{2}  \theta_{7}
+ 16 \theta_{1}^{2}  \theta_{4}
+ 29 \theta_{3}^{2}  \theta_{4}
- 3 \theta_{9} \theta_{1}
+ 3 \theta_{9} \theta_{3}
+ \tfrac{11}{2} \theta_{8} \theta_{1}
- 8 \theta_{8} \theta_{3}
+ 2 \theta_{7}^{2} 
+ 5 \theta_{4}^{2} 
\\&\ \quad
- \theta_{13}
- \theta_{14}
+ \theta_{15}
,
\\
\bar{\alpha}^{(4)}_{6} &=  3 \theta_{1}^{4} 
- 11 \theta_{1}^{3}  \theta_{3}
+ 17 \theta_{1}^{2}  \theta_{3}^{2} 
- 12 \theta_{1} \theta_{3}^{3} 
+ 6 \theta_{3}^{4} 
+ 3 \theta_{1}^{2}  \theta_{4}
- 3 \theta_{1}^{2}  \theta_{5}
- \theta_{1}^{2}  \theta_{7}
- 7 \theta_{1} \theta_{3} \theta_{4}
+ 5 \theta_{1} \theta_{3} \theta_{5}
+ 7 \theta_{1} \theta_{3} \theta_{7}
+ 5 \theta_{3}^{2}  \theta_{4}
- 4 \theta_{3}^{2}  \theta_{5}
\\&\ \quad
- 5 \theta_{3}^{2}  \theta_{7}
+ \theta_{1} \theta_{10}
+ 3 \theta_{1} \theta_{8}
- 2 \theta_{1} \theta_{9}
+ 3 \theta_{10} \theta_{3}
- 3 \theta_{3} \theta_{8}
+ \theta_{4} \theta_{5}
- \theta_{5} \theta_{7}
- \theta_{13}
,
\\
\bar{\alpha}^{(4)}_{7} &= 
- 3 \theta_{8} \theta_{5}
+ 6 \theta_{1} \theta_{5}^{2} 
- 9 \theta_{3} \theta_{5}^{2} 
+ 10 \theta_{1} \theta_{7} \theta_{5}
- 16 \theta_{3} \theta_{7} \theta_{5}
+ 27 \theta_{7} \theta_{1}^{2}  \theta_{3}
+ 16 \theta_{1}^{3}  \theta_{4}
+ 43 \theta_{3}^{2}  \theta_{1} \theta_{4}
- 31 \theta_{3}^{2}  \theta_{1} \theta_{5}
- 24 \theta_{3}^{2}  \theta_{1} \theta_{7}
+ 34 \theta_{5} \theta_{1}^{2}  \theta_{3}
+ 3 \theta_{3}^{5} 
\\
&\ \quad
+ 38 \theta_{3}^{2}  \theta_{1}^{3} 
- 26 \theta_{3}^{3}  \theta_{1}^{2} 
+ \theta_{3}^{4}  \theta_{1}
- 8 \theta_{3}^{3}  \theta_{4}
+ 11 \theta_{3}^{3}  \theta_{5}
+ 3 \theta_{3}^{3}  \theta_{7}
+ \tfrac{9}{2} \theta_{4} \theta_{8}
+ \tfrac{1}{2} \theta_{10} \theta_{4}
+ 12 \theta_{1} \theta_{4}^{2} 
- 24 \theta_{3} \theta_{4}^{2} 
- \tfrac{1}{2} \theta_{10} \theta_{7}
- \tfrac{5}{2} \theta_{7} \theta_{8}
+ 3 \theta_{1} \theta_{7}^{2} 
\\
&\ \quad
- 6 \theta_{3} \theta_{7}^{2} 
+ \theta_{13} \theta_{3}
- \theta_{12} \theta_{1}
+ 2 \theta_{12} \theta_{3}
- \theta_{14} \theta_{1}
+ 2 \theta_{14} \theta_{3}
- \theta_{15} \theta_{3}
+ \tfrac{3}{2} \theta_{1}^{2}  \theta_{8}
- 2 \theta_{3}^{2}  \theta_{8}
+ 4 \theta_{3}^{2}  \theta_{9}
+ \tfrac{1}{2} \theta_{1}^{2}  \theta_{10}
- 5 \theta_{10} \theta_{3}^{2} 
- \theta_{1} \theta_{3} \theta_{8}
\end{align*}
\begin{align*}
&\ \quad
- \theta_{1} \theta_{10} \theta_{3}
- 17 \theta_{1} \theta_{4} \theta_{5}
- 13 \theta_{1} \theta_{4} \theta_{7}
+ 29 \theta_{3} \theta_{4} \theta_{5}
+ 25 \theta_{3} \theta_{4} \theta_{7}
- 23 \theta_{1}^{4}  \theta_{3}
- 11 \theta_{1}^{3}  \theta_{5}
- 9 \theta_{1}^{3}  \theta_{7}
+ 5 \theta_{1}^{5} 
- 51 \theta_{4} \theta_{1}^{2}  \theta_{3}
- \theta_{18}
,
\\
\bar{\alpha}^{(4)}_{8} &=  \tfrac{1}{4}\big( 2 \theta_{1}^{3} 
- 6 \theta_{1}^{2}  \theta_{3}
+ 8 \theta_{1} \theta_{3}^{2} 
- 6 \theta_{3}^{3} 
+ 2 \theta_{1} \theta_{4}
- 2 \theta_{1} \theta_{5}
- 2 \theta_{1} \theta_{7}
- 2 \theta_{3} \theta_{4}
+ 2 \theta_{3} \theta_{5}
+ 2 \theta_{3} \theta_{7}
- \theta_{10}
+ \theta_{8}\big)^{2} 
,
\\
\bar{\alpha}^{(4)}_{9} &= 
- 4 \theta_{1} \theta_{5}^{3} 
- 8 \theta_{1} \theta_{7}^{2}  \theta_{5}
- 3 \theta_{9} \theta_{3} \theta_{10}
+ 4 \theta_{3} \theta_{5}^{3} 
- 6 \theta_{4} \theta_{5} \theta_{8}
- 5 \theta_{4} \theta_{7} \theta_{8}
- 26 \theta_{3} \theta_{4} \theta_{5} \theta_{7}
+ 26 \theta_{1} \theta_{4} \theta_{5} \theta_{7}
+ 12 \theta_{3} \theta_{1}^{3}  \theta_{9}
- 123 \theta_{3} \theta_{1}^{2}  \theta_{4}^{2} 
- 71 \theta_{3} \theta_{1}^{2}  \theta_{5}^{2} 
\\&\ \quad
- 51 \theta_{3} \theta_{1}^{2}  \theta_{7}^{2} 
+ 22 \theta_{1}^{3}  \theta_{5}^{2} 
+ 14 \theta_{1}^{3}  \theta_{7}^{2} 
- 8 \theta_{1} \theta_{7} \theta_{5}^{2} 
- 8 \theta_{1}^{2}  \theta_{4} \theta_{10}
- 16 \theta_{3} \theta_{4} \theta_{5}^{2} 
- 11 \theta_{3}^{2}  \theta_{4} \theta_{10}
- 3 \theta_{3} \theta_{1} \theta_{5} \theta_{10}
+ 3 \theta_{3} \theta_{1} \theta_{7} \theta_{10}
+ 16 \theta_{3} \theta_{1} \theta_{4} \theta_{10}
\\&\ \quad
- 2 \theta_{3} \theta_{4} \theta_{15}
+ \tfrac{5}{2} \theta_{8}^{2}  \theta_{1}
- \theta_{8} \theta_{14}
- 63 \theta_{3}^{7} 
+ \theta_{10} \theta_{14}
- 2 \theta_{3}^{2}  \theta_{5} \theta_{10}
+ \theta_{3}^{2}  \theta_{7} \theta_{10}
+ 14 \theta_{1}^{7} 
- 6 \theta_{15} \theta_{1}^{2}  \theta_{3}
+ 8 \theta_{15} \theta_{1} \theta_{3}^{2} 
- 6 \theta_{15} \theta_{3}^{3} 
+ 188 \theta_{3} \theta_{1}^{2}  \theta_{4} \theta_{5}
\\&\ \quad
+ 162 \theta_{3} \theta_{1}^{2}  \theta_{4} \theta_{7}
- 110 \theta_{3} \theta_{1}^{2}  \theta_{5} \theta_{7}
+ 14 \theta_{1} \theta_{4} \theta_{7}^{2} 
- 2 \theta_{7} \theta_{10} \theta_{5}
+ 2 \theta_{7} \theta_{8} \theta_{5}
+ 8 \theta_{3} \theta_{7} \theta_{5}^{2} 
+ 516 \theta_{1}^{3}  \theta_{3}^{2}  \theta_{4}
- 380 \theta_{1}^{3}  \theta_{3}^{2}  \theta_{5}
- 318 \theta_{1}^{3}  \theta_{3}^{2}  \theta_{7}
\\&\ \quad
- 672 \theta_{1}^{2}  \theta_{3}^{3}  \theta_{4}
+ 466 \theta_{1}^{2}  \theta_{3}^{3}  \theta_{5}
+ 412 \theta_{1}^{2}  \theta_{3}^{3}  \theta_{7}
+ 492 \theta_{1} \theta_{3}^{4}  \theta_{4}
- 310 \theta_{1} \theta_{3}^{4}  \theta_{5}
- 298 \theta_{1} \theta_{3}^{4}  \theta_{7}
+ 104 \theta_{3}^{3}  \theta_{4} \theta_{5}
+ 104 \theta_{3}^{3}  \theta_{4} \theta_{7}
- 60 \theta_{3}^{3}  \theta_{5} \theta_{7}
\\&\ \quad
- 26 \theta_{1}^{2}  \theta_{3}^{2}  \theta_{9}
+ 30 \theta_{1} \theta_{3}^{3}  \theta_{9}
+ 154 \theta_{1} \theta_{3}^{2}  \theta_{4}^{2} 
+ 78 \theta_{1} \theta_{3}^{2}  \theta_{5}^{2} 
+ 72 \theta_{1} \theta_{3}^{2}  \theta_{7}^{2} 
- 24 \theta_{1} \theta_{3}^{3}  \theta_{10}
+ 6 \theta_{1}^{2}  \theta_{3}^{2}  \theta_{10}
+ 3 \theta_{3} \theta_{1}^{3}  \theta_{10}
- 224 \theta_{1} \theta_{3}^{2}  \theta_{4} \theta_{5}
\\&\ \quad
- 210 \theta_{1} \theta_{3}^{2}  \theta_{4} \theta_{7}
+ 134 \theta_{1} \theta_{3}^{2}  \theta_{5} \theta_{7}
+ 2 \theta_{3} \theta_{7} \theta_{15}
- 220 \theta_{3} \theta_{1}^{4}  \theta_{4}
+ 170 \theta_{3} \theta_{1}^{4}  \theta_{5}
+ 134 \theta_{3} \theta_{1}^{4}  \theta_{7}
- 32 \theta_{1}^{5}  \theta_{5}
- 24 \theta_{1}^{5}  \theta_{7}
- 2 \theta_{1}^{3}  \theta_{14}
- \tfrac{19}{4} \theta_{8}^{2}  \theta_{3}
\\&\ \quad
- 56 \theta_{1}^{3}  \theta_{4} \theta_{5}
- 46 \theta_{1}^{3}  \theta_{4} \theta_{7}
- 2 \theta_{1}^{2}  \theta_{4} \theta_{9}
+ 18 \theta_{1}^{2}  \theta_{4} \theta_{8}
- 2 \theta_{1} \theta_{4} \theta_{14}
+ 6 \theta_{10} \theta_{4} \theta_{5}
+ 5 \theta_{10} \theta_{4} \theta_{7}
+ 16 \theta_{1} \theta_{4} \theta_{5}^{2} 
+ 22 \theta_{3} \theta_{4}^{2}  \theta_{5}
+ 20 \theta_{3} \theta_{4}^{2}  \theta_{7}
\\&\ \quad
- 2 \theta_{1} \theta_{7} \theta_{15}
+ 2 \theta_{3} \theta_{5} \theta_{15}
- 22 \theta_{1} \theta_{4}^{2}  \theta_{5}
- 20 \theta_{1} \theta_{4}^{2}  \theta_{7}
+ 4 \theta_{10} \theta_{1}^{2}  \theta_{5}
+ 2 \theta_{3} \theta_{4} \theta_{14}
+ 6 \theta_{14} \theta_{3}^{3} 
+ 2 \theta_{1}^{3}  \theta_{15}
- 8 \theta_{1} \theta_{14} \theta_{3}^{2} 
+ 2 \theta_{1} \theta_{7} \theta_{14}
+ 2 \theta_{1} \theta_{5} \theta_{14}
\\&\ \quad
+ 32 \theta_{1}^{3}  \theta_{5} \theta_{7}
- 2 \theta_{1} \theta_{5} \theta_{15}
+ 29 \theta_{1} \theta_{7} \theta_{3} \theta_{8}
+ 8 \theta_{3} \theta_{7}^{2}  \theta_{5}
+ 39 \theta_{3}^{4}  \theta_{8}
- 14 \theta_{8} \theta_{1}^{2}  \theta_{5}
- 10 \theta_{8} \theta_{1}^{2}  \theta_{7}
- 2 \theta_{9} \theta_{1}^{4} 
+ 342 \theta_{1}^{5}  \theta_{3}^{2} 
- 677 \theta_{1}^{4}  \theta_{3}^{3} 
+ 862 \theta_{1}^{3}  \theta_{3}^{4} 
\\&\ \quad
- 700 \theta_{1}^{2}  \theta_{3}^{5} 
+ 330 \theta_{1} \theta_{3}^{6} 
- 162 \theta_{3}^{5}  \theta_{4}
+ 84 \theta_{3}^{5}  \theta_{5}
+ 102 \theta_{3}^{5}  \theta_{7}
- 18 \theta_{3}^{4}  \theta_{9}
- 77 \theta_{3}^{3}  \theta_{4}^{2} 
- 33 \theta_{3}^{3}  \theta_{5}^{2} 
- 39 \theta_{3}^{3}  \theta_{7}^{2} 
- 103 \theta_{3} \theta_{1}^{6} 
- 48 \theta_{1} \theta_{3} \theta_{4} \theta_{8}
\\&\ \quad
+ 35 \theta_{1} \theta_{5} \theta_{3} \theta_{8}
+ 27 \theta_{3}^{4}  \theta_{10}
- 2 \theta_{10} \theta_{1}^{4} 
- \tfrac{5}{2} \theta_{10}^{2}  \theta_{1}
- \tfrac{3}{2} \theta_{8} \theta_{3} \theta_{10}
- 14 \theta_{3} \theta_{4} \theta_{7}^{2} 
+ 2 \theta_{1} \theta_{4} \theta_{15}
+ 5 \theta_{4}^{2}  \theta_{8}
- 5 \theta_{10} \theta_{4}^{2} 
+ 10 \theta_{1} \theta_{4}^{3} 
- 10 \theta_{3} \theta_{4}^{3} 
\\&\ \quad
- 4 \theta_{1} \theta_{7}^{3} 
+ 4 \theta_{3} \theta_{7}^{3} 
+ 2 \theta_{7}^{2}  \theta_{8}
- 2 \theta_{10} \theta_{7}^{2} 
+ \tfrac{25}{4} \theta_{3} \theta_{10}^{2} 
+ \theta_{8} \theta_{15}
+ 2 \theta_{8} \theta_{5}^{2} 
- 6 \theta_{3}^{2}  \theta_{4} \theta_{9}
+ 6 \theta_{3}^{2}  \theta_{5} \theta_{9}
+ 6 \theta_{3}^{2}  \theta_{7} \theta_{9}
+ 8 \theta_{3} \theta_{1} \theta_{4} \theta_{9}
- 8 \theta_{3} \theta_{1} \theta_{5} \theta_{9}
\\&\ \quad
- 8 \theta_{3} \theta_{1} \theta_{7} \theta_{9}
- 2 \theta_{3} \theta_{7} \theta_{14}
- 2 \theta_{3} \theta_{5} \theta_{14}
+ 6 \theta_{3} \theta_{1}^{2}  \theta_{14}
+ 40 \theta_{1}^{5}  \theta_{4}
+ 36 \theta_{1}^{3}  \theta_{4}^{2} 
+ 12 \theta_{8} \theta_{1}^{4} 
- 55 \theta_{1}^{3}  \theta_{3} \theta_{8}
+ 100 \theta_{1}^{2}  \theta_{3}^{2}  \theta_{8}
- 94 \theta_{1} \theta_{3}^{3}  \theta_{8}
\\&\ \quad
- 20 \theta_{3}^{2}  \theta_{5} \theta_{8}
- 23 \theta_{3}^{2}  \theta_{7} \theta_{8}
+ 33 \theta_{3}^{2}  \theta_{4} \theta_{8}
+ 3 \theta_{9} \theta_{3} \theta_{8}
+ \theta_{9} \theta_{10} \theta_{1}
+ 2 \theta_{9} \theta_{1}^{2}  \theta_{5}
+ 2 \theta_{9} \theta_{1}^{2}  \theta_{7}
- 2 \theta_{10} \theta_{5}^{2} 
- \theta_{9} \theta_{8} \theta_{1}
- \theta_{10} \theta_{15}
,
\\
\alpha^{(5)}_{1} &= 
- \theta_{1}^{3} 
+ 3 \theta_{1}^{2}  \theta_{3}
- 4 \theta_{1} \theta_{3}^{2} 
+ 3 \theta_{3}^{3} 
- \theta_{1} \theta_{4}
+ \theta_{1} \theta_{5}
+ \theta_{1} \theta_{7}
+ \theta_{3} \theta_{4}
- \theta_{3} \theta_{5}
- \theta_{3} \theta_{7}
+ \tfrac{1}{2} \theta_{10}
- \tfrac{1}{2} \theta_{8}
,
\\
\bar{\alpha}^{(5)}_{1} &=  \theta_{1}^{3} 
- 3 \theta_{1}^{2}  \theta_{3}
+ 4 \theta_{1} \theta_{3}^{2} 
- 3 \theta_{3}^{3} 
+ \theta_{1} \theta_{4}
- \theta_{1} \theta_{5}
- \theta_{1} \theta_{7}
- \theta_{3} \theta_{4}
+ \theta_{3} \theta_{5}
+ \theta_{3} \theta_{7}
- \tfrac{1}{2} \theta_{10}
+ \tfrac{1}{2} \theta_{8}
,
\\
\bar{\alpha}^{(5)}_{2} &=  2 \theta_{7} \theta_{5}
+ 2 \theta_{5}^{2} 
- 6 \theta_{4} \theta_{5}
- 5 \theta_{4} \theta_{7}
+ 20 \theta_{1} \theta_{5} \theta_{3}
+ 14 \theta_{1} \theta_{7} \theta_{3}
- 33 \theta_{1} \theta_{3} \theta_{4}
+ 9 \theta_{3}^{4} 
+ 7 \theta_{1}^{4} 
+ \tfrac{5}{2} \theta_{10} \theta_{1}
- \tfrac{13}{2} \theta_{10} \theta_{3}
- 30 \theta_{1}^{3}  \theta_{3}
+ 50 \theta_{1}^{2}  \theta_{3}^{2} 
\\&\ \quad
- 39 \theta_{1} \theta_{3}^{3} 
- 9 \theta_{1}^{2}  \theta_{5}
- 5 \theta_{1}^{2}  \theta_{7}
- 10 \theta_{3}^{2}  \theta_{5}
- 13 \theta_{3}^{2}  \theta_{7}
+ 13 \theta_{1}^{2}  \theta_{4}
+ 23 \theta_{3}^{2}  \theta_{4}
- \theta_{9} \theta_{1}
+ 3 \theta_{9} \theta_{3}
+ \tfrac{5}{2} \theta_{8} \theta_{1}
- \tfrac{9}{2} \theta_{8} \theta_{3}
+ 2 \theta_{7}^{2} 
+ 5 \theta_{4}^{2} 
- \theta_{14}
+ \theta_{15}
,
\\
\bar{\alpha}^{(5)}_{3} &= 
- \tfrac{1}{4}\big( 2 \theta_{1}^{3} 
- 6 \theta_{1}^{2}  \theta_{3}
+ 8 \theta_{1} \theta_{3}^{2} 
- 6 \theta_{3}^{3} 
+ 2 \theta_{1} \theta_{4}
- 2 \theta_{1} \theta_{5}
- 2 \theta_{1} \theta_{7}
- 2 \theta_{3} \theta_{4}
+ 2 \theta_{3} \theta_{5}
+ 2 \theta_{3} \theta_{7}
- \theta_{10}
+ \theta_{8}\big)\big( 14 \theta_{1}^{4} 
- 60 \theta_{1}^{3}  \theta_{3}
+ 100 \theta_{1}^{2}  \theta_{3}^{2} 
\\&\ \ \ \qquad
- 78 \theta_{1} \theta_{3}^{3} 
+ 18 \theta_{3}^{4} 
+ 26 \theta_{1}^{2}  \theta_{4}
- 18 \theta_{1}^{2}  \theta_{5}
- 10 \theta_{1}^{2}  \theta_{7}
- 66 \theta_{1} \theta_{3} \theta_{4}
+ 40 \theta_{1} \theta_{3} \theta_{5}
+ 28 \theta_{1} \theta_{3} \theta_{7}
+ 46 \theta_{3}^{2}  \theta_{4}
- 20 \theta_{3}^{2}  \theta_{5}
- 26 \theta_{3}^{2}  \theta_{7}
\\&\ \ \ \qquad
+ 5 \theta_{1} \theta_{10}
+ 5 \theta_{1} \theta_{8}
- 2 \theta_{1} \theta_{9}
- 13 \theta_{10} \theta_{3}
- 9 \theta_{3} \theta_{8}
+ 6 \theta_{3} \theta_{9}
+ 10 \theta_{4}^{2} 
- 12 \theta_{4} \theta_{5}
- 10 \theta_{4} \theta_{7}
+ 4 \theta_{5}^{2} 
+ 4 \theta_{5} \theta_{7}
+ 4 \theta_{7}^{2} 
- 2 \theta_{14}
+ 2 \theta_{15}\big)
.
\end{align*}%
\end{footnotesize}%
Therefore, 
by  Corollary~\ref{cor:hom2},
there exists a $K$-algebra  isomorphism
$\varphi : P^f(\mathbb{E}_7) \to P(\mathbb{E}_7)$
defined on arrows by formulas from
\ref{secA:hom},
with the coefficients 
${\alpha}^{(i)}_{j_i}, \bar{\alpha}^{(i)}_{j_i} \in K$,
for $i=0,\dots,5$, $j = 1, \dots,  j_i$, 
$j_0 = 11, j_1 = 5, j_2 = 15, j_3 = 17, j_4 = 9, j_5 = 3$,
taking values defined in
\ref{secA:simpl} 
and computed in
\ref{secA:coeff}.
Consequently, the algebras 
$P(\mathbb{E}_7)$
and
$P^f(\mathbb{E}_7)$
are isomorphic.


\section{Calculations for type $\mathbb{E}_8$}
\label{app:E8}

In this appendix 
we present 
the sketch of the
construction of 
a $K$-algebra isomorphism 
$\varphi : P^f(\mathbb{E}_8) \to P(\mathbb{E}_8)$
for a given admissible deforming element $f$
according to the algorithm described in the main part of the article.
This appendix is organized as follows.
In \ref{secB:adm} we 
describe
the equations
for the coefficients of $f$ indicated by 
the admissibility condition of $f$.
In \ref{secB:hom} we present the general
form of the homomorphisms from
$P^f(\mathbb{E}_8)$ to $P(\mathbb{E}_8)$
(in particular we present the chosen basis elements
of $P(\mathbb{E}_8)$).
In \ref{secB:simpl} we describe the
chosen method of simplifying the system of equations
for the coefficients of an isomorphism from
$P^f(\mathbb{E}_8)$ to $P(\mathbb{E}_8)$.
\ref{secB:solutions} 
is devoted to 
considerations 
on solving the obtained system of equations
and analysing the obtained solution.

We recall that 
considered case
is much more complicated than the case of type $\mathbb{E}_7$ 
(described in \ref{app:E7})
and the full formulas obtained in the
calculations
are to long to be presented in the article.
Hence we only provide here suitable 
choices in the subsequent steps of the algorithm
and describe the general shape of obtained formulas
without presenting their full form 
(we refer for them to \cite{B:E8zip})
nor going into details.

We refer to 
the main part of the article
for description of
applied algorithms
as well as for theoretical background and more references.

\subsection{Equations derived from the admissibility condition}
\label{secB:adm}

We carry out the steps described in 
Section~\ref{sec:adm}.
First we compute a base of 
$R(\mathbb{E}_8)$
(see Algorithm~\ref{alg:base-REn} for details).
We 
take as a basis the set
$\{
1,
x, y, 
xy,
yx,
yy,
xyx,
xyy,
yxy,
yyx,
xyxy,
xyyx,
yxyx,
\linebreak
yxyy,
yyxy,
xyxyx,
xyxyy,
xyyxy,
yxyxy,
yxyyx,
yyxyx,
xyxyxy,
xyxyyx,
xyyxyx,
yxyxyx,
yxyxyy,
yxyyxy,
xyxyxyx,
xyxyxyy,
\linebreak
xyxyyxy,
yxyxyxy,
yxyxyyx,
yxyyxyx,
xyxyxyxy,
xyxyxyyx,
xyxyyxyx,
yxyxyxyx,
yxyxyxyy,
yxyxyyxy,
xyxyxyxyx,
xyxyxyxyy,
\linebreak
xyxyxyyxy,
yxyxyxyxy,
yxyxyxyyx,
yxyxyyxyx,
xyxyxyxyxy,
xyxyxyxyyx,
xyxyxyyxyx,
yxyxyxyxyx,
yxyxyxyxyy,
\linebreak
xyxyxyxyxyx,
xyxyxyxyxyy,
yxyxyxyxyxy,
yxyxyxyxyyx,
xyxyxyxyxyxy,
xyxyxyxyxyyx,
yxyxyxyxyxyx,
xyxyxyxyxyxyx,
\linebreak
yxyxyxyxyxyxy,
xyxyxyxyxyxyxy
\}$. 
Hen\-ce we know that each element $f \in \rad R(\mathbb{E}_8)$ is of the form
\begin{align*}
f(x,&y) = 
\theta_{1} xy
+ \theta_{2} yx
+ \theta_{3} yy
+ \theta_{4} xyx
+ \theta_{5} xyy
+ \theta_{6} yxy
+ \theta_{7} yyx
+ \theta_{8} xyxy
+ \theta_{9} xyyx
+ \theta_{10} yxyx
+ \theta_{11} yxyy
\\&
+ \theta_{12} yyxy
+ \theta_{13} xyxyx
+ \theta_{14} xyxyy
+ \theta_{15} xyyxy
+ \theta_{16} yxyxy
+ \theta_{17} yxyyx
+ \theta_{18} yyxyx
+ \theta_{19} xyxyxy
\\&
+ \theta_{20} xyxyyx
+ \theta_{21} xyyxyx
+ \theta_{22} yxyxyx
+ \theta_{23} yxyxyy
+ \theta_{24} yxyyxy
+ \theta_{25} xyxyxyx
+ \theta_{26} xyxyxyy
\\&
+ \theta_{27} xyxyyxy
+ \theta_{28} yxyxyxy
+ \theta_{29} yxyxyyx
+ \theta_{30} yxyyxyx
+ \theta_{31} xyxyxyxy
+ \theta_{32} xyxyxyyx
\\&
+ \theta_{33} xyxyyxyx
+ \theta_{34} yxyxyxyx
+ \theta_{35} yxyxyxyy
+ \theta_{36} yxyxyyxy
+ \theta_{37} xyxyxyxyx
+ \theta_{38} xyxyxyxyy
\\&
+ \theta_{39} xyxyxyyxy
+ \theta_{40} yxyxyxyxy
+ \theta_{41} yxyxyxyyx
+ \theta_{42} yxyxyyxyx
+ \theta_{43} xyxyxyxyxy
\\&
+ \theta_{44} xyxyxyxyyx
+ \theta_{45} xyxyxyyxyx
+ \theta_{46} yxyxyxyxyx
+ \theta_{47} yxyxyxyxyy
+ \theta_{48} xyxyxyxyxyx
\\&
+ \theta_{49} xyxyxyxyxyy
+ \theta_{50} yxyxyxyxyxy
+ \theta_{51} yxyxyxyxyyx
+ \theta_{52} xyxyxyxyxyxy
+ \theta_{53} xyxyxyxyxyyx
\\&
+ \theta_{54} yxyxyxyxyxyx
+ \theta_{55} xyxyxyxyxyxyx
+ \theta_{56} yxyxyxyxyxyxy
+ \theta_{57} xyxyxyxyxyxyxy
\end{align*}
for some $\theta_1, \dots, \theta_{57} \in K$.
We fix these coefficients.
Further, we compute 
$(x+y+f(x,y))^5$ 
(see Sections \ref{sec:adm} and \ref{sec:baseEn} for details).
We obtain the formula of the form 
\begin{align*}
\big(x+y&+f(x,y)\big)^5 = 
(\theta_{1}+\theta_{2}-2\theta_{3})xyxyxy
+(\theta_{1}+\theta_{2}-2\theta_{3})xyxyyx
+(\theta_{1}+\theta_{2}-2\theta_{3})xyyxyx
\\&
+(\theta_{1}+\theta_{2}-2\theta_{3})yxyxyx
+(3\theta_4-2\theta_5+\theta_6-2\theta_7+2\theta_1^2+
\theta_1\theta_2-2\theta_1\theta_3+2\theta_2^2
-2\theta_2\theta_3-\theta_3^2) xyxyxyx
\\&
+(3\theta_4-2\theta_5+\theta_6-2\theta_7+2\theta_1^2
+\theta_1\theta_2-2\theta_1\theta_3+2\theta_2^2
-2\theta_2\theta_3-\theta_3^2) yxyxyxy
\\&
+(3\theta_8-2\theta_9+3\theta_{10}-2\theta_{11}-2\theta_{12}
+6\theta_1\theta_4-4\theta_1\theta_5+2\theta_1\theta_6
-2\theta_1\theta_7
+3\theta_2\theta_4
+\theta_2\theta_6
\\&\ \ \quad
-2\theta_2\theta_7
-2\theta_3\theta_5
-2\theta_3\theta_7
+3\theta_1^3+\theta_1^2\theta_2-2\theta_1^2\theta_3
+2\theta_1\theta_2^2
-4\theta_1\theta_2\theta_3
-\theta_1\theta_3^2+\theta_2^3) xyxyxyxy 
+ \dots
\end{align*}
(where appear 30 basis elements of $R(\mathbb{E}_8)$).
See \cite[\texttt{e8-adm-power.txt}]{B:E8zip} for its full form.
Hence $f$ satisfy the admissibility condition if and only if 
the corresponding 30 
equalities 
(starting with ``$\theta_{1}+\theta_{2}-2\theta_{3} = 0$'')
are satisfied.
See \cite[\texttt{e8-adm-equations.txt}]{B:E8zip} for full list of these equations.

We recall that following 
\cite[Theorem]{B:soc}
there is a non-isomorphic deformation of $P(\mathbb{E}_8)$
in characteristic $2$.
Hence we may assume that $K$ is of characteristic different from $2$.
Then, applying Algorithm~\ref{alg:adm} to these equations
with the following chosen sequence of coefficients
\[
\theta_2,
\theta_6,
\theta_{12},
\theta_{17},
\theta_{18},
\theta_{19},
\theta_{28},
\theta_{36},
\theta_{40},
\theta_{45},
\]
we obtain the required set of 10 independent equations:
\begin{align*}
\theta_2&=2\theta_3-\theta_1,
\\
\theta_6&=-3\theta_1^2+6\theta_1\theta_3-3\theta_3^2-3\theta_4+2\theta_5+2\theta_7,
\\
\theta_{12}&=\theta_1^2\theta_3-2\theta_1\theta_3^2+\theta_3^3-\theta_1\theta_5
+\theta_1\theta_7+\theta_3\theta_5-\theta_3\theta_7
-\theta_{11}-\theta_9+\tfrac{3}{2}(\theta_{10}+\theta_8),
\\
\theta_{17}&=-9\theta_1^4+38\theta_1^3\theta_3-62\theta_1^2\theta_3^2
+46\theta_1\theta_3^3-13\theta_3^4-27\theta_1^2\theta_4+10\theta_1^2\theta_5
+14\theta_1^2\theta_7+54\theta_1\theta_3\theta_4
\\&\,\quad
-18\theta_1\theta_3\theta_5
-30\theta_1\theta_3\theta_7-27\theta_3^2\theta_4+8\theta_3^2\theta_5
+16\theta_3^2\theta_7-4\theta_1\theta_11+6\theta_1\theta_8-2\theta_1\theta_9
+4\theta_11\theta_3
\\&\,\quad
-6\theta_3\theta_8+2\theta_3\theta_9-18\theta_4^2
+18\theta_4\theta_5+18\theta_4\theta_7-5\theta_5^2-8\theta_5\theta_7-5\theta_7^2
+3\theta_{13}-2\theta_{15}+3\theta_{16},
\\&\ \ \vdots
\\
\theta_{45}&=4\theta_1\theta_4\theta_9^2-\theta_{10}\theta_{27}-3\theta_{10}\theta_5\theta_{14}
-2\theta_7\theta_9^2\theta_1+5\theta_{14}\theta_3\theta_4^2+\tfrac{15}{2}\theta_{10}\theta_{14}\theta_3^2
+2\theta_{14}\theta_5\theta_9+\theta_{20}\theta_4\theta_5
\\&\,\quad
-\theta_{22}\theta_5\theta_7
-\theta_{21}\theta_4\theta_7-\theta_{23}\theta_4^2-\theta_{20}\theta_4^2
+49\theta_3\theta_5\theta_7\theta_8\theta_1
-93\theta_1\theta_5\theta_7\theta_11\theta_3
\\&\,\quad
-\tfrac{1421}{2}\theta_1^2\theta_3\theta_4\theta_5\theta_7
+746\theta_1\theta_3^2\theta_4\theta_5\theta_7
+109\theta_1\theta_3\theta_4\theta_5\theta_9
+111\theta_1\theta_3\theta_4\theta_7\theta_9 
+ \dots
\end{align*}%
See 
\cite[\texttt{e8-substitutions.txt}]{B:E8zip}
for list of these equations
(we note that only the last of them alone would take about 8--10 pages 
of the article, so it is not possible to show all them here in their full form).

\subsection{General form of homomorphism}
\label{secB:hom}

Let $f$ be a given admissible element 
of the structure described in 
\ref{secB:adm}.
In order to construct an isomorphism 
$\varphi : P^f(\mathbb{E}_8) \to P(\mathbb{E}_8)$
we want to  
find the coefficients satisfying the assumptions
of Corollary~\ref{cor:hom2}.
We start with 
computing the base of $P(\mathbb{E}_8)$,
according to Algorithm~\ref{alg:base-PEn}.
In particular, for each arrow $\alpha \in Q_1$ we
compute a basis of 
$e_{s(\alpha)} P(\mathbb{E}_8) e_{t(\alpha)}$.
Then we conclude that our constructed isomorphism
should be given by the following formulas
\begin{footnotesize}
\setlength{\jot}{0pt}
	\begin{align*}
	\varphi(a_{0}) &= 
	a_{0}
	+ \alpha^{(0)}_{1} a_{0} \bar{a}_{2} a_{2}
	+ \alpha^{(0)}_{2} a_{0} \bar{a}_{2} a_{2} \bar{a}_{0} a_{0}
	+ \alpha^{(0)}_{3} a_{0} \bar{a}_{2} a_{2} \bar{a}_{2} a_{2}
	+ \alpha^{(0)}_{4} a_{0} \bar{a}_{2} a_{2} \bar{a}_{0} a_{0} \bar{a}_{2} a_{2}
	+ \alpha^{(0)}_{5} a_{0} \bar{a}_{2} a_{2} \bar{a}_{2} a_{2} \bar{a}_{0} a_{0}
	+ \alpha^{(0)}_{6} a_{0} \bar{a}_{2} a_{2} \bar{a}_{0} a_{0} \bar{a}_{2} a_{2} \bar{a}_{0} a_{0}
	\\&\quad
	+ \alpha^{(0)}_{7} a_{0} \bar{a}_{2} a_{2} \bar{a}_{0} a_{0} \bar{a}_{2} a_{2} \bar{a}_{2} a_{2}
	+ \alpha^{(0)}_{8} a_{0} \bar{a}_{2} a_{2} \bar{a}_{2} a_{2} \bar{a}_{0} a_{0} \bar{a}_{2} a_{2}
	+ \alpha^{(0)}_{9} a_{0} \bar{a}_{2} a_{2} \bar{a}_{0} a_{0} \bar{a}_{2} a_{2} \bar{a}_{0} a_{0} \bar{a}_{2} a_{2}
	+ \alpha^{(0)}_{10} a_{0} \bar{a}_{2} a_{2} \bar{a}_{0} a_{0} \bar{a}_{2} a_{2} \bar{a}_{2} a_{2} \bar{a}_{0} a_{0}
	\\&\quad
	+ \alpha^{(0)}_{11} a_{0} \bar{a}_{2} a_{2} \bar{a}_{2} a_{2} \bar{a}_{0} a_{0} \bar{a}_{2} a_{2} \bar{a}_{0} a_{0}
	+ \alpha^{(0)}_{12} a_{0} \bar{a}_{2} a_{2} \bar{a}_{0} a_{0} \bar{a}_{2} a_{2} \bar{a}_{0} a_{0} \bar{a}_{2} a_{2} \bar{a}_{0} a_{0}
	+ \alpha^{(0)}_{13} a_{0} \bar{a}_{2} a_{2} \bar{a}_{0} a_{0} \bar{a}_{2} a_{2} \bar{a}_{0} a_{0} \bar{a}_{2} a_{2} \bar{a}_{2} a_{2}
	\\&\quad
	+ \alpha^{(0)}_{14} a_{0} \bar{a}_{2} a_{2} \bar{a}_{0} a_{0} \bar{a}_{2} a_{2} \bar{a}_{2} a_{2} \bar{a}_{0} a_{0} \bar{a}_{2} a_{2}
	+ \alpha^{(0)}_{15} a_{0} \bar{a}_{2} a_{2} \bar{a}_{0} a_{0} \bar{a}_{2} a_{2} \bar{a}_{0} a_{0} \bar{a}_{2} a_{2} \bar{a}_{0} a_{0} \bar{a}_{2} a_{2}
	+ \alpha^{(0)}_{16} a_{0} \bar{a}_{2} a_{2} \bar{a}_{0} a_{0} \bar{a}_{2} a_{2} \bar{a}_{0} a_{0} \bar{a}_{2} a_{2} \bar{a}_{2} a_{2} \bar{a}_{0} a_{0}
	\\&\quad
	+ \alpha^{(0)}_{17} a_{0} \bar{a}_{2} a_{2} \bar{a}_{0} a_{0} \bar{a}_{2} a_{2} \bar{a}_{2} a_{2} \bar{a}_{0} a_{0} \bar{a}_{2} a_{2} \bar{a}_{0} a_{0}
	+ \alpha^{(0)}_{18} a_{0} \bar{a}_{2} a_{2} \bar{a}_{0} a_{0} \bar{a}_{2} a_{2} \bar{a}_{0} a_{0} \bar{a}_{2} a_{2} \bar{a}_{0} a_{0} \bar{a}_{2} a_{2} \bar{a}_{0} a_{0}
	\\&\quad
	+ \alpha^{(0)}_{19} a_{0} \bar{a}_{2} a_{2} \bar{a}_{0} a_{0} \bar{a}_{2} a_{2} \bar{a}_{0} a_{0} \bar{a}_{2} a_{2} \bar{a}_{0} a_{0} \bar{a}_{2} a_{2} \bar{a}_{2} a_{2}
	+ \alpha^{(0)}_{20} a_{0} \bar{a}_{2} a_{2} \bar{a}_{0} a_{0} \bar{a}_{2} a_{2} \bar{a}_{0} a_{0} \bar{a}_{2} a_{2} \bar{a}_{2} a_{2} \bar{a}_{0} a_{0} \bar{a}_{2} a_{2}
	\\&\quad
	+ \alpha^{(0)}_{21} a_{0} \bar{a}_{2} a_{2} \bar{a}_{0} a_{0} \bar{a}_{2} a_{2} \bar{a}_{0} a_{0} \bar{a}_{2} a_{2} \bar{a}_{0} a_{0} \bar{a}_{2} a_{2} \bar{a}_{0} a_{0} \bar{a}_{2} a_{2}
	+ \alpha^{(0)}_{22} a_{0} \bar{a}_{2} a_{2} \bar{a}_{0} a_{0} \bar{a}_{2} a_{2} \bar{a}_{0} a_{0} \bar{a}_{2} a_{2} \bar{a}_{0} a_{0} \bar{a}_{2} a_{2} \bar{a}_{2} a_{2} \bar{a}_{0} a_{0}
	\\&\quad
	+ \alpha^{(0)}_{23} a_{0} \bar{a}_{2} a_{2} \bar{a}_{0} a_{0} \bar{a}_{2} a_{2} \bar{a}_{0} a_{0} \bar{a}_{2} a_{2} \bar{a}_{2} a_{2} \bar{a}_{0} a_{0} \bar{a}_{2} a_{2} \bar{a}_{0} a_{0}
	+ \alpha^{(0)}_{24} a_{0} \bar{a}_{2} a_{2} \bar{a}_{0} a_{0} \bar{a}_{2} a_{2} \bar{a}_{0} a_{0} \bar{a}_{2} a_{2} \bar{a}_{0} a_{0} \bar{a}_{2} a_{2} \bar{a}_{0} a_{0} \bar{a}_{2} a_{2} \bar{a}_{0} a_{0}
	\\&\quad
	+ \alpha^{(0)}_{25} a_{0} \bar{a}_{2} a_{2} \bar{a}_{0} a_{0} \bar{a}_{2} a_{2} \bar{a}_{0} a_{0} \bar{a}_{2} a_{2} \bar{a}_{0} a_{0} \bar{a}_{2} a_{2} \bar{a}_{0} a_{0} \bar{a}_{2} a_{2} \bar{a}_{2} a_{2}
	+ \alpha^{(0)}_{26} a_{0} \bar{a}_{2} a_{2} \bar{a}_{0} a_{0} \bar{a}_{2} a_{2} \bar{a}_{0} a_{0} \bar{a}_{2} a_{2} \bar{a}_{0} a_{0} \bar{a}_{2} a_{2} \bar{a}_{0} a_{0} \bar{a}_{2} a_{2} \bar{a}_{0} a_{0} \bar{a}_{2} a_{2}
	\\&\quad
	+ \alpha^{(0)}_{27} a_{0} \bar{a}_{2} a_{2} \bar{a}_{0} a_{0} \bar{a}_{2} a_{2} \bar{a}_{0} a_{0} \bar{a}_{2} a_{2} \bar{a}_{0} a_{0} \bar{a}_{2} a_{2} \bar{a}_{0} a_{0} \bar{a}_{2} a_{2} \bar{a}_{2} a_{2} \bar{a}_{0} a_{0}
	+ \alpha^{(0)}_{28} a_{0} \bar{a}_{2} a_{2} \bar{a}_{0} a_{0} \bar{a}_{2} a_{2} \bar{a}_{0} a_{0} \bar{a}_{2} a_{2} \bar{a}_{0} a_{0} \bar{a}_{2} a_{2} \bar{a}_{0} a_{0} \bar{a}_{2} a_{2} \bar{a}_{0} a_{0} \bar{a}_{2} a_{2} \bar{a}_{0} a_{0}
	\\&\quad
	+ \alpha^{(0)}_{29} a_{0} \bar{a}_{2} a_{2} \bar{a}_{0} a_{0} \bar{a}_{2} a_{2} \bar{a}_{0} a_{0} \bar{a}_{2} a_{2} \bar{a}_{0} a_{0} \bar{a}_{2} a_{2} \bar{a}_{0} a_{0} \bar{a}_{2} a_{2} \bar{a}_{0} a_{0} \bar{a}_{2} a_{2} \bar{a}_{0} a_{0} \bar{a}_{2} a_{2}
	\\
	\varphi(\bar{a}_{0}) &= 
	\bar{a}_{0}
	+ \bar{\alpha}^{(0)}_{1} \bar{a}_{2} a_{2} \bar{a}_{0}
	+ \bar{\alpha}^{(0)}_{2} \bar{a}_{0} a_{0} \bar{a}_{2} a_{2} \bar{a}_{0}
	+ \bar{\alpha}^{(0)}_{3} \bar{a}_{2} a_{2} \bar{a}_{2} a_{2} \bar{a}_{0}
	+ \bar{\alpha}^{(0)}_{4} \bar{a}_{0} a_{0} \bar{a}_{2} a_{2} \bar{a}_{2} a_{2} \bar{a}_{0}
	+ \bar{\alpha}^{(0)}_{5} \bar{a}_{2} a_{2} \bar{a}_{0} a_{0} \bar{a}_{2} a_{2} \bar{a}_{0}
	+ \bar{\alpha}^{(0)}_{6} \bar{a}_{0} a_{0} \bar{a}_{2} a_{2} \bar{a}_{0} a_{0} \bar{a}_{2} a_{2} \bar{a}_{0}
	\\&\quad
	+ \bar{\alpha}^{(0)}_{7} \bar{a}_{2} a_{2} \bar{a}_{0} a_{0} \bar{a}_{2} a_{2} \bar{a}_{2} a_{2} \bar{a}_{0}
	+ \bar{\alpha}^{(0)}_{8} \bar{a}_{2} a_{2} \bar{a}_{2} a_{2} \bar{a}_{0} a_{0} \bar{a}_{2} a_{2} \bar{a}_{0}
	+ \bar{\alpha}^{(0)}_{9} \bar{a}_{0} a_{0} \bar{a}_{2} a_{2} \bar{a}_{0} a_{0} \bar{a}_{2} a_{2} \bar{a}_{2} a_{2} \bar{a}_{0}
	+ \bar{\alpha}^{(0)}_{10} \bar{a}_{0} a_{0} \bar{a}_{2} a_{2} \bar{a}_{2} a_{2} \bar{a}_{0} a_{0} \bar{a}_{2} a_{2} \bar{a}_{0}
	\\&\quad
	+ \bar{\alpha}^{(0)}_{11} \bar{a}_{2} a_{2} \bar{a}_{0} a_{0} \bar{a}_{2} a_{2} \bar{a}_{0} a_{0} \bar{a}_{2} a_{2} \bar{a}_{0}
	+ \bar{\alpha}^{(0)}_{12} \bar{a}_{0} a_{0} \bar{a}_{2} a_{2} \bar{a}_{0} a_{0} \bar{a}_{2} a_{2} \bar{a}_{0} a_{0} \bar{a}_{2} a_{2} \bar{a}_{0}
	+ \bar{\alpha}^{(0)}_{13} \bar{a}_{2} a_{2} \bar{a}_{0} a_{0} \bar{a}_{2} a_{2} \bar{a}_{0} a_{0} \bar{a}_{2} a_{2} \bar{a}_{2} a_{2} \bar{a}_{0}
	\\&\quad
	+ \bar{\alpha}^{(0)}_{14} \bar{a}_{2} a_{2} \bar{a}_{0} a_{0} \bar{a}_{2} a_{2} \bar{a}_{2} a_{2} \bar{a}_{0} a_{0} \bar{a}_{2} a_{2} \bar{a}_{0}
	+ \bar{\alpha}^{(0)}_{15} \bar{a}_{0} a_{0} \bar{a}_{2} a_{2} \bar{a}_{0} a_{0} \bar{a}_{2} a_{2} \bar{a}_{0} a_{0} \bar{a}_{2} a_{2} \bar{a}_{2} a_{2} \bar{a}_{0}
	+ \bar{\alpha}^{(0)}_{16} \bar{a}_{0} a_{0} \bar{a}_{2} a_{2} \bar{a}_{0} a_{0} \bar{a}_{2} a_{2} \bar{a}_{2} a_{2} \bar{a}_{0} a_{0} \bar{a}_{2} a_{2} \bar{a}_{0}
	\\&\quad
	+ \bar{\alpha}^{(0)}_{17} \bar{a}_{2} a_{2} \bar{a}_{0} a_{0} \bar{a}_{2} a_{2} \bar{a}_{0} a_{0} \bar{a}_{2} a_{2} \bar{a}_{0} a_{0} \bar{a}_{2} a_{2} \bar{a}_{0}
	+ \bar{\alpha}^{(0)}_{18} \bar{a}_{0} a_{0} \bar{a}_{2} a_{2} \bar{a}_{0} a_{0} \bar{a}_{2} a_{2} \bar{a}_{0} a_{0} \bar{a}_{2} a_{2} \bar{a}_{0} a_{0} \bar{a}_{2} a_{2} \bar{a}_{0}
	\\&\quad
	+ \bar{\alpha}^{(0)}_{19} \bar{a}_{2} a_{2} \bar{a}_{0} a_{0} \bar{a}_{2} a_{2} \bar{a}_{0} a_{0} \bar{a}_{2} a_{2} \bar{a}_{0} a_{0} \bar{a}_{2} a_{2} \bar{a}_{2} a_{2} \bar{a}_{0}
	+ \bar{\alpha}^{(0)}_{20} \bar{a}_{2} a_{2} \bar{a}_{0} a_{0} \bar{a}_{2} a_{2} \bar{a}_{0} a_{0} \bar{a}_{2} a_{2} \bar{a}_{2} a_{2} \bar{a}_{0} a_{0} \bar{a}_{2} a_{2} \bar{a}_{0}
	\\&\quad
	+ \bar{\alpha}^{(0)}_{21} \bar{a}_{0} a_{0} \bar{a}_{2} a_{2} \bar{a}_{0} a_{0} \bar{a}_{2} a_{2} \bar{a}_{0} a_{0} \bar{a}_{2} a_{2} \bar{a}_{0} a_{0} \bar{a}_{2} a_{2} \bar{a}_{2} a_{2} \bar{a}_{0}
	+ \bar{\alpha}^{(0)}_{22} \bar{a}_{0} a_{0} \bar{a}_{2} a_{2} \bar{a}_{0} a_{0} \bar{a}_{2} a_{2} \bar{a}_{0} a_{0} \bar{a}_{2} a_{2} \bar{a}_{2} a_{2} \bar{a}_{0} a_{0} \bar{a}_{2} a_{2} \bar{a}_{0}
	\\&\quad
	+ \bar{\alpha}^{(0)}_{23} \bar{a}_{2} a_{2} \bar{a}_{0} a_{0} \bar{a}_{2} a_{2} \bar{a}_{0} a_{0} \bar{a}_{2} a_{2} \bar{a}_{0} a_{0} \bar{a}_{2} a_{2} \bar{a}_{0} a_{0} \bar{a}_{2} a_{2} \bar{a}_{0}
	+ \bar{\alpha}^{(0)}_{24} \bar{a}_{0} a_{0} \bar{a}_{2} a_{2} \bar{a}_{0} a_{0} \bar{a}_{2} a_{2} \bar{a}_{0} a_{0} \bar{a}_{2} a_{2} \bar{a}_{0} a_{0} \bar{a}_{2} a_{2} \bar{a}_{0} a_{0} \bar{a}_{2} a_{2} \bar{a}_{0}
	\\&\quad
	+ \bar{\alpha}^{(0)}_{25} \bar{a}_{2} a_{2} \bar{a}_{0} a_{0} \bar{a}_{2} a_{2} \bar{a}_{0} a_{0} \bar{a}_{2} a_{2} \bar{a}_{0} a_{0} \bar{a}_{2} a_{2} \bar{a}_{0} a_{0} \bar{a}_{2} a_{2} \bar{a}_{2} a_{2} \bar{a}_{0}
	+ \bar{\alpha}^{(0)}_{26} \bar{a}_{0} a_{0} \bar{a}_{2} a_{2} \bar{a}_{0} a_{0} \bar{a}_{2} a_{2} \bar{a}_{0} a_{0} \bar{a}_{2} a_{2} \bar{a}_{0} a_{0} \bar{a}_{2} a_{2} \bar{a}_{0} a_{0} \bar{a}_{2} a_{2} \bar{a}_{2} a_{2} \bar{a}_{0}
\end{align*}
\begin{align*}
	&\quad
	+ \bar{\alpha}^{(0)}_{27} \bar{a}_{2} a_{2} \bar{a}_{0} a_{0} \bar{a}_{2} a_{2} \bar{a}_{0} a_{0} \bar{a}_{2} a_{2} \bar{a}_{0} a_{0} \bar{a}_{2} a_{2} \bar{a}_{0} a_{0} \bar{a}_{2} a_{2} \bar{a}_{0} a_{0} \bar{a}_{2} a_{2} \bar{a}_{0}
	+ \bar{\alpha}^{(0)}_{28} \bar{a}_{0} a_{0} \bar{a}_{2} a_{2} \bar{a}_{0} a_{0} \bar{a}_{2} a_{2} \bar{a}_{0} a_{0} \bar{a}_{2} a_{2} \bar{a}_{0} a_{0} \bar{a}_{2} a_{2} \bar{a}_{0} a_{0} \bar{a}_{2} a_{2} \bar{a}_{0} a_{0} \bar{a}_{2} a_{2} \bar{a}_{0}
	\\&\quad
	+ \bar{\alpha}^{(0)}_{29} \bar{a}_{2} a_{2} \bar{a}_{0} a_{0} \bar{a}_{2} a_{2} \bar{a}_{0} a_{0} \bar{a}_{2} a_{2} \bar{a}_{0} a_{0} \bar{a}_{2} a_{2} \bar{a}_{0} a_{0} \bar{a}_{2} a_{2} \bar{a}_{0} a_{0} \bar{a}_{2} a_{2} \bar{a}_{0} a_{0} \bar{a}_{2} a_{2} \bar{a}_{0}
	\\
	\varphi(a_{1}) &= 
	a_{1}
	+ \alpha^{(1)}_{1} a_{1} a_{2} \bar{a}_{0} a_{0} \bar{a}_{2}
	+ \alpha^{(1)}_{2} a_{1} a_{2} \bar{a}_{0} a_{0} \bar{a}_{2} a_{2} \bar{a}_{2}
	+ \alpha^{(1)}_{3} a_{1} a_{2} \bar{a}_{0} a_{0} \bar{a}_{2} a_{2} \bar{a}_{0} a_{0} \bar{a}_{2}
	+ \alpha^{(1)}_{4} a_{1} a_{2} \bar{a}_{0} a_{0} \bar{a}_{2} a_{2} \bar{a}_{0} a_{0} \bar{a}_{2} a_{2} \bar{a}_{2}
	\\&\quad
	+ \alpha^{(1)}_{5} a_{1} a_{2} \bar{a}_{0} a_{0} \bar{a}_{2} a_{2} \bar{a}_{2} a_{2} \bar{a}_{0} a_{0} \bar{a}_{2}
	+ \alpha^{(1)}_{6} a_{1} a_{2} \bar{a}_{0} a_{0} \bar{a}_{2} a_{2} \bar{a}_{0} a_{0} \bar{a}_{2} a_{2} \bar{a}_{0} a_{0} \bar{a}_{2}
	+ \alpha^{(1)}_{7} a_{1} a_{2} \bar{a}_{0} a_{0} \bar{a}_{2} a_{2} \bar{a}_{0} a_{0} \bar{a}_{2} a_{2} \bar{a}_{2} a_{2} \bar{a}_{0} a_{0} \bar{a}_{2}
	\\&\quad
	+ \alpha^{(1)}_{8} a_{1} a_{2} \bar{a}_{0} a_{0} \bar{a}_{2} a_{2} \bar{a}_{0} a_{0} \bar{a}_{2} a_{2} \bar{a}_{0} a_{0} \bar{a}_{2} a_{2} \bar{a}_{0} a_{0} \bar{a}_{2}
	+ \alpha^{(1)}_{9} a_{1} a_{2} \bar{a}_{0} a_{0} \bar{a}_{2} a_{2} \bar{a}_{0} a_{0} \bar{a}_{2} a_{2} \bar{a}_{2} a_{2} \bar{a}_{0} a_{0} \bar{a}_{2} a_{2} \bar{a}_{2}
	\\&\quad
	+ \alpha^{(1)}_{10} a_{1} a_{2} \bar{a}_{0} a_{0} \bar{a}_{2} a_{2} \bar{a}_{0} a_{0} \bar{a}_{2} a_{2} \bar{a}_{0} a_{0} \bar{a}_{2} a_{2} \bar{a}_{0} a_{0} \bar{a}_{2} a_{2} \bar{a}_{2}
	+ \alpha^{(1)}_{11} a_{1} a_{2} \bar{a}_{0} a_{0} \bar{a}_{2} a_{2} \bar{a}_{0} a_{0} \bar{a}_{2} a_{2} \bar{a}_{0} a_{0} \bar{a}_{2} a_{2} \bar{a}_{0} a_{0} \bar{a}_{2} a_{2} \bar{a}_{0} a_{0} \bar{a}_{2}
	\\&\quad
	+ \alpha^{(1)}_{12} a_{1} a_{2} \bar{a}_{0} a_{0} \bar{a}_{2} a_{2} \bar{a}_{0} a_{0} \bar{a}_{2} a_{2} \bar{a}_{0} a_{0} \bar{a}_{2} a_{2} \bar{a}_{0} a_{0} \bar{a}_{2} a_{2} \bar{a}_{0} a_{0} \bar{a}_{2} a_{2} \bar{a}_{2}
	\\&\quad
	+ \alpha^{(1)}_{13} a_{1} a_{2} \bar{a}_{0} a_{0} \bar{a}_{2} a_{2} \bar{a}_{0} a_{0} \bar{a}_{2} a_{2} \bar{a}_{0} a_{0} \bar{a}_{2} a_{2} \bar{a}_{0} a_{0} \bar{a}_{2} a_{2} \bar{a}_{0} a_{0} \bar{a}_{2} a_{2} \bar{a}_{2} a_{2} \bar{a}_{0} a_{0} \bar{a}_{2}
	\\
	\varphi(\bar{a}_{1}) &= 
	\bar{a}_{1}
	+ \bar{\alpha}^{(1)}_{1} a_{2} \bar{a}_{0} a_{0} \bar{a}_{2} \bar{a}_{1}
	+ \bar{\alpha}^{(1)}_{2} a_{2} \bar{a}_{2} a_{2} \bar{a}_{0} a_{0} \bar{a}_{2} \bar{a}_{1}
	+ \bar{\alpha}^{(1)}_{3} a_{2} \bar{a}_{0} a_{0} \bar{a}_{2} a_{2} \bar{a}_{0} a_{0} \bar{a}_{2} \bar{a}_{1}
	+ \bar{\alpha}^{(1)}_{4} a_{2} \bar{a}_{0} a_{0} \bar{a}_{2} a_{2} \bar{a}_{2} a_{2} \bar{a}_{0} a_{0} \bar{a}_{2} \bar{a}_{1}
	\\&\quad
	+ \bar{\alpha}^{(1)}_{5} a_{2} \bar{a}_{2} a_{2} \bar{a}_{0} a_{0} \bar{a}_{2} a_{2} \bar{a}_{0} a_{0} \bar{a}_{2} \bar{a}_{1}
	+ \bar{\alpha}^{(1)}_{6} a_{2} \bar{a}_{0} a_{0} \bar{a}_{2} a_{2} \bar{a}_{0} a_{0} \bar{a}_{2} a_{2} \bar{a}_{0} a_{0} \bar{a}_{2} \bar{a}_{1}
	+ \bar{\alpha}^{(1)}_{7} a_{2} \bar{a}_{0} a_{0} \bar{a}_{2} a_{2} \bar{a}_{0} a_{0} \bar{a}_{2} a_{2} \bar{a}_{2} a_{2} \bar{a}_{0} a_{0} \bar{a}_{2} \bar{a}_{1}
	\\&\quad
	+ \bar{\alpha}^{(1)}_{8} a_{2} \bar{a}_{0} a_{0} \bar{a}_{2} a_{2} \bar{a}_{0} a_{0} \bar{a}_{2} a_{2} \bar{a}_{0} a_{0} \bar{a}_{2} a_{2} \bar{a}_{0} a_{0} \bar{a}_{2} \bar{a}_{1}
	+ \bar{\alpha}^{(1)}_{9} a_{2} \bar{a}_{2} a_{2} \bar{a}_{0} a_{0} \bar{a}_{2} a_{2} \bar{a}_{0} a_{0} \bar{a}_{2} a_{2} \bar{a}_{2} a_{2} \bar{a}_{0} a_{0} \bar{a}_{2} \bar{a}_{1}
	\\&\quad
	+ \bar{\alpha}^{(1)}_{10} a_{2} \bar{a}_{0} a_{0} \bar{a}_{2} a_{2} \bar{a}_{0} a_{0} \bar{a}_{2} a_{2} \bar{a}_{0} a_{0} \bar{a}_{2} a_{2} \bar{a}_{2} a_{2} \bar{a}_{0} a_{0} \bar{a}_{2} \bar{a}_{1}
	+ \bar{\alpha}^{(1)}_{11} a_{2} \bar{a}_{0} a_{0} \bar{a}_{2} a_{2} \bar{a}_{0} a_{0} \bar{a}_{2} a_{2} \bar{a}_{0} a_{0} \bar{a}_{2} a_{2} \bar{a}_{0} a_{0} \bar{a}_{2} a_{2} \bar{a}_{0} a_{0} \bar{a}_{2} \bar{a}_{1}
	\\&\quad
	+ \bar{\alpha}^{(1)}_{12} a_{2} \bar{a}_{0} a_{0} \bar{a}_{2} a_{2} \bar{a}_{0} a_{0} \bar{a}_{2} a_{2} \bar{a}_{0} a_{0} \bar{a}_{2} a_{2} \bar{a}_{0} a_{0} \bar{a}_{2} a_{2} \bar{a}_{2} a_{2} \bar{a}_{0} a_{0} \bar{a}_{2} \bar{a}_{1}
	\\&\quad
	+ \bar{\alpha}^{(1)}_{13} a_{2} \bar{a}_{0} a_{0} \bar{a}_{2} a_{2} \bar{a}_{0} a_{0} \bar{a}_{2} a_{2} \bar{a}_{0} a_{0} \bar{a}_{2} a_{2} \bar{a}_{0} a_{0} \bar{a}_{2} a_{2} \bar{a}_{0} a_{0} \bar{a}_{2} a_{2} \bar{a}_{2} a_{2} \bar{a}_{0} a_{0} \bar{a}_{2} \bar{a}_{1}
	\\
	\varphi(a_{2}) &= 
	a_{2}
	+ \alpha^{(2)}_{1} a_{2} \bar{a}_{0} a_{0}
	+ \alpha^{(2)}_{2} a_{2} \bar{a}_{2} a_{2}
	+ \alpha^{(2)}_{3} a_{2} \bar{a}_{0} a_{0} \bar{a}_{2} a_{2}
	+ \alpha^{(2)}_{4} a_{2} \bar{a}_{2} a_{2} \bar{a}_{0} a_{0}
	+ \alpha^{(2)}_{5} a_{2} \bar{a}_{0} a_{0} \bar{a}_{2} a_{2} \bar{a}_{0} a_{0}
	+ \alpha^{(2)}_{6} a_{2} \bar{a}_{0} a_{0} \bar{a}_{2} a_{2} \bar{a}_{2} a_{2}
	\\&\quad
	+ \alpha^{(2)}_{7} a_{2} \bar{a}_{2} a_{2} \bar{a}_{0} a_{0} \bar{a}_{2} a_{2}
	+ \alpha^{(2)}_{8} a_{2} \bar{a}_{0} a_{0} \bar{a}_{2} a_{2} \bar{a}_{0} a_{0} \bar{a}_{2} a_{2}
	+ \alpha^{(2)}_{9} a_{2} \bar{a}_{0} a_{0} \bar{a}_{2} a_{2} \bar{a}_{2} a_{2} \bar{a}_{0} a_{0}
	+ \alpha^{(2)}_{10} a_{2} \bar{a}_{2} a_{2} \bar{a}_{0} a_{0} \bar{a}_{2} a_{2} \bar{a}_{0} a_{0}
	\\&\quad
	+ \alpha^{(2)}_{11} a_{2} \bar{a}_{2} a_{2} \bar{a}_{0} a_{0} \bar{a}_{2} a_{2} \bar{a}_{2} a_{2}
	+ \alpha^{(2)}_{12} a_{2} \bar{a}_{0} a_{0} \bar{a}_{2} a_{2} \bar{a}_{0} a_{0} \bar{a}_{2} a_{2} \bar{a}_{0} a_{0}
	+ \alpha^{(2)}_{13} a_{2} \bar{a}_{0} a_{0} \bar{a}_{2} a_{2} \bar{a}_{0} a_{0} \bar{a}_{2} a_{2} \bar{a}_{2} a_{2}
	+ \alpha^{(2)}_{14} a_{2} \bar{a}_{0} a_{0} \bar{a}_{2} a_{2} \bar{a}_{2} a_{2} \bar{a}_{0} a_{0} \bar{a}_{2} a_{2}
	\\&\quad
	+ \alpha^{(2)}_{15} a_{2} \bar{a}_{2} a_{2} \bar{a}_{0} a_{0} \bar{a}_{2} a_{2} \bar{a}_{0} a_{0} \bar{a}_{2} a_{2}
	+ \alpha^{(2)}_{16} a_{2} \bar{a}_{0} a_{0} \bar{a}_{2} a_{2} \bar{a}_{0} a_{0} \bar{a}_{2} a_{2} \bar{a}_{0} a_{0} \bar{a}_{2} a_{2}
	+ \alpha^{(2)}_{17} a_{2} \bar{a}_{0} a_{0} \bar{a}_{2} a_{2} \bar{a}_{0} a_{0} \bar{a}_{2} a_{2} \bar{a}_{2} a_{2} \bar{a}_{0} a_{0}
	\\&\quad
	+ \alpha^{(2)}_{18} a_{2} \bar{a}_{0} a_{0} \bar{a}_{2} a_{2} \bar{a}_{2} a_{2} \bar{a}_{0} a_{0} \bar{a}_{2} a_{2} \bar{a}_{0} a_{0}
	+ \alpha^{(2)}_{19} a_{2} \bar{a}_{2} a_{2} \bar{a}_{0} a_{0} \bar{a}_{2} a_{2} \bar{a}_{0} a_{0} \bar{a}_{2} a_{2} \bar{a}_{2} a_{2}
	+ \alpha^{(2)}_{20} a_{2} \bar{a}_{0} a_{0} \bar{a}_{2} a_{2} \bar{a}_{0} a_{0} \bar{a}_{2} a_{2} \bar{a}_{0} a_{0} \bar{a}_{2} a_{2} \bar{a}_{0} a_{0}
	\\&\quad
	+ \alpha^{(2)}_{21} a_{2} \bar{a}_{0} a_{0} \bar{a}_{2} a_{2} \bar{a}_{0} a_{0} \bar{a}_{2} a_{2} \bar{a}_{0} a_{0} \bar{a}_{2} a_{2} \bar{a}_{2} a_{2}
	+ \alpha^{(2)}_{22} a_{2} \bar{a}_{0} a_{0} \bar{a}_{2} a_{2} \bar{a}_{0} a_{0} \bar{a}_{2} a_{2} \bar{a}_{2} a_{2} \bar{a}_{0} a_{0} \bar{a}_{2} a_{2}
	+ \alpha^{(2)}_{23} a_{2} \bar{a}_{2} a_{2} \bar{a}_{0} a_{0} \bar{a}_{2} a_{2} \bar{a}_{0} a_{0} \bar{a}_{2} a_{2} \bar{a}_{2} a_{2} \bar{a}_{0} a_{0}
	\\&\quad
	+ \alpha^{(2)}_{24} a_{2} \bar{a}_{0} a_{0} \bar{a}_{2} a_{2} \bar{a}_{0} a_{0} \bar{a}_{2} a_{2} \bar{a}_{0} a_{0} \bar{a}_{2} a_{2} \bar{a}_{0} a_{0} \bar{a}_{2} a_{2}
	+ \alpha^{(2)}_{25} a_{2} \bar{a}_{0} a_{0} \bar{a}_{2} a_{2} \bar{a}_{0} a_{0} \bar{a}_{2} a_{2} \bar{a}_{0} a_{0} \bar{a}_{2} a_{2} \bar{a}_{2} a_{2} \bar{a}_{0} a_{0}
	\\&\quad
	+ \alpha^{(2)}_{26} a_{2} \bar{a}_{0} a_{0} \bar{a}_{2} a_{2} \bar{a}_{0} a_{0} \bar{a}_{2} a_{2} \bar{a}_{2} a_{2} \bar{a}_{0} a_{0} \bar{a}_{2} a_{2} \bar{a}_{0} a_{0}
	+ \alpha^{(2)}_{27} a_{2} \bar{a}_{2} a_{2} \bar{a}_{0} a_{0} \bar{a}_{2} a_{2} \bar{a}_{0} a_{0} \bar{a}_{2} a_{2} \bar{a}_{2} a_{2} \bar{a}_{0} a_{0} \bar{a}_{2} a_{2}
	\\
	&\quad
	+ \alpha^{(2)}_{28} a_{2} \bar{a}_{0} a_{0} \bar{a}_{2} a_{2} \bar{a}_{0} a_{0} \bar{a}_{2} a_{2} \bar{a}_{0} a_{0} \bar{a}_{2} a_{2} \bar{a}_{0} a_{0} \bar{a}_{2} a_{2} \bar{a}_{0} a_{0}
	+ \alpha^{(2)}_{29} a_{2} \bar{a}_{0} a_{0} \bar{a}_{2} a_{2} \bar{a}_{0} a_{0} \bar{a}_{2} a_{2} \bar{a}_{0} a_{0} \bar{a}_{2} a_{2} \bar{a}_{0} a_{0} \bar{a}_{2} a_{2} \bar{a}_{2} a_{2}
	\\&\quad
	+ \alpha^{(2)}_{30} a_{2} \bar{a}_{0} a_{0} \bar{a}_{2} a_{2} \bar{a}_{0} a_{0} \bar{a}_{2} a_{2} \bar{a}_{0} a_{0} \bar{a}_{2} a_{2} \bar{a}_{2} a_{2} \bar{a}_{0} a_{0} \bar{a}_{2} a_{2}
	+ \alpha^{(2)}_{31} a_{2} \bar{a}_{2} a_{2} \bar{a}_{0} a_{0} \bar{a}_{2} a_{2} \bar{a}_{0} a_{0} \bar{a}_{2} a_{2} \bar{a}_{2} a_{2} \bar{a}_{0} a_{0} \bar{a}_{2} a_{2} \bar{a}_{0} a_{0}
	\\&\quad
	+ \alpha^{(2)}_{32} a_{2} \bar{a}_{0} a_{0} \bar{a}_{2} a_{2} \bar{a}_{0} a_{0} \bar{a}_{2} a_{2} \bar{a}_{0} a_{0} \bar{a}_{2} a_{2} \bar{a}_{0} a_{0} \bar{a}_{2} a_{2} \bar{a}_{0} a_{0} \bar{a}_{2} a_{2}
	+ \alpha^{(2)}_{33} a_{2} \bar{a}_{0} a_{0} \bar{a}_{2} a_{2} \bar{a}_{0} a_{0} \bar{a}_{2} a_{2} \bar{a}_{0} a_{0} \bar{a}_{2} a_{2} \bar{a}_{0} a_{0} \bar{a}_{2} a_{2} \bar{a}_{2} a_{2} \bar{a}_{0} a_{0}
	\\&\quad
	+ \alpha^{(2)}_{34} a_{2} \bar{a}_{0} a_{0} \bar{a}_{2} a_{2} \bar{a}_{0} a_{0} \bar{a}_{2} a_{2} \bar{a}_{0} a_{0} \bar{a}_{2} a_{2} \bar{a}_{2} a_{2} \bar{a}_{0} a_{0} \bar{a}_{2} a_{2} \bar{a}_{0} a_{0}
	+ \alpha^{(2)}_{35} a_{2} \bar{a}_{0} a_{0} \bar{a}_{2} a_{2} \bar{a}_{0} a_{0} \bar{a}_{2} a_{2} \bar{a}_{0} a_{0} \bar{a}_{2} a_{2} \bar{a}_{0} a_{0} \bar{a}_{2} a_{2} \bar{a}_{0} a_{0} \bar{a}_{2} a_{2} \bar{a}_{0} a_{0}
	\\&\quad
	+ \alpha^{(2)}_{36} a_{2} \bar{a}_{0} a_{0} \bar{a}_{2} a_{2} \bar{a}_{0} a_{0} \bar{a}_{2} a_{2} \bar{a}_{0} a_{0} \bar{a}_{2} a_{2} \bar{a}_{0} a_{0} \bar{a}_{2} a_{2} \bar{a}_{0} a_{0} \bar{a}_{2} a_{2} \bar{a}_{2} a_{2}
	+ \alpha^{(2)}_{37} a_{2} \bar{a}_{0} a_{0} \bar{a}_{2} a_{2} \bar{a}_{0} a_{0} \bar{a}_{2} a_{2} \bar{a}_{0} a_{0} \bar{a}_{2} a_{2} \bar{a}_{0} a_{0} \bar{a}_{2} a_{2} \bar{a}_{0} a_{0} \bar{a}_{2} a_{2} \bar{a}_{0} a_{0} \bar{a}_{2} a_{2}
	\\&\quad
	+ \alpha^{(2)}_{38} a_{2} \bar{a}_{0} a_{0} \bar{a}_{2} a_{2} \bar{a}_{0} a_{0} \bar{a}_{2} a_{2} \bar{a}_{0} a_{0} \bar{a}_{2} a_{2} \bar{a}_{0} a_{0} \bar{a}_{2} a_{2} \bar{a}_{0} a_{0} \bar{a}_{2} a_{2} \bar{a}_{2} a_{2} \bar{a}_{0} a_{0}
	\\&\quad
	+ \alpha^{(2)}_{39} a_{2} \bar{a}_{0} a_{0} \bar{a}_{2} a_{2} \bar{a}_{0} a_{0} \bar{a}_{2} a_{2} \bar{a}_{0} a_{0} \bar{a}_{2} a_{2} \bar{a}_{0} a_{0} \bar{a}_{2} a_{2} \bar{a}_{0} a_{0} \bar{a}_{2} a_{2} \bar{a}_{0} a_{0} \bar{a}_{2} a_{2} \bar{a}_{0} a_{0}
	\\
	\varphi(\bar{a}_{2}) &= 
	\bar{a}_{2}
	+ \bar{\alpha}^{(2)}_{1} \bar{a}_{0} a_{0} \bar{a}_{2}
	+ \bar{\alpha}^{(2)}_{2} \bar{a}_{2} a_{2} \bar{a}_{2}
	+ \bar{\alpha}^{(2)}_{3} \bar{a}_{0} a_{0} \bar{a}_{2} a_{2} \bar{a}_{2}
	+ \bar{\alpha}^{(2)}_{4} \bar{a}_{2} a_{2} \bar{a}_{0} a_{0} \bar{a}_{2}
	+ \bar{\alpha}^{(2)}_{5} \bar{a}_{0} a_{0} \bar{a}_{2} a_{2} \bar{a}_{0} a_{0} \bar{a}_{2}
	+ \bar{\alpha}^{(2)}_{6} \bar{a}_{2} a_{2} \bar{a}_{0} a_{0} \bar{a}_{2} a_{2} \bar{a}_{2}
	\\&\quad
	+ \bar{\alpha}^{(2)}_{7} \bar{a}_{2} a_{2} \bar{a}_{2} a_{2} \bar{a}_{0} a_{0} \bar{a}_{2}
	+ \bar{\alpha}^{(2)}_{8} \bar{a}_{0} a_{0} \bar{a}_{2} a_{2} \bar{a}_{0} a_{0} \bar{a}_{2} a_{2} \bar{a}_{2}
	+ \bar{\alpha}^{(2)}_{9} \bar{a}_{0} a_{0} \bar{a}_{2} a_{2} \bar{a}_{2} a_{2} \bar{a}_{0} a_{0} \bar{a}_{2}
	+ \bar{\alpha}^{(2)}_{10} \bar{a}_{2} a_{2} \bar{a}_{0} a_{0} \bar{a}_{2} a_{2} \bar{a}_{0} a_{0} \bar{a}_{2}
	\\&\quad
	+ \bar{\alpha}^{(2)}_{11} \bar{a}_{2} a_{2} \bar{a}_{2} a_{2} \bar{a}_{0} a_{0} \bar{a}_{2} a_{2} \bar{a}_{2}
	+ \bar{\alpha}^{(2)}_{12} \bar{a}_{0} a_{0} \bar{a}_{2} a_{2} \bar{a}_{0} a_{0} \bar{a}_{2} a_{2} \bar{a}_{0} a_{0} \bar{a}_{2}
	+ \bar{\alpha}^{(2)}_{13} \bar{a}_{0} a_{0} \bar{a}_{2} a_{2} \bar{a}_{2} a_{2} \bar{a}_{0} a_{0} \bar{a}_{2} a_{2} \bar{a}_{2}
	+ \bar{\alpha}^{(2)}_{14} \bar{a}_{2} a_{2} \bar{a}_{0} a_{0} \bar{a}_{2} a_{2} \bar{a}_{0} a_{0} \bar{a}_{2} a_{2} \bar{a}_{2}
	\\&\quad
	+ \bar{\alpha}^{(2)}_{15} \bar{a}_{2} a_{2} \bar{a}_{0} a_{0} \bar{a}_{2} a_{2} \bar{a}_{2} a_{2} \bar{a}_{0} a_{0} \bar{a}_{2}
	+ \bar{\alpha}^{(2)}_{16} \bar{a}_{0} a_{0} \bar{a}_{2} a_{2} \bar{a}_{0} a_{0} \bar{a}_{2} a_{2} \bar{a}_{0} a_{0} \bar{a}_{2} a_{2} \bar{a}_{2}
	+ \bar{\alpha}^{(2)}_{17} \bar{a}_{0} a_{0} \bar{a}_{2} a_{2} \bar{a}_{0} a_{0} \bar{a}_{2} a_{2} \bar{a}_{2} a_{2} \bar{a}_{0} a_{0} \bar{a}_{2}
	\\&\quad
	+ \bar{\alpha}^{(2)}_{18} \bar{a}_{2} a_{2} \bar{a}_{0} a_{0} \bar{a}_{2} a_{2} \bar{a}_{0} a_{0} \bar{a}_{2} a_{2} \bar{a}_{0} a_{0} \bar{a}_{2}
	+ \bar{\alpha}^{(2)}_{19} \bar{a}_{2} a_{2} \bar{a}_{0} a_{0} \bar{a}_{2} a_{2} \bar{a}_{2} a_{2} \bar{a}_{0} a_{0} \bar{a}_{2} a_{2} \bar{a}_{2}
	+ \bar{\alpha}^{(2)}_{20} \bar{a}_{0} a_{0} \bar{a}_{2} a_{2} \bar{a}_{0} a_{0} \bar{a}_{2} a_{2} \bar{a}_{0} a_{0} \bar{a}_{2} a_{2} \bar{a}_{0} a_{0} \bar{a}_{2}
	\\&\quad
	+ \bar{\alpha}^{(2)}_{21} \bar{a}_{0} a_{0} \bar{a}_{2} a_{2} \bar{a}_{0} a_{0} \bar{a}_{2} a_{2} \bar{a}_{2} a_{2} \bar{a}_{0} a_{0} \bar{a}_{2} a_{2} \bar{a}_{2}
	+ \bar{\alpha}^{(2)}_{22} \bar{a}_{2} a_{2} \bar{a}_{0} a_{0} \bar{a}_{2} a_{2} \bar{a}_{0} a_{0} \bar{a}_{2} a_{2} \bar{a}_{0} a_{0} \bar{a}_{2} a_{2} \bar{a}_{2}
	+ \bar{\alpha}^{(2)}_{23} \bar{a}_{2} a_{2} \bar{a}_{0} a_{0} \bar{a}_{2} a_{2} \bar{a}_{0} a_{0} \bar{a}_{2} a_{2} \bar{a}_{2} a_{2} \bar{a}_{0} a_{0} \bar{a}_{2}
	\\&\quad
	+ \bar{\alpha}^{(2)}_{24} \bar{a}_{0} a_{0} \bar{a}_{2} a_{2} \bar{a}_{0} a_{0} \bar{a}_{2} a_{2} \bar{a}_{0} a_{0} \bar{a}_{2} a_{2} \bar{a}_{0} a_{0} \bar{a}_{2} a_{2} \bar{a}_{2}
	+ \bar{\alpha}^{(2)}_{25} \bar{a}_{0} a_{0} \bar{a}_{2} a_{2} \bar{a}_{0} a_{0} \bar{a}_{2} a_{2} \bar{a}_{0} a_{0} \bar{a}_{2} a_{2} \bar{a}_{2} a_{2} \bar{a}_{0} a_{0} \bar{a}_{2}
	\\&\quad
	+ \bar{\alpha}^{(2)}_{26} \bar{a}_{2} a_{2} \bar{a}_{0} a_{0} \bar{a}_{2} a_{2} \bar{a}_{0} a_{0} \bar{a}_{2} a_{2} \bar{a}_{0} a_{0} \bar{a}_{2} a_{2} \bar{a}_{0} a_{0} \bar{a}_{2}
	+ \bar{\alpha}^{(2)}_{27} \bar{a}_{2} a_{2} \bar{a}_{0} a_{0} \bar{a}_{2} a_{2} \bar{a}_{0} a_{0} \bar{a}_{2} a_{2} \bar{a}_{2} a_{2} \bar{a}_{0} a_{0} \bar{a}_{2} a_{2} \bar{a}_{2}
	\\&\quad
	+ \bar{\alpha}^{(2)}_{28} \bar{a}_{0} a_{0} \bar{a}_{2} a_{2} \bar{a}_{0} a_{0} \bar{a}_{2} a_{2} \bar{a}_{0} a_{0} \bar{a}_{2} a_{2} \bar{a}_{0} a_{0} \bar{a}_{2} a_{2} \bar{a}_{0} a_{0} \bar{a}_{2}
	+ \bar{\alpha}^{(2)}_{29} \bar{a}_{0} a_{0} \bar{a}_{2} a_{2} \bar{a}_{0} a_{0} \bar{a}_{2} a_{2} \bar{a}_{0} a_{0} \bar{a}_{2} a_{2} \bar{a}_{2} a_{2} \bar{a}_{0} a_{0} \bar{a}_{2} a_{2} \bar{a}_{2}
	\\&\quad
	+ \bar{\alpha}^{(2)}_{30} \bar{a}_{2} a_{2} \bar{a}_{0} a_{0} \bar{a}_{2} a_{2} \bar{a}_{0} a_{0} \bar{a}_{2} a_{2} \bar{a}_{0} a_{0} \bar{a}_{2} a_{2} \bar{a}_{0} a_{0} \bar{a}_{2} a_{2} \bar{a}_{2}
	+ \bar{\alpha}^{(2)}_{31} \bar{a}_{2} a_{2} \bar{a}_{0} a_{0} \bar{a}_{2} a_{2} \bar{a}_{0} a_{0} \bar{a}_{2} a_{2} \bar{a}_{0} a_{0} \bar{a}_{2} a_{2} \bar{a}_{2} a_{2} \bar{a}_{0} a_{0} \bar{a}_{2}
	\\&\quad
	+ \bar{\alpha}^{(2)}_{32} \bar{a}_{0} a_{0} \bar{a}_{2} a_{2} \bar{a}_{0} a_{0} \bar{a}_{2} a_{2} \bar{a}_{0} a_{0} \bar{a}_{2} a_{2} \bar{a}_{0} a_{0} \bar{a}_{2} a_{2} \bar{a}_{0} a_{0} \bar{a}_{2} a_{2} \bar{a}_{2}
	+ \bar{\alpha}^{(2)}_{33} \bar{a}_{0} a_{0} \bar{a}_{2} a_{2} \bar{a}_{0} a_{0} \bar{a}_{2} a_{2} \bar{a}_{0} a_{0} \bar{a}_{2} a_{2} \bar{a}_{0} a_{0} \bar{a}_{2} a_{2} \bar{a}_{2} a_{2} \bar{a}_{0} a_{0} \bar{a}_{2}
	\\&\quad
	+ \bar{\alpha}^{(2)}_{34} \bar{a}_{2} a_{2} \bar{a}_{0} a_{0} \bar{a}_{2} a_{2} \bar{a}_{0} a_{0} \bar{a}_{2} a_{2} \bar{a}_{0} a_{0} \bar{a}_{2} a_{2} \bar{a}_{0} a_{0} \bar{a}_{2} a_{2} \bar{a}_{0} a_{0} \bar{a}_{2}
	+ \bar{\alpha}^{(2)}_{35} \bar{a}_{0} a_{0} \bar{a}_{2} a_{2} \bar{a}_{0} a_{0} \bar{a}_{2} a_{2} \bar{a}_{0} a_{0} \bar{a}_{2} a_{2} \bar{a}_{0} a_{0} \bar{a}_{2} a_{2} \bar{a}_{0} a_{0} \bar{a}_{2} a_{2} \bar{a}_{0} a_{0} \bar{a}_{2}
	\\&\quad
	+ \bar{\alpha}^{(2)}_{36} \bar{a}_{2} a_{2} \bar{a}_{0} a_{0} \bar{a}_{2} a_{2} \bar{a}_{0} a_{0} \bar{a}_{2} a_{2} \bar{a}_{0} a_{0} \bar{a}_{2} a_{2} \bar{a}_{0} a_{0} \bar{a}_{2} a_{2} \bar{a}_{2} a_{2} \bar{a}_{0} a_{0} \bar{a}_{2}
	+ \bar{\alpha}^{(2)}_{37} \bar{a}_{0} a_{0} \bar{a}_{2} a_{2} \bar{a}_{0} a_{0} \bar{a}_{2} a_{2} \bar{a}_{0} a_{0} \bar{a}_{2} a_{2} \bar{a}_{0} a_{0} \bar{a}_{2} a_{2} \bar{a}_{0} a_{0} \bar{a}_{2} a_{2} \bar{a}_{2} a_{2} \bar{a}_{0} a_{0} \bar{a}_{2}
	\\&\quad
	+ \bar{\alpha}^{(2)}_{38} \bar{a}_{2} a_{2} \bar{a}_{0} a_{0} \bar{a}_{2} a_{2} \bar{a}_{0} a_{0} \bar{a}_{2} a_{2} \bar{a}_{0} a_{0} \bar{a}_{2} a_{2} \bar{a}_{0} a_{0} \bar{a}_{2} a_{2} \bar{a}_{0} a_{0} \bar{a}_{2} a_{2} \bar{a}_{0} a_{0} \bar{a}_{2}
	\\&\quad
	+ \bar{\alpha}^{(2)}_{39} \bar{a}_{0} a_{0} \bar{a}_{2} a_{2} \bar{a}_{0} a_{0} \bar{a}_{2} a_{2} \bar{a}_{0} a_{0} \bar{a}_{2} a_{2} \bar{a}_{0} a_{0} \bar{a}_{2} a_{2} \bar{a}_{0} a_{0} \bar{a}_{2} a_{2} \bar{a}_{0} a_{0} \bar{a}_{2} a_{2} \bar{a}_{0} a_{0} \bar{a}_{2}
	\\
	\varphi(a_{3}) &= 
	a_{3}
	+ \alpha^{(3)}_{1} \bar{a}_{0} a_{0} a_{3}
	+ \alpha^{(3)}_{2} \bar{a}_{2} a_{2} a_{3}
	+ \alpha^{(3)}_{3} \bar{a}_{0} a_{0} \bar{a}_{2} a_{2} a_{3}
	+ \alpha^{(3)}_{4} \bar{a}_{2} a_{2} \bar{a}_{0} a_{0} a_{3}
	+ \alpha^{(3)}_{5} \bar{a}_{2} a_{2} \bar{a}_{2} a_{2} a_{3}
	+ \alpha^{(3)}_{6} \bar{a}_{0} a_{0} \bar{a}_{2} a_{2} \bar{a}_{0} a_{0} a_{3}
	\\&\quad
	+ \alpha^{(3)}_{7} \bar{a}_{0} a_{0} \bar{a}_{2} a_{2} \bar{a}_{2} a_{2} a_{3}
	+ \alpha^{(3)}_{8} \bar{a}_{2} a_{2} \bar{a}_{0} a_{0} \bar{a}_{2} a_{2} a_{3}
	+ \alpha^{(3)}_{9} \bar{a}_{2} a_{2} \bar{a}_{2} a_{2} \bar{a}_{0} a_{0} a_{3}
	+ \alpha^{(3)}_{10} \bar{a}_{0} a_{0} \bar{a}_{2} a_{2} \bar{a}_{0} a_{0} \bar{a}_{2} a_{2} a_{3}
	+ \alpha^{(3)}_{11} \bar{a}_{0} a_{0} \bar{a}_{2} a_{2} \bar{a}_{2} a_{2} \bar{a}_{0} a_{0} a_{3}
	\\&\quad
	+ \alpha^{(3)}_{12} \bar{a}_{2} a_{2} \bar{a}_{0} a_{0} \bar{a}_{2} a_{2} \bar{a}_{0} a_{0} a_{3}
	+ \alpha^{(3)}_{13} \bar{a}_{2} a_{2} \bar{a}_{0} a_{0} \bar{a}_{2} a_{2} \bar{a}_{2} a_{2} a_{3}
	+ \alpha^{(3)}_{14} \bar{a}_{0} a_{0} \bar{a}_{2} a_{2} \bar{a}_{0} a_{0} \bar{a}_{2} a_{2} \bar{a}_{0} a_{0} a_{3}
	+ \alpha^{(3)}_{15} \bar{a}_{0} a_{0} \bar{a}_{2} a_{2} \bar{a}_{0} a_{0} \bar{a}_{2} a_{2} \bar{a}_{2} a_{2} a_{3}
	\\&\quad
	+ \alpha^{(3)}_{16} \bar{a}_{2} a_{2} \bar{a}_{0} a_{0} \bar{a}_{2} a_{2} \bar{a}_{0} a_{0} \bar{a}_{2} a_{2} a_{3}
	+ \alpha^{(3)}_{17} \bar{a}_{2} a_{2} \bar{a}_{0} a_{0} \bar{a}_{2} a_{2} \bar{a}_{2} a_{2} \bar{a}_{0} a_{0} a_{3}
	+ \alpha^{(3)}_{18} \bar{a}_{2} a_{2} \bar{a}_{2} a_{2} \bar{a}_{0} a_{0} \bar{a}_{2} a_{2} \bar{a}_{0} a_{0} a_{3}
	\\&\quad
	+ \alpha^{(3)}_{19} \bar{a}_{0} a_{0} \bar{a}_{2} a_{2} \bar{a}_{0} a_{0} \bar{a}_{2} a_{2} \bar{a}_{0} a_{0} \bar{a}_{2} a_{2} a_{3}
	+ \alpha^{(3)}_{20} \bar{a}_{0} a_{0} \bar{a}_{2} a_{2} \bar{a}_{0} a_{0} \bar{a}_{2} a_{2} \bar{a}_{2} a_{2} \bar{a}_{0} a_{0} a_{3}
	+ \alpha^{(3)}_{21} \bar{a}_{0} a_{0} \bar{a}_{2} a_{2} \bar{a}_{2} a_{2} \bar{a}_{0} a_{0} \bar{a}_{2} a_{2} \bar{a}_{0} a_{0} a_{3}
	\\&\quad
	+ \alpha^{(3)}_{22} \bar{a}_{2} a_{2} \bar{a}_{0} a_{0} \bar{a}_{2} a_{2} \bar{a}_{0} a_{0} \bar{a}_{2} a_{2} \bar{a}_{0} a_{0} a_{3}
	+ \alpha^{(3)}_{23} \bar{a}_{2} a_{2} \bar{a}_{0} a_{0} \bar{a}_{2} a_{2} \bar{a}_{0} a_{0} \bar{a}_{2} a_{2} \bar{a}_{2} a_{2} a_{3}
	+ \alpha^{(3)}_{24} \bar{a}_{0} a_{0} \bar{a}_{2} a_{2} \bar{a}_{0} a_{0} \bar{a}_{2} a_{2} \bar{a}_{0} a_{0} \bar{a}_{2} a_{2} \bar{a}_{0} a_{0} a_{3}
	\\&\quad
	+ \alpha^{(3)}_{25} \bar{a}_{0} a_{0} \bar{a}_{2} a_{2} \bar{a}_{0} a_{0} \bar{a}_{2} a_{2} \bar{a}_{0} a_{0} \bar{a}_{2} a_{2} \bar{a}_{2} a_{2} a_{3}
	+ \alpha^{(3)}_{26} \bar{a}_{2} a_{2} \bar{a}_{0} a_{0} \bar{a}_{2} a_{2} \bar{a}_{0} a_{0} \bar{a}_{2} a_{2} \bar{a}_{0} a_{0} \bar{a}_{2} a_{2} a_{3}
	+ \alpha^{(3)}_{27} \bar{a}_{2} a_{2} \bar{a}_{0} a_{0} \bar{a}_{2} a_{2} \bar{a}_{0} a_{0} \bar{a}_{2} a_{2} \bar{a}_{2} a_{2} \bar{a}_{0} a_{0} a_{3}
	\\&\quad
	+ \alpha^{(3)}_{28} \bar{a}_{2} a_{2} \bar{a}_{0} a_{0} \bar{a}_{2} a_{2} \bar{a}_{2} a_{2} \bar{a}_{0} a_{0} \bar{a}_{2} a_{2} \bar{a}_{0} a_{0} a_{3}
	+ \alpha^{(3)}_{29} \bar{a}_{0} a_{0} \bar{a}_{2} a_{2} \bar{a}_{0} a_{0} \bar{a}_{2} a_{2} \bar{a}_{0} a_{0} \bar{a}_{2} a_{2} \bar{a}_{0} a_{0} \bar{a}_{2} a_{2} a_{3}
	\\&\quad
	+ \alpha^{(3)}_{30} \bar{a}_{0} a_{0} \bar{a}_{2} a_{2} \bar{a}_{0} a_{0} \bar{a}_{2} a_{2} \bar{a}_{0} a_{0} \bar{a}_{2} a_{2} \bar{a}_{2} a_{2} \bar{a}_{0} a_{0} a_{3}
	+ \alpha^{(3)}_{31} \bar{a}_{0} a_{0} \bar{a}_{2} a_{2} \bar{a}_{0} a_{0} \bar{a}_{2} a_{2} \bar{a}_{2} a_{2} \bar{a}_{0} a_{0} \bar{a}_{2} a_{2} \bar{a}_{0} a_{0} a_{3}
	\\&\quad
	+ \alpha^{(3)}_{32} \bar{a}_{2} a_{2} \bar{a}_{0} a_{0} \bar{a}_{2} a_{2} \bar{a}_{0} a_{0} \bar{a}_{2} a_{2} \bar{a}_{0} a_{0} \bar{a}_{2} a_{2} \bar{a}_{0} a_{0} a_{3}
	+ \alpha^{(3)}_{33} \bar{a}_{2} a_{2} \bar{a}_{0} a_{0} \bar{a}_{2} a_{2} \bar{a}_{0} a_{0} \bar{a}_{2} a_{2} \bar{a}_{0} a_{0} \bar{a}_{2} a_{2} \bar{a}_{2} a_{2} a_{3}
	\\&\quad
	+ \alpha^{(3)}_{34} \bar{a}_{0} a_{0} \bar{a}_{2} a_{2} \bar{a}_{0} a_{0} \bar{a}_{2} a_{2} \bar{a}_{0} a_{0} \bar{a}_{2} a_{2} \bar{a}_{0} a_{0} \bar{a}_{2} a_{2} \bar{a}_{0} a_{0} a_{3}
	+ \alpha^{(3)}_{35} \bar{a}_{0} a_{0} \bar{a}_{2} a_{2} \bar{a}_{0} a_{0} \bar{a}_{2} a_{2} \bar{a}_{0} a_{0} \bar{a}_{2} a_{2} \bar{a}_{0} a_{0} \bar{a}_{2} a_{2} \bar{a}_{2} a_{2} a_{3}
	\\&\quad
	+ \alpha^{(3)}_{36} \bar{a}_{2} a_{2} \bar{a}_{0} a_{0} \bar{a}_{2} a_{2} \bar{a}_{0} a_{0} \bar{a}_{2} a_{2} \bar{a}_{0} a_{0} \bar{a}_{2} a_{2} \bar{a}_{0} a_{0} \bar{a}_{2} a_{2} a_{3}
	+ \alpha^{(3)}_{37} \bar{a}_{2} a_{2} \bar{a}_{0} a_{0} \bar{a}_{2} a_{2} \bar{a}_{0} a_{0} \bar{a}_{2} a_{2} \bar{a}_{2} a_{2} \bar{a}_{0} a_{0} \bar{a}_{2} a_{2} \bar{a}_{0} a_{0} a_{3}
	\\&\quad
	+ \alpha^{(3)}_{38} \bar{a}_{0} a_{0} \bar{a}_{2} a_{2} \bar{a}_{0} a_{0} \bar{a}_{2} a_{2} \bar{a}_{0} a_{0} \bar{a}_{2} a_{2} \bar{a}_{0} a_{0} \bar{a}_{2} a_{2} \bar{a}_{0} a_{0} \bar{a}_{2} a_{2} a_{3}
	+ \alpha^{(3)}_{39} \bar{a}_{0} a_{0} \bar{a}_{2} a_{2} \bar{a}_{0} a_{0} \bar{a}_{2} a_{2} \bar{a}_{0} a_{0} \bar{a}_{2} a_{2} \bar{a}_{2} a_{2} \bar{a}_{0} a_{0} \bar{a}_{2} a_{2} \bar{a}_{0} a_{0} a_{3}
	\\&\quad
	+ \alpha^{(3)}_{40} \bar{a}_{2} a_{2} \bar{a}_{0} a_{0} \bar{a}_{2} a_{2} \bar{a}_{0} a_{0} \bar{a}_{2} a_{2} \bar{a}_{0} a_{0} \bar{a}_{2} a_{2} \bar{a}_{0} a_{0} \bar{a}_{2} a_{2} \bar{a}_{0} a_{0} a_{3}
	+ \alpha^{(3)}_{41} \bar{a}_{2} a_{2} \bar{a}_{0} a_{0} \bar{a}_{2} a_{2} \bar{a}_{0} a_{0} \bar{a}_{2} a_{2} \bar{a}_{0} a_{0} \bar{a}_{2} a_{2} \bar{a}_{0} a_{0} \bar{a}_{2} a_{2} \bar{a}_{2} a_{2} a_{3}
	\\&\quad
	+ \alpha^{(3)}_{42} \bar{a}_{0} a_{0} \bar{a}_{2} a_{2} \bar{a}_{0} a_{0} \bar{a}_{2} a_{2} \bar{a}_{0} a_{0} \bar{a}_{2} a_{2} \bar{a}_{0} a_{0} \bar{a}_{2} a_{2} \bar{a}_{0} a_{0} \bar{a}_{2} a_{2} \bar{a}_{0} a_{0} a_{3}
	+ \alpha^{(3)}_{43} \bar{a}_{0} a_{0} \bar{a}_{2} a_{2} \bar{a}_{0} a_{0} \bar{a}_{2} a_{2} \bar{a}_{0} a_{0} \bar{a}_{2} a_{2} \bar{a}_{0} a_{0} \bar{a}_{2} a_{2} \bar{a}_{0} a_{0} \bar{a}_{2} a_{2} \bar{a}_{2} a_{2} a_{3}
	\\&\quad
	+ \alpha^{(3)}_{44} \bar{a}_{2} a_{2} \bar{a}_{0} a_{0} \bar{a}_{2} a_{2} \bar{a}_{0} a_{0} \bar{a}_{2} a_{2} \bar{a}_{0} a_{0} \bar{a}_{2} a_{2} \bar{a}_{0} a_{0} \bar{a}_{2} a_{2} \bar{a}_{0} a_{0} \bar{a}_{2} a_{2} a_{3}
	+ \alpha^{(3)}_{45} \bar{a}_{0} a_{0} \bar{a}_{2} a_{2} \bar{a}_{0} a_{0} \bar{a}_{2} a_{2} \bar{a}_{0} a_{0} \bar{a}_{2} a_{2} \bar{a}_{0} a_{0} \bar{a}_{2} a_{2} \bar{a}_{0} a_{0} \bar{a}_{2} a_{2} \bar{a}_{0} a_{0} \bar{a}_{2} a_{2} a_{3}
\end{align*}
\begin{align*}
&\quad
	+ \alpha^{(3)}_{46} \bar{a}_{2} a_{2} \bar{a}_{0} a_{0} \bar{a}_{2} a_{2} \bar{a}_{0} a_{0} \bar{a}_{2} a_{2} \bar{a}_{0} a_{0} \bar{a}_{2} a_{2} \bar{a}_{0} a_{0} \bar{a}_{2} a_{2} \bar{a}_{0} a_{0} \bar{a}_{2} a_{2} \bar{a}_{0} a_{0} a_{3}
	\\&\quad
	+ \alpha^{(3)}_{47} \bar{a}_{0} a_{0} \bar{a}_{2} a_{2} \bar{a}_{0} a_{0} \bar{a}_{2} a_{2} \bar{a}_{0} a_{0} \bar{a}_{2} a_{2} \bar{a}_{0} a_{0} \bar{a}_{2} a_{2} \bar{a}_{0} a_{0} \bar{a}_{2} a_{2} \bar{a}_{0} a_{0} \bar{a}_{2} a_{2} \bar{a}_{0} a_{0} a_{3}
	\\
	\varphi(\bar{a}_{3}) &= 
	\bar{a}_{3}
	+ \bar{\alpha}^{(3)}_{1} \bar{a}_{3} \bar{a}_{0} a_{0}
	+ \bar{\alpha}^{(3)}_{2} \bar{a}_{3} \bar{a}_{2} a_{2}
	+ \bar{\alpha}^{(3)}_{3} \bar{a}_{3} \bar{a}_{0} a_{0} \bar{a}_{2} a_{2}
	+ \bar{\alpha}^{(3)}_{4} \bar{a}_{3} \bar{a}_{2} a_{2} \bar{a}_{0} a_{0}
	+ \bar{\alpha}^{(3)}_{5} \bar{a}_{3} \bar{a}_{2} a_{2} \bar{a}_{2} a_{2}
	+ \bar{\alpha}^{(3)}_{6} \bar{a}_{3} \bar{a}_{0} a_{0} \bar{a}_{2} a_{2} \bar{a}_{0} a_{0}
	\\&\quad
	+ \bar{\alpha}^{(3)}_{7} \bar{a}_{3} \bar{a}_{0} a_{0} \bar{a}_{2} a_{2} \bar{a}_{2} a_{2}
	+ \bar{\alpha}^{(3)}_{8} \bar{a}_{3} \bar{a}_{2} a_{2} \bar{a}_{0} a_{0} \bar{a}_{2} a_{2}
	+ \bar{\alpha}^{(3)}_{9} \bar{a}_{3} \bar{a}_{2} a_{2} \bar{a}_{2} a_{2} \bar{a}_{0} a_{0}
	+ \bar{\alpha}^{(3)}_{10} \bar{a}_{3} \bar{a}_{0} a_{0} \bar{a}_{2} a_{2} \bar{a}_{0} a_{0} \bar{a}_{2} a_{2}
	+ \bar{\alpha}^{(3)}_{11} \bar{a}_{3} \bar{a}_{0} a_{0} \bar{a}_{2} a_{2} \bar{a}_{2} a_{2} \bar{a}_{0} a_{0}
	\\&\quad
	+ \bar{\alpha}^{(3)}_{12} \bar{a}_{3} \bar{a}_{2} a_{2} \bar{a}_{0} a_{0} \bar{a}_{2} a_{2} \bar{a}_{0} a_{0}
	+ \bar{\alpha}^{(3)}_{13} \bar{a}_{3} \bar{a}_{2} a_{2} \bar{a}_{0} a_{0} \bar{a}_{2} a_{2} \bar{a}_{2} a_{2}
	+ \bar{\alpha}^{(3)}_{14} \bar{a}_{3} \bar{a}_{0} a_{0} \bar{a}_{2} a_{2} \bar{a}_{0} a_{0} \bar{a}_{2} a_{2} \bar{a}_{0} a_{0}
	+ \bar{\alpha}^{(3)}_{15} \bar{a}_{3} \bar{a}_{0} a_{0} \bar{a}_{2} a_{2} \bar{a}_{0} a_{0} \bar{a}_{2} a_{2} \bar{a}_{2} a_{2}
	\\&\quad
	+ \bar{\alpha}^{(3)}_{16} \bar{a}_{3} \bar{a}_{0} a_{0} \bar{a}_{2} a_{2} \bar{a}_{2} a_{2} \bar{a}_{0} a_{0} \bar{a}_{2} a_{2}
	+ \bar{\alpha}^{(3)}_{17} \bar{a}_{3} \bar{a}_{2} a_{2} \bar{a}_{0} a_{0} \bar{a}_{2} a_{2} \bar{a}_{0} a_{0} \bar{a}_{2} a_{2}
	+ \bar{\alpha}^{(3)}_{18} \bar{a}_{3} \bar{a}_{2} a_{2} \bar{a}_{0} a_{0} \bar{a}_{2} a_{2} \bar{a}_{2} a_{2} \bar{a}_{0} a_{0}
	\\&\quad
	+ \bar{\alpha}^{(3)}_{19} \bar{a}_{3} \bar{a}_{0} a_{0} \bar{a}_{2} a_{2} \bar{a}_{0} a_{0} \bar{a}_{2} a_{2} \bar{a}_{0} a_{0} \bar{a}_{2} a_{2}
	+ \bar{\alpha}^{(3)}_{20} \bar{a}_{3} \bar{a}_{0} a_{0} \bar{a}_{2} a_{2} \bar{a}_{0} a_{0} \bar{a}_{2} a_{2} \bar{a}_{2} a_{2} \bar{a}_{0} a_{0}
	+ \bar{\alpha}^{(3)}_{21} \bar{a}_{3} \bar{a}_{0} a_{0} \bar{a}_{2} a_{2} \bar{a}_{2} a_{2} \bar{a}_{0} a_{0} \bar{a}_{2} a_{2} \bar{a}_{0} a_{0}
	\\&\quad
	+ \bar{\alpha}^{(3)}_{22} \bar{a}_{3} \bar{a}_{2} a_{2} \bar{a}_{0} a_{0} \bar{a}_{2} a_{2} \bar{a}_{0} a_{0} \bar{a}_{2} a_{2} \bar{a}_{0} a_{0}
	+ \bar{\alpha}^{(3)}_{23} \bar{a}_{3} \bar{a}_{2} a_{2} \bar{a}_{0} a_{0} \bar{a}_{2} a_{2} \bar{a}_{2} a_{2} \bar{a}_{0} a_{0} \bar{a}_{2} a_{2}
	+ \bar{\alpha}^{(3)}_{24} \bar{a}_{3} \bar{a}_{0} a_{0} \bar{a}_{2} a_{2} \bar{a}_{0} a_{0} \bar{a}_{2} a_{2} \bar{a}_{0} a_{0} \bar{a}_{2} a_{2} \bar{a}_{0} a_{0}
	\\&\quad
	+ \bar{\alpha}^{(3)}_{25} \bar{a}_{3} \bar{a}_{0} a_{0} \bar{a}_{2} a_{2} \bar{a}_{0} a_{0} \bar{a}_{2} a_{2} \bar{a}_{0} a_{0} \bar{a}_{2} a_{2} \bar{a}_{2} a_{2}
	+ \bar{\alpha}^{(3)}_{26} \bar{a}_{3} \bar{a}_{0} a_{0} \bar{a}_{2} a_{2} \bar{a}_{0} a_{0} \bar{a}_{2} a_{2} \bar{a}_{2} a_{2} \bar{a}_{0} a_{0} \bar{a}_{2} a_{2}
	+ \bar{\alpha}^{(3)}_{27} \bar{a}_{3} \bar{a}_{2} a_{2} \bar{a}_{0} a_{0} \bar{a}_{2} a_{2} \bar{a}_{0} a_{0} \bar{a}_{2} a_{2} \bar{a}_{0} a_{0} \bar{a}_{2} a_{2}
	\\&\quad
	+ \bar{\alpha}^{(3)}_{28} \bar{a}_{3} \bar{a}_{2} a_{2} \bar{a}_{0} a_{0} \bar{a}_{2} a_{2} \bar{a}_{2} a_{2} \bar{a}_{0} a_{0} \bar{a}_{2} a_{2} \bar{a}_{0} a_{0}
	+ \bar{\alpha}^{(3)}_{29} \bar{a}_{3} \bar{a}_{0} a_{0} \bar{a}_{2} a_{2} \bar{a}_{0} a_{0} \bar{a}_{2} a_{2} \bar{a}_{0} a_{0} \bar{a}_{2} a_{2} \bar{a}_{0} a_{0} \bar{a}_{2} a_{2}
	\\&\quad
	+ \bar{\alpha}^{(3)}_{30} \bar{a}_{3} \bar{a}_{0} a_{0} \bar{a}_{2} a_{2} \bar{a}_{0} a_{0} \bar{a}_{2} a_{2} \bar{a}_{0} a_{0} \bar{a}_{2} a_{2} \bar{a}_{2} a_{2} \bar{a}_{0} a_{0}
	+ \bar{\alpha}^{(3)}_{31} \bar{a}_{3} \bar{a}_{0} a_{0} \bar{a}_{2} a_{2} \bar{a}_{0} a_{0} \bar{a}_{2} a_{2} \bar{a}_{2} a_{2} \bar{a}_{0} a_{0} \bar{a}_{2} a_{2} \bar{a}_{0} a_{0}
	\\&\quad
	+ \bar{\alpha}^{(3)}_{32} \bar{a}_{3} \bar{a}_{2} a_{2} \bar{a}_{0} a_{0} \bar{a}_{2} a_{2} \bar{a}_{0} a_{0} \bar{a}_{2} a_{2} \bar{a}_{0} a_{0} \bar{a}_{2} a_{2} \bar{a}_{0} a_{0}
	+ \bar{\alpha}^{(3)}_{33} \bar{a}_{3} \bar{a}_{2} a_{2} \bar{a}_{0} a_{0} \bar{a}_{2} a_{2} \bar{a}_{0} a_{0} \bar{a}_{2} a_{2} \bar{a}_{0} a_{0} \bar{a}_{2} a_{2} \bar{a}_{2} a_{2}
	\\&\quad
	+ \bar{\alpha}^{(3)}_{34} \bar{a}_{3} \bar{a}_{0} a_{0} \bar{a}_{2} a_{2} \bar{a}_{0} a_{0} \bar{a}_{2} a_{2} \bar{a}_{0} a_{0} \bar{a}_{2} a_{2} \bar{a}_{0} a_{0} \bar{a}_{2} a_{2} \bar{a}_{0} a_{0}
	+ \bar{\alpha}^{(3)}_{35} \bar{a}_{3} \bar{a}_{0} a_{0} \bar{a}_{2} a_{2} \bar{a}_{0} a_{0} \bar{a}_{2} a_{2} \bar{a}_{0} a_{0} \bar{a}_{2} a_{2} \bar{a}_{0} a_{0} \bar{a}_{2} a_{2} \bar{a}_{2} a_{2}
	\\&\quad
	+ \bar{\alpha}^{(3)}_{36} \bar{a}_{3} \bar{a}_{0} a_{0} \bar{a}_{2} a_{2} \bar{a}_{0} a_{0} \bar{a}_{2} a_{2} \bar{a}_{0} a_{0} \bar{a}_{2} a_{2} \bar{a}_{2} a_{2} \bar{a}_{0} a_{0} \bar{a}_{2} a_{2}
	+ \bar{\alpha}^{(3)}_{37} \bar{a}_{3} \bar{a}_{2} a_{2} \bar{a}_{0} a_{0} \bar{a}_{2} a_{2} \bar{a}_{0} a_{0} \bar{a}_{2} a_{2} \bar{a}_{0} a_{0} \bar{a}_{2} a_{2} \bar{a}_{0} a_{0} \bar{a}_{2} a_{2}
	\\&\quad
	+ \bar{\alpha}^{(3)}_{38} \bar{a}_{3} \bar{a}_{0} a_{0} \bar{a}_{2} a_{2} \bar{a}_{0} a_{0} \bar{a}_{2} a_{2} \bar{a}_{0} a_{0} \bar{a}_{2} a_{2} \bar{a}_{0} a_{0} \bar{a}_{2} a_{2} \bar{a}_{0} a_{0} \bar{a}_{2} a_{2}
	+ \bar{\alpha}^{(3)}_{39} \bar{a}_{3} \bar{a}_{0} a_{0} \bar{a}_{2} a_{2} \bar{a}_{0} a_{0} \bar{a}_{2} a_{2} \bar{a}_{0} a_{0} \bar{a}_{2} a_{2} \bar{a}_{0} a_{0} \bar{a}_{2} a_{2} \bar{a}_{2} a_{2} \bar{a}_{0} a_{0}
	\\&\quad
	+ \bar{\alpha}^{(3)}_{40} \bar{a}_{3} \bar{a}_{0} a_{0} \bar{a}_{2} a_{2} \bar{a}_{0} a_{0} \bar{a}_{2} a_{2} \bar{a}_{0} a_{0} \bar{a}_{2} a_{2} \bar{a}_{2} a_{2} \bar{a}_{0} a_{0} \bar{a}_{2} a_{2} \bar{a}_{0} a_{0}
	+ \bar{\alpha}^{(3)}_{41} \bar{a}_{3} \bar{a}_{2} a_{2} \bar{a}_{0} a_{0} \bar{a}_{2} a_{2} \bar{a}_{0} a_{0} \bar{a}_{2} a_{2} \bar{a}_{0} a_{0} \bar{a}_{2} a_{2} \bar{a}_{0} a_{0} \bar{a}_{2} a_{2} \bar{a}_{2} a_{2}
	\\&\quad
	+ \bar{\alpha}^{(3)}_{42} \bar{a}_{3} \bar{a}_{0} a_{0} \bar{a}_{2} a_{2} \bar{a}_{0} a_{0} \bar{a}_{2} a_{2} \bar{a}_{0} a_{0} \bar{a}_{2} a_{2} \bar{a}_{0} a_{0} \bar{a}_{2} a_{2} \bar{a}_{0} a_{0} \bar{a}_{2} a_{2} \bar{a}_{0} a_{0}
	+ \bar{\alpha}^{(3)}_{43} \bar{a}_{3} \bar{a}_{0} a_{0} \bar{a}_{2} a_{2} \bar{a}_{0} a_{0} \bar{a}_{2} a_{2} \bar{a}_{0} a_{0} \bar{a}_{2} a_{2} \bar{a}_{0} a_{0} \bar{a}_{2} a_{2} \bar{a}_{0} a_{0} \bar{a}_{2} a_{2} \bar{a}_{2} a_{2}
	\\&\quad
	+ \bar{\alpha}^{(3)}_{44} \bar{a}_{3} \bar{a}_{2} a_{2} \bar{a}_{0} a_{0} \bar{a}_{2} a_{2} \bar{a}_{0} a_{0} \bar{a}_{2} a_{2} \bar{a}_{0} a_{0} \bar{a}_{2} a_{2} \bar{a}_{0} a_{0} \bar{a}_{2} a_{2} \bar{a}_{2} a_{2} \bar{a}_{0} a_{0}
	+ \bar{\alpha}^{(3)}_{45} \bar{a}_{3} \bar{a}_{0} a_{0} \bar{a}_{2} a_{2} \bar{a}_{0} a_{0} \bar{a}_{2} a_{2} \bar{a}_{0} a_{0} \bar{a}_{2} a_{2} \bar{a}_{0} a_{0} \bar{a}_{2} a_{2} \bar{a}_{0} a_{0} \bar{a}_{2} a_{2} \bar{a}_{0} a_{0} \bar{a}_{2} a_{2}
	\\&\quad
	+ \bar{\alpha}^{(3)}_{46} \bar{a}_{3} \bar{a}_{0} a_{0} \bar{a}_{2} a_{2} \bar{a}_{0} a_{0} \bar{a}_{2} a_{2} \bar{a}_{0} a_{0} \bar{a}_{2} a_{2} \bar{a}_{0} a_{0} \bar{a}_{2} a_{2} \bar{a}_{0} a_{0} \bar{a}_{2} a_{2} \bar{a}_{2} a_{2} \bar{a}_{0} a_{0}
	\\&\quad
	+ \bar{\alpha}^{(3)}_{47} \bar{a}_{3} \bar{a}_{0} a_{0} \bar{a}_{2} a_{2} \bar{a}_{0} a_{0} \bar{a}_{2} a_{2} \bar{a}_{0} a_{0} \bar{a}_{2} a_{2} \bar{a}_{0} a_{0} \bar{a}_{2} a_{2} \bar{a}_{0} a_{0} \bar{a}_{2} a_{2} \bar{a}_{0} a_{0} \bar{a}_{2} a_{2} \bar{a}_{0} a_{0}
	\\
	\varphi(a_{4}) &= 
	a_{4}
	+ \alpha^{(4)}_{1} \bar{a}_{3} a_{3} a_{4}
	+ \alpha^{(4)}_{2} \bar{a}_{3} \bar{a}_{0} a_{0} a_{3} a_{4}
	+ \alpha^{(4)}_{3} \bar{a}_{3} \bar{a}_{2} a_{2} a_{3} a_{4}
	+ \alpha^{(4)}_{4} \bar{a}_{3} \bar{a}_{0} a_{0} \bar{a}_{2} a_{2} a_{3} a_{4}
	+ \alpha^{(4)}_{5} \bar{a}_{3} \bar{a}_{2} a_{2} \bar{a}_{0} a_{0} a_{3} a_{4}
	+ \alpha^{(4)}_{6} \bar{a}_{3} \bar{a}_{0} a_{0} \bar{a}_{2} a_{2} \bar{a}_{0} a_{0} a_{3} a_{4}
	\\&\quad
	+ \alpha^{(4)}_{7} \bar{a}_{3} \bar{a}_{0} a_{0} \bar{a}_{2} a_{2} \bar{a}_{2} a_{2} a_{3} a_{4}
	+ \alpha^{(4)}_{8} \bar{a}_{3} \bar{a}_{2} a_{2} \bar{a}_{0} a_{0} \bar{a}_{2} a_{2} a_{3} a_{4}
	+ \alpha^{(4)}_{9} \bar{a}_{3} \bar{a}_{0} a_{0} \bar{a}_{2} a_{2} \bar{a}_{0} a_{0} \bar{a}_{2} a_{2} a_{3} a_{4}
	+ \alpha^{(4)}_{10} \bar{a}_{3} \bar{a}_{2} a_{2} \bar{a}_{0} a_{0} \bar{a}_{2} a_{2} \bar{a}_{0} a_{0} a_{3} a_{4}
	\\&\quad
	+ \alpha^{(4)}_{11} \bar{a}_{3} \bar{a}_{2} a_{2} \bar{a}_{0} a_{0} \bar{a}_{2} a_{2} \bar{a}_{2} a_{2} a_{3} a_{4}
	+ \alpha^{(4)}_{12} \bar{a}_{3} \bar{a}_{0} a_{0} \bar{a}_{2} a_{2} \bar{a}_{0} a_{0} \bar{a}_{2} a_{2} \bar{a}_{0} a_{0} a_{3} a_{4}
	+ \alpha^{(4)}_{13} \bar{a}_{3} \bar{a}_{0} a_{0} \bar{a}_{2} a_{2} \bar{a}_{0} a_{0} \bar{a}_{2} a_{2} \bar{a}_{2} a_{2} a_{3} a_{4}
	\\&\quad
	+ \alpha^{(4)}_{14} \bar{a}_{3} \bar{a}_{2} a_{2} \bar{a}_{0} a_{0} \bar{a}_{2} a_{2} \bar{a}_{0} a_{0} \bar{a}_{2} a_{2} a_{3} a_{4}
	+ \alpha^{(4)}_{15} \bar{a}_{3} \bar{a}_{0} a_{0} \bar{a}_{2} a_{2} \bar{a}_{0} a_{0} \bar{a}_{2} a_{2} \bar{a}_{0} a_{0} \bar{a}_{2} a_{2} a_{3} a_{4}
	+ \alpha^{(4)}_{16} \bar{a}_{3} \bar{a}_{0} a_{0} \bar{a}_{2} a_{2} \bar{a}_{2} a_{2} \bar{a}_{0} a_{0} \bar{a}_{2} a_{2} \bar{a}_{0} a_{0} a_{3} a_{4}
	\\&\quad
	+ \alpha^{(4)}_{17} \bar{a}_{3} \bar{a}_{2} a_{2} \bar{a}_{0} a_{0} \bar{a}_{2} a_{2} \bar{a}_{0} a_{0} \bar{a}_{2} a_{2} \bar{a}_{0} a_{0} a_{3} a_{4}
	+ \alpha^{(4)}_{18} \bar{a}_{3} \bar{a}_{0} a_{0} \bar{a}_{2} a_{2} \bar{a}_{0} a_{0} \bar{a}_{2} a_{2} \bar{a}_{0} a_{0} \bar{a}_{2} a_{2} \bar{a}_{0} a_{0} a_{3} a_{4}
	\\&\quad
	+ \alpha^{(4)}_{19} \bar{a}_{3} \bar{a}_{0} a_{0} \bar{a}_{2} a_{2} \bar{a}_{0} a_{0} \bar{a}_{2} a_{2} \bar{a}_{0} a_{0} \bar{a}_{2} a_{2} \bar{a}_{2} a_{2} a_{3} a_{4}
	+ \alpha^{(4)}_{20} \bar{a}_{3} \bar{a}_{2} a_{2} \bar{a}_{0} a_{0} \bar{a}_{2} a_{2} \bar{a}_{2} a_{2} \bar{a}_{0} a_{0} \bar{a}_{2} a_{2} \bar{a}_{0} a_{0} a_{3} a_{4}
	\\&\quad
	+ \alpha^{(4)}_{21} \bar{a}_{3} \bar{a}_{0} a_{0} \bar{a}_{2} a_{2} \bar{a}_{0} a_{0} \bar{a}_{2} a_{2} \bar{a}_{0} a_{0} \bar{a}_{2} a_{2} \bar{a}_{0} a_{0} \bar{a}_{2} a_{2} a_{3} a_{4}
	+ \alpha^{(4)}_{22} \bar{a}_{3} \bar{a}_{0} a_{0} \bar{a}_{2} a_{2} \bar{a}_{0} a_{0} \bar{a}_{2} a_{2} \bar{a}_{2} a_{2} \bar{a}_{0} a_{0} \bar{a}_{2} a_{2} \bar{a}_{0} a_{0} a_{3} a_{4}
	\\&\quad
	+ \alpha^{(4)}_{23} \bar{a}_{3} \bar{a}_{2} a_{2} \bar{a}_{0} a_{0} \bar{a}_{2} a_{2} \bar{a}_{0} a_{0} \bar{a}_{2} a_{2} \bar{a}_{0} a_{0} \bar{a}_{2} a_{2} \bar{a}_{0} a_{0} a_{3} a_{4}
	+ \alpha^{(4)}_{24} \bar{a}_{3} \bar{a}_{0} a_{0} \bar{a}_{2} a_{2} \bar{a}_{0} a_{0} \bar{a}_{2} a_{2} \bar{a}_{0} a_{0} \bar{a}_{2} a_{2} \bar{a}_{0} a_{0} \bar{a}_{2} a_{2} \bar{a}_{0} a_{0} a_{3} a_{4}
	\\&\quad
	+ \alpha^{(4)}_{25} \bar{a}_{3} \bar{a}_{2} a_{2} \bar{a}_{0} a_{0} \bar{a}_{2} a_{2} \bar{a}_{0} a_{0} \bar{a}_{2} a_{2} \bar{a}_{0} a_{0} \bar{a}_{2} a_{2} \bar{a}_{0} a_{0} \bar{a}_{2} a_{2} a_{3} a_{4}
	+ \alpha^{(4)}_{26} \bar{a}_{3} \bar{a}_{0} a_{0} \bar{a}_{2} a_{2} \bar{a}_{0} a_{0} \bar{a}_{2} a_{2} \bar{a}_{0} a_{0} \bar{a}_{2} a_{2} \bar{a}_{0} a_{0} \bar{a}_{2} a_{2} \bar{a}_{0} a_{0} \bar{a}_{2} a_{2} a_{3} a_{4}
	\\&\quad
	+ \alpha^{(4)}_{27} \bar{a}_{3} \bar{a}_{0} a_{0} \bar{a}_{2} a_{2} \bar{a}_{0} a_{0} \bar{a}_{2} a_{2} \bar{a}_{0} a_{0} \bar{a}_{2} a_{2} \bar{a}_{2} a_{2} \bar{a}_{0} a_{0} \bar{a}_{2} a_{2} \bar{a}_{0} a_{0} a_{3} a_{4}
	+ \alpha^{(4)}_{28} \bar{a}_{3} \bar{a}_{0} a_{0} \bar{a}_{2} a_{2} \bar{a}_{0} a_{0} \bar{a}_{2} a_{2} \bar{a}_{0} a_{0} \bar{a}_{2} a_{2} \bar{a}_{0} a_{0} \bar{a}_{2} a_{2} \bar{a}_{0} a_{0} \bar{a}_{2} a_{2} \bar{a}_{0} a_{0} a_{3} a_{4}
	\\&\quad
	+ \alpha^{(4)}_{29} \bar{a}_{3} \bar{a}_{0} a_{0} \bar{a}_{2} a_{2} \bar{a}_{0} a_{0} \bar{a}_{2} a_{2} \bar{a}_{0} a_{0} \bar{a}_{2} a_{2} \bar{a}_{0} a_{0} \bar{a}_{2} a_{2} \bar{a}_{0} a_{0} \bar{a}_{2} a_{2} \bar{a}_{0} a_{0} \bar{a}_{2} a_{2} a_{3} a_{4}
	\\
	\varphi(\bar{a}_{4}) &= 
	\bar{a}_{4}
	+ \bar{\alpha}^{(4)}_{1} \bar{a}_{4} \bar{a}_{3} a_{3}
	+ \bar{\alpha}^{(4)}_{2} \bar{a}_{4} \bar{a}_{3} \bar{a}_{0} a_{0} a_{3}
	+ \bar{\alpha}^{(4)}_{3} \bar{a}_{4} \bar{a}_{3} \bar{a}_{2} a_{2} a_{3}
	+ \bar{\alpha}^{(4)}_{4} \bar{a}_{4} \bar{a}_{3} \bar{a}_{0} a_{0} \bar{a}_{2} a_{2} a_{3}
	+ \bar{\alpha}^{(4)}_{5} \bar{a}_{4} \bar{a}_{3} \bar{a}_{2} a_{2} \bar{a}_{0} a_{0} a_{3}
	+ \bar{\alpha}^{(4)}_{6} \bar{a}_{4} \bar{a}_{3} \bar{a}_{0} a_{0} \bar{a}_{2} a_{2} \bar{a}_{0} a_{0} a_{3}
	\\&\quad
	+ \bar{\alpha}^{(4)}_{7} \bar{a}_{4} \bar{a}_{3} \bar{a}_{0} a_{0} \bar{a}_{2} a_{2} \bar{a}_{2} a_{2} a_{3}
	+ \bar{\alpha}^{(4)}_{8} \bar{a}_{4} \bar{a}_{3} \bar{a}_{2} a_{2} \bar{a}_{0} a_{0} \bar{a}_{2} a_{2} a_{3}
	+ \bar{\alpha}^{(4)}_{9} \bar{a}_{4} \bar{a}_{3} \bar{a}_{0} a_{0} \bar{a}_{2} a_{2} \bar{a}_{0} a_{0} \bar{a}_{2} a_{2} a_{3}
	+ \bar{\alpha}^{(4)}_{10} \bar{a}_{4} \bar{a}_{3} \bar{a}_{0} a_{0} \bar{a}_{2} a_{2} \bar{a}_{2} a_{2} \bar{a}_{0} a_{0} a_{3}
	\\&\quad
	+ \bar{\alpha}^{(4)}_{11} \bar{a}_{4} \bar{a}_{3} \bar{a}_{2} a_{2} \bar{a}_{0} a_{0} \bar{a}_{2} a_{2} \bar{a}_{2} a_{2} a_{3}
	+ \bar{\alpha}^{(4)}_{12} \bar{a}_{4} \bar{a}_{3} \bar{a}_{0} a_{0} \bar{a}_{2} a_{2} \bar{a}_{0} a_{0} \bar{a}_{2} a_{2} \bar{a}_{0} a_{0} a_{3}
	+ \bar{\alpha}^{(4)}_{13} \bar{a}_{4} \bar{a}_{3} \bar{a}_{0} a_{0} \bar{a}_{2} a_{2} \bar{a}_{0} a_{0} \bar{a}_{2} a_{2} \bar{a}_{2} a_{2} a_{3}
	\\&\quad
	+ \bar{\alpha}^{(4)}_{14} \bar{a}_{4} \bar{a}_{3} \bar{a}_{2} a_{2} \bar{a}_{0} a_{0} \bar{a}_{2} a_{2} \bar{a}_{2} a_{2} \bar{a}_{0} a_{0} a_{3}
	+ \bar{\alpha}^{(4)}_{15} \bar{a}_{4} \bar{a}_{3} \bar{a}_{0} a_{0} \bar{a}_{2} a_{2} \bar{a}_{0} a_{0} \bar{a}_{2} a_{2} \bar{a}_{0} a_{0} \bar{a}_{2} a_{2} a_{3}
	+ \bar{\alpha}^{(4)}_{16} \bar{a}_{4} \bar{a}_{3} \bar{a}_{0} a_{0} \bar{a}_{2} a_{2} \bar{a}_{0} a_{0} \bar{a}_{2} a_{2} \bar{a}_{2} a_{2} \bar{a}_{0} a_{0} a_{3}
	\\&\quad
	+ \bar{\alpha}^{(4)}_{17} \bar{a}_{4} \bar{a}_{3} \bar{a}_{0} a_{0} \bar{a}_{2} a_{2} \bar{a}_{2} a_{2} \bar{a}_{0} a_{0} \bar{a}_{2} a_{2} \bar{a}_{0} a_{0} a_{3}
	+ \bar{\alpha}^{(4)}_{18} \bar{a}_{4} \bar{a}_{3} \bar{a}_{0} a_{0} \bar{a}_{2} a_{2} \bar{a}_{0} a_{0} \bar{a}_{2} a_{2} \bar{a}_{0} a_{0} \bar{a}_{2} a_{2} \bar{a}_{0} a_{0} a_{3}
	\\&\quad
	+ \bar{\alpha}^{(4)}_{19} \bar{a}_{4} \bar{a}_{3} \bar{a}_{0} a_{0} \bar{a}_{2} a_{2} \bar{a}_{0} a_{0} \bar{a}_{2} a_{2} \bar{a}_{0} a_{0} \bar{a}_{2} a_{2} \bar{a}_{2} a_{2} a_{3}
	+ \bar{\alpha}^{(4)}_{20} \bar{a}_{4} \bar{a}_{3} \bar{a}_{2} a_{2} \bar{a}_{0} a_{0} \bar{a}_{2} a_{2} \bar{a}_{2} a_{2} \bar{a}_{0} a_{0} \bar{a}_{2} a_{2} \bar{a}_{0} a_{0} a_{3}
	\\&\quad
	+ \bar{\alpha}^{(4)}_{21} \bar{a}_{4} \bar{a}_{3} \bar{a}_{0} a_{0} \bar{a}_{2} a_{2} \bar{a}_{0} a_{0} \bar{a}_{2} a_{2} \bar{a}_{0} a_{0} \bar{a}_{2} a_{2} \bar{a}_{0} a_{0} \bar{a}_{2} a_{2} a_{3}
	+ \bar{\alpha}^{(4)}_{22} \bar{a}_{4} \bar{a}_{3} \bar{a}_{0} a_{0} \bar{a}_{2} a_{2} \bar{a}_{0} a_{0} \bar{a}_{2} a_{2} \bar{a}_{0} a_{0} \bar{a}_{2} a_{2} \bar{a}_{2} a_{2} \bar{a}_{0} a_{0} a_{3}
	\\&\quad
	+ \bar{\alpha}^{(4)}_{23} \bar{a}_{4} \bar{a}_{3} \bar{a}_{0} a_{0} \bar{a}_{2} a_{2} \bar{a}_{0} a_{0} \bar{a}_{2} a_{2} \bar{a}_{2} a_{2} \bar{a}_{0} a_{0} \bar{a}_{2} a_{2} \bar{a}_{0} a_{0} a_{3}
	+ \bar{\alpha}^{(4)}_{24} \bar{a}_{4} \bar{a}_{3} \bar{a}_{0} a_{0} \bar{a}_{2} a_{2} \bar{a}_{0} a_{0} \bar{a}_{2} a_{2} \bar{a}_{0} a_{0} \bar{a}_{2} a_{2} \bar{a}_{0} a_{0} \bar{a}_{2} a_{2} \bar{a}_{0} a_{0} a_{3}
	\\&\quad
	+ \bar{\alpha}^{(4)}_{25} \bar{a}_{4} \bar{a}_{3} \bar{a}_{0} a_{0} \bar{a}_{2} a_{2} \bar{a}_{0} a_{0} \bar{a}_{2} a_{2} \bar{a}_{0} a_{0} \bar{a}_{2} a_{2} \bar{a}_{0} a_{0} \bar{a}_{2} a_{2} \bar{a}_{2} a_{2} a_{3}
	+ \bar{\alpha}^{(4)}_{26} \bar{a}_{4} \bar{a}_{3} \bar{a}_{0} a_{0} \bar{a}_{2} a_{2} \bar{a}_{0} a_{0} \bar{a}_{2} a_{2} \bar{a}_{0} a_{0} \bar{a}_{2} a_{2} \bar{a}_{0} a_{0} \bar{a}_{2} a_{2} \bar{a}_{0} a_{0} \bar{a}_{2} a_{2} a_{3}
	\\&\quad
	+ \bar{\alpha}^{(4)}_{27} \bar{a}_{4} \bar{a}_{3} \bar{a}_{0} a_{0} \bar{a}_{2} a_{2} \bar{a}_{0} a_{0} \bar{a}_{2} a_{2} \bar{a}_{0} a_{0} \bar{a}_{2} a_{2} \bar{a}_{2} a_{2} \bar{a}_{0} a_{0} \bar{a}_{2} a_{2} \bar{a}_{0} a_{0} a_{3}
	+ \bar{\alpha}^{(4)}_{28} \bar{a}_{4} \bar{a}_{3} \bar{a}_{0} a_{0} \bar{a}_{2} a_{2} \bar{a}_{0} a_{0} \bar{a}_{2} a_{2} \bar{a}_{0} a_{0} \bar{a}_{2} a_{2} \bar{a}_{0} a_{0} \bar{a}_{2} a_{2} \bar{a}_{0} a_{0} \bar{a}_{2} a_{2} \bar{a}_{0} a_{0} a_{3}
	\\&\quad
	+ \bar{\alpha}^{(4)}_{29} \bar{a}_{4} \bar{a}_{3} \bar{a}_{0} a_{0} \bar{a}_{2} a_{2} \bar{a}_{0} a_{0} \bar{a}_{2} a_{2} \bar{a}_{0} a_{0} \bar{a}_{2} a_{2} \bar{a}_{0} a_{0} \bar{a}_{2} a_{2} \bar{a}_{0} a_{0} \bar{a}_{2} a_{2} \bar{a}_{0} a_{0} \bar{a}_{2} a_{2} a_{3}
	\\
	\varphi(a_{5}) &= 
	a_{5}
	+ \alpha^{(5)}_{1} \bar{a}_{4} a_{4} a_{5}
	+ \alpha^{(5)}_{2} \bar{a}_{4} \bar{a}_{3} \bar{a}_{0} a_{0} a_{3} a_{4} a_{5}
	+ \alpha^{(5)}_{3} \bar{a}_{4} \bar{a}_{3} \bar{a}_{0} a_{0} \bar{a}_{2} a_{2} a_{3} a_{4} a_{5}
	+ \alpha^{(5)}_{4} \bar{a}_{4} \bar{a}_{3} \bar{a}_{2} a_{2} \bar{a}_{0} a_{0} a_{3} a_{4} a_{5}
	+ \alpha^{(5)}_{5} \bar{a}_{4} \bar{a}_{3} \bar{a}_{0} a_{0} \bar{a}_{2} a_{2} \bar{a}_{0} a_{0} a_{3} a_{4} a_{5}
	\\&\quad
	+ \alpha^{(5)}_{6} \bar{a}_{4} \bar{a}_{3} \bar{a}_{2} a_{2} \bar{a}_{0} a_{0} \bar{a}_{2} a_{2} a_{3} a_{4} a_{5}
	+ \alpha^{(5)}_{7} \bar{a}_{4} \bar{a}_{3} \bar{a}_{0} a_{0} \bar{a}_{2} a_{2} \bar{a}_{0} a_{0} \bar{a}_{2} a_{2} a_{3} a_{4} a_{5}
	+ \alpha^{(5)}_{8} \bar{a}_{4} \bar{a}_{3} \bar{a}_{0} a_{0} \bar{a}_{2} a_{2} \bar{a}_{0} a_{0} \bar{a}_{2} a_{2} \bar{a}_{0} a_{0} a_{3} a_{4} a_{5}
	\\&\quad
	+ \alpha^{(5)}_{9} \bar{a}_{4} \bar{a}_{3} \bar{a}_{0} a_{0} \bar{a}_{2} a_{2} \bar{a}_{0} a_{0} \bar{a}_{2} a_{2} \bar{a}_{0} a_{0} \bar{a}_{2} a_{2} a_{3} a_{4} a_{5}
	+ \alpha^{(5)}_{10} \bar{a}_{4} \bar{a}_{3} \bar{a}_{0} a_{0} \bar{a}_{2} a_{2} \bar{a}_{2} a_{2} \bar{a}_{0} a_{0} \bar{a}_{2} a_{2} \bar{a}_{0} a_{0} a_{3} a_{4} a_{5}
	\\&\quad
	+ \alpha^{(5)}_{11} \bar{a}_{4} \bar{a}_{3} \bar{a}_{0} a_{0} \bar{a}_{2} a_{2} \bar{a}_{0} a_{0} \bar{a}_{2} a_{2} \bar{a}_{0} a_{0} \bar{a}_{2} a_{2} \bar{a}_{0} a_{0} a_{3} a_{4} a_{5}
	+ \alpha^{(5)}_{12} \bar{a}_{4} \bar{a}_{3} \bar{a}_{2} a_{2} \bar{a}_{0} a_{0} \bar{a}_{2} a_{2} \bar{a}_{2} a_{2} \bar{a}_{0} a_{0} \bar{a}_{2} a_{2} \bar{a}_{0} a_{0} a_{3} a_{4} a_{5}
	\\&\quad
	+ \alpha^{(5)}_{13} \bar{a}_{4} \bar{a}_{3} \bar{a}_{0} a_{0} \bar{a}_{2} a_{2} \bar{a}_{0} a_{0} \bar{a}_{2} a_{2} \bar{a}_{2} a_{2} \bar{a}_{0} a_{0} \bar{a}_{2} a_{2} \bar{a}_{0} a_{0} a_{3} a_{4} a_{5}
	+ \alpha^{(5)}_{14} \bar{a}_{4} \bar{a}_{3} \bar{a}_{0} a_{0} \bar{a}_{2} a_{2} \bar{a}_{0} a_{0} \bar{a}_{2} a_{2} \bar{a}_{0} a_{0} \bar{a}_{2} a_{2} \bar{a}_{2} a_{2} \bar{a}_{0} a_{0} \bar{a}_{2} a_{2} \bar{a}_{0} a_{0} a_{3} a_{4} a_{5}
	\\&\quad
	+ \alpha^{(5)}_{15} \bar{a}_{4} \bar{a}_{3} \bar{a}_{0} a_{0} \bar{a}_{2} a_{2} \bar{a}_{0} a_{0} \bar{a}_{2} a_{2} \bar{a}_{0} a_{0} \bar{a}_{2} a_{2} \bar{a}_{0} a_{0} \bar{a}_{2} a_{2} \bar{a}_{0} a_{0} \bar{a}_{2} a_{2} \bar{a}_{0} a_{0} a_{3} a_{4} a_{5}
	\\
	\varphi(\bar{a}_{5}) &= 
	\bar{a}_{5}
	+ \bar{\alpha}^{(5)}_{1} \bar{a}_{5} \bar{a}_{4} a_{4}
	+ \bar{\alpha}^{(5)}_{2} \bar{a}_{5} \bar{a}_{4} \bar{a}_{3} \bar{a}_{0} a_{0} a_{3} a_{4}
	+ \bar{\alpha}^{(5)}_{3} \bar{a}_{5} \bar{a}_{4} \bar{a}_{3} \bar{a}_{0} a_{0} \bar{a}_{2} a_{2} a_{3} a_{4}
	+ \bar{\alpha}^{(5)}_{4} \bar{a}_{5} \bar{a}_{4} \bar{a}_{3} \bar{a}_{2} a_{2} \bar{a}_{0} a_{0} a_{3} a_{4}
	+ \bar{\alpha}^{(5)}_{5} \bar{a}_{5} \bar{a}_{4} \bar{a}_{3} \bar{a}_{0} a_{0} \bar{a}_{2} a_{2} \bar{a}_{0} a_{0} a_{3} a_{4}
	\\&\quad
	+ \bar{\alpha}^{(5)}_{6} \bar{a}_{5} \bar{a}_{4} \bar{a}_{3} \bar{a}_{0} a_{0} \bar{a}_{2} a_{2} \bar{a}_{2} a_{2} a_{3} a_{4}
	+ \bar{\alpha}^{(5)}_{7} \bar{a}_{5} \bar{a}_{4} \bar{a}_{3} \bar{a}_{0} a_{0} \bar{a}_{2} a_{2} \bar{a}_{0} a_{0} \bar{a}_{2} a_{2} a_{3} a_{4}
	+ \bar{\alpha}^{(5)}_{8} \bar{a}_{5} \bar{a}_{4} \bar{a}_{3} \bar{a}_{0} a_{0} \bar{a}_{2} a_{2} \bar{a}_{0} a_{0} \bar{a}_{2} a_{2} \bar{a}_{0} a_{0} a_{3} a_{4}
	\\&\quad
	+ \bar{\alpha}^{(5)}_{9} \bar{a}_{5} \bar{a}_{4} \bar{a}_{3} \bar{a}_{0} a_{0} \bar{a}_{2} a_{2} \bar{a}_{0} a_{0} \bar{a}_{2} a_{2} \bar{a}_{0} a_{0} \bar{a}_{2} a_{2} a_{3} a_{4}
	+ \bar{\alpha}^{(5)}_{10} \bar{a}_{5} \bar{a}_{4} \bar{a}_{3} \bar{a}_{0} a_{0} \bar{a}_{2} a_{2} \bar{a}_{2} a_{2} \bar{a}_{0} a_{0} \bar{a}_{2} a_{2} \bar{a}_{0} a_{0} a_{3} a_{4}
	\\&\quad
	+ \bar{\alpha}^{(5)}_{11} \bar{a}_{5} \bar{a}_{4} \bar{a}_{3} \bar{a}_{0} a_{0} \bar{a}_{2} a_{2} \bar{a}_{0} a_{0} \bar{a}_{2} a_{2} \bar{a}_{0} a_{0} \bar{a}_{2} a_{2} \bar{a}_{0} a_{0} a_{3} a_{4}
	+ \bar{\alpha}^{(5)}_{12} \bar{a}_{5} \bar{a}_{4} \bar{a}_{3} \bar{a}_{0} a_{0} \bar{a}_{2} a_{2} \bar{a}_{0} a_{0} \bar{a}_{2} a_{2} \bar{a}_{0} a_{0} \bar{a}_{2} a_{2} \bar{a}_{2} a_{2} a_{3} a_{4}
	\\&\quad
	+ \bar{\alpha}^{(5)}_{13} \bar{a}_{5} \bar{a}_{4} \bar{a}_{3} \bar{a}_{0} a_{0} \bar{a}_{2} a_{2} \bar{a}_{0} a_{0} \bar{a}_{2} a_{2} \bar{a}_{0} a_{0} \bar{a}_{2} a_{2} \bar{a}_{0} a_{0} \bar{a}_{2} a_{2} a_{3} a_{4}
	+ \bar{\alpha}^{(5)}_{14} \bar{a}_{5} \bar{a}_{4} \bar{a}_{3} \bar{a}_{0} a_{0} \bar{a}_{2} a_{2} \bar{a}_{0} a_{0} \bar{a}_{2} a_{2} \bar{a}_{0} a_{0} \bar{a}_{2} a_{2} \bar{a}_{2} a_{2} \bar{a}_{0} a_{0} \bar{a}_{2} a_{2} \bar{a}_{0} a_{0} a_{3} a_{4}
	\\&\quad
	+ \bar{\alpha}^{(5)}_{15} \bar{a}_{5} \bar{a}_{4} \bar{a}_{3} \bar{a}_{0} a_{0} \bar{a}_{2} a_{2} \bar{a}_{0} a_{0} \bar{a}_{2} a_{2} \bar{a}_{0} a_{0} \bar{a}_{2} a_{2} \bar{a}_{0} a_{0} \bar{a}_{2} a_{2} \bar{a}_{0} a_{0} \bar{a}_{2} a_{2} \bar{a}_{0} a_{0} a_{3} a_{4}
	\\
	\varphi(a_{6}) &= 
	a_{6}
	+ \alpha^{(6)}_{1} \bar{a}_{5} \bar{a}_{4} \bar{a}_{3} \bar{a}_{0} a_{0} a_{3} a_{4} a_{5} a_{6}
	+ \alpha^{(6)}_{2} \bar{a}_{5} \bar{a}_{4} \bar{a}_{3} \bar{a}_{2} a_{2} \bar{a}_{0} a_{0} a_{3} a_{4} a_{5} a_{6}
	+ \alpha^{(6)}_{3} \bar{a}_{5} \bar{a}_{4} \bar{a}_{3} \bar{a}_{0} a_{0} \bar{a}_{2} a_{2} \bar{a}_{0} a_{0} \bar{a}_{2} a_{2} \bar{a}_{0} a_{0} a_{3} a_{4} a_{5} a_{6}
	\\&\quad
	+ \alpha^{(6)}_{4} \bar{a}_{5} \bar{a}_{4} \bar{a}_{3} \bar{a}_{0} a_{0} \bar{a}_{2} a_{2} \bar{a}_{2} a_{2} \bar{a}_{0} a_{0} \bar{a}_{2} a_{2} \bar{a}_{0} a_{0} a_{3} a_{4} a_{5} a_{6}
	+ \alpha^{(6)}_{5} \bar{a}_{5} \bar{a}_{4} \bar{a}_{3} \bar{a}_{0} a_{0} \bar{a}_{2} a_{2} \bar{a}_{0} a_{0} \bar{a}_{2} a_{2} \bar{a}_{0} a_{0} \bar{a}_{2} a_{2} \bar{a}_{2} a_{2} \bar{a}_{0} a_{0} \bar{a}_{2} a_{2} \bar{a}_{0} a_{0} a_{3} a_{4} a_{5} a_{6}
	\\
	\varphi(\bar{a}_{6}) &= 
	\bar{a}_{6}
	+ \bar{\alpha}^{(6)}_{1} \bar{a}_{6} \bar{a}_{5} \bar{a}_{4} \bar{a}_{3} \bar{a}_{0} a_{0} a_{3} a_{4} a_{5}
	+ \bar{\alpha}^{(6)}_{2} \bar{a}_{6} \bar{a}_{5} \bar{a}_{4} \bar{a}_{3} \bar{a}_{0} a_{0} \bar{a}_{2} a_{2} a_{3} a_{4} a_{5}
	+ \bar{\alpha}^{(6)}_{3} \bar{a}_{6} \bar{a}_{5} \bar{a}_{4} \bar{a}_{3} \bar{a}_{0} a_{0} \bar{a}_{2} a_{2} \bar{a}_{0} a_{0} \bar{a}_{2} a_{2} \bar{a}_{0} a_{0} a_{3} a_{4} a_{5}
	\\&\quad
	+ \bar{\alpha}^{(6)}_{4} \bar{a}_{6} \bar{a}_{5} \bar{a}_{4} \bar{a}_{3} \bar{a}_{0} a_{0} \bar{a}_{2} a_{2} \bar{a}_{0} a_{0} \bar{a}_{2} a_{2} \bar{a}_{0} a_{0} \bar{a}_{2} a_{2} a_{3} a_{4} a_{5}
	+ \bar{\alpha}^{(6)}_{5} \bar{a}_{6} \bar{a}_{5} \bar{a}_{4} \bar{a}_{3} \bar{a}_{0} a_{0} \bar{a}_{2} a_{2} \bar{a}_{0} a_{0} \bar{a}_{2} a_{2} \bar{a}_{0} a_{0} \bar{a}_{2} a_{2} \bar{a}_{2} a_{2} \bar{a}_{0} a_{0} \bar{a}_{2} a_{2} \bar{a}_{0} a_{0} a_{3} a_{4} a_{5}
	\end{align*}%
\end{footnotesize}%
with coefficients ${\alpha}^{(i)}_{j_i}, \bar{\alpha}^{(i)}_{j_i} \in K$,
for $i=0,\dots,6$, $j = 1, \dots,  j_i$, 
$j_0 = 29, j_1 = 13, j_2 = 39, j_3 = 47, j_4 = 29, j_5 = 15, j_6 = 5$.

\subsection{Simplified system of equations}
\label{secB:simpl}

In order to obtain the simplified system of equations
(see Section~\ref{sec:solv} for details)
we will assume that in the above definition of $\varphi$ 
we have $\alpha^{(2)}_{1} = 1$ and that the following
coefficients are equal zero:
\begin{small}
\setlength{\jot}{2pt}
\begin{gather*}
\alpha^{(0)}_{2},
\alpha^{(0)}_{6},
\alpha^{(0)}_{10},
\alpha^{(0)}_{11},
\alpha^{(0)}_{16},
\alpha^{(0)}_{22},
\alpha^{(0)}_{24},
\alpha^{(0)}_{26},
\alpha^{(0)}_{27},
\alpha^{(0)}_{28},
\alpha^{(0)}_{29},
\bar{\alpha}^{(0)}_{4},
\bar{\alpha}^{(0)}_{12},
\bar{\alpha}^{(0)}_{13},
\bar{\alpha}^{(0)}_{16},
\bar{\alpha}^{(0)}_{18},
\bar{\alpha}^{(0)}_{21},
\bar{\alpha}^{(0)}_{22},
\bar{\alpha}^{(0)}_{23},
\bar{\alpha}^{(0)}_{25},
\alpha^{(1)}_{6},
\alpha^{(1)}_{8},
\bar{\alpha}^{(1)}_{2},
\bar{\alpha}^{(1)}_{3},
\\
\bar{\alpha}^{(1)}_{4},
\bar{\alpha}^{(1)}_{5},
\bar{\alpha}^{(1)}_{6},
\bar{\alpha}^{(1)}_{7},
\bar{\alpha}^{(1)}_{8},
\bar{\alpha}^{(1)}_{9},
\bar{\alpha}^{(1)}_{10},
\bar{\alpha}^{(1)}_{11},
\bar{\alpha}^{(1)}_{12},
\alpha^{(2)}_{15},
\alpha^{(2)}_{16},
\alpha^{(2)}_{29},
\alpha^{(2)}_{30},
\alpha^{(2)}_{31},
\alpha^{(2)}_{32},
\alpha^{(2)}_{35},
\alpha^{(2)}_{36},
\alpha^{(2)}_{39},
\bar{\alpha}^{(2)}_{3},
\bar{\alpha}^{(2)}_{4},
\bar{\alpha}^{(2)}_{5},
\bar{\alpha}^{(2)}_{6},
\bar{\alpha}^{(2)}_{7},
\\
\bar{\alpha}^{(2)}_{8},
\bar{\alpha}^{(2)}_{9},
\bar{\alpha}^{(2)}_{10},
\bar{\alpha}^{(2)}_{11},
\bar{\alpha}^{(2)}_{12},
\bar{\alpha}^{(2)}_{13},
\bar{\alpha}^{(2)}_{14},
\bar{\alpha}^{(2)}_{15},
\bar{\alpha}^{(2)}_{16},
\bar{\alpha}^{(2)}_{17},
\bar{\alpha}^{(2)}_{18},
\bar{\alpha}^{(2)}_{19},
\bar{\alpha}^{(2)}_{20},
\bar{\alpha}^{(2)}_{21},
\bar{\alpha}^{(2)}_{22},
\bar{\alpha}^{(2)}_{23},
\bar{\alpha}^{(2)}_{24},
\bar{\alpha}^{(2)}_{25},
\bar{\alpha}^{(2)}_{26},
\bar{\alpha}^{(2)}_{27},
\bar{\alpha}^{(2)}_{28},
\bar{\alpha}^{(2)}_{29},
\bar{\alpha}^{(2)}_{30},
\\
\bar{\alpha}^{(2)}_{31},
\bar{\alpha}^{(2)}_{32},
\bar{\alpha}^{(2)}_{33},
\bar{\alpha}^{(2)}_{34},
\bar{\alpha}^{(2)}_{35},
\bar{\alpha}^{(2)}_{36},
\bar{\alpha}^{(2)}_{37},
\bar{\alpha}^{(2)}_{38},
\alpha^{(3)}_{2},
\alpha^{(3)}_{3},
\alpha^{(3)}_{4},
\alpha^{(3)}_{5},
\alpha^{(3)}_{6},
\alpha^{(3)}_{7},
\alpha^{(3)}_{8},
\alpha^{(3)}_{9},
\alpha^{(3)}_{10},
\alpha^{(3)}_{11},
\alpha^{(3)}_{12},
\alpha^{(3)}_{13},
\alpha^{(3)}_{14},
\alpha^{(3)}_{15},
\alpha^{(3)}_{16},
\\
\alpha^{(3)}_{17},
\alpha^{(3)}_{18},
\alpha^{(3)}_{19},
\alpha^{(3)}_{20},
\alpha^{(3)}_{21},
\alpha^{(3)}_{22},
\alpha^{(3)}_{23},
\alpha^{(3)}_{24},
\alpha^{(3)}_{25},
\alpha^{(3)}_{26},
\alpha^{(3)}_{27},
\alpha^{(3)}_{28},
\alpha^{(3)}_{29},
\alpha^{(3)}_{30},
\alpha^{(3)}_{31},
\alpha^{(3)}_{32},
\alpha^{(3)}_{33},
\alpha^{(3)}_{34},
\alpha^{(3)}_{35},
\alpha^{(3)}_{36},
\alpha^{(3)}_{37},
\alpha^{(3)}_{38},
\alpha^{(3)}_{39},
\\
\alpha^{(3)}_{40},
\alpha^{(3)}_{41},
\alpha^{(3)}_{42},
\alpha^{(3)}_{43},
\alpha^{(3)}_{44},
\alpha^{(3)}_{45},
\alpha^{(3)}_{46},
\alpha^{(3)}_{47},
\bar{\alpha}^{(3)}_{1},
\bar{\alpha}^{(3)}_{10},
\bar{\alpha}^{(3)}_{17},
\bar{\alpha}^{(3)}_{45},
\bar{\alpha}^{(3)}_{47},
\alpha^{(4)}_{2},
\alpha^{(4)}_{3},
\alpha^{(4)}_{4},
\alpha^{(4)}_{5},
\alpha^{(4)}_{6},
\alpha^{(4)}_{7},
\alpha^{(4)}_{8},
\alpha^{(4)}_{9},
\alpha^{(4)}_{10},
\alpha^{(4)}_{11},
\\
\alpha^{(4)}_{12},
\alpha^{(4)}_{13},
\alpha^{(4)}_{14},
\alpha^{(4)}_{15},
\alpha^{(4)}_{16},
\alpha^{(4)}_{17},
\alpha^{(4)}_{18},
\alpha^{(4)}_{19},
\alpha^{(4)}_{20},
\alpha^{(4)}_{21},
\alpha^{(4)}_{22},
\alpha^{(4)}_{23},
\alpha^{(4)}_{24},
\alpha^{(4)}_{25},
\alpha^{(4)}_{26},
\alpha^{(4)}_{27},
\alpha^{(4)}_{28},
\bar{\alpha}^{(4)}_{1},
\bar{\alpha}^{(4)}_{4},
\bar{\alpha}^{(4)}_{5},
\bar{\alpha}^{(4)}_{6},
\bar{\alpha}^{(4)}_{10},
\bar{\alpha}^{(4)}_{17},
\\
\bar{\alpha}^{(4)}_{19},
\bar{\alpha}^{(4)}_{27},
\alpha^{(5)}_{2},
\alpha^{(5)}_{3},
\alpha^{(5)}_{4},
\alpha^{(5)}_{5},
\alpha^{(5)}_{6},
\alpha^{(5)}_{7},
\alpha^{(5)}_{8},
\alpha^{(5)}_{9},
\alpha^{(5)}_{10},
\alpha^{(5)}_{11},
\alpha^{(5)}_{12},
\alpha^{(5)}_{13},
\alpha^{(5)}_{14},
\bar{\alpha}^{(5)}_{1},
\bar{\alpha}^{(5)}_{2},
\bar{\alpha}^{(5)}_{8},
\bar{\alpha}^{(5)}_{15},
\alpha^{(6)}_{2},
\alpha^{(6)}_{3},
\alpha^{(6)}_{4},
\bar{\alpha}^{(6)}_{5}.
\end{gather*}
\end{small}%
Further,
from 
the relations of $P(\mathbb{E}_8)$
(see Section~\ref{sec:eq})
we obtain 
the set 
of 170 equations 
\cite[\texttt{e8-equations-reduced.txt}]{B:E8zip}.
Then,
we apply to it the set of 10 
formulas 
for
$
\theta_2,
\theta_6,
\theta_{12}$, $
\theta_{17}$, $
\theta_{18}$, $
\theta_{19}$, $
\theta_{28},
\theta_{36},
\theta_{40},
\theta_{45}
$
described
in \ref{secB:adm}
and
solve it
(see the next section for details)
for the remaining variables 
\begin{small}
\setlength{\jot}{2pt}
\begin{gather*}
\alpha^{(0)}_{1},
\alpha^{(0)}_{3},
\alpha^{(0)}_{4},
\alpha^{(0)}_{5},
\alpha^{(0)}_{7},
\alpha^{(0)}_{8},
\alpha^{(0)}_{9},
\alpha^{(0)}_{12},
\alpha^{(0)}_{13},
\alpha^{(0)}_{14},
\alpha^{(0)}_{15},
\alpha^{(0)}_{17},
\alpha^{(0)}_{18},
\alpha^{(0)}_{19},
\alpha^{(0)}_{20},
\alpha^{(0)}_{21},
\alpha^{(0)}_{23},
\alpha^{(0)}_{25},
\bar{\alpha}^{(0)}_{1},
\bar{\alpha}^{(0)}_{2},
\bar{\alpha}^{(0)}_{3},
\bar{\alpha}^{(0)}_{5},
\bar{\alpha}^{(0)}_{6},
\bar{\alpha}^{(0)}_{7},
\\
\bar{\alpha}^{(0)}_{8},
\bar{\alpha}^{(0)}_{9},
\bar{\alpha}^{(0)}_{10},
\bar{\alpha}^{(0)}_{11},
\bar{\alpha}^{(0)}_{14},
\bar{\alpha}^{(0)}_{15},
\bar{\alpha}^{(0)}_{17},
\bar{\alpha}^{(0)}_{19},
\bar{\alpha}^{(0)}_{20},
\bar{\alpha}^{(0)}_{24},
\bar{\alpha}^{(0)}_{26},
\bar{\alpha}^{(0)}_{27},
\bar{\alpha}^{(0)}_{28},
\bar{\alpha}^{(0)}_{29},
\alpha^{(1)}_{1},
\alpha^{(1)}_{2},
\alpha^{(1)}_{3},
\alpha^{(1)}_{4},
\alpha^{(1)}_{5},
\alpha^{(1)}_{7},
\alpha^{(1)}_{9},
\alpha^{(1)}_{10},
\alpha^{(1)}_{11},
\alpha^{(1)}_{12},
\\
\alpha^{(1)}_{13},
\bar{\alpha}^{(1)}_{1},
\bar{\alpha}^{(1)}_{13},
\alpha^{(2)}_{2},
\alpha^{(2)}_{3},
\alpha^{(2)}_{4},
\alpha^{(2)}_{5},
\alpha^{(2)}_{6},
\alpha^{(2)}_{7},
\alpha^{(2)}_{8},
\alpha^{(2)}_{9},
\alpha^{(2)}_{10},
\alpha^{(2)}_{11},
\alpha^{(2)}_{12},
\alpha^{(2)}_{13},
\alpha^{(2)}_{14},
\alpha^{(2)}_{17},
\alpha^{(2)}_{18},
\alpha^{(2)}_{19},
\alpha^{(2)}_{20},
\alpha^{(2)}_{21},
\alpha^{(2)}_{22},
\alpha^{(2)}_{23},
\alpha^{(2)}_{24},
\\
\alpha^{(2)}_{25},
\alpha^{(2)}_{26},
\alpha^{(2)}_{27},
\alpha^{(2)}_{28},
\alpha^{(2)}_{33},
\alpha^{(2)}_{34},
\alpha^{(2)}_{37},
\alpha^{(2)}_{38},
\bar{\alpha}^{(2)}_{1},
\bar{\alpha}^{(2)}_{2},
\bar{\alpha}^{(2)}_{39},
\alpha^{(3)}_{1},
\bar{\alpha}^{(3)}_{2},
\bar{\alpha}^{(3)}_{3},
\bar{\alpha}^{(3)}_{4},
\bar{\alpha}^{(3)}_{5},
\bar{\alpha}^{(3)}_{6},
\bar{\alpha}^{(3)}_{7},
\bar{\alpha}^{(3)}_{8},
\bar{\alpha}^{(3)}_{9},
\bar{\alpha}^{(3)}_{11},
\bar{\alpha}^{(3)}_{12},
\bar{\alpha}^{(3)}_{13},
\bar{\alpha}^{(3)}_{14},
\\
\bar{\alpha}^{(3)}_{15},
\bar{\alpha}^{(3)}_{16},
\bar{\alpha}^{(3)}_{18},
\bar{\alpha}^{(3)}_{19},
\bar{\alpha}^{(3)}_{20},
\bar{\alpha}^{(3)}_{21},
\bar{\alpha}^{(3)}_{22},
\bar{\alpha}^{(3)}_{23},
\bar{\alpha}^{(3)}_{24},
\bar{\alpha}^{(3)}_{25},
\bar{\alpha}^{(3)}_{26},
\bar{\alpha}^{(3)}_{27},
\bar{\alpha}^{(3)}_{28},
\bar{\alpha}^{(3)}_{29},
\bar{\alpha}^{(3)}_{30},
\bar{\alpha}^{(3)}_{31},
\bar{\alpha}^{(3)}_{32},
\bar{\alpha}^{(3)}_{33},
\bar{\alpha}^{(3)}_{34},
\bar{\alpha}^{(3)}_{35},
\bar{\alpha}^{(3)}_{36},
\bar{\alpha}^{(3)}_{37},
\bar{\alpha}^{(3)}_{38},
\bar{\alpha}^{(3)}_{39},
\\
\bar{\alpha}^{(3)}_{40},
\bar{\alpha}^{(3)}_{41},
\bar{\alpha}^{(3)}_{42},
\bar{\alpha}^{(3)}_{43},
\bar{\alpha}^{(3)}_{44},
\bar{\alpha}^{(3)}_{46},
\alpha^{(4)}_{1},
\alpha^{(4)}_{29},
\bar{\alpha}^{(4)}_{2},
\bar{\alpha}^{(4)}_{3},
\bar{\alpha}^{(4)}_{7},
\bar{\alpha}^{(4)}_{8},
\bar{\alpha}^{(4)}_{9},
\bar{\alpha}^{(4)}_{11},
\bar{\alpha}^{(4)}_{12},
\bar{\alpha}^{(4)}_{13},
\bar{\alpha}^{(4)}_{14},
\bar{\alpha}^{(4)}_{15},
\bar{\alpha}^{(4)}_{16},
\bar{\alpha}^{(4)}_{18},
\bar{\alpha}^{(4)}_{20},
\bar{\alpha}^{(4)}_{21},
\bar{\alpha}^{(4)}_{22},
\bar{\alpha}^{(4)}_{23},
\\
\bar{\alpha}^{(4)}_{24},
\bar{\alpha}^{(4)}_{25},
\bar{\alpha}^{(4)}_{26},
\bar{\alpha}^{(4)}_{28},
\bar{\alpha}^{(4)}_{29},
\alpha^{(5)}_{1},
\alpha^{(5)}_{15},
\bar{\alpha}^{(5)}_{3},
\bar{\alpha}^{(5)}_{4},
\bar{\alpha}^{(5)}_{5},
\bar{\alpha}^{(5)}_{6},
\bar{\alpha}^{(5)}_{7},
\bar{\alpha}^{(5)}_{9},
\bar{\alpha}^{(5)}_{10},
\bar{\alpha}^{(5)}_{11},
\bar{\alpha}^{(5)}_{12},
\bar{\alpha}^{(5)}_{13},
\bar{\alpha}^{(5)}_{14},
\alpha^{(6)}_{1},
\alpha^{(6)}_{5},
\bar{\alpha}^{(6)}_{1},
\bar{\alpha}^{(6)}_{2},
\bar{\alpha}^{(6)}_{3},
\bar{\alpha}^{(6)}_{4}.
\end{gather*}
\end{small}%

We note that 
we need to perform 
above simplification 
in order
to be able to
solve the obtained system of equation 
within an ``acceptable'' period of time.
It let us reduce 
the number of equations 
from 178 to 170,
the complexity of these equations,
and
the number of variables 
from 354 to 168 
(see \cite[\texttt{e8-equations-full.txt} and \texttt{e8-equations-reduced.txt}]{B:E8zip}
for details).
Unfortunately, the obtained set of equations
is still to large to be presented here
(we note that its size is about one thousand times bigger than
size of substitutions considered in 
\ref{secB:adm}).

\subsection{On solving the obtained set of equations} 
\label{secB:solutions}

It was already mentioned in 
Section~\ref{sec:solv},
that system of equations obtained 
in \ref{secB:simpl}
is to complicated to be solved manually,
and natural approach is to use
some dedicated software package
to solve it.
In our case we
decided to use
Maple~2016 (which was later replaced by Maple~2019).

We may divide the process into following three steps:
\begin{enumerate}
	\setlength{\itemsep}{0pt}
	\item
	solving set of solutions;
	\item
	evaluate (expand) the obtained result;
	\item
	verify if it satisfy the assumptions. 
\end{enumerate}
Invocation of procedures for these steps are 
listed in the Maple Text file \cite[\texttt{commands.txt}]{B:E8zip}.
Below we briefly describe these steps.

We start with defining 
the sequence of substitutions,
set of variables and 
system of equations. 
Then we invoke procedure
\texttt{solve}().
In the result we obtain
sequence of formulas describing
the solutions. 
In particular it contains 
369930 polynomials which roots
are used 
(in the form of the construction ``\texttt{RootOf}()'') 
in formulas 
for values of
the required coefficients ${\alpha}^{(i)}_{j_i}$.
Moreover, 
all denominators from the quotients appearing 
in these formulas belong to the set $\{2,4,8,16,32,64\}$.
We recall that from the general assumption 
the 
computations are performed over
an algebraically closed field 
of characteristic different from $2$.
Hence we know that 
we perform the calculations in right characteristic and
that we can find roots of all these polynomials.

It suffices to find only one solution, 
so in the next step
we invoke the procedure 
\texttt{allvalues}()
on the result from the above step
and take the first element of the obtained sequence
of solutions.
Chosen solution contain
the formulas for all variables.
These formulas contain 369930 square roots 
of expressions build from fixed invariables $\theta_i$.
All denominators from the quotients appearing 
in these formulas again belong to the set $\{2,4,8,16,32,64\}$.
We note that
in order to list the denominators we may use some
simple recursive procedure 
(see \cite[\texttt{denominators.txt}]{B:E8zip} for an example).
We may also use some 
external
tools to parse and analyse the 
formulas saved in ``raw'' form to the text files.
Summing up, we conclude that
this system of equations has a solution
(for arbitrary algebraically closed field $K$
of characteristic different from $2$).
Hence,
following Corollary~\ref{cor:hom2},
there exist an isomorphism
$\varphi : P^f(\mathbb{E}_8) \to P(\mathbb{E}_8)$.
defined on arrows by formulas from
\ref{secB:hom},
with the coefficients 
${\alpha}^{(i)}_{j_i}, \bar{\alpha}^{(i)}_{j_i} \in K$
of the solution of this set of equations.
Hence algebras 
$P(\mathbb{E}_8)$
and
$P^f(\mathbb{E}_8)$
are isomorphic.

\medskip

We end this section with some notes regarding 
complexity
of the problem of solving the obtained 
set of equations.

We note that solving 
the simplified system of equations, 
described in \ref{secB:simpl},
has a relatively low memory and computational
complexity.
Maple 2016 on 
a computer equipped with an Intel Xeon processor E5-1620 v2 (4 cores, 3.70 GHz)
and 128 GB RAM memory
solves it in about 12 hours and these computations 
utilize
about
10.2 GB RAM.
The size of the file with a single solution
saved in a ``raw'' Maple Text format is about 260 MB,
so it is also
``relatively'' small.

On the other hand, 
the process of selecting the right variables 
to be fixed with zero values
has higher memory and time requirements.
Verifying 
(on the above mentioned machine)
if the set of equations with
150--160 coefficients fixed with 0 value
has a solution,
usually takes
4--5 weeks, and utilize about 64GB RAM.
Moreover, in the case of  positive answer,
the data obtained to the analysis are significantly larger
(above 1 GB) and requires external tools
for further analysis
(just parsing the obtained 
term  
(representing the obtained solution)
to be printed in text mode 
in Maple text console could take several days
(about two days in the case of 1.7 GB file)
and showing the solution in graphical interface
would require to increase the limit of the terms
of the output above reasonable value).

We note that 
the choices presented in 
\ref{secB:simpl}  
are expected to not be optimal.
Clearly, it can result in
obtaining a 
more complicated 
solution 
then necessary. 
Solving the set of equations
without any simplification could give us
some hits on choosing right variables, but
it is a very complicated problem
(see note below).
We also note that we fix the coefficient
$\alpha^{(2)}_{1}$ with value $1$
in order to avoid 
considering cases 
and avoid 
denominators different from positive integers
(containing expressions with $\theta_i$, for $i \in \{1,\dots,57\}$).
This construction also seems to be a little artificial
and it is expected to not be necessary for a better
choice of variables.

Finally, we note that
solving the non-simplified set of equations
(without fixing any coefficients)
has even higher requirements.
Such a computation was also 
launched
on the above mentioned computer.
But the computation was interrupted after 2 month 
because of 
rapidly increasing number of reallocated sectors
on the hard disc designated for swap.
It was also started on a more powerful
computer -- with double Intel Xeon processors E5-2650L v2 (20 cores, 1.90 GHz).
Unfortunately, 
the calculations were accidentally terminated 
after 5 months
(the Maple process has allocated 420 GB RAM in that moment)
due to the license server upgrade.
These calculation were started again and are in progress
(for about half a year)
at the moment of preparing this article.
\end{appendix}

\section*{Conclusion}

This article presents an approach to the problem of
the existence of non-trivial deformed preprojective 
algebras of Dynkin types $\EE_6$, $\EE_7$, $\EE_8$, 
in which the computer calculations play an essential role.
Its application leads to the results 
on periodicity and classification of the considered algebras. 
In particular, we confirm Conjecture~\ref{conj:1.6}
in the case of Dynkin types $\EE_6$, $\EE_7$, $\EE_8$.
Among others it asserts that each of the deformed preprojective 
algebras of Dynkin types $\EE_6$, $\EE_7$, $\EE_8$ is  periodic.

Unfortunately, the technique we apply in Sections 3-8 is of  very high computational complexity.
However, it is not clear  for us if the conjecture can be proved 
without   computer computation. 
Note, that
in   our algorithms described  in Sections 3-8, there are several steps  where 
we could make some choices that have  a significant impact 
on the overall computational complexity. 

As we mentioned in the introduction, 
we strongly believe that our  technique  presented in the paper 
can also be useful in  solving  representation theory problems of a similar type.

\section*{Acknowledgements}

I would like to 
express a deep gratitude to
Professor Andrzej Skowro\'nski,
who was the supervisor of my PhD dissertation in 2003, 
for the inspiration and encouragement 
to deal with the subject of this article, 
as well as 
for many years of his support 
and stimulative cooperation.

\end{document}